\documentclass[12pt,oneside,english]{amsart}
\usepackage[T1]{fontenc}
\usepackage[latin9]{inputenc}
\usepackage[a4paper]{geometry}
\usepackage{babel}
\usepackage{amsthm}
\usepackage{amssymb}
\usepackage{graphicx}
\usepackage{esint}
\usepackage[unicode=true,pdfusetitle,
 bookmarks=true,bookmarksnumbered=false,bookmarksopen=false,
 breaklinks=false,pdfborder={0 0 1},backref=false,colorlinks=false]
 {hyperref}

\makeatletter
\numberwithin{equation}{section}
\numberwithin{figure}{section}
\theoremstyle{plain}
\newtheorem{thm}{\protect\theoremname}
  \theoremstyle{remark}
  \newtheorem{rem}[thm]{\protect\remarkname}
 \theoremstyle{definition}
 \newtheorem*{defn*}{\protect\definitionname}
  \theoremstyle{definition}
  \newtheorem*{problem*}{\protect\problemname}
  \theoremstyle{definition}
  \newtheorem{defn}[thm]{\protect\definitionname}
  \theoremstyle{plain}
  \newtheorem{lem}[thm]{\protect\lemmaname}
  \theoremstyle{plain}
  \newtheorem{prop}[thm]{\protect\propositionname}
  \theoremstyle{plain}
  \newtheorem*{prop*}{\protect\propositionname}
  \theoremstyle{plain}
  \newtheorem*{thm*}{\protect\theoremname}
  \theoremstyle{plain}
  \newtheorem*{lem*}{\protect\lemmaname}

\makeatletter
\renewcommand\section{\@startsection{section}{1}{\z@}%
{-3.5ex \@plus -1ex \@minus -.2ex}%
{2.3ex \@plus.2ex}
{\normalfont\large\bfseries}}

\makeatother

  \providecommand{\definitionname}{Definition}
  \providecommand{\lemmaname}{Lemma}
  \providecommand{\problemname}{Problem}
  \providecommand{\propositionname}{Proposition}
  \providecommand{\remarkname}{Remark}
  \providecommand{\theoremname}{Theorem}
\providecommand{\theoremname}{Theorem}

\begin{document}

\title{Ising Interfaces and Free Boundary Conditions}

\author{Clément Hongler and Kalle Kytölä}
\begin{abstract}
We study the interfaces arising in the two-dimensional Ising model
at critical temperature, without magnetic field. We show that in the
presence of free boundary conditions between plus and minus spins,
the scaling limit of these interfaces can be described by a variant
of SLE, called dipolar SLE$\left(3\right)$. This generalizes a celebrated
result of Chelkak and Smirnov \cite{chelkak-smirnov-ii,chelkak-smirnov-iii}
and proves a conjecture of Bauer, Bernard and Houdayer \cite{bauer-bernard-houdayer}.
We mention two possible applications of our result.
\end{abstract}

\address{Department of Mathematics, Columbia University. 2990 Broadway, New
York, NY 10027, USA.}

\email{hongler@math.columbia.edu}

\address{Department of Mathematics and Statistics, P. O. Box 68, FIN--00014
University of Helsinki, Finland}

\email{kalle.kytola@helsinki.fi}

\maketitle
\tableofcontents{}

\section{Ising model and conformal invariance}

The Ising model is one of the most investigated models for order-disorder
phase transitions: having a simple formulation, it exhibits a rich
and interesting behavior. In two dimensions, the model can be understood
at a high level of precision from both mathematical and physical viewpoints,
using a variety of techniques. 

Recall that the Ising model on a graph $\mathcal{G}$ is defined by
a Gibbs probability measure on configurations of $\pm1$ (or \emph{up/down})
spins located on the vertices of $\mathcal{G}$: it is a random assignment
$\left(\sigma_{x}\right)_{x\in\mathcal{V}}$ of $\pm1$ spins to the
vertices $\mathcal{V}$ of $\mathcal{G}$ and the probability of a
state is proportional to its Boltzmann weight $e^{-\beta\mathbf{H}}$,
where $\beta>0$ is the inverse temperature of the model and $\mathbf{H}$
is the Hamiltonian, or energy, of the state $\sigma$. In the Ising
model with no external magnetic field, we have $\mathbf{H}:=-\sum_{i\sim j}\sigma_{i}\sigma_{j}$,
where the sum is over all the pairs of adjacent vertices of $\mathcal{G}$.

The model favors lower energy configurations, and hence local alignment
of spins: if two adjacent spins are aligned, their contribution to
the energy is smaller than if they are different. The strength of
this alignment effect is modulated by $\beta$. The central question
is whether this local interaction gives rise to a long-range order
as we take a large graph, or whether its effects remain confined to
a small scale. This question turns out to depend crucially on $\beta$:
in dimension greater or equal to $2$, there exists $\beta_{c}>0$
such that for $\beta<\beta_{c}$, the system is basically disordered
(except on small scales), while for $\beta>\beta_{c}$, a long-range
ferromagnetic order arises (spins retain a positive correlation at
arbitrarily large distance). 

The Ising model was introduced by Lenz in 1920 \cite{lenz} and its
one-dimensional version was studied by Ising \cite{ising}. In 1936,
Peierls showed the existence of a phase transition in the Ising model
in dimension two and higher by looking at the interfaces between clusters
(connected components) of up and down spins \cite{peierls}. In 1941,
Kramers and Wannier determined the value of the critical temperature
on $\mathbb{Z}^{2}$, thanks to a remarkable duality result, which
also deals with interfaces and is now named after them \cite{kramers-wannier}.
In 1944, Onsager computed exactly the free energy of the model at
arbitrary temperatures, thus allowing for a derivation of the thermodynamic
properties of the model \cite{onsager}. Since then, the two-dimensional
Ising model has attracted a lot of attention, and great progress has
been made, making it possible to understand the model at a rather
unique level of precision \cite{baxter,mccoy-wu-i,palmer}.

Arguably, the most intriguing and physically relevant phase of the
Ising model is the critical phase and its vicinity. The advent of
the Renormalization Group in the 1960s (see \cite{fisher} for a historical
exposition) yielded a deep physical understanding (though non-rigorous)
of this regime and suggested the existence of a scaling limit of the
model, a universal object with continuous symmetries.

The idea that the critical scaling limits in two dimensions are conformally
invariant, together with the introduction of an operator algebra for
the Ising model \cite{kadanoff-ceva}, suggested the description of
the Ising model by Conformal Field Theory (CFT), a theory initiated
by Belavin, Polyakov and Zamolodchikov \cite{belavin-polyakov-zamolodchikov-i,belavin-polyakov-zamolodchikov-ii}:
there should be a quantum field theory underlying the critical scaling
limit, invariant by conformal transformations.

One of the most spectacular results of CFT is the prediction of exact
formulae for the correlation functions of various models, in particular
the Ising model. The development of boundary CFT, initiated by Cardy
\cite{cardy-i}, subsequentially allowed to understand in a precise
way the effect of the geometry of the surface on which the model lives,
and the effect of various boundary conditions. One of the most emblematic
successes of boundary CFT was Cardy's crossing probability formula
for percolation in a conformal rectangle \cite{cardy-iii}, whose
numerical verification gave one of the most convincing evidence of
the full conformal invariance of that model, that is, the conformal
invariance by the infinite-dimensional family of the conformal mappings.

The introduction of Schramm's SLE curves \cite{schramm-i} in 1999
was the starting point of the development of the mathematical subject
of conformal invariant processes. A precise sense of conformal invariance
of statistical mechanics models was given, in terms of the (scaling
limit of the) curves arising in the models. Shortly thereafter, the
conformal invariance of the scaling limit of critical percolation
on the triangular lattice was proven by Smirnov \cite{smirnov-iv},
and similar results were derived for a number of other models \cite{kenyon,lawler-schramm-werner-ii,miller,schramm-sheffield}.
More recently, major progress has been realized for the Ising model
and its random-cluster representation (also known as FK representation),
where the interfaces arising with so-called Dobrushin boundary conditions
at criticality have been shown to be conformally invariant in the
scaling limit by Chelkak and Smirnov \cite{smirnov-i,smirnov-ii,smirnov-iii,chelkak-smirnov-ii,chelkak-smirnov-iii}.

While being definite breakthroughs, these results do not answer directly
all questions about the conformal invariance of the Ising model. They
show conformal invariance of scaling limits of the interfaces arising
in a particular setup. From these results, much information can be
inferred, and other scaling limit results for other types of interfaces
can be obtained: for instance the convergence of all the interfaces
arising with certain boundary conditions can then be expected and
in principle proved, as was done for percolation \cite{camia-newman-ii,smirnov-v}.
However, proving such results is in general highly non-trivial. Moreover,
there is one type of boundary conditions, conjectured to be conformally
invariant, which is not directly tractable from the existing results:
the free boundary conditions, which do not appear in the setup of
the result of Chelkak and Smirnov.

In this paper, we generalize the result of Chelkak and Smirnov to
the case when free boundary conditions enter the picture. To prove
our result, we relate it to the rigorous computation of a (dual) boundary
CFT correlation function, which is obtained by using both recent results
concerning the boundary correlation functions of the model and existing
SLE results (for dual models). Our result relies mostly on the following
recent results:
\begin{itemize}
\item The convergence of critical FK-Ising interfaces to SLE$\left(16/3\right)$
\cite{smirnov-i}.
\item The scaling limits of Ising and FK-Ising fermionic observables \cite{chelkak-smirnov-ii,hongler-i,hongler-smirnov-ii}.
\item The precompactness and Löwner regularity of interfaces satisfying
crossing estimates \cite{kemppainen-smirnov}. 
\item Crossing estimates for critical Ising and FK-Ising models \cite{chelkak-smirnov-ii,duminil-copin-hongler-nolin}.
\end{itemize}
A first promising application of our theorem is the conformal invariance
of crossing probabilities investigated by Langlands, Lewis and Saint-Aubin
\cite{langlands-lewis-staubin}: we can represent the crossing events
that they consider in terms of an exploration process, whose conformally
invariant scaling limit can be identified using our result. A second
potential application is the proof that the collection of the Ising
model interfaces converges to the Conformal Loop Ensemble (CLE) introduced
by Sheffield \cite{sheffield-ii}. This also suggests the introduction
of a new object to describe the collection of interfaces with free
boundary conditions.

\specialsection*{Acknowledgements}

The authors would like to thank Dmitry Chelkak and Stanislav Smirnov
for many enlightening conversations and useful advice, as well as
Konstantin Izyurov, Wendelin Werner, Hugo Duminil-Copin, Vincent Beffara,
Yvan Velenik, Pierre Nolin, Antti Kemppainen, Scott Sheffield, Julien
Dubédat and Fredrik Johansson Viklund for interesting discussions.

This research was partially supported by the Swiss NSF, the European
Research Council AG CONFRA, the Academy of Finland and by the National
Science Foundation under grant DMS-1106588.

\section{Main result}

\subsection{Statement of the main theorem}

The most natural setup to study the Ising model interfaces consists
in the \emph{Dobrushin boundary conditions}: take a suitable discretization
of a simply connected domain, split the boundary into two connected
pieces, and consider the Ising model at critical temperature on this
discretization, conditioning the spins on one piece to be $+1$ and
the ones on the other piece to be $-1$. An interface naturally arises
between the $+$ and $-$ spin clusters of the two pieces of the boundary
(see Figure \ref{fig:dobrushin-dipolar}); for a more precise definition,
see Sections \ref{sub:graph-domain}, \ref{sub:ising-model} \ref{sub:interface}
below. 

The conformal invariance of the scaling limit of the interfaces appearing
in the critical Ising model on the square lattice (as well as on more
general graphs) with these boundary conditions was recently shown
by Chelkak and Smirnov. At subcritical temperature ($\beta>\beta_{c}$),
these interfaces were shown by Pfister and Velenik to converge to
a straight line \cite{pfister-velenik}.

Our result is the proof of a conjecture of Bauer, Bernard and Houdayer
\cite{bauer-bernard-houdayer}. The result deals with what appears
to be the most natural setup involving free boundary conditions, expected
to be the third type (in addition to $+$ and $-$) of conformally
invariant boundary conditions \cite{di-francesco-mathieu-senechal}. 
\begin{thm}
\label{thm:main-thm}Let $\left(D_{\delta},r_{\delta},\ell_{\delta},b_{\delta}\right)_{\delta>0}$
be a family of (simply connected) discrete square grid domains of
mesh size $\delta$ with three boundary marked points approximating
a continuous domain $\left(D,r,\ell,b\right)$ as $\delta\to0$. 

Consider the Ising model at critical temperature on the faces of $\left(D_{\delta},r_{\delta},\ell_{\delta},b_{\delta}\right)$
with free boundary condition on the counterclockwise arc $\left[r_{\delta},\ell_{\delta}\right]$,
$-$ boundary condition on $\left[\ell_{\delta},b_{\delta}\right]$
and $+$ boundary condition on $\left[b_{\delta},r_{\delta}\right]$
(see Figure \ref{fig:dobrushin-dipolar}).

Then, as $\delta\to0$, the law of the initial segments of the interface
$\gamma_{\delta}$ emanating at $b_{\delta}$, that separates the
$-$ spin cluster of $\left[\ell_{\delta},b_{\delta}\right]$ and
the $+$ spin cluster of $\left[b_{\delta},r_{\delta}\right]$ and
ends on $\left[r_{\delta},\ell_{\delta}\right]$, converges to the
law of dipolar SLE$\left(3\right)$ in $\left(D,r,\ell,b\right)$. 

The convergence is locally uniform with respect to the domains.
\end{thm}
The discrete domains are defined in Section \ref{sub:graph-domain},
the Ising model with boundary conditions in Section \ref{sub:ising-model},
the interface $\gamma_{\delta}$ in Section \ref{sub:interface},
dipolar SLE in Section \ref{sub:dipolar-sle}. The notions of convergence
and uniformity involved are briefly discussed in Section \ref{sub:uniformity-convergence}. 
\begin{rem}
With the recently announced results in \cite{chelkak,chelkak-duminil-copin-hongler},
one can consider the scaling limit of the whole discrete interface
$\gamma_{\delta}$ and not just of its initial segments (see Section
\ref{sub:precompactness}).
\end{rem}

\begin{rem}
In this article we only consider square grid domains for simplicity,
although our result can be generalized to other lattices as well,
using techniques introduced in \cite{chelkak-smirnov-i,chelkak-smirnov-ii}.
\end{rem}
\begin{figure}

\includegraphics[width=7cm]{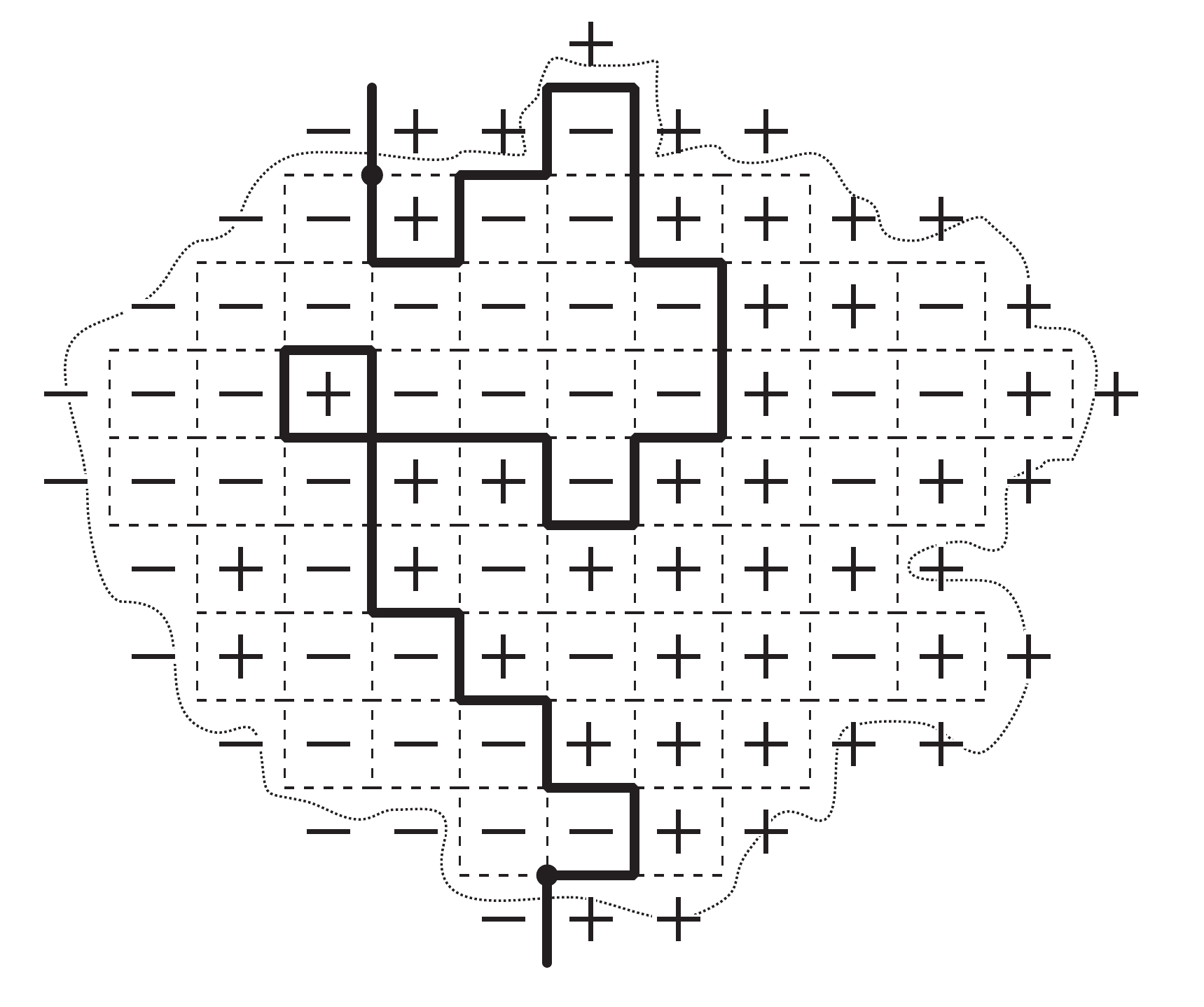}\includegraphics[width=7cm]{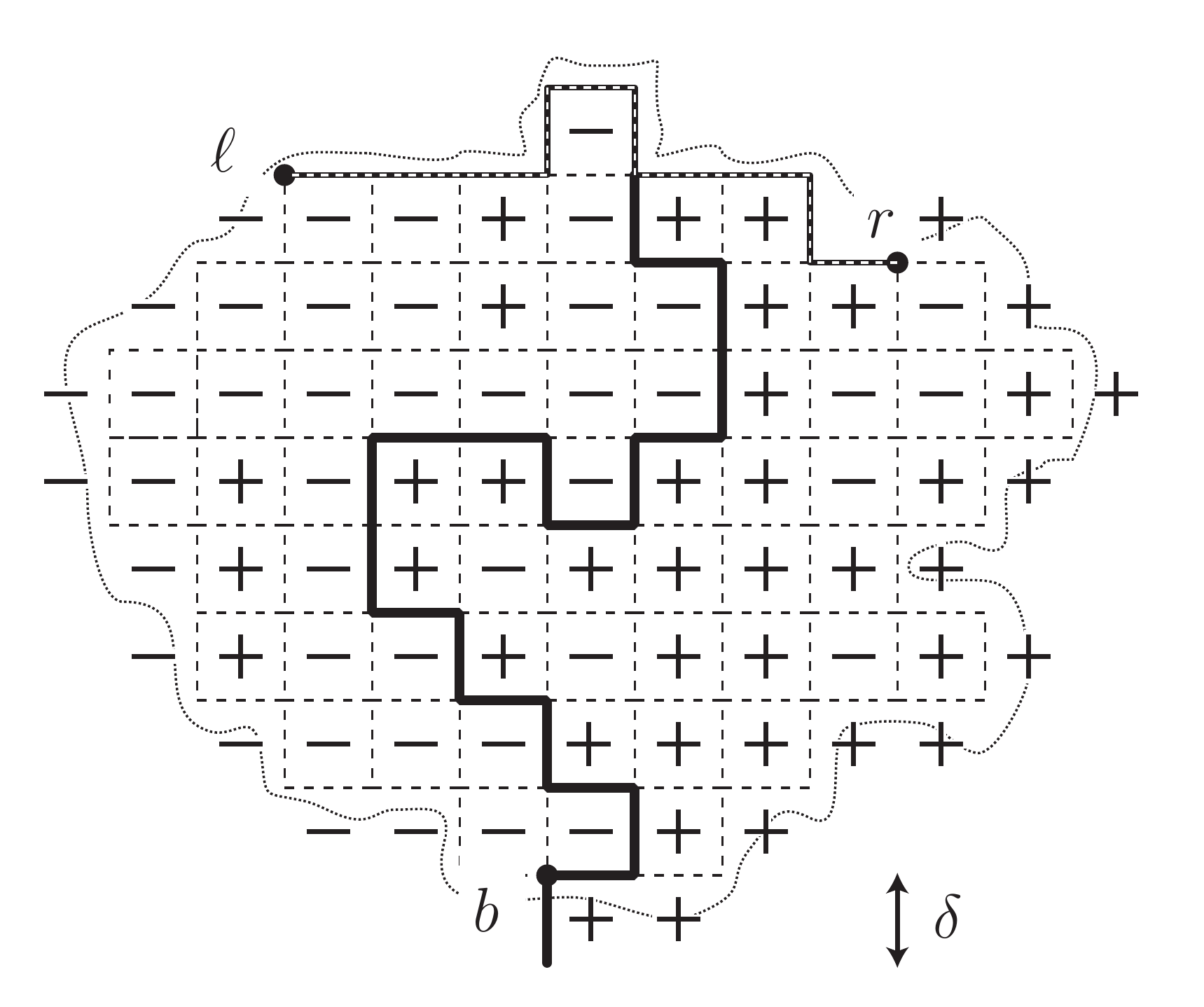}

\caption{\label{fig:dobrushin-dipolar}On the left: Ising interface with Dobrushin
boundary conditions. On the right: Ising interface with free boundary
condition on $\left[r,\ell\right]$, $-$ boundary condition on $\left[\ell,b\right]$
and $+$ boundary condition on $\left[b,r\right]$.}
\end{figure}

\subsection{Graph and domain\label{sub:graph-domain}}

Let us now give the notation that will be used throughout this paper:
\begin{itemize}
\item For $\delta>0$, we denote by $\mathbb{C}_{\delta}:=\delta\mathbb{Z}^{2}$
the square grid of mesh size $\delta$.
\item A discrete square grid domain $\Omega_{\delta}$ is a simply connected
graph made of the union of faces of $\mathbb{C}_{\delta}$; its boundary
$\partial\Omega_{\delta}$ is a simple closed curve made of edges
of $\mathbb{C}_{\delta}$; when necessary we will identify $\Omega_{\delta}$
with the Jordan domain of $\mathbb{C}$ bounded by $\partial\Omega_{\delta}$.
\item For any two vertices $x,y\in\partial\Omega_{\delta}$, we denote by
$\left[x,y\right]\subset\partial\Omega_{\delta}$ the counterclockwise
arc between $x$ and $y$. 
\item When needed we will identify each edge of $\partial\Omega_{\delta}$
with the face of $\mathbb{C}_{\delta}\setminus\Omega_{\delta}$ that
is adjacent to it. 
\item We denote by $\left(\Omega_{\delta},a_{\delta}^{1},\ldots,a_{\delta}^{k}\right)$
a discrete domain $\Omega_{\delta}$ with $k$ marked vertices $a_{\delta}^{1},\ldots,a_{\delta}^{k}\in\partial\Omega_{\delta}$
appearing in counterclockwise order. 
\item We call an arc $\left[x_{\delta},y_{\delta}\right]\subset\partial\Omega_{\delta}$
whose edges are all vertical a \emph{vertical arc.}
\end{itemize}
We will omit a number of $\delta$ subscripts when they will be clear
from the context, in particular when we will be discussing purely
discrete statements.

\subsection{Ising model\label{sub:ising-model}}

For concreteness and simplicity we only define here the Ising model
in the setup needed for our result, that is, the critical Ising model
on the faces $\mathcal{F}$ of a discrete domain $D_{\delta}$ with
free boundary condition on $\left[r,\ell\right]$, $-$ boundary condition
on $\left[\ell,b\right]$ and $+$ boundary condition on $\left[b,r\right]$.
We call these boundary conditions \emph{dipolar boundary conditions
on $\left(D_{\delta},r,\ell,b\right)$}. The probability space is
\[
\mathcal{S}:=\left\{ \left(\sigma_{f}\right)_{f\in\mathcal{F}}:\,\,\sigma_{f}\in\left\{ \pm1\right\} \,\,\forall f\in\mathcal{F}\right\} .
\]
The probability of a spin configuration $\sigma\in\mathcal{S}$ is
given by $\mathbb{P}\left\{ \sigma\right\} :=\frac{1}{\mathcal{Z}}e^{-\beta\mathbf{H}\left(\sigma\right)}$,
where 
\begin{itemize}
\item the inverse temperature $\beta$ is equal to its critical value $\frac{1}{2}\ln\left(\sqrt{2}+1\right)$;
\item the energy $\mathbf{H}\left(\sigma\right)$ is defined by
\begin{equation}
\mathbf{H}\left(\sigma\right):=-\left(\sum_{f\sim g}\sigma_{f}\sigma_{g}+\sum_{f\sim\left[b,r\right]}\sigma_{f}-\sum_{f\sim\left[\ell,b\right]}\sigma_{f}\right),\label{eq:hamiltonian-ising}
\end{equation}
where the first sum is over all pairs of adjacent faces in $\mathcal{F}$,
the second and third ones are respectively over all faces adjacent
to an edge of $\left[b,r\right]$ and $\left[\ell,b\right]$ (a face
appears several times in the sum if it is adjacent to several such
edges);
\item the partition function $\mathcal{Z}$ is defined as $\sum_{\sigma\in\mathcal{S}}\exp\left(-\beta\mathbf{H}\left(\sigma\right)\right)$. 
\end{itemize}
Notice that the boundary conditions only appear in the Hamiltonian.
Another way of formulating the boundary conditions is to say that
there is a $+1$ spin at the faces identified with $\left[b,r\right]$,
that there is a $-1$ spin on those identified with $\left[\ell,b\right]$
and that there are no spins on the faces identified with $\left[r,\ell\right]$.

\subsection{Interface\label{sub:interface}}

The boundary conditions of the Ising model in $\left(D_{\delta},r,\ell,b\right)$
defined in the previous subsection (dipolar boundary conditions) naturally
generate an interface between the $-$ cluster of the arc $\left[\ell,b\right]$
and the $+$ cluster of the arc $\left[b,r\right]$. For any configuration
$\sigma\in\mathcal{S}$ (where $\mathcal{S}$ is as in Section \ref{sub:ising-model}),
we can find a path $\gamma_{\delta}$ made of edges of $D_{\delta}$,
that starts at $b$ and ends on $\left[r,\ell\right]$, and such that
$\gamma_{\delta}$ has only faces with $-$ spins on its left (possibly
including the faces identified with $\left[\ell,b\right]$) and faces
with $+$ spins on its right (possibly including those identified
with $\left[b,r\right]$), as shown on Figure \ref{fig:dobrushin-dipolar}.
We call such a path an \emph{admissible interface}. 

As the square grid is not a trivalent graph, there might be different
admissible choices of the interfaces, yielding ambiguities in the
definition of the interface $\gamma_{\delta}$. These ambiguities
turn out to be irrelevant in the scaling limit, but for definiteness,
we will make the following convention.
\begin{defn*}
We define the interface $\gamma_{\delta}$ to be the left-most admissible
interface.
\end{defn*}
Exactly the same arguments as the ones we use in this paper give that
the right-most admissible interface converges to the same limit as
the left-most one, and hence all admissible choices also converge
to the same limit. 

For technical reasons, we will consider \emph{initial segments} of
the interface, that is, the interface stopped as it hits an $\epsilon$-neighborhood
of $\left[r,\ell\right]$, for an $\epsilon>0$ fixed. The scaling
limit of the initial segments hence means: we let the mesh size $\delta\to0$
with $\epsilon>0$ fixed and after that let $\epsilon\to0$.

\subsection{Dipolar SLE and Loewner chains in the strip\label{sub:dipolar-sle}}

Schramm-Loewner Evolutions \cite{schramm-i} are the natural candidates
for the conformally invariant scaling limits of discrete curves in
two dimensions, as shown by Schramm's principle (see also \cite{kemppainen}
for an extension of this principle relevant for our setup). See \cite{lawler-ii}
for a reference about SLE processes.

We now define the variant of SLE suited for our purposes, which is
called \emph{dipolar SLE$\left(\kappa\right)$} (see \cite{bauer-bernard-houdayer}).
It can be viewed as a particular case of the more general SLE$\left(\kappa;\rho\right)$
processes \cite{werner-ii,schramm-wilson}, which will be introduced
in Section \ref{sec:sle-variants}. 

Dipolar SLE$\left(\kappa\right)$ has been shown to be the scaling
limit of the loop-erased random walk from a point to an arc (when
$\kappa=2$) \cite{zhan} and of discrete Gaussian free field level
lines with certain symmetric boundary conditions (when $\kappa=4$)
\cite{schramm-sheffield}.

\subsubsection{Dipolar SLE$\left(\kappa\right)$\label{sub:dipolar-sle-subsub}}

For $\kappa\geq0$, dipolar SLE$\left(\kappa\right)$ is naturally
defined on the strip $\mathbb{S}:=\left\{ z\in\mathbb{C}:0<\Im\mathfrak{m}\left(z\right)<\pi\right\} $,
as a Loewner chain (see Figure \ref{fig:strip-loewner-chain}). 

A Loewner chain in the strip is defined by the following flow equation
\begin{eqnarray*}
\partial_{t}g_{t}\left(z\right) & = & \coth\left(\frac{g_{t}\left(z\right)-U_{t}}{2}\right)\\
g_{0}\left(z\right) & = & z
\end{eqnarray*}
where $\left(U_{t}\right)_{t\geq0}$ is a continuous real-valued function,
called the \emph{driving function}. Consider the Loewner chain obtained
by taking as driving function $\left(\sqrt{\kappa}B_{t}\right)_{t\geq0}$,
where $B_{t}$ is a standard one-dimensional Brownian motion. We call
this chain the \emph{dipolar SLE$\left(\kappa\right)$ Loewner chain}.
For each $t\geq0$, let $S_{t}\subset\mathbb{S}$ be the set of points
for which the flow is well-defined up to time $t$. The following
properties are valid at all times $t\geq0$:
\begin{itemize}
\item $g_{t}:S_{t}\to\mathbb{S}$ is a conformal mapping, with $\lim_{z\to+\infty}g_{t}\left(z\right)-z-t=0$
and $\lim_{z\to-\infty}g_{t}\left(z\right)-z+t=0$. 
\item $S_{t}$ is the unbounded connected component of $\mathbb{S}\setminus\gamma\left[0,t\right]$,
where $\gamma\subset\overline{\mathbb{S}}$ is a curve, called the
\emph{trace}, which is such that $g_{t}\left(\gamma\left(t\right)\right)=U_{t}$. 
\item $\gamma\left(0\right)=0$ and $\gamma\left(t\right)$ tends to a point
on the upper side of $\mathbb{S}$ as $t\to\infty$. 
\end{itemize}
Dipolar SLE$\left(\kappa\right)$ in the strip $\mathbb{S}$ is the
trace $\gamma$, considered as an (oriented) unparametrized curve. 

In a domain $\left(D,r,\ell,b\right)$, dipolar SLE$\left(\kappa\right)$
is defined as the image of dipolar SLE$\left(\kappa\right)$ by the
conformal mapping $\varphi:\mathbb{S}\to D$, with $\varphi\left(0\right)=b$,
$\varphi\left(\infty\right)=r$, $\varphi\left(-\infty\right)=\ell$.
In the case we are interested in (i.e. $\kappa=3$), dipolar SLE$\left(\kappa\right)$
is almost surely a simple curve -- this is true for all $\kappa\in\left[0,4\right]$
(see \cite{lawler-ii} for a proof in the case of chordal SLE$\left(\kappa\right)$
-- chordal and strip SLE$\left(\kappa\right)$ are absolutely continuous
with respect to each other\cite{schramm-wilson}).

\subsubsection{Loewner chain in the strip\label{sub:dipolar-loewner-chain}}

As explained above, given a real-valued continuous function $\left(U_{t}\right)_{t\geq0}$,
we can generate a Loewner chain in $\mathbb{S}$ and hence a family
of shrinking subdomains $\left(S_{t}\right)_{t\geq0}$ of $\mathbb{S}$,
with $S_{t}\subset S_{s}$ for any $t\geq s$ and $S_{0}=\mathbb{S}$.
Conversely, it can be shown (see \cite{lawler-ii}) that any such
family of subdomains $\left(S_{t}\right)_{t\geq0}$ satisfying a certain
local growth property can be realized (after time reparametrization)
as a Loewner chain in the strip, guided by a continuous driving function
$\left(V_{t}\right)_{t\geq0}$.

\begin{figure}
\includegraphics[width=10cm]{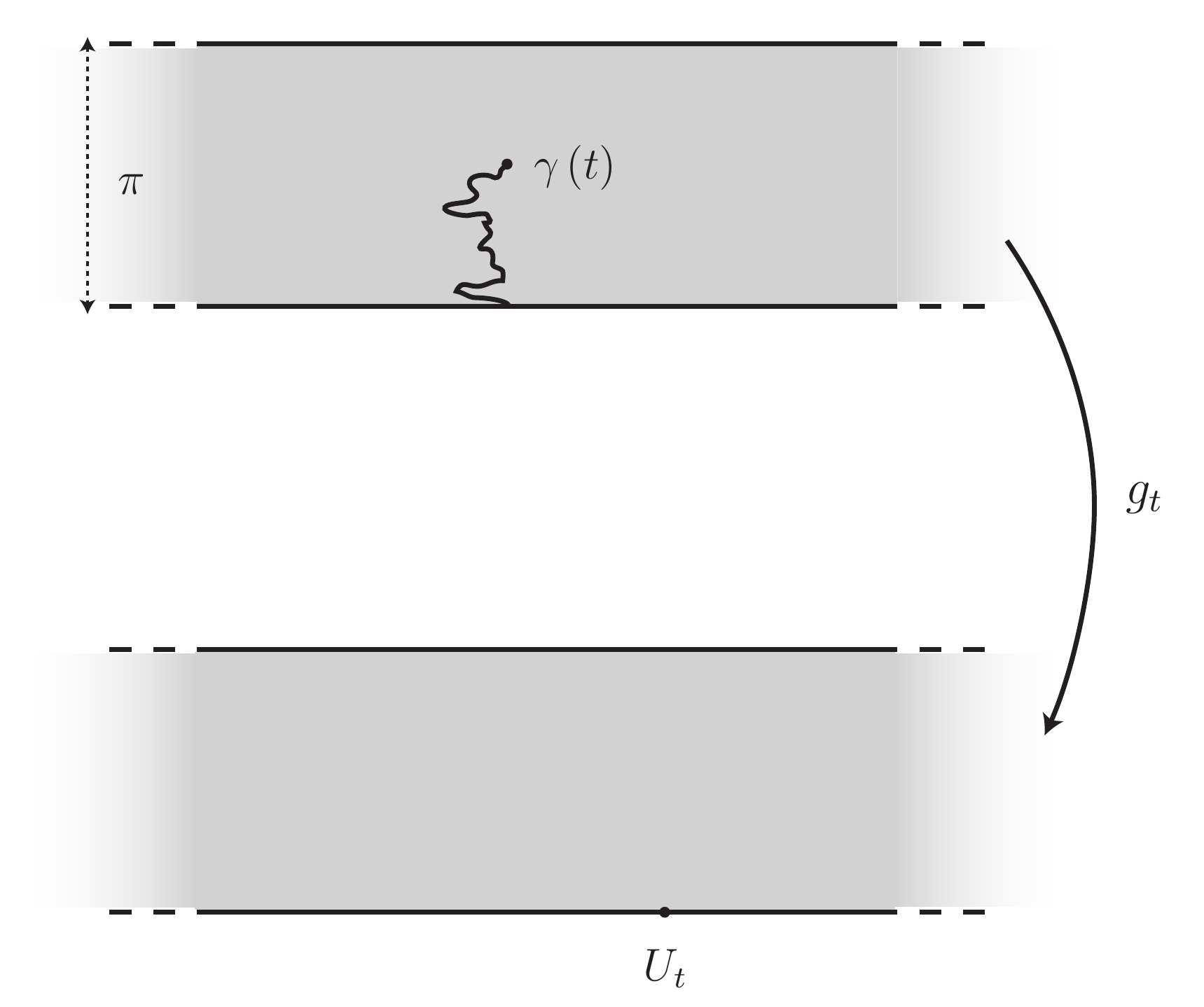}

\caption{\label{fig:strip-loewner-chain}Loewner chain in the strip}

\end{figure}

\subsection{Convergence and uniformity\label{sub:uniformity-convergence}}

As for most SLE convergence results, there are actually several types
of convergence results that can be obtained with our techniques: the
strength of the result we get depends on how well (in which topology)
the discrete domains $\left(D_{\delta},r_{\delta},\ell_{\delta},b_{\delta}\right)$
approximate the continuous domain $\left(D,r,\ell,b\right)$. For
definiteness and simplicity, we will use a rather strong topology,
which is best suited for applications. 

For two oriented simple curves $\gamma_{1}$, $\gamma_{2}$ in the
complex plane, we define $\mathrm{d}_{\infty}\left(\gamma_{1},\gamma_{2}\right)$
by 
\[
\mathrm{d}_{\infty}\left(\gamma_{1},\gamma_{2}\right):=\inf_{\zeta_{1},\zeta_{2}}\|\zeta_{1}-\zeta_{2}\|_{\infty},
\]
where the infimum is taken over all orientation-preserving parametrizations
$\zeta_{1}$ and $\zeta_{2}$ of $\gamma_{1}$ and $\gamma_{2}$ respectively.
Let $\mathsf{C}$ be the completion of the set of simple curves for
this metric. For two domains $\left(D_{1},a_{1}^{1},\ldots,a_{n}^{1}\right)$
and $\left(D_{2},a_{2}^{1},\ldots,a_{n}^{2}\right)$ with $n$ marked
boundary points such that $\partial D_{1},\partial D_{2}\in\mathsf{C}$,
we define 
\[
\mathrm{d}_{\infty}\left[\left(D_{1},a_{1}^{1},\ldots,a_{n}^{1}\right),\left(D_{2},a_{1}^{1},\ldots,a_{n}^{2}\right)\right]=\mathrm{d}_{\infty}\left(\partial D_{1},\partial D_{2}\right)+\sum_{i=1}^{n}\left|a_{i}^{1}-a_{i}^{2}\right|.
\]
 We can now define the type of convergence we will work with:
\begin{itemize}
\item We say that $\left(D_{\delta},r_{\delta},\ell_{\delta},b_{\delta}\right)\to\left(D,r,\ell,b\right)$
if $\mathrm{d}_{\infty}\left[\left(D_{\delta},r_{\delta},\ell_{\delta},b_{\delta}\right),\left(D,r,\ell,b\right)\right]\to0$
as $\delta\to0$.
\item We say that the interface $\gamma_{\delta}$ converges in law to the
dipolar SLE trace $\gamma$ as $\delta\to0$ if for any $\epsilon>0$,
there exists $\delta_{0}>0$ such that for any $\delta\leq\delta_{0}$,
there exists a coupling of $\gamma_{\delta}$ and $\gamma$ such that
$\mathbb{P}\left\{ \mathrm{d}_{\infty}\left(\gamma_{\delta},\gamma\right)>\epsilon\right\} \leq\epsilon$.
This is equivalent to saying that $\gamma_{\delta}$ converges weakly
to $\gamma$. 
\end{itemize}
What we mean by locally uniform convergence in Theorem \ref{thm:main-thm}
is: for any $R>0$ and any $\epsilon>0$, there exists $\delta_{0},\epsilon_{0}>0$
such that for any $\delta\leq\delta_{0}$, for any discrete domain
$\left(D_{\delta},r_{\delta},\ell_{\delta},b_{\delta}\right)$ of
diameter smaller than $R$, such that
\[
\mathrm{d}_{\infty}\left[\left(D_{\delta},r_{\delta},\ell_{\delta},b_{\delta}\right),\left(D,r,\ell,b\right)\right]\leq\epsilon_{0},
\]
we have that there exists a coupling of the interface $\gamma_{\delta}$
in $\left(D_{\delta},r_{\delta},\ell_{\delta},b_{\delta}\right)$
and the SLE $\gamma$ in $\left(D,r,\ell,b\right)$ such that: 
\[
\mathbb{P}\left\{ \mathrm{d}_{\infty}\left(\gamma_{\delta},\gamma\right)>\epsilon\right\} \leq\epsilon.
\]

\subsection{Interesting features of the proof}

Although our proof follows a classical strategy for proving convergence
results to SLE, it involves a number of ideas that are new in the
subject. In particular, we find the following features worth pointing
out:
\begin{itemize}
\item Our new martingale observable is not a discrete holomorphic or discrete
harmonic function, for it does not satisfy local relations. Instead,
it is merely defined on the boundary of the domain where we are considering
it. For that reason, it requires more than discrete complex analysis
to show the convergence of the observable to a conformally invariant
limit.
\item To understand the scaling limit of the Ising model, one introduces
and studies the scaling limit of a dual Ising model.
\item One uses SLE$\left(16/3\right)$ to obtain a convergence result to
SLE$\left(3\right)$: scaling limits of correlation functions of the
dual Ising model can be expressed as SLE$\left(16/3\right)$ integrals
that can then be computed using Itô's calculus.
\item The proof illustrates the usefulness of obtaining exact results for
quantities like the spin correlations to derive a qualitative result,
the conformal symmetry of certain Ising interfaces.
\item The non-universal (lattice-dependent) multiplicative constants appearing
in the exact formulae for the correlation functions that we compute
turn out to be useful to show the convergence to a universal limit.
\item The proof demonstrates the possibility to use local Riemann charts
together with discrete complex analysis to understand boundary correlation
functions for the Ising model on rough domains.
\end{itemize}

\subsection{Structure of the paper}

In Section \ref{sec:possible-applications}, we give two possible
applications of our result, to the computation of crossing probabilities
and to the convergence of the Ising interfaces to Conformal Loop Ensembles. 

The rest of this paper is then devoted to the proof of Theorem \ref{thm:main-thm}.
The global strategy is the following:
\begin{itemize}
\item In Section \ref{sec:proof-main-result}, Theorem \ref{thm:main-thm}
is reduced to a key theorem (Theorem \ref{thm:martingale-observable}),
which is the existence of a so-called continuous martingale observable
available in the scaling limit, following a path that has now become
standard in the SLE subject.
\item In Section \ref{sec:proof-key-thm}, one constructs a discrete martingale
observable for the interface (Proposition \ref{pro:discrete-martingale},
proven in Section \ref{sec:disc-mart-property}). The heart of the
matter to prove the key Theorem \ref{thm:martingale-observable} is
to show that the discrete martingale observable converges to the continuous
one. This convergence result is decomposed into four ingredients (Propositions
\ref{pro:representation-obs-as-correlation}, \ref{pro:int-rep-corr},
\ref{pro:conv-discrete-expectations} and \ref{pro:sle-averages}),
which are proven in Sections \ref{sec:kw-duality}, \ref{sec:fk-representation},
\ref{sec:cv-element-corr-func} and \ref{sec:fk-integrals-to-sle-integrals}
respectively. 
\item The discrete complex analysis techniques required to prove the results
of Section \ref{sec:cv-element-corr-func} are finally presented in
Section \ref{sec:discrete-complex-analysis}.
\end{itemize}

\section{Possible Applications\label{sec:possible-applications}}

\subsection{Crossing probabilities and free boundary conditions\label{sub:crossing-probabilities}}

In \cite{langlands-lewis-staubin}, Langlands, Lewis and Saint-Aubin
investigated numerical evidence for the conformal invariance of the
Ising model, taking a approach similar to the one of \cite{langlands-pouliot-staubin}
for percolation. They considered probabilities of crossings made of
$+$ spins in conformal rectangles (simply connected domains with
four marked boundary points), with free boundary conditions, and concluded
the conformal invariance of the scaling limit of these probabilities.
More precisely, they gave numerical evidence suggesting the following:
\begin{problem*}
Show that for any simply connected domain with four marked boundary
points $\left(D,a^{1},a^{2},a^{3},a^{4}\right)$, there exists a correlation
function $\mathfrak{C}\left(D,a^{1},a^{2},a^{3},a^{4}\right)$, which
is conformally invariant in the sense that 
\[
\mathfrak{C}\left(\varphi\left(D\right),\varphi\left(a^{1}\right),\varphi\left(a^{2}\right),\varphi\left(a^{3}\right),\varphi\left(a^{4}\right)\right)=\mathfrak{C}\left(D,a^{1},a^{2},a^{3},a^{4}\right)
\]
for any conformal mapping $\varphi:D\to\varphi\left(D\right)$ and
such that if $\left(D_{\delta},a_{\delta}^{1},a_{\delta}^{2},a_{\delta}^{3},a_{\delta}^{4}\right)$
is a family of discrete domains approximating $\left(D,a^{1},a^{2},a^{3},a^{4}\right)$
and we consider the critical Ising model on $\left(D_{\delta},a_{\delta}^{1},a_{\delta}^{2},a_{\delta}^{3},a_{\delta}^{4}\right)$
with free boundary conditions, we have
\[
\mathbb{P}_{D_{\delta}}\left\{ \mbox{there is a crossing of \ensuremath{+}spins }\left[a_{\delta}^{1},a_{\delta}^{2}\right]\leftrightsquigarrow\left[a_{\delta}^{3},a_{\delta}^{4}\right]\right\} \underset{\delta\to0}{\longrightarrow}\mathfrak{C}\left(D,a^{1},a^{2},a^{3},a^{4}\right).
\]
We say that there is a crossing of $+$ spins $\left[a_{\delta}^{1},a_{\delta}^{2}\right]\leftrightsquigarrow\left[a_{\delta}^{3},a_{\delta}^{4}\right]$
when there is a connected component of $D_{\delta}$ that is adjacent
to $\left[a_{\delta}^{1},a_{\delta}^{2}\right]$ and $\left[a_{\delta}^{3},a_{\delta}^{4}\right]$,
the spins at the vertices thereof are all $+1$. 
\end{problem*}
In a subsequent paper, the authors and Hugo Duminil-Copin will show
this conformal invariance result, whose proof relies on Theorem \ref{thm:main-thm}.
The strategy resembles the SLE-based derivation of Cardy's formula
for percolation (see \cite{lawler-schramm-werner-v}):
\begin{itemize}
\item One translates the crossing events in terms of hitting probabilities
for a discrete exploration path: construct an exploration process
$\iota_{\delta}$ started at $a_{\delta}^{1}$ that has $-$ spins
on its left and $+$ spins on its right, and which {}``pretends''
that there are $-$ spins on $\left[a_{\delta}^{4},a_{\delta}^{1}\right]$
and that there are $+$ spins on $\left[a_{\delta}^{1},a_{\delta}^{2}\right]$.
Depending on whether $\iota_{\delta}$ first hits $\left[a_{\delta}^{3},a_{\delta}^{4}\right]$
or $\left[a_{\delta}^{2},a_{\delta}^{3}\right]$, there is or there
is not a crossing of $+$ spins $\left[a_{\delta}^{1},a_{\delta}^{2}\right]\leftrightsquigarrow\left[a_{\delta}^{3},a_{\delta}^{4}\right]$
(see Figure \ref{fig:crossing-probabilities}). 
\item One shows that the discrete exploration path converges in law to a
conformally invariant continuous process: using a priori estimates,
one gets that the subsequential scaling limits of the process are
instantenously reflected on $\partial D_{\delta}$, and, using the
main result of the present paper, that the excursions are described
by dipolar SLE$\left(3\right)$. 
\end{itemize}
\begin{figure}

\includegraphics[width=11cm]{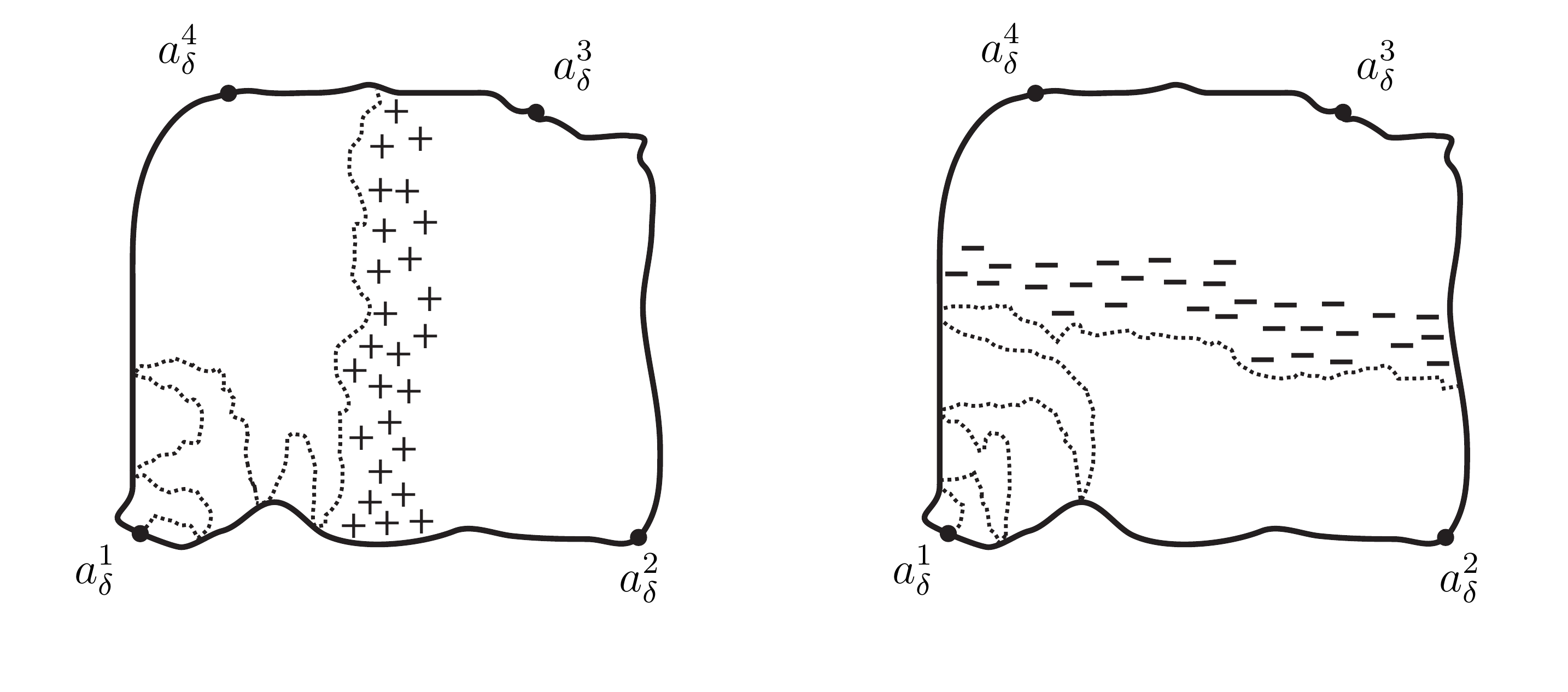}

\caption{\label{fig:crossing-probabilities}Depending on whether there is or
there is not a crossing of $+$ spins $\left[a_{\delta}^{1},a_{\delta}^{2}\right]\leftrightsquigarrow\left[a_{\delta}^{3},a_{\delta}^{4}\right]$,
the exploration process $\iota_{\delta}$ (dotted) first hits $\left[a_{\delta}^{3},a_{\delta}^{4}\right]$
or $\left[a_{\delta}^{2},a_{\delta}^{3}\right]$; if there is no $+$
crossing, there is a star-crossing (see \cite{werner-iii}) of $-$
spins $\left[a_{\delta}^{2},a_{\delta}^{3}\right]\leftrightsquigarrow\left[a_{\delta}^{4},a_{\delta}^{1}\right]$.}
\end{figure}

\subsection{Conformal loop ensembles\label{sub:cle}}

The most natural geometrical object to describe an Ising model configuration
on a discrete domain $D_{\delta}$ (with $+$ boundary conditions,
say) is probably the collection of all interfaces between $+$ and
$-$ spin clusters (in other words: put a dual edge between any pair
of spins with opposite signs), which form a collection of nested loops
on the lattice. The study of such contours dates back to Peierls,
who showed the existence of a phase transition by such considerations
\cite{peierls} (see also the contours of Lemma \ref{lem:low-t-phi}).

At critical temperature, it is natural to expect these random loops
to have a scaling limit, and the limiting loops to look like a variant
of SLE$\left(3\right)$. This limit, called Conformal Loop Ensemble
(CLE) with $\kappa$ parameter equal to $3$, is indeed a random collection
of continuous loops that can be constructed from SLE. 

The CLE$\left(\kappa\right)$ processes, introduced in \cite{sheffield-ii},
are defined for $\kappa\in(8/3,8]$, and they are the conjectural
scaling limits of loops arising in various lattice models; for $\kappa\in(8/3,4]$,
they also can be constructed from a Brownian loop soup \cite{sheffield-werner-i}. 

A very useful characterization result gives that the CLE$\left(\kappa\right)$'s
are the unique objects satisfying conformal invariance and an analog
of the domain Markov property (that many lattice models satisfy on
discrete level) \cite{sheffield-werner-ii}. 

It is reasonable to expect that the convergence of all the loops of
a lattice model to CLE$\left(\kappa\right)$ follows from the convergence
of a single interface between marked boundary point to SLE$\left(\kappa\right)$.
This has been worked out in detail for the case of percolation ($\kappa=6$)
\cite{camia-newman-ii}, and is work in progress for FK-Ising model
($\kappa=16/3$) \cite{kemppainen-smirnov-ii}; there is also closely
related work in progress for the uniform spanning tree $\left(\kappa=8\right)$
\cite{benoist-dubedat}.

For the Ising model, the situation seems more complicated, although
it might be possible to derive the convergence to CLE$\left(3\right)$
directly from the convergence of interfaces with $+/-$ boundary conditions
to chordal SLE$\left(3\right)$. 

The core idea for both percolation and FK-Ising is to construct an
exploration process on discrete level, that starts from a point on
the boundary and explores all the loops of the model; what makes this
idea work is that macroscopic loops touch the boundary with probability
tending to $1$ as the mesh size $\delta\to0$. This way, the discrete
process enters the bulk automatically and is instantaneously reflected
on the boundary; its excursions can be identified using the convergence
results to chordal SLE$\left(\kappa\right)$. 

The problem is that such an approach with chordal SLE cannot work,
at least without modification, for the Ising model: indeed, with probability
tending to $1$, there are no macroscopic loops touching the boundary,
as is witnessed by the fact that the CLE$\left(3\right)$ loops do
not touch each other, or simply that the SLE$\left(3\right)$ trace
is a simple curve. Hence, if we use the same discrete construction
as for the FK-Ising model, the exploration process will get stuck
on the boundary of the domain and will find no loop. It is reasonable
to expect that this construction works if one introduces small jumps
in the exploration process, or if we introduce some randomization
procedure, but this seems rather subtle to handle

We propose here an alternative approach, which allows to explore the
loops of the model with an exploration process. It relies on the two
convergence results
\begin{itemize}
\item The convergence of the FK-Ising interfaces to CLE$\left(16/3\right)$
(result of Kemppainen and Smirnov \cite{kemppainen-smirnov-ii}).
\item The identification of the arcs appearing for the  Ising model with
free boundary conditions, using dipolar SLE$\left(3\right)$.
\end{itemize}
Let us explain how from these two ingredients one can show the conformal
invariance of the loops of the model. The Ising model with $+$ boundary
conditions can be coupled with an FK-Ising model with wired boundary
conditions; through this coupling, the Ising model spin configurations
are obtained by assigning independent random $\pm1$ values (with
probabilities $1/2-1/2$) to the vertices of each FK cluster (see
Theorem \ref{thm:fk-to-ising}). If we look at the wired cluster $\Gamma$
attached to the boundary of the domain, then the (inner) boundary
of this cluster consists of disjoint loops $\gamma_{\delta}^{\left(1\right)},\ldots,\gamma_{\delta}^{\left(n\right)}$.
Conditionally on $\gamma_{\delta}^{\left(1\right)},\ldots,\gamma_{\delta}^{\left(n\right)}$,
the law of the spins (obtained through the coupling) in the domains
$\Omega_{\delta}^{\left(1\right)},\ldots,\Omega_{\delta}^{\left(n\right)}$
are the laws of independent critical Ising models with free boundary
conditions (see Figure \ref{fig:cle}). The spins of $\Gamma$ are
all set to $+$. Hence, we know that there are no Ising loops in $\Gamma$:
all the Ising loops (having $+$ spins outside and $-$ spins inside)
appear inside the domains $\Omega_{\delta}^{\left(k\right)}$. Let
us denote by $\lambda_{\delta}^{\left(k;j\right)}$ ($1\leq j\leq k_{m}$)
the loops appearing in $\Omega_{\delta}^{\left(k\right)}$. A number
of the loops $\lambda_{\delta}^{\left(k;j\right)}$ touch the loop
$\gamma_{k}$ and they can be reconstructed by pasting boundary arcs
with parts of $\gamma_{k}$. 

\begin{figure}
\includegraphics[width=12cm]{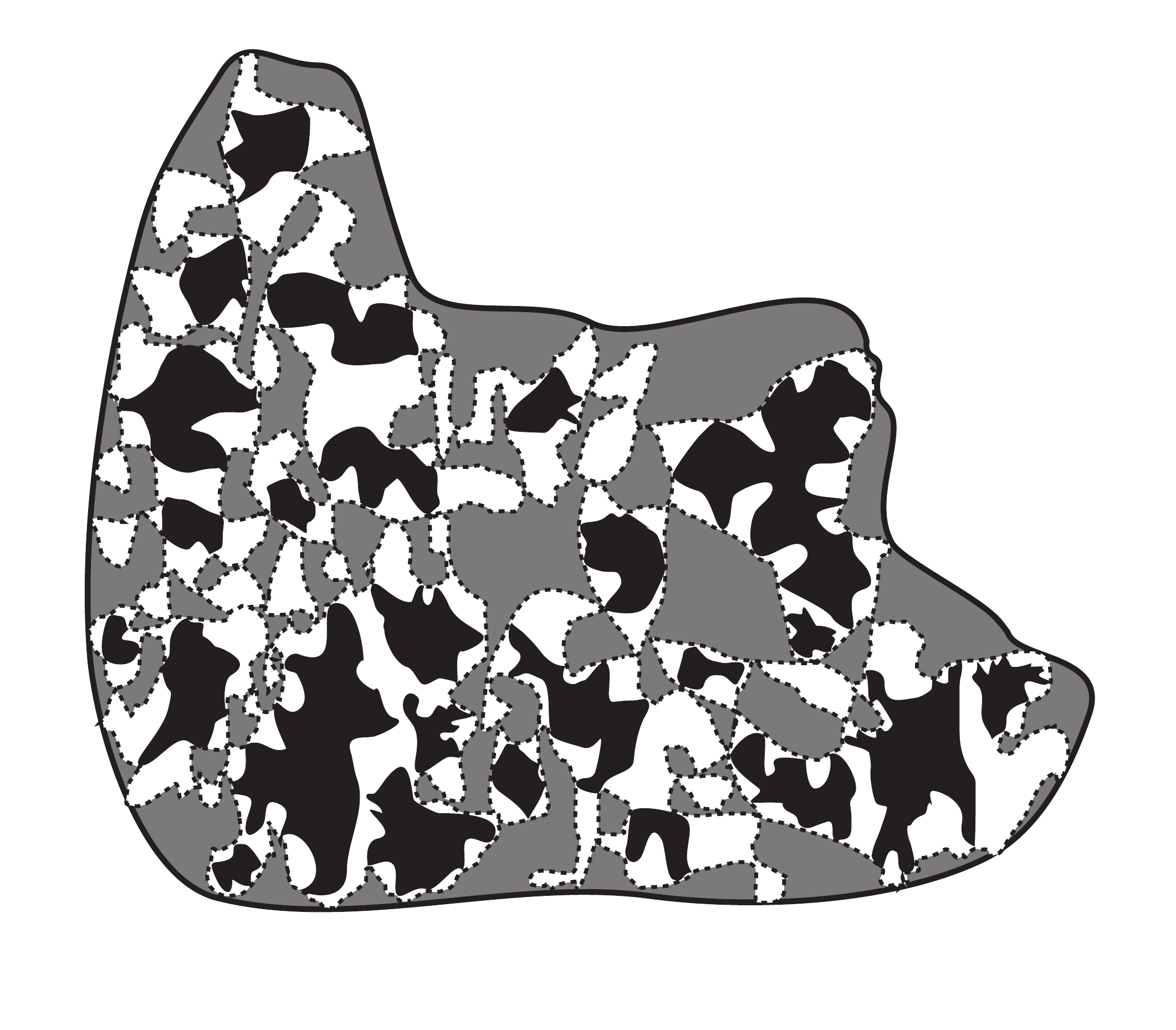}

\caption{\label{fig:cle}Coupling of FK-Ising and Ising interfaces. The dashed
loops, which are bounding white regions, are part of the boundary
of the FK cluster $\Gamma$ (they are the $\gamma_{\delta}^{\left(j\right)}$
loops). Conditionally on these loops, the boundary conditions for
the Ising model inside them are purely free. The boundaries of the
black areas are the $\lambda_{\delta}^{\left(k;j\right)}$ loops (they
have $+$ spins outside and $-$ spins inside).}
\end{figure}

The point is that the the loops $\left\{ \gamma_{\delta}^{\left(i\right)}:i\right\} $
converge to CLE$\left(16/3\right)$ loops as $\delta\to0$ and that
the arcs on the loops $\left\{ \gamma_{\delta}^{\left(i\right)}:i\right\} $
converge to a free SLE$\left(3\right)$ tree, discussed below. Hence,
we can construct the scaling limit of the loops $\lambda_{\delta}^{\left(k;i\right)}$,
which are the outermost loops of the Ising model with $+$ boundary
conditions in the original domain. The process can then be iterated
inside the loops thus obtained, and all the loops will be eventually
discovered. As the whole construction is conformally invariant, so
is the scaling limit of the collection of the loops arising in the
model; by the characterization of \cite{sheffield-werner-ii}, we
deduce that the collection of all the loops is CLE$\left(3\right)$. 

Let us now describe the free SLE$\left(3\right)$ tree in a domain
$\Omega$, which describes the scaling limit of the arcs linking the
boundary with free boundary conditions. It is a continuous tree, such
that any two points $x,y\in\partial\Omega$ are linked by a branch
which is an exploration process as that {}``pretends'' there is
$+$ boundary condition on $\left[x,y\right]$ and $-$ boundary condition
on $\left[y,x\right]$. For any three points $x,y,z\in\partial\Omega$,
we can couple the branch $\iota^{x\to y}$ from $x$ to $y$ and the
branch $\iota^{x\to z}$ from $x$ to $z$ in such a way that they
are the same, up to the first time $\tau$ when $x$ and $z$ are
disconnected by that branch; the remaining of the curve is independent. 

Note that our technique also allows to describe the scaling limit
of the collection of all the interfaces appearing with various boundary
conditions, in particular purely free boundary conditions.

\section{Proof of the main result\label{sec:proof-main-result}}

We now outline the proof of our main result, Theorem \ref{thm:main-thm}.
The key argument is the martingale observable result (Theorem \ref{thm:martingale-observable}
in Section \ref{sub:martingale-observable}). Together with the precompactness
result given in Section \ref{sub:precompactness}, the key theorem
is used to obtain Theorem \ref{thm:main-thm} in Section \ref{sub:identification-limit}.

\subsection{\label{sub:precompactness}Precompactness}

The first ingredient in the proof of Theorem \ref{thm:main-thm} is
a precompactness result, which allows to extract subsequential limits.
\begin{thm}
\label{thm:precompactness}With the assumptions and the notation of
Theorem 1, the laws of the initial segments of the interfaces $\gamma_{\delta}$
form a tight family, and their subsequential scaling limits are almost
surely curves. 
\end{thm}
This result follows from standard estimates (Russo-Seymour-Welsh-type
crossing bounds) and its proof is exactly the same as the one for
the Dobrushin setup \cite{chelkak-smirnov-iii}. These are a priori
uniform estimates for some crossing probabilities follow from \cite{chelkak-smirnov-ii,duminil-copin-hongler-nolin}.
The framework built in \cite{kemppainen-smirnov} gives the result.
By these arguments, one shows that the law of the curves stopped upon
reaching an $\epsilon$-neighborhood of $\left[r,\ell\right]$ are
tight and by diagonal extraction one can let $\epsilon\to0$. 

To interchange the limits (which allows to consider the scaling limit
of the discrete interfaces and not just initial segments thereof),
one needs additional control on the end of the interface: one has
to ensure that the discrete interface $\gamma_{\delta}$ hits the
arc $\left[r,\ell\right]$ with high probability once it gets close
to that arc. This can be deduced from strong RSW a priori estimates,
which have been recently announced \cite{chelkak,chelkak-duminil-copin-hongler}.

\subsection{\label{sub:martingale-observable}Martingale observable}

This subsection contains the key result for proving the main theorem:
it is the part which is really specific to our setup.

Let us first define what is known as an observable in the SLE literature,
and as a correlation function in the CFT literature: it is a function
of a domain with marked points. This observable will play a crucial
role in the proof of the main theorem.
\begin{defn}
\label{def:cts-obs}For any domain $\left(\Omega,r,\ell,x,z\right)$
such that $z$ is on a smooth part of $\partial\Omega$, we denote
by $\Phi\left(\Omega,r,\ell,x,z\right)\in\mathbb{R}$ the quantity
defined by
\[
\Phi\left(\Omega,r,\ell,x,z\right):=\sqrt{\frac{\sqrt{2}+1}{2\pi}}\left|\psi'\left(z\right)\right|^{\frac{1}{2}}\coth\left(\frac{\psi\left(z\right)}{2}\right)
\]
where $\psi$ is the conformal map from $\Omega$ to the strip $\mathbb{S}:=\left\{ z\in\mathbb{C}:0<\Im\mathfrak{m}\left(z\right)<\pi\right\} $
such that $\psi\left(x\right)=0$, $\psi\left(r\right)=+\infty$,
$\psi\left(\ell\right)=-\infty$. 
\end{defn}
The key theorem to prove the convergence of the interface to SLE is
the martingale property of the function $\Phi$. It brings to the
continuous level all the information about the Ising model that we
need to identify the curve. 
\begin{thm}
\label{thm:martingale-observable}Assume that the arc $\left[b,r\right]$
contains a vertical part $\mathfrak{v}$ and that for each $\delta>0$
the discretization $\left(D_{\delta},r_{\delta},\ell_{\delta},b_{\delta}\right)$
contains a vertical part $\mathfrak{v}_{\delta}\subset\left[b_{\delta},r_{\delta}\right]$
that converges to $\mathfrak{v}$ as $\delta\to0$. Let $\gamma$
have the law of any subsequential limit of (the initial segments of)
discrete interfaces $\gamma_{\delta_{n}}$ for a sequence $\left(\delta_{n}\right)_{n\geq0}$
with $\delta_{n}\to0$ as $n\to\infty$. Let $D_{t}$ be the connected
component of $\Omega\setminus\gamma_{t}$ containing the arc $\left[r,\ell\right]$.
Let $\tau\in\left[0,\infty\right]$ be the first time $t$ when $\gamma$
hits $\mathfrak{v}\cup\left[r,\ell\right]$.

Then for any $z\in\mathfrak{v}$, we have that 
\[
\left(\Phi\left(D_{t\wedge\tau},r,\ell,\gamma\left(t\wedge\tau\right),z\right)\right)_{t\geq0}
\]
 is a continuous local martingale.
\end{thm}
In physical terms, the observable $\Phi$ hence plays the role of
a one-parameter family of (stochastic) conservation laws, indexed
by $z$. The proof of this theorem is discussed in Section \ref{thm:martingale-observable}.

\subsection{\label{sub:identification-limit}Identification of the scaling limit}

The following technical lemma, shown in Appendix A, allows us to fit
with the framework of Theorem \ref{thm:martingale-observable}:
\begin{lem}
\label{lem:vert-boundary-assumption}To prove Theorem \ref{thm:main-thm},
we can assume that the domain $D$ is such that the arc $\left[b,r\right]$
contains a vertical part $\mathfrak{v}$ and that the discrete domains
$D_{\delta}$ are such that the arc $\left[b_{\delta},r_{\delta}\right]$
contains a vertical part $\mathfrak{v}_{\delta}$ converging to $\mathfrak{v}$
as $\delta\to0$. 
\end{lem}
We can now give the proof of the main theorem:
\begin{proof}[Proof of Theorem \ref{thm:main-thm}]
By Theorem \ref{thm:precompactness}, we can extract subsequential
scaling limits of the (initial segments of the) discrete interfaces
$\left(\gamma_{\delta}\right)_{\delta>0}$ as $\delta\to0$. It remains
to identify any subsequential scaling limit $\gamma$ as dipolar SLE$\left(3\right)$
in $\left(D,r,\ell,b\right)$. Let us assume that we have a vertical
part of the boundary $\mathfrak{v}\subset\left[b,r\right]$ as in
Lemma \ref{lem:vert-boundary-assumption}. Thanks to the key Theorem
\ref{thm:martingale-observable}, we can follow a procedure which
has become standard in the SLE subject \cite{smirnov-i,makarov-smirnov}
to identify the scaling limit of the interface:
\begin{itemize}
\item We describe the growing random curve by its complementary, and look,
for each time $t\geq0$, at the domain $D$ slitted by the curve $\gamma\left[0,t\right]$
(we pick an arbitrary parametrization of $\gamma$). Our interface
is now described by shrinking domains $\left(D_{t},r,\ell,\gamma\left(t\right)\right)_{t\geq0}$,
where $D_{t}\subset D$ is the connected component of $D\setminus\gamma\left[0,t\right]$
that contains $\left[r,\ell\right]$. 
\item For any $z\in\mathfrak{v}$, the process $\left(\Phi\left(D_{t},r,\ell,\gamma\left(t\right),z\right)\right)_{t\geq0}$,
stopped as $\gamma\left(t\right)$ hits $\mathfrak{v}\cup\left[r,\ell\right]$,
is a continuous local martingale (Theorem \ref{thm:martingale-observable}).
\item We map $D$ to the strip $\mathbb{S}$ by the conformal mapping $\psi:D\to\mathbb{S}$
such that $\psi\left(b\right)=0$, $\psi\left(r\right)=\infty$, $\psi\left(\ell\right)=-\infty$.
\item We look at the process $\left(S_{t}\right)_{t\geq0}$, where for any
$t\geq0$, $S_{t}$ is the unbounded connected component of $\mathbb{S}\setminus\psi\left(\gamma\left[0,t\right]\right)$.
\item As explained in Section \ref{sub:dipolar-loewner-chain}, we can encode
$\left(S_{t}\right)_{t\geq0}$ by a strip Loewner chain $\left(g_{t}:S_{t}\to\mathbb{S}\right)_{t\geq0}$,
with driving process $\left(V_{t}\right)_{t\geq0}$: after time reparametrization,
we have
\begin{eqnarray*}
\partial_{t}g_{t}\left(z\right) & = & \coth\left(\frac{g_{t}\left(z\right)-V_{t}}{2}\right),\\
g_{0}\left(z\right) & = & z.
\end{eqnarray*}

\item By conformal covariance of $\Phi$ and its explicit formula on the
strip, we deduce that 
\[
\left(\left|g_{t}'\left(z\right)\right|^{\frac{1}{2}}\coth\left(\frac{1}{2}\left(g_{t}\left(z\right)-V_{t}\right)\right)\right)_{t\geq0}
\]
 is a continuous local martingale for any $z\in\psi\left(\mathfrak{v}\right)$.
\item Since $g_{t}\left(z\right)$ and $g_{t}'\left(z\right)$ are differentiable
in time (and hence of finite variation) and since $g_{t}'\left(z\right)$
never vanishes, we deduce that $\left(V_{t}\right)_{t\geq0}$ is a
continuous semi-martingale.
\item Using that $\left(V_{t}\right)_{t\geq0}$ is a continuous semi-martingale,
we can apply Itô's calculus to get that 
\begin{eqnarray*}
 &  & \mathrm{d}\left(\left|g_{t}'\left(z\right)\right|^{\frac{1}{2}}\coth\left(\frac{g_{t}\left(z\right)-V_{t}}{2}\right)\right)\\
 & = & \frac{\left|g_{t}'\left(z\right)\right|^{\frac{1}{2}}}{4}\left(\frac{2\mathrm{d}V_{t}}{\sinh^{2}\left(\frac{g_{t}\left(z\right)-V_{t}}{2}\right)}+\frac{\cosh\left(\frac{g_{t}\left(z\right)-V_{t}}{2}\right)}{\sinh^{3}\left(\frac{g_{t}\left(z\right)-V_{t}}{2}\right)}\left(\mathrm{d}\left\langle V_{t},V_{t}\right\rangle -3\mathrm{d}t\right)\right),
\end{eqnarray*}
for any $z\in\psi\left(\mathfrak{v}\right)$. We obtain that $V_{t}$
is driftless and that $\mathrm{d}\left\langle V_{t},V_{t}\right\rangle =3\mathrm{d}t$.
\item Since we moreover have $V_{0}=0$ (as $\psi\left(\gamma\right)$ starts
growing at $0$), it follows from Lévy's characterization theorem
that $\left(V_{t}\right)_{t}$ has the law of $\left(\sqrt{3}B_{t}\right)_{t}$,
where $\left(B_{t}\right)_{t}$ is a standard Brownian motion.
\item This shows that $\psi\left(\gamma\right)$ has the law of dipolar
SLE$\left(3\right)$ in $\left(\mathbb{S},\infty,-\infty,0\right)$
and hence that $\gamma$ has the law of dipolar SLE$\left(3\right)$
in $\left(D,r,\ell,b\right)$.
\end{itemize}
\end{proof}

\section{\label{sec:proof-key-thm}Proof of the key theorem: the martingale
observable}

\subsection{\label{sub:disc-obs}The discrete martingale observable}

In this section we give the main steps for the proof of the key theorem
(Theorem \ref{thm:martingale-observable}). Let us first define a
discrete version of the observable $\Phi$ introduced in Section \ref{sub:martingale-observable},
which allows to make the connection with the Ising model:
\begin{defn}
\label{def:disc-obs}Let $\left(\Omega_{\delta},r,\ell,x,z\right)$
be a discrete domain with four marked boundary points. We denote by
$\Phi_{\delta}\left(\Omega_{\delta},r,\ell,x,z\right)$ the quantity
defined by
\[
\Phi_{\delta}\left(\Omega_{\delta},r,\ell,x,z\right):=\frac{\tilde{\mathcal{Z}}\left(\Omega_{\delta},r,\ell,x,z\right)}{\mathcal{Z}\left(\Omega_{\delta},r,\ell,x\right)},
\]
where
\begin{itemize}
\item $\mathcal{Z}\left(\Omega_{\delta},r,\ell,x\right)$ is the partition
function of the critical Ising model on the faces of $\left(\Omega_{\delta},r,\ell,x\right)$
with dipolar boundary conditions (free on $\left[r,\ell\right]$,
$-$ on $\left[\ell,x\right]$ and $+$ on $\left[x,r\right]$), as
defined in Section \ref{sub:ising-model}.
\item $\tilde{\mathcal{Z}}\left(\Omega_{\delta},r,\ell,x,z\right)$ is the
partition function of the critical Ising model on the faces of $\left(\Omega_{\delta},r,\ell,x,z\right)$
with the following modified boundary conditions: free on $\left[r,\ell\right]$,
$-$ on $\left[\ell,x\right]$, $+$ on $\left[x,z\right]$ and $-$
on $\left[z,r\right]$; more precisely, it is obtained by replacing
$\left[b,r\right]$ by $\left[x,z\right]$ and $\left[\ell,b\right]$
by $\left[\ell,x\right]\cup\left[z,r\right]$ in Equation \ref{eq:hamiltonian-ising}.
\end{itemize}
\end{defn}
\begin{figure}
\includegraphics[width=11cm]{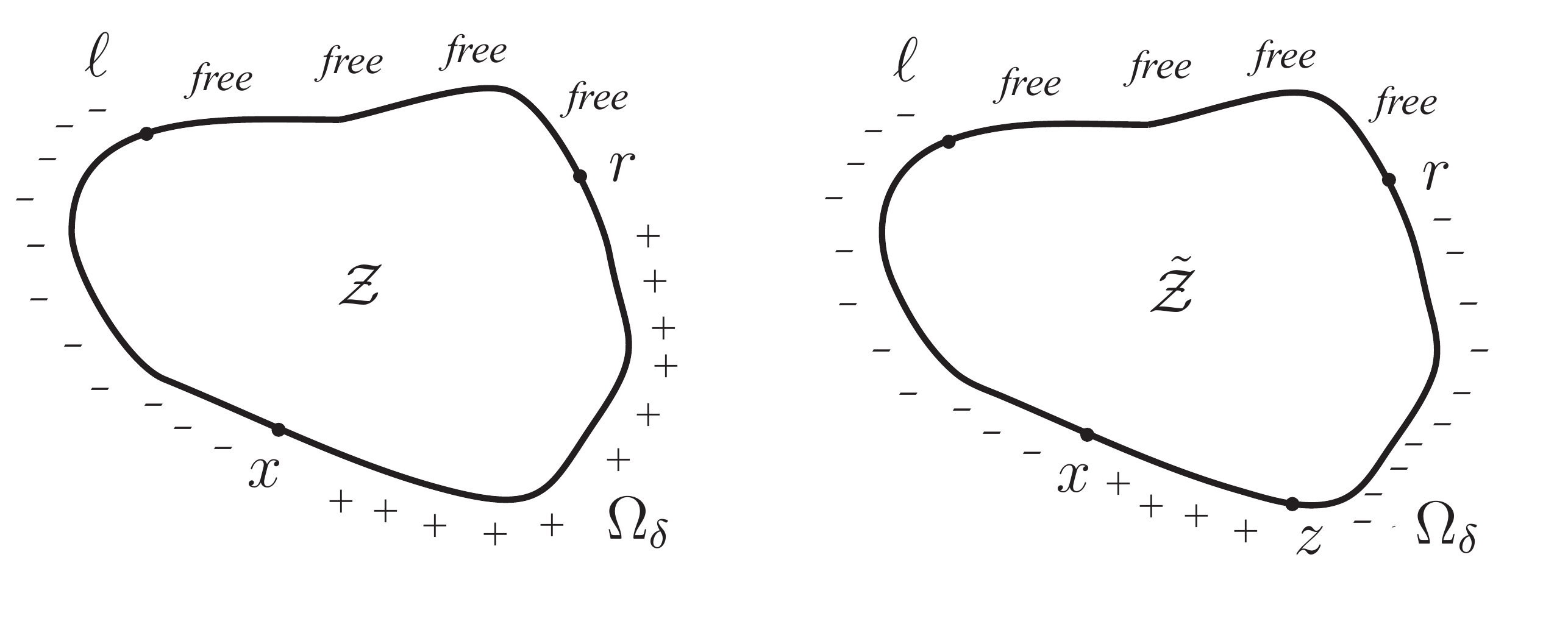}

\caption{The boundary conditions for the partition functions $\mathcal{Z}\left(\Omega_{\delta},r,\ell,x\right)$
and $\tilde{\mathcal{Z}}\left(\Omega_{\delta},r,\ell,x,z\right)$.}
\end{figure}

The following proposition is the analogue of the key theorem for $\Phi_{\delta}$
and the discrete interface. 
\begin{prop}
\label{pro:discrete-martingale}Let $\left(\gamma_{\delta}\left(n\right)\right)_{n\geq0}$
have the law of the interface emanating at $a$ in the critical Ising
model on $\left(D_{\delta},a,b,c\right)$ with dipolar boundary conditions,
parametrized by the number of steps. For any $z\in\left[a,b\right]$,
we have that 
\[
\left(\Phi_{\delta}\left(D_{\delta}\setminus\gamma_{\delta}\left[0,n\right],r,\ell,\gamma_{\delta}\left(n\right),z\right)\right)_{n\wedge\tau\left(z\right)}
\]
 is a discrete time martingale, where 
\[
\tau\left(z\right):=\inf\left\{ n:\gamma_{\delta}\left(n\right)\in\left[z,\ell\right]\right\} .
\]

\end{prop}
See Section \ref{sub:proof-disc-mart-property} for a precise definition
of $D_{\delta}\setminus\gamma_{\delta}\left[0,n\right]$. The proof
of Proposition \ref{pro:discrete-martingale}, which is in essence
combinatorial, is also given in Section \ref{sub:proof-disc-mart-property}.

The key theorem, which is the martingale property of $\Phi$ for the
subsequential scaling limits of $\left(\gamma_{\delta}\right)_{\delta>0}$
is hence a consequence of the following observable convergence theorem,
as it is inherited from discrete level:
\begin{thm}
\label{thm:disc-obs-to-cts-obs}If $\left(\Omega_{\delta},r_{\delta},\ell_{\delta},x_{\delta},z_{\delta}\right)\to\left(\Omega,r,\ell,x,z\right)$
as $\delta\to0$, then 
\[
\frac{1}{\sqrt{\delta}}\Phi_{\delta}\left(\Omega_{\delta},r_{\delta},\ell_{\delta},x_{\delta},z_{\delta}\right)\underset{\delta\to0}{\longrightarrow}\Phi\left(\Omega,r,\ell,x,z\right).
\]
The convergence is locally uniform with respect to the domains.
\end{thm}
We are now in position to prove the key theorem:
\begin{proof}[Proof of Theorem \ref{thm:martingale-observable}]
By Proposition \ref{pro:discrete-martingale}, $\Phi_{\delta}$ is
a discrete martingale with respect to the Ising interface, when parametrized
by the lattice steps. Since $\frac{1}{\sqrt{\delta}}\Phi_{\delta}$
converges to $\Phi$ as $\delta\to0$, we deduce that $t\mapsto\Phi\left(D\setminus\gamma\left[0,t\right],r,\ell,\gamma\left(t\right),z\right)$
is a local martingale. The time continuity follows from the continuity
of $\Phi$ with respect to the domain.
\end{proof}
The heart of the matter, namely the connection between discrete and
continuous worlds, is therefore contained in Theorem \ref{thm:disc-obs-to-cts-obs},
the proof thereof exploits special features of the Ising model, both
combinatorial and analytical. It is discussed in the following subsections.

\subsection{Four ingredients\label{sub:four-ingredients}}

In this subsection, we give four propositions, which together allow
to deduce the observable convergence theorem (Theorem \ref{thm:disc-obs-to-cts-obs}),
as will be explained in Section \ref{sub:disc-mart-obs-cv}.

\subsubsection{Correlation function representation}

The first proposition allows for a representation of $\Phi_{\delta}$
in terms of discrete correlation functions on a dual Ising model:
\begin{prop}
\label{pro:representation-obs-as-correlation}If we consider the Ising
model on the \emph{vertices} of $\left(\Omega_{\delta},r,\ell,x,z\right)$
with $+$ boundary condition on the arc $\left[r,\ell\right]$ and
free boundary condition on the arc $\left[\ell,r\right]$, we have
\[
\frac{\mathbb{E}_{\Omega_{\delta}}^{\left[r,\ell\right]_{+}}\left[\sigma_{x}\sigma_{z}\right]}{\mathbb{E}_{\Omega_{\delta}}^{\left[r,\ell\right]_{+}}\left[\sigma_{x}\right]}=\Phi_{\delta}\left(\Omega_{\delta},r,\ell,x,z\right).
\]

\end{prop}
The proof is given in Section \ref{sec:kw-duality}.

\subsubsection{FK representation}

The second proposition makes use of the Fortuin-Kasteleyn (FK) representation
of the dual Ising model (defined in Section \ref{sub:fk-model}) to
give an expression for $\Phi_{\delta}\left(\Omega_{\delta},r_{\delta},\ell_{\delta},x_{\delta},z_{\delta}\right)$
in terms of expectations over FK interfaces of simple correlation
functions:
\begin{prop}
\label{pro:int-rep-corr}We have
\[
\Phi_{\delta}\left(\Omega_{\delta},r,\ell,x,z\right)=\mathrm{E}_{\delta}^{\mathrm{A}}\left(\Omega_{\delta},r,\ell,x,z\right)+\mathrm{E}_{\delta}^{\mathrm{B}}\left(\Omega_{\delta},r,\ell,x,z\right),
\]
where
\begin{eqnarray*}
\mathrm{E}_{\delta}^{\mathrm{A}}\left(\Omega_{\delta},r,\ell,x,z\right) & := & \frac{\mathbb{E}_{\delta}^{\mathrm{A}}\left[\mathbb{E}_{\Omega_{\delta}\setminus\lambda_{\delta}}^{\mathrm{free}}\left[\sigma_{x}\sigma_{z}\right]\right]}{\mathbb{E}_{\Omega_{\delta}}^{\left[r,\ell\right]_{+}}\left[\sigma_{x}\right]},\\
\mathrm{E}_{\delta}^{\mathrm{B}}\left(\Omega_{\delta},r,\ell,x,z\right) & := & \mathbb{E}_{\delta}^{\mathrm{B}}\left[\mathbb{E}_{\Omega_{\delta}\setminus\tilde{\lambda}_{\delta}}^{\left[r,x\right]_{+}}\left[\sigma_{z}\right]\right],
\end{eqnarray*}
where 
\begin{itemize}
\item the expectation $\mathbb{E}_{\delta}^{\mathrm{A}}$ is taken over
the realizations $\lambda_{\delta}$ of a critical FK-Ising interface
in $\Omega_{\delta}$ from $\ell$ to $r$;
\item the expectation $\mathbb{E}_{\delta}^{\mathrm{B}}$ is taken over
all realizations $\tilde{\lambda}_{\delta}$ of a critical FK interface
in $\Omega_{\delta}$ from $\ell$ to $r$, conditioned to pass through
$x$, stopped at $x$;
\item the correlation $\mathbb{E}_{\Omega_{\delta}\setminus\lambda_{\delta}}^{\mathrm{free}}\left[\sigma_{x}\sigma_{z}\right]$
is the correlation of the spins at $x$ and $z$ of the critical Ising
model on $\Upsilon_{\delta}$ with free boundary conditions, where
$\Upsilon_{\delta}$ is the connected component of $\Omega_{\delta}\setminus\lambda_{\delta}$
(the graph $\Omega_{\delta}$ with the edges crossed by $\lambda_{\delta}$
removed) containing $x$ and $z$;
\item the correlation \textup{$\mathbb{E}_{\Omega_{\delta}\setminus\tilde{\lambda}_{\delta}}^{\left[r,x\right]_{+}}\left[\sigma_{z}\right]$}\textup{\emph{
is the magnetization at $z$ of the critical Ising model on the connected
component of $\Omega_{\delta}\setminus\tilde{\lambda}_{\delta}$ that
contains $z$, with $+$ boundary condition on $\left[r,x\right]$
and free boundary condition on $\left[x,r\right]$. }}
\end{itemize}
\end{prop}
The proof of this proposition, as well as the definition of the FK
model and of its interface, are given in Section \ref{sec:fk-representation}.

\subsubsection{From discrete to continuum}

Let us now introduce the continuous analogues of the discrete objects
appearing in Proposition \ref{pro:int-rep-corr}: the correlation
functions and the curves. The convergence of the discrete correlation
functions to the continuous ones is dealt with in Section \ref{sec:cv-element-corr-func}. 
\begin{defn}
\label{def:corr-func}If $\Omega$ is a simply connected domain we
define the correlation functions $\left\langle \sigma_{x}\sigma_{z}\right\rangle _{\Omega}^{\mathrm{free}}$
and $\left\langle \sigma_{x}\right\rangle _{\Omega}^{\left[r,\ell\right]_{+}}$
by 
\begin{eqnarray*}
\left\langle \sigma_{x}\sigma_{z}\right\rangle _{\Omega}^{\mathrm{free}} & := & \left|\eta'\left(x\right)\right|^{\frac{1}{2}}\left|\eta'\left(y\right)\right|^{\frac{1}{2}}\left\langle \sigma_{\eta\left(x\right)}\sigma_{\eta\left(y\right)}\right\rangle _{\mathbb{H}}^{\mathrm{free}}\\
\left\langle \sigma_{\eta_{2}}\sigma_{\eta_{1}}\right\rangle _{\mathbb{H}}^{\mathrm{free}} & := & \frac{\sqrt{2}+1}{\pi}\frac{1}{\left|\eta_{1}-\eta_{2}\right|}\\
\left\langle \sigma_{x}\right\rangle _{\Omega}^{\left[r,\ell\right]_{+}} & := & \left|\eta'\left(x\right)\right|^{\frac{1}{2}}\left\langle \sigma_{\eta\left(x\right)}\right\rangle _{\mathbb{H}}^{\left[\eta\left(r\right),\eta\left(\ell\right)\right]_{+}},\\
\left\langle \sigma_{\eta_{1}}\right\rangle _{\mathbb{H}}^{\left[\eta_{\infty},\eta_{0}\right]_{+}} & := & \sqrt{\frac{\sqrt{2}+1}{2\pi}}\sqrt{\frac{\left|\eta_{\infty}-\eta_{0}\right|}{\left|\eta_{1}-\eta_{0}\right|\left|\eta_{1}-\eta_{\infty}\right|}}.
\end{eqnarray*}
for any conformal mapping $\eta:\Omega\to\mathbb{H}$ and any $x,y,r,\ell\in\partial\Omega$,
provided the right hand sides are well-defined (these definitions
are independent of the choice of $\eta$).
\end{defn}
The FK-Ising interfaces converge to variants of SLE$\left(16/3\right)$
that are defined in Section \ref{sec:sle-variants}.

The third proposition gives the convergence of the discrete expectations
$\mathrm{E}_{\delta}^{\mathrm{A}}$ and $\mathrm{E}_{\delta}^{\mathrm{B}}$
as $\delta\to0$ to continuous expectations. The definitions of chordal
SLE$\left(\kappa\right)$ and SLE$\left(\kappa;\rho\right)$ are given
in Section \ref{sec:sle-variants}. 
\begin{prop}
\label{pro:conv-discrete-expectations}As $\delta\to0$ and $\left(\Omega_{\delta},r_{\delta},\ell_{\delta},x_{\delta},z_{\delta}\right)\to\left(\Omega,r,\ell,x,z\right)$,
we have
\begin{eqnarray*}
\frac{1}{\sqrt{\delta}}\mathrm{E}_{\delta}^{\mathrm{A}}\left(\Omega_{\delta},r_{\delta},\ell_{\delta},x_{\delta},z_{\delta}\right) & \underset{\delta\to0}{\longrightarrow} & \mathrm{E}^{\mathrm{A}}\left(\Omega,r,\ell,x,z\right)\\
\frac{1}{\sqrt{\delta}}\mathrm{E}_{\delta}^{\mathrm{B}}\left(\Omega_{\delta},r_{\delta},\ell_{\delta},x_{\delta},z_{\delta}\right) & \underset{\delta\to0}{\longrightarrow} & \mathrm{E}^{\mathrm{B}}\left(\Omega,r,\ell,x,z\right),
\end{eqnarray*}
where the continuous expectations $\mathrm{E}^{\mathrm{A}}\left(\Omega,r,\ell,x,z\right)$
and $\mathrm{E}^{\mathrm{B}}\left(\Omega,r,\ell,x,z\right)$ are defined
by
\begin{eqnarray*}
\mathrm{E}^{\mathrm{A}}\left(\Omega,r,\ell,x,z\right) & := & \mathbb{E}^{\mathrm{A}}\left[\frac{\left\langle \sigma_{x}\sigma_{z}\right\rangle _{\Omega\setminus\lambda}^{\mathrm{free}}}{\left\langle \sigma_{x}\right\rangle _{\Omega}^{\left[r,\ell\right]_{+}}}\right],\\
\mathrm{E}^{\mathrm{B}}\left(\Omega,r,\ell,x,z\right) & := & \mathbb{E}^{\mathrm{B}}\left[\left\langle \sigma_{z}\right\rangle _{\Omega\setminus\tilde{\lambda}}^{\left[r,\ell\right]_{+}}\right],
\end{eqnarray*}
where 
\begin{itemize}
\item the correlation functions are as in Definition \ref{def:corr-func}.
\item the expectation $\mathbb{E}^{\mathrm{A}}$ is over the realizations
$\lambda$ of a chordal SLE$\left(16/3\right)$ trace from $\ell$
to $r$.
\item the expectation $\mathbb{E}^{\mathrm{B}}$ is over the realizations
$\tilde{\lambda}$ of an SLE$\left(16/3;-8/3\right)$ trace starting
from $\ell$, with observation point $r$ and force point $x$, stopped
upon hitting $x$. 
\item the integrand in $\mathbb{E}^{\mathrm{A}}$ is defined as $0$ if
$x$ and $z$ are in different connected components of $\Omega\setminus\lambda$.
\end{itemize}
The convergence is locally uniform with respect to the domains.
\end{prop}
The proof is given in Section \ref{sec:fk-integrals-to-sle-integrals}.
\begin{rem}
Notice that the ratio 
\[
\frac{\left\langle \sigma_{x}\sigma_{z}\right\rangle _{\Omega\setminus\lambda}^{\mathrm{free}}}{\left\langle \sigma_{x}\right\rangle _{\Omega}^{\left[r,\ell\right]_{+}}}
\]
 is well-defined even if $x$ lies on a rough part of $\partial\Omega$;
on the other hand, recall that we assumed that $z$ is on a vertical
part of $\partial\Omega$.
\end{rem}

\subsubsection{Computations in the continuum}

The fourth and last proposition is the explicit computation of $\mathrm{E}^{\mathrm{A}}$
and $\mathrm{E}^{\mathrm{B}}$:
\begin{prop}
\label{pro:sle-averages}We have 
\begin{eqnarray*}
\mathrm{E}^{\mathrm{A}}\left(\Omega,r,\ell,x,z\right) & = & \left|\eta{}_{\Omega}'\left(z\right)\right|^{\frac{1}{2}}\mathrm{E}^{\mathrm{A}}\left(\mathbb{H},\eta{}_{\Omega}\left(r\right),\eta{}_{\Omega}\left(\ell\right),\eta{}_{\Omega}\left(x\right),\eta{}_{\Omega}\left(z\right)\right)\\
\mathrm{E}^{\mathrm{B}}\left(\Omega,r,\ell,x,z\right) & = & \left|\eta{}_{\Omega}'\left(z\right)\right|^{\frac{1}{2}}\mathrm{E}^{\mathrm{B}}\left(\mathbb{H},\eta{}_{\Omega}\left(r\right),\eta{}_{\Omega}\left(\ell\right),\eta{}_{\Omega}\left(x\right),\eta{}_{\Omega}\left(z\right)\right)\\
\Phi\left(\Omega,r,\ell,x,z\right) & = & \mathrm{E}^{\mathrm{A}}\left(\Omega,r,\ell,x,z\right)+\mathrm{E}^{\mathrm{B}}\left(\Omega,r,\ell,x,z\right)
\end{eqnarray*}
for any conformal mapping $\eta_{\Omega}:\Omega\to\mathbb{H}$. On
$\mathbb{H}$, we have, if $\eta_{\infty}<\eta_{0}<\eta_{1}<\eta_{2}$,
\begin{eqnarray*}
\mathrm{E}^{\mathrm{A}}\left(\mathbb{H},\eta_{\infty},\eta_{0},\eta_{1},\eta_{2}\right) & = & C_{\mathrm{A}}\frac{\left(\eta_{1}-\eta_{\infty}\right)^{1/2}}{\left(\eta_{2}-\eta_{\infty}\right)^{1/2}\left(\eta_{2}-\eta_{1}\right)^{1/2}}\int_{\chi}^{1}\frac{\left(1-\zeta\right)}{\zeta^{3/4}\left(\zeta-\chi\right)^{3/4}}\mathrm{d}\zeta\\
\mathrm{E}^{\mathrm{B}}\left(\mathbb{H},\eta_{\infty},\eta_{0},\eta_{1},\eta_{2}\right) & = & C_{\mathrm{B}}\frac{\left(\eta_{1}-\eta_{\infty}\right)^{1/2}}{\left(\eta_{2}-\eta_{\infty}\right)^{1/2}\left(\eta_{2}-\eta_{1}\right)^{1/2}}\cdot\frac{1}{\sqrt{\chi}}\int_{1-\chi}^{1}\frac{\zeta^{1/2}}{\left(\zeta-1+\chi\right)^{1/4}\left(1-\zeta\right)^{1/4}}\mathrm{d}\zeta
\end{eqnarray*}
where the cross-ratio $\chi$ is defined by 
\begin{eqnarray*}
\chi & := & \frac{\eta_{0}-\eta_{\infty}}{\eta_{2}-\eta_{0}}\frac{\eta_{2}-\eta_{1}}{\eta_{1}-\eta_{\infty}}
\end{eqnarray*}
and the constants $C_{\mathrm{A}}$ and $C_{\mathrm{B}}$ are: 
\begin{eqnarray*}
C_{\mathrm{A}} & := & \frac{\sqrt{2\sqrt{2}+2}}{\pi}\frac{\Gamma\left(\frac{3}{4}\right)}{\Gamma\left(\frac{1}{4}\right)},\\
C_{\mathrm{B}} & := & \frac{2\sqrt{1+\sqrt{2}}}{\pi^{3/2}}.
\end{eqnarray*}

\end{prop}

\subsection{Convergence of the discrete martingale observable\label{sub:disc-mart-obs-cv}}

From the four propositions of the previous subsection, we obtain the
proof of Theorem \ref{thm:disc-obs-to-cts-obs}:
\begin{itemize}
\item By Proposition \ref{pro:representation-obs-as-correlation}, we can
represent $\Phi_{\delta}$ as a ratio of discrete spin correlation
functions of the dual Ising model.
\item By Proposition \ref{pro:int-rep-corr}, we can represent the discrete
spin correlation functions as expectations of simple correlation functions
computed on random domains determined by FK interfaces.
\item By Proposition \ref{pro:conv-discrete-expectations}, the FK expectations,
renormalized by $\frac{1}{\sqrt{\delta}}$, converge to SLE expectations.
\item By Proposition \ref{pro:sle-averages}, the sum of the two SLE expectations
is equal to $\Phi$. 
\end{itemize}
We hence deduce the theorem: $\frac{1}{\sqrt{\delta}}\Phi_{\delta}\left(\Omega_{\delta},r_{\delta},\ell_{\delta},x_{\delta},z_{\delta}\right)\to\Phi\left(\Omega,r,\ell,x,z\right)$
as $\delta\to0$.

\section{\label{sec:disc-mart-property}The discrete martingale property}

In this section, we prove the martingale property of the discrete
observable.

\subsection{\label{sub:low-t-expansion}Low-temperature expansion}

Let us first give a graphical representation of the discrete observable
$\Phi_{\delta}\left(\Omega_{\delta},r,\ell,x,z\right)$ defined in
Section \ref{sub:disc-obs} as 
\[
\frac{\tilde{\mathcal{Z}}\left(\Omega_{\delta},r,\ell,x,z\right)}{\mathcal{Z}\left(\Omega_{\delta},r,\ell,x\right)},
\]
where the denominator is the partition function of the critical Ising
model with dipolar boundary conditions and the numerator is the partition
function of the critical Ising model with modified boundary conditions
(see Definition \ref{def:disc-obs}).

We call a collection of edges of $\Omega_{\delta}$ a \emph{contour}.
The following lemma gives a contour representation of $\Phi_{\delta}$,
known as low-temperature expansion:
\begin{lem}
\label{lem:low-t-phi}We have 
\[
\frac{\tilde{\mathcal{Z}}\left(\Omega_{\delta},r,\ell,x,z\right)}{\mathcal{Z}\left(\Omega_{\delta},r,\ell,x\right)}=\frac{\sum_{\tilde{\omega}\in\tilde{\mathcal{C}}}\alpha^{\left|\tilde{\omega}\right|}}{\sum_{\omega\in\mathcal{C}}\alpha^{\left|\omega\right|}},
\]
where $\alpha=\exp\left(-2\beta_{c}\right)=\sqrt{2}-1$ and where
\begin{itemize}
\item $\mathcal{C}$ is the set of contours arising as interfaces (between
$+$ and $-$ spins) of the configuration of spins appearing in $\mathcal{Z}\left(\Omega_{\delta},r,\ell,x\right)$
(treating the spins on $\left[\ell,x\right]$ as $-$ and the spins
on $\left[x,r\right]$ as $+$): $\mathcal{C}$ is the set of contours
$\omega$ such that each vertex of $\Omega_{\delta}\setminus\left(\left\{ x\right\} \cup\left[r,\ell\right]\right)$
belongs to an even number of edges of $\omega$ and such that $x$
belongs to an odd number of edges of $\omega$.
\item $\tilde{\mathcal{C}}$ is the set of contours arising as interfaces
of the spin configurations appearing in $\tilde{\mathcal{Z}}\left(\Omega_{\delta},r,\ell,x,z\right)$
(treating the spins on $\left[\ell,x\right]\cup\left[z,r\right]$
as $-$ and the spins on $\left[x,z\right]$ as $+$): $\tilde{\mathcal{C}}$
is the set of contours $\tilde{\omega}$ such that each vertex of
$\Omega_{\delta}\setminus\left(\left\{ x,z\right\} \cup\left[r,\ell\right]\right)$
belongs to an even number of edges of $\tilde{\omega}$ and such that
$x$ and $z$ belong to an odd number of edges of $\tilde{\omega}$.
\end{itemize}
\end{lem}
\begin{proof}
Let $\mathcal{E}$ denote the set of edges of $\Omega_{\delta}$ that
are not on $\left[r,\ell\right]$. 

By definition of the contour sets as interfaces between the spins,
each edge $e\in\mathcal{E}$ present in a configuration $\omega\in\mathcal{C}$
corresponds to a pair of adjacent spins $i\sim j$ that are of opposite
signs and hence the contribution $-\sigma_{i}\sigma_{j}$ to the energy
(defined by Equation \ref{eq:hamiltonian-ising}) of that pair of
spins is $+1$ (the pairs of spins possibly include the spins on $\left[\ell,x\right]$
and the ones on $\left[x,r\right]$). On the other hand, each vacant
edge in a configuration $\omega\in\mathcal{C}$ (i.e. each edge of
$\Omega_{\delta}$ that does not belong to $\omega$) corresponds
to a pair of adjacent spins that have the same sign, and the corresponding
contribution of that pair to the energy is $-1$.

As the number of pairs we are summing over is $\left|\mathcal{E}\right|$
(the number of elements in $\mathcal{E}$), we have that the energy
$\mathbf{H}\left(\sigma\right)$ of a spin configuration $\sigma$
corresponding to a contour $\omega\in\mathcal{C}$ is equal to $2\left|\omega\right|-\left|\mathcal{E}\right|$.
By exactly the same considerations, the energy of a spin configuration
$\tilde{\sigma}$ corresponding to a contour $\tilde{\omega}\in\tilde{\mathcal{C}}$
is equal to $2\left|\tilde{\omega}\right|-\left|\mathcal{E}\right|$.
Hence, we obtain
\begin{eqnarray*}
\frac{\tilde{\mathcal{Z}}\left(\Omega_{\delta},r,\ell,x,z\right)}{\mathcal{Z}\left(\Omega_{\delta},r,\ell,x\right)} & = & \frac{\sum_{\tilde{\omega}\in\tilde{\mathcal{C}}}\exp\left(-\beta\left(2\left|\tilde{\omega}\right|-\left|\mathcal{E}\right|\right)\right)}{\sum_{\omega\in\mathcal{C}}\exp\left(-\beta\left(2\left|\omega\right|-\left|\mathcal{E}\right|\right)\right)}\\
 & = & \frac{\sum_{\tilde{\omega}\in\tilde{\mathcal{C}}}\alpha^{\left|\tilde{\omega}\right|}}{\sum_{\omega\in\mathcal{C}}\alpha^{\left|\omega\right|}}.
\end{eqnarray*}
\end{proof}
\begin{rem}
\label{rem:low-t-contours}The contours in $\mathcal{C}$ contain
(possibly non-simple) loops, with an interface from $x$ to $\left[r,\ell\right]$
and possibly arcs pairing vertices of $\left[r,\ell\right]$ (see
Figure \ref{fig:contour-c}). The contours in $\tilde{\mathcal{C}}$
contain loops and arcs, plus either an interface from $x$ to $z$
or two interfaces, emanating at $x$ and $z$, and both ending on
$\left[r,\ell\right]$ (see Figure \ref{fig:contour-c-tilde}).

\begin{figure}
\includegraphics[width=7cm]{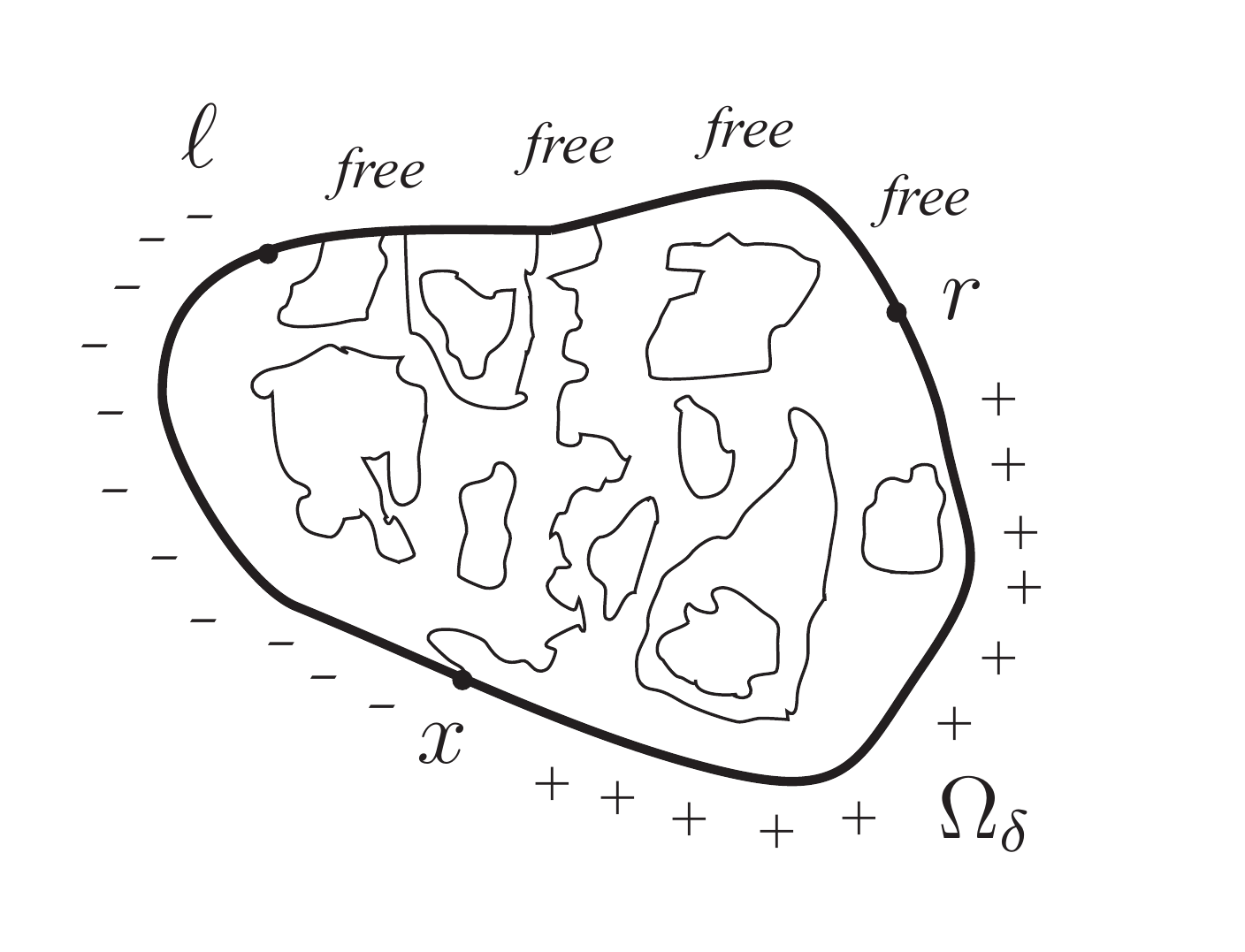}

\caption{\label{fig:contour-c}Schematic sketch of the contours in $\mathcal{C}$.}
\end{figure}

\end{rem}
\begin{figure}
\includegraphics[width=12cm]{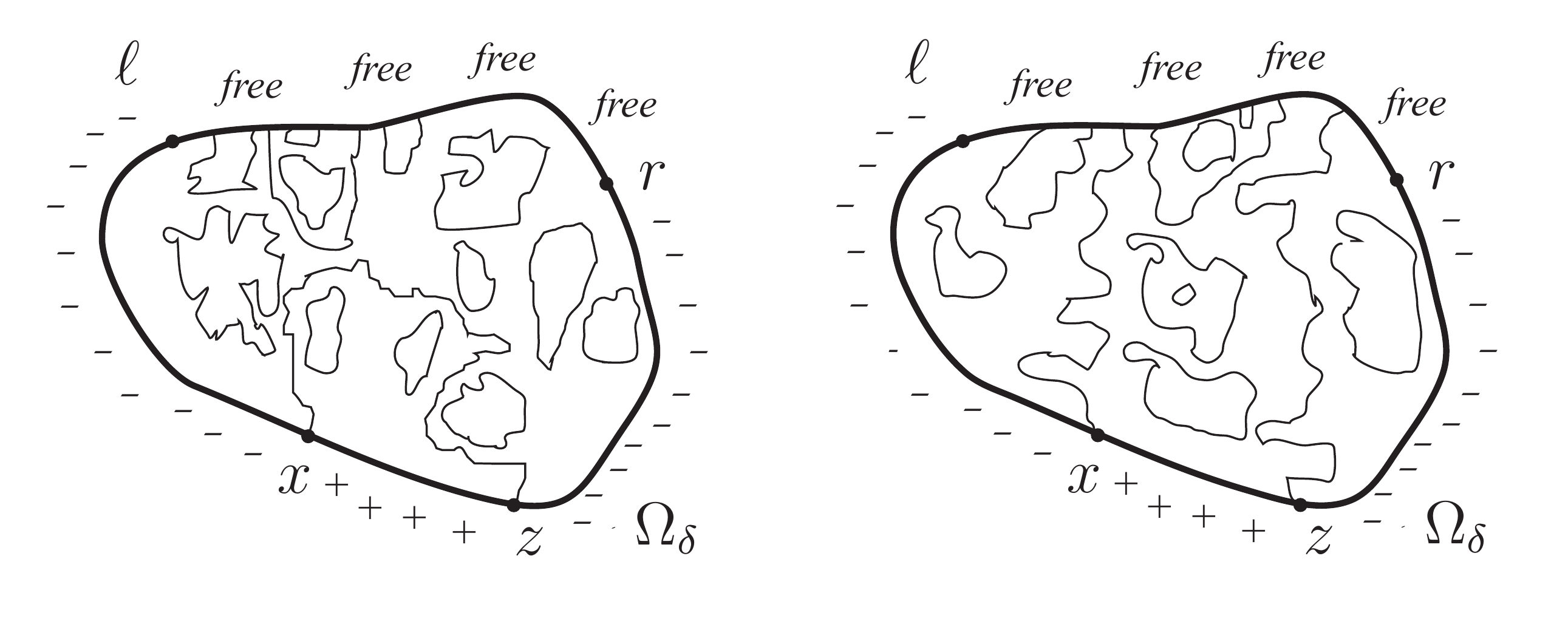}

\caption{\label{fig:contour-c-tilde}Schematic sketch of the contours in $\tilde{\mathcal{C}}$.}

\end{figure}

\subsection{\label{sub:proof-disc-mart-property}Proof of Proposition \ref{pro:discrete-martingale}}

Thanks to the low-temperature expansion detailed in the previous subsection,
we can now prove the discrete martingale property:
\begin{prop*}[Proposition \ref{pro:discrete-martingale}]
Let $\left(\gamma_{\delta}\left(n\right)\right)_{n\geq0}$ have the
law of the interface arising at $b$ in the critical Ising model on
$\left(D_{\delta},r,\ell,b\right)$ with dipolar boundary conditions,
parametrized by the number of steps. For any $z\in\left[b,r\right]$,
\[
\left(\Phi_{\delta}\left(D_{\delta}\setminus\gamma_{\delta}\left[0,n\right],r,\ell,\gamma_{\delta}\left(n\right),z\right)\right)_{n\wedge\tau\left(z\right)}
\]
 is a discrete time martingale, where 
\[
\tau\left(z\right):=\inf\left\{ n:\gamma_{\delta}\left(n\right)\in\left[z,\ell\right]\right\} .
\]

\end{prop*}
The discrete domain $D_{\delta}\setminus\gamma_{\delta}\left[0,n\right]$
is defined as the connected component of 
\[
D_{\delta}\setminus\left(\gamma_{\delta}^{\mbox{Left Side}}\cup\gamma_{\delta}^{\mbox{Right Side}}\right)
\]
 that contains $r$, $\ell$ and $z$, where
\begin{itemize}
\item $\gamma_{\delta}^{\mbox{Left Side}}$ is the set of faces of $D_{\delta}$
that either share an edge or a vertex with the left side of $\gamma_{\delta}\left[0,n-1\right]$;
\item $\gamma_{\delta}^{\mbox{Right Side}}$ is the set of faces of $D_{\delta}$
that share an edge with the right side of $\gamma_{\delta}\left[0,n\right]$.\end{itemize}
\begin{proof}[Proof of Proposition \ref{pro:discrete-martingale}]
We prove that, conditionally on the event that $0\leq n<\tau\left(z\right)$
(if $n\geq\tau\left(z\right)$, there is nothing to prove):
\[
\mathbb{E}\left[\Phi_{\delta}\left(D_{\delta}\setminus\gamma_{\delta}\left[0,n+1\right],r,\ell,\gamma_{\delta}\left(n+1\right),z\right)|\gamma_{\delta}\left[0,n\right]\right]=\Phi_{\delta}\left(D_{\delta}\setminus\gamma_{\delta}\left[0,n\right],r,\ell,\gamma_{\delta}\left(n\right),z\right).
\]
By Lemma \ref{lem:low-t-phi}, we obtain 
\begin{eqnarray*}
\frac{\tilde{\mathcal{Z}}\left(D_{\delta}\setminus\gamma_{\delta}\left[0,n\right],r,\ell,\gamma_{\delta}\left(n\right),z\right)}{\mathcal{Z}\left(D_{\delta}\setminus\gamma_{\delta}\left[0,n\right],r,\ell,\gamma_{\delta}\left(n\right)\right)} & = & \frac{\sum_{\omega\in\tilde{\mathcal{C}}}\alpha^{\left|\omega\right|}}{\sum_{\omega\in\mathcal{C}}\alpha^{\left|\omega\right|}},
\end{eqnarray*}
where $\alpha=\sqrt{2}-1$ and $\mathcal{C},\tilde{\mathcal{C}}$
are as in Lemma \ref{lem:low-t-phi} (where the domain $\Omega_{\delta}$
is now $D_{\delta}\setminus\gamma_{\delta}\left[0,n\right]$). 

Let $L,S,R$ be the three vertices adjacent that are the possible
values for $\gamma_{\delta}\left(n+1\right)$, if the interface turns
left, goes straight or turns right after $\gamma_{\delta}\left(n\right)$
(see Figure \ref{fig:disc-exploration}) and let $e_{L},e_{S},e_{R}$
be the edges from $\gamma_{\delta}\left(n\right)$ to $L,S,R$. Every
contour configuration in $\mathcal{C}$ or $\tilde{\mathcal{C}}$
contains an interface emanating from $\gamma_{\delta}\left(n\right)$,
and we hence split $\mathcal{C}$ as $\mathcal{C}_{L}\cup\mathcal{C}_{S}\cup\mathcal{C}_{R}$
and $\tilde{\mathcal{C}}$ as $\tilde{\mathcal{C}}_{L}\cup\tilde{\mathcal{C}}_{S}\cup\tilde{\mathcal{C}}_{R}$,
where
\begin{eqnarray*}
\mathcal{C}_{L}:=\left\{ \omega\in\mathcal{C}:e_{L}\in\omega\right\} , & \mathcal{C}_{S}:=\left\{ \omega\in\mathcal{C}:e_{S}\in\omega,e_{L}\notin\omega\right\} , & \mathcal{C}_{R}:=\left\{ \omega\in\mathcal{C}:e_{R}\in\omega,e_{L}\notin\omega\right\} ,\\
\tilde{\mathcal{C}}_{L}:=\left\{ \omega\in\tilde{\mathcal{C}}:e_{L}\in\omega\right\} , & \tilde{\mathcal{C}}_{S}:=\left\{ \omega\in\mathcal{\tilde{C}}:e_{S}\in\omega,e_{L}\notin\omega\right\} , & \tilde{\mathcal{C}}_{R}:=\left\{ \omega\in\tilde{\mathcal{C}}:e_{R}\in\omega,e_{L}\notin\omega\right\} .
\end{eqnarray*}
Notice that the slight asymmetry of these definitions follows our
convention of turning left whenever there is an ambiguity. Writing
\[
\mathcal{Z}_{L}:=\sum_{\omega\in\mathcal{C}_{L}}\alpha^{\left|\omega\right|},\,\,\,\,\tilde{\mathcal{Z}}_{L}:=\sum_{\omega\in\tilde{\mathcal{C}}_{L}}\alpha^{\left|\omega\right|},\,\,\,\,\mathcal{Z}_{S}:=\sum_{\omega\in\mathcal{C}_{S}}\alpha^{\left|\omega\right|},\ldots
\]
we get $\mathcal{Z}=\mathcal{Z}_{L}+\mathcal{Z}_{S}+\mathcal{Z}_{R}$
and $\tilde{\mathcal{Z}}=\tilde{\mathcal{Z}}_{L}+\tilde{\mathcal{Z}}_{S}+\tilde{\mathcal{Z}}_{R}$
and it is easy to check that we have 
\begin{eqnarray*}
\mathbf{P}_{\mathrm{Left}}:=\mathbb{P}\left[\gamma_{\delta}\left(n+1\right)=L|\gamma_{\delta}\left[0,n\right]\right] & = & \frac{\mathcal{Z}_{L}}{\mathcal{Z}},\\
\mathbf{P}_{\mathrm{Straight}}:=\mathbb{P}\left[\gamma_{\delta}\left(n+1\right)=S|\gamma_{\delta}\left[0,n\right]\right] & = & \frac{\mathcal{Z}_{S}}{\mathcal{Z}},\\
\mathbf{P}_{\mathrm{Right}}:=\mathbb{P}\left[\gamma_{\delta}\left(n+1\right)=R|\gamma_{\delta}\left[0,n\right]\right] & = & \frac{\mathcal{Z}_{R}}{\mathcal{Z}}.
\end{eqnarray*}
Now remark that 
\[
\Phi_{\delta}\left(D\setminus\gamma_{\delta}\left[0,n+1\right]\right)=\begin{cases}
\frac{\tilde{\mathcal{Z}}_{L}}{\mathcal{Z}_{L}} & \mbox{ on event }L,\\
\frac{\tilde{\mathcal{Z}}_{S}}{\mathcal{Z}_{S}} & \mbox{ on event }S,\\
\frac{\tilde{\mathcal{Z}}_{R}}{\mathcal{Z}_{R}} & \mbox{ on event }R.
\end{cases}
\]
The martingale property is thus verified by the following calculation:
\begin{eqnarray*}
 &  & \mathbb{E}\left[\Phi_{\delta}\left(D\setminus\gamma_{\delta}\left[0,n+1\right],r,\ell,\gamma_{\delta}\left(n+1\right),z\right)|\gamma_{\delta}\left[0,n\right]\right]\\
 & = & \mathbf{P}_{\mathrm{Left}}\frac{\tilde{\mathcal{Z}}_{L}}{\mathcal{Z}_{L}}+\mathbf{P}_{\mathrm{Straight}}\frac{\tilde{\mathcal{Z}}_{S}}{\mathcal{Z}_{S}}+\mathbf{P}_{\mathrm{Right}}\frac{\tilde{\mathcal{Z}}_{R}}{\mathcal{Z}_{R}}\\
 & = & \frac{\mathcal{Z}_{L}}{\mathcal{Z}}\frac{\tilde{\mathcal{Z}}_{L}}{\mathcal{Z}_{L}}+\frac{\mathcal{Z}_{S}}{\mathcal{Z}}\frac{\tilde{\mathcal{Z}}_{S}}{\mathcal{Z}_{S}}+\frac{\mathcal{Z}_{R}}{\mathcal{Z}}\frac{\tilde{\mathcal{Z}}_{R}}{\mathcal{Z}_{R}}=\frac{\tilde{\mathcal{Z}}_{L}+\tilde{\mathcal{Z}}_{S}+\tilde{\mathcal{Z}}_{R}}{\mathcal{Z}}\\
 & = & \Phi_{\delta}\left(D\setminus\gamma_{\delta}\left[0,n\right],r,\ell,\gamma_{\delta}\left(n\right),z\right).
\end{eqnarray*}

\end{proof}
\begin{figure}

\includegraphics[width=9cm]{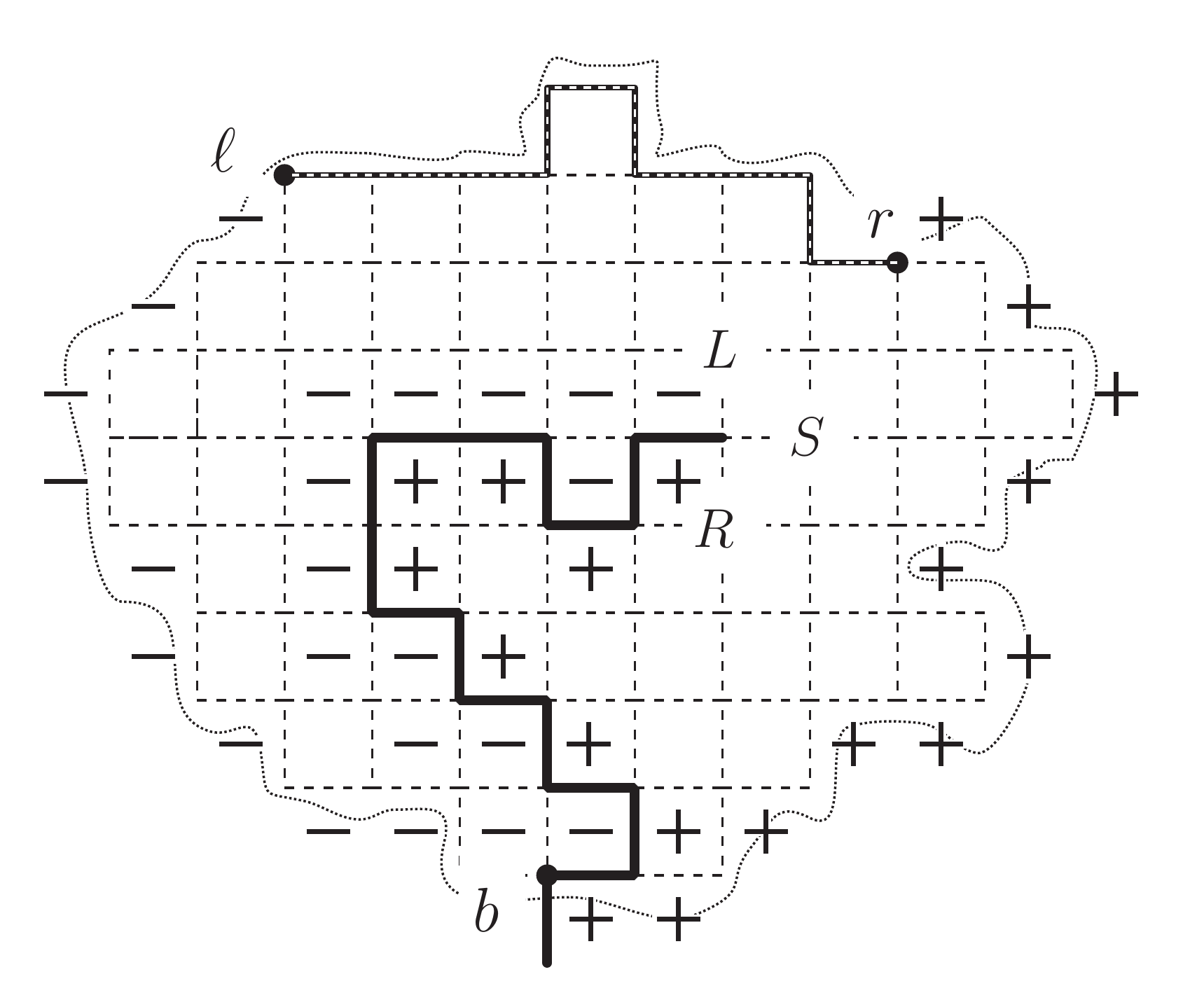}

\caption{\label{fig:disc-exploration}}

\end{figure}

\section{Kramers-Wannier duality and spin correlations\label{sec:kw-duality}}

In this section, we use Kramers-Wannier duality \cite{kramers-wannier}
to represent the observable $\Phi_{\delta}\left(\Omega_{\delta},r,\ell,x,z\right)$
as a ratio of spin correlations on the vertices of $\Omega_{\delta}$
with dual boundary conditions. This section uses in a crucial way
the fact that we are at the critical temperature. For more details
about Kramers-Wannier duality, see \cite{kadanoff-ceva}, or \cite[Chapter 1]{palmer}.

To understand the quantity $\Phi_{\delta}\left(\Omega_{\delta},r,\ell,x,z\right)$,
which is a ratio of partition functions of Ising models living on
the faces of $\Omega_{\delta}$ (as in Figure \ref{fig:disc-exploration})
we need to introduce a dual Ising model (see Figure \ref{fig:dual-ising}),
which lives on the vertices of $\Omega_{\delta}$: it is defined exactly
like before, with the probability of a spin configuration $\left(\sigma_{x}\right)_{x\in\mathcal{V}}$
proportional to $\exp\left(-\beta\mathbf{H}\left(\sigma\right)\right)$,
where 
\[
\mathbf{H}\left(\sigma\right):=-\sum_{x\sim y}\sigma_{x}\sigma_{y},
\]
the sum being over all pairs of adjacent vertices. The boundary conditions
that we will need are somewhat simpler: we simply condition the vertices
on the arc $\left[r,\ell\right]$ to $+1$, and let free the spins
on the arc $\left[\ell,r\right]\setminus\left\{ \ell,r\right\} $
(which we will denote $\left[\ell,r\right]$ for convenience). 

\begin{figure}

\includegraphics[width=11cm]{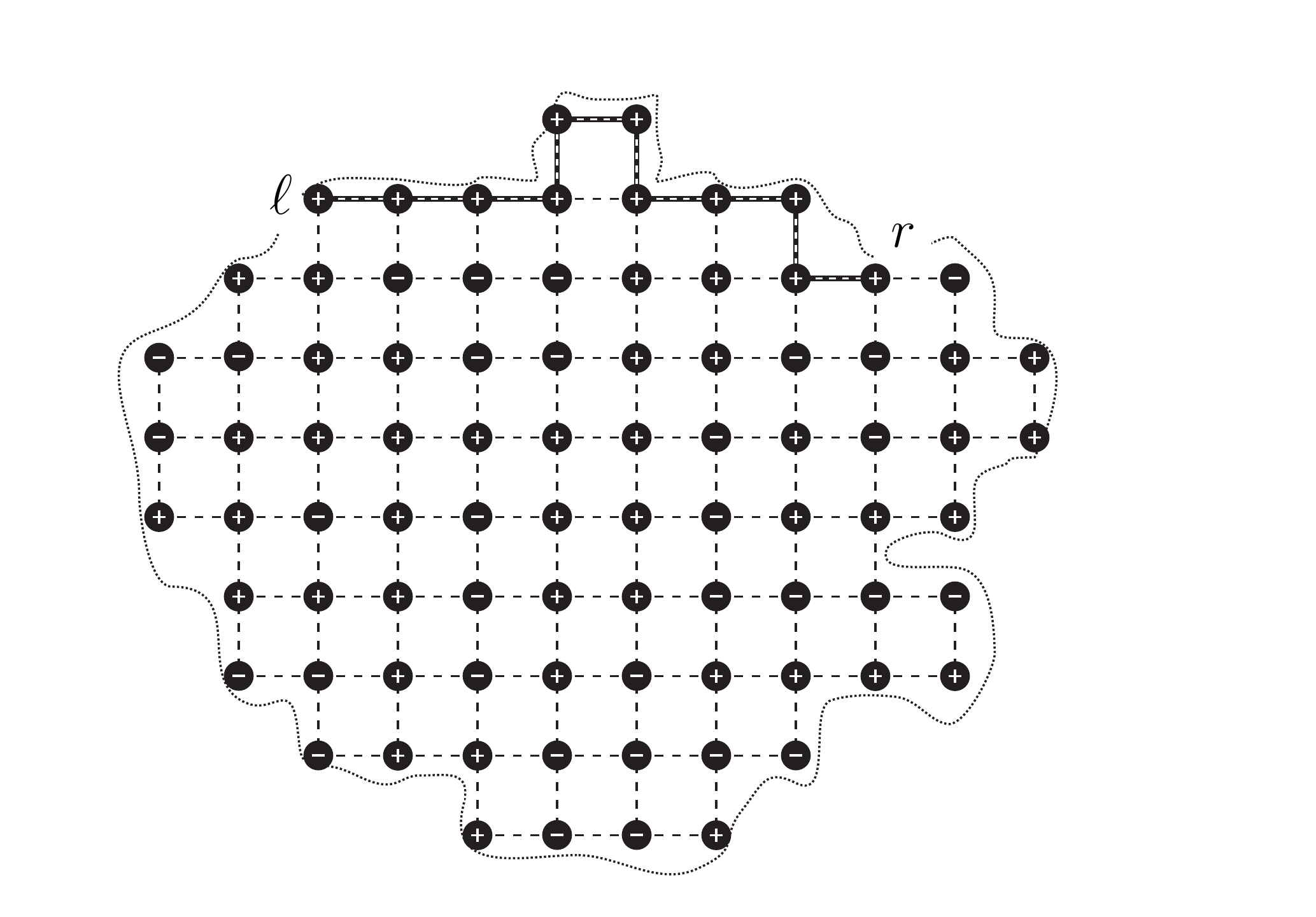}

\caption{\label{fig:dual-ising}Dual Ising model living on the vertices of
$\Omega_{\delta}$, with dual boundary conditions.}

\end{figure}

\begin{prop*}[Proposition \ref{pro:representation-obs-as-correlation}]
If we consider the Ising model on the vertices of the graph $\Omega_{\delta}$
with $+$ boundary condition on the arc $\left[r,\ell\right]$ and
free boundary condition on the arc $\left[\ell,r\right]$, we have
\[
\frac{\mathbb{E}_{\Omega_{\delta}}^{\left[r,\ell\right]_{+}}\left[\sigma_{x}\sigma_{z}\right]}{\mathbb{E}_{\Omega_{\delta}}^{\left[r,\ell\right]_{+}}\left[\sigma_{x}\right]}=\Phi_{\delta}\left(\Omega_{\delta},r,\ell,x,z\right).
\]
\end{prop*}
\begin{proof}
We use the Kramers-Wannier duality technique, also known as the high-temperature
expansion of the Ising model. Starting from the left-hand side, we
have
\begin{eqnarray*}
\frac{\mathbb{E}_{\Omega_{\delta}}^{\left[r,\ell\right]_{+}}\left[\sigma_{x}\sigma_{z}\right]}{\mathbb{E}_{\Omega_{\delta}}^{\left[r,\ell\right]_{+}}\left[\sigma_{x}\right]} & = & \frac{\sum_{\sigma}\sigma_{x}\sigma_{z}\prod_{\left\langle i,j\right\rangle \in\mathcal{E}}\exp\left(\beta_{c}\sigma_{i}\sigma_{j}\right)}{\sum_{\sigma}\sigma_{x}\prod_{\left\langle i,j\right\rangle \in\mathcal{E}}\exp\left(\beta_{c}\sigma_{i}\sigma_{j}\right)}\\
 & = & \frac{\sum_{\sigma}\sigma_{x}\sigma_{z}\prod_{\left\langle i,j\right\rangle \in\mathcal{E}}\left(\cosh\beta_{c}+\sigma_{i}\sigma_{j}\sinh\beta_{c}\right)}{\sum_{\sigma}\sigma_{x}\prod_{\left\langle i,j\right\rangle \in\mathcal{E}}\left(\cosh\beta_{c}+\sigma_{i}\sigma_{j}\sinh\beta_{c}\right)}\\
 & = & \frac{\sum_{\sigma}\sigma_{x}\sigma_{z}\prod_{\left\langle i,j\right\rangle \in\mathcal{E}}\left(1+\alpha\sigma_{i}\sigma_{j}\right)}{\sum_{\sigma}\sigma_{x}\prod_{\left\langle i,j\right\rangle \in\mathcal{E}}\left(1+\alpha\sigma_{i}\sigma_{j}\right)}\\
 & = & \frac{\sum_{\sigma}\sigma_{x}\sigma_{z}\sum_{E\subset\mathcal{E}}\alpha^{\left|E\right|}\prod_{\left\langle i,j\right\rangle \in E}\sigma_{i}\sigma_{j}}{\sum_{\sigma}\sigma_{x}\sum_{E\subset\mathcal{E}}\alpha^{\left|E\right|}\prod_{\left\langle i,j\right\rangle \in E}\sigma_{i}\sigma_{j}}\\
 & = & \frac{\sum_{E\subset\mathcal{E}}\alpha^{\left|E\right|}\sum_{\sigma}\sigma_{x}\sigma_{z}\prod_{\left\langle i,j\right\rangle \in E}\sigma_{i}\sigma_{j}}{\sum_{E\subset\mathcal{E}}\alpha^{\left|E\right|}\sum_{\sigma}\sigma_{x}\prod_{\left\langle i,j\right\rangle \in E}\sigma_{i}\sigma_{j}}\\
 & = & \frac{\sum_{E\in\tilde{\mathcal{C}}}\alpha^{\left|E\right|}}{\sum_{E\in\mathcal{C}}\alpha^{\left|E\right|}}\\
 & = & \Phi_{\delta}\left(\Omega_{\delta},r,\ell,x,z\right)
\end{eqnarray*}
For the second equation, we used that $\exp=\sinh+\cosh$, parity
of $\sinh$ and $\cosh$ and $\sigma_{i}\sigma_{j}=\pm1$. For the
third one, we used that $\alpha=\sqrt{2}-1=\tanh\beta_{c}$. For the
fourth, we expanded the product over all the edges.

The subtle point is the sixth; let us first look at the denominator.
We used that if $E\subset\mathcal{E}$ is such that a vertex $v\in\Omega\setminus\left[r,\ell\right]$
arises an odd number of times in the product $\sigma_{x}\prod_{\left\langle i,j\right\rangle \in E}\sigma_{i}\sigma_{j}$,
then the sum over all the spin configurations of this product vanishes
by symmetry; hence the only $E\subset\mathcal{E}$ giving a nonzero
contribution (which is then the number of possible spin configurations),
are the contours in $\mathcal{C}$ (defined in Lemma \ref{lem:low-t-phi}):
the vertex $x$ must belong to an odd number of edges, the vertices
in $\Omega_{\delta}\setminus\left(\left[r,\ell\right]\cup\left\{ x\right\} \right)$
must belong to an even number of edges and the vertices on the arc
$\left[r,\ell\right]$ are free to belong to an arbitrary number of
edges (since we are not summing over the spins at these vertices,
which are set to $+1$, because of the boundary condition). Hence,
it is easy to see that the $E\subset\mathcal{E}$ which have nonzero
contributions are precisely the $E\in\mathcal{C}$. Similarly, the
terms giving a nonzero contribution to the numerators are the $E\subset\mathcal{E}$
such that the vertices $x$ and $z$ belong to an odd number of edges
of $E$, the vertices in $\Omega_{\delta}\setminus\left(\left[r,\ell\right]\cup\left\{ x,z\right\} \right)$
belong to an even number of edges of $E$ and the vertices in $\left[r,\ell\right]$
are free to belong to an arbitrary number of edges: those are precisely
the contours in $\tilde{\mathcal{C}}$ (defined in Lemma \ref{lem:low-t-phi}).
\end{proof}

\section{FK representation and connection events\label{sec:fk-representation}}

In the previous section, we showed (Proposition \ref{pro:representation-obs-as-correlation})
that

\[
\Phi_{\delta}\left(\Omega_{\delta},r,\ell,x,z\right)=\frac{\mathbb{E}_{\Omega_{\delta}}^{\left[r,\ell\right]_{+}}\left[\sigma_{x}\sigma_{z}\right]}{\mathbb{E}_{\Omega_{\delta}}^{\left[r,\ell\right]_{+}}\left[\sigma_{x}\right]}.
\]
In this section, we use the Fortuin-Kasteleyn representation of the
Ising model to express the correlation functions of the right-hand
side as a sum of correlation functions of simpler form (in the sense
of boundary conditions) computed on random domains. Notice that this
part does not use the fact that the temperature is critical.

\subsection{The FK model\label{sub:fk-model}}

The Fortuin-Kasteleyn (FK) model on a graph $\mathcal{G}=\left(\mathcal{V},\mathcal{E}\right)$
is a bond percolation model (i.e. a random subset of edges of $\mathcal{E}$)
with two parameters $p\in\left[0,1\right]$ and $q\geq0$, where the
probability of an edge configuration $\omega\subset\mathcal{E}$ is
proportional to 
\[
\left(\frac{p}{1-p}\right)^{\#\mathrm{edges}\left(\omega\right)}q^{\#\mathrm{clusters}\left(\omega\right)},
\]
where the clusters of $\omega$ are the connected components of the
graph $\left(\mathcal{V},\omega\right)$ (the graph with vertices
$\mathcal{V}$ and edges $\omega$). Given a (deterministic) subset
$\mathfrak{b}$ of $\mathcal{V}$ (typically a part of the boundary
of $\mathcal{V}$, if $\mathcal{V}$ is a domain) can introduce \emph{wired}
boundary condition, by declaring the vertices of $\mathfrak{b}$ to
be in the same cluster (even if they are not linked by an edge in
$\mathcal{E}$). When we do not wire boundary vertices, we call them
\emph{free}. 

We will be interested in connection events for the FK model: for two
vertices $a,b\in\mathcal{V}$ , we will denote by $\left\{ a\leftrightsquigarrow b\right\} $
the event that $a$ and $b$ belong to the same cluster of the FK
configuration. We will denote, for $B\subset\mathcal{V}$, by $\left\{ a\leftrightsquigarrow B\right\} $
the event that $\left\{ a\leftrightsquigarrow b\right\} $ occurs
for some $b\in B$. 

When $q=2$, the FK model provides a graphical representation of the
Ising model at inverse temperature $\beta=\ln\left(\frac{1}{\sqrt{1-p}}\right)$
-- the $p$ corresponding to $\beta_{c}=\frac{1}{2}\ln\left(\sqrt{2}+1\right)$
is hence $p_{c}=\frac{\sqrt{2}}{\sqrt{2}+1}$. In this paper, we will
always assume that $q=2$, and we will only be interested in the $p=p_{c}$
case. 
\begin{thm}[Fortuin-Kastelyn]
\label{thm:fk-to-ising}Let $\omega\subset E$ have the law of an
FK configuration on the domain $\left(\Omega_{\delta},r,\ell\right)$
with parameters $p=1-e^{-2\beta}$ and $q=2$, with wired boundary
condition on $\left[r,\ell\right]$ and free boundary condition on
$\left[\ell,r\right]$. Assign the spin value $+1$ to each of the
vertices of the cluster containing $\left[r,\ell\right]$. For each
other cluster of $\omega$, assign the same $\pm1$ value to the spin
at the vertices of this cluster, with probability $\frac{1}{2}-\frac{1}{2}$,
independently of the other clusters. Then the law of the spin configuration
$\sigma\in\left\{ \pm1\right\} ^{\Omega_{\delta}}$ obtained via this
procedure is that of an Ising model with inverse temperature $\beta$,
$+$ boundary condition on $\left[r,\ell\right]$ and free boundary
condition on $\left[\ell,r\right]$. 
\end{thm}
See \cite[Chapter 1]{grimmett} for a proof.

We call the FK model with $q=2$ \emph{the FK-Ising model}, or just
the FK model for short. When in addition $p=p_{c}$, we will refer
to it as the \emph{critical FK-Ising model}.

\subsection{Correlation functions as connection probabilities}

In this subsection, we use Theorem \ref{thm:fk-to-ising} to give
an FK representation of the Ising spin correlations: the FK-Ising
model allows to understand how the influence between spins spreads
through the graph.
\begin{lem}
\label{lem:spin-corr-as-connections}If we consider the Ising model
(at any temperature) on $\left(\Omega_{\delta},r,\ell\right)$ with
$+$ boundary condition on $\left[r,\ell\right]$ and free boundary
condition on $\left[\ell,r\right]$ and the corresponding FK model
on $\left(\Omega_{\delta},r,\ell\right)$ with wired boundary condition
on $\left[r,\ell\right]$ and free boundary condition on $\left[\ell,r\right]$,
then for any $x,z\in\Omega_{\delta}$ we have
\begin{eqnarray*}
\mathbb{E}_{\Omega_{\delta}}^{\left[r,\ell\right]_{+}}\left[\sigma_{x}\right] & = & \mathbb{P}_{\Omega_{\delta}}^{\left[r,\ell\right]_{\mathrm{w}}}\left\{ x\leftrightsquigarrow\left[r,\ell\right]\right\} \\
\mathbb{E}_{\Omega_{\delta}}^{\left[r,\ell\right]_{+}}\left[\sigma_{x}\sigma_{z}\right] & = & \mathbb{P}_{\Omega_{\delta}}^{\left[r,\ell\right]_{\mathrm{w}}}\left\{ x\leftrightsquigarrow z\right\} \\
 & = & \mathbb{P}_{\Omega_{\delta}}^{\left[r,\ell\right]_{\mathrm{w}}}\left\{ x\leftrightsquigarrow z\leftrightsquigarrow\left[r,\ell\right]\right\} +\mathbb{P}_{\Omega_{\delta}}^{\left[r,\ell\right]_{\mathrm{w}}}\left\{ x\leftrightsquigarrow z\not\leftrightsquigarrow\left[r,\ell\right]\right\} 
\end{eqnarray*}
\end{lem}
\begin{proof}
We use the FK representation of the Ising model (Theorem \ref{thm:fk-to-ising}),
sampling an Ising configuration from an FK-Ising configuration. 

To obtain the first identity, it suffices to see that, conditionally
on $\left\{ x\leftrightsquigarrow\left[r,\ell\right]\right\} $, the
spin $\sigma_{x}$ takes the value $1$ (and hence is of expectation
$1$), and conditionally on $\left\{ x\not\leftrightsquigarrow\left[r,\ell\right]\right\} $,
the spin $\sigma_{x}$ takes the values $-1$ and $+1$ with equal
probabilites (hence is of expectation $0$). 

To obtain the second identity, notice that, conditionally on $\left\{ x\leftrightsquigarrow z\right\} $,
the spins $\sigma_{x}$ and $\sigma_{z}$ are the same (and hence
their expected product is $1$), that conditionally on $\left\{ x\not\leftrightsquigarrow z\right\} $,
they are independent (and since a centered $\pm1$ is sampled for
either $x$ or $z$ or both, the expected product is $0$).
\end{proof}

\subsection{Discrete vertex domains\label{sub:disc-vertex-domain}}

We will need to consider graphs that are slightly more general than
the discrete domains defined in Section \ref{sub:graph-domain}. Let
$\mathbb{C}_{\delta}$ be the square grid of mesh size $\delta$. 
\begin{itemize}
\item We call a subgraph $\Omega_{\delta}$ of $\mathbb{C}_{\delta}$ a
\emph{discrete vertex domain} if it is a connected and simply connected
(i.e. each face of $\Omega_{\delta}$ is a face of $\mathbb{C}_{\delta}$). 
\item We denote by $\partial\Omega_{\delta}$ the Jordan curve that lives
between $\Omega_{\delta}$ and the complementary of its dual (see
Figure \ref{fig:discrete-vertex-domain}). 
\item When needed, we will identify $\Omega_{\delta}$ with the Jordan domain
bounded by $\partial\Omega_{\delta}$.
\item We denote by $\partial_{0}\Omega_{\delta}$ the set of vertices of
$\Omega_{\delta}$ at distance less than $\frac{\delta}{2}$ to the
curve $\partial\Omega_{\delta}$.
\item We identify any given arc $\left[v,w\right]\subset\partial\Omega_{\delta}$
with the vertices of $\partial_{0}\Omega_{\delta}$ at distance less
than $\frac{\delta}{2}$ to $\left[v,w\right]$. 
\item When there is no ambiguity, we identify the vertices of $\partial_{0}\Omega_{\delta}$
with the closest points of $\partial\Omega_{\delta}$.
\end{itemize}
Let us remark that all the discrete domains (as defined in Section
\ref{sub:graph-domain}) are discrete vertex domains, but that the
converse is not true. 

\begin{figure}

\includegraphics[width=8cm]{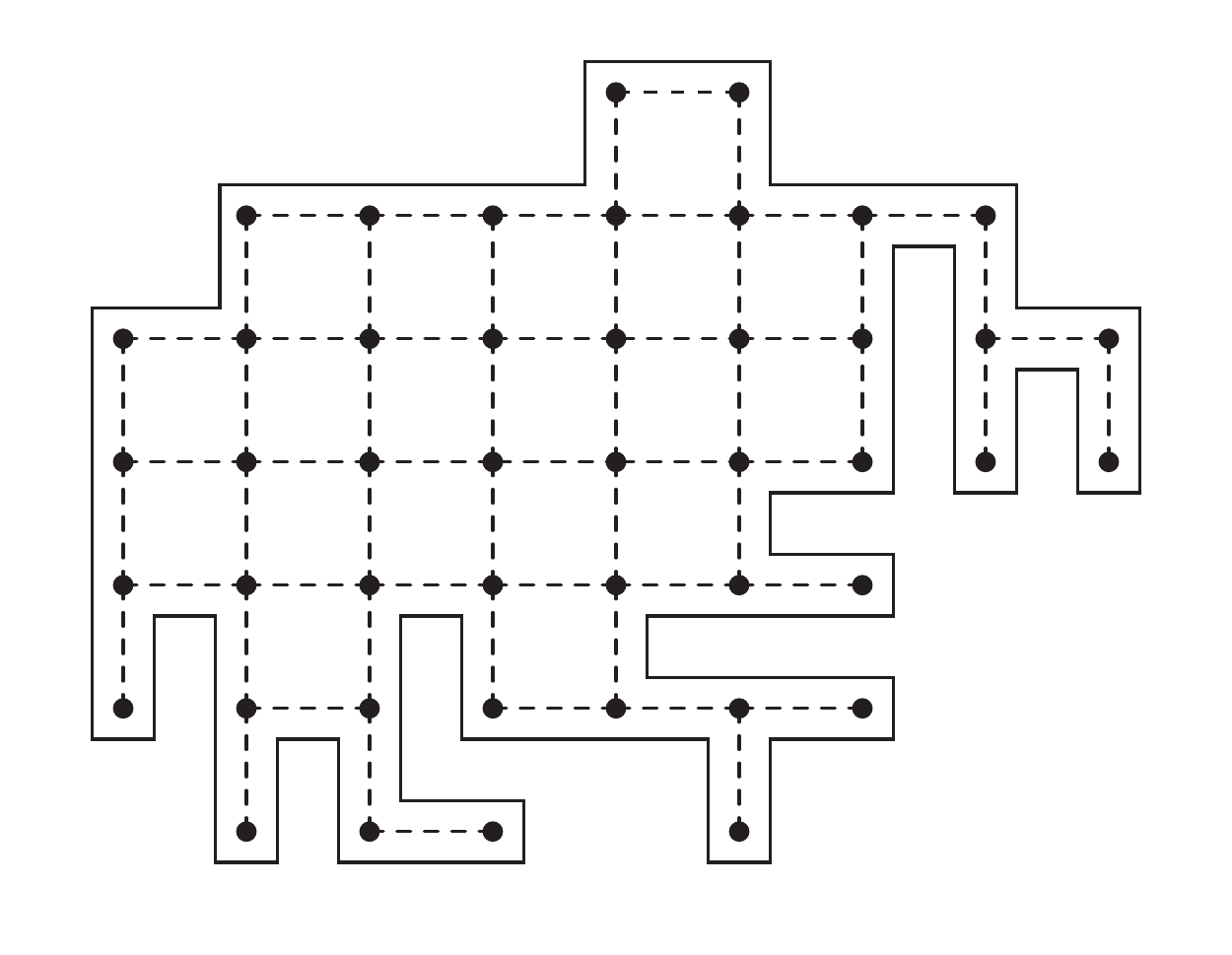}

\caption{\label{fig:discrete-vertex-domain}Discrete vertex domain.}

\end{figure}

\subsection{Interfaces, screening effects and random domains}

The FK setup of Theorem \ref{thm:fk-to-ising} and Lemma \ref{lem:spin-corr-as-connections},
with wired boundary condition on an arc $\left[r,\ell\right]$ of
a discrete domain $\left(\Omega_{\delta},r,\ell\right)$ and free
boundary condition on the other arc $\left[\ell,r\right]$ naturally
generates an interface $\lambda_{\delta}$, which is the boundary
of the FK cluster of the arc $\left[r,\ell\right]$, which links $r$
and $\ell$; we will always orient it from $\ell$ to $r$. The FK
interface lives between $\Omega_{\delta}$ and its dual graph (see
Figure \ref{fig:fk-interface}) and is qualitatively very different
from an Ising model interface: at critical temperature, its scaling
limit is SLE$\left(16/3\right)$.

\begin{figure}
\includegraphics[width=9cm]{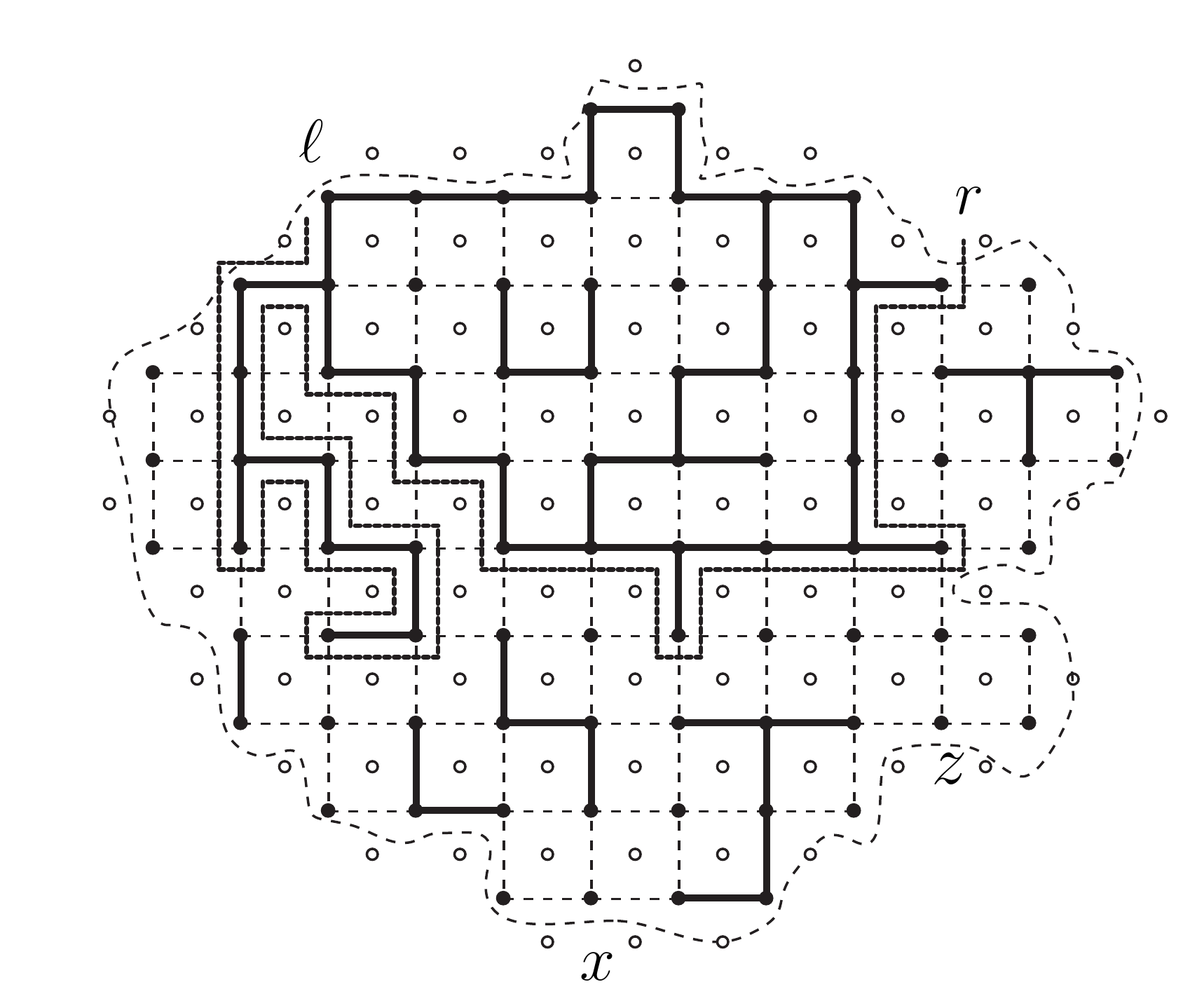}

\caption{\label{fig:fk-interface}The FK interface}

\end{figure}

The FK interface has convenient properties and is very well understood
thanks to discrete complex analysis techniques introduced by Smirnov
\cite{smirnov-ii,smirnov-i}, which is the reason why it plays a crucial
role in our analysis. We will in particular make essential use of
the following properties:
\begin{itemize}
\item The domain Markov property (also known as spatial Markov property):
we have equivalences between:

\begin{enumerate}
\item the conditional law of the FK model in $\left(\Omega_{\delta},r,\ell\right)$,
knowing the interface $\lambda_{\delta}\left[0,n\right]$ (or an initial
segment of it);
\item the law of the FK model on the connected components of $\Omega_{\delta}\setminus\lambda_{\delta}\left[0,n\right]$,
with each connected component disconnected from $\lambda_{\delta}\left(n\right)$
having either purely wired or purely free boundary conditions, depending
on whether it is on the left or the right of $\lambda_{\delta}\left[0,n\right]$. 
\end{enumerate}

Hence, we can split a domain with mixed boundary conditions into a
collection of random subdomains with simpler boundary conditions.
The domain Markov property is a direct consequence of a fundamental
screening property of the FK model, which asserts that the conditional
law of an FK configuration outside of some domain can be encoded through
boundary conditions (describing which boundary vertices are in the
same cluster). 

\item The boundary hitting probabilities: the event that the FK interface
$\lambda_{\delta}$ hits a point $y$ on the free arc $\left[\ell,r\right]$
(or more precisely: is such that $\lambda_{\delta}\cup\left[r,\ell\right]$
disconnects $y$ from $\infty$) is the same as the event that $\left\{ y\leftrightsquigarrow\left[r,\ell\right]\right\} $
(i.e. that $y$ belongs to the same cluster as $\left[r,\ell\right]$).
The boundary hitting probabilities can then be computed in the scaling
limit thanks to discrete complex analysis results concerning this
interface: the discrete holomorphic observable introduced in \cite{smirnov-ii},
which is complexified version of the passage probability, exactly
gives this passage probability on the boundary of the domain.
\end{itemize}
From the above properties, we immediately deduce the following lemma,
which will be instrumental in the next subsection:
\begin{lem}
\label{lem:interface-passage-probabilities}With the notation of Lemma
\ref{lem:spin-corr-as-connections}, we have
\begin{eqnarray*}
\mathbb{P}_{\Omega_{\delta}}^{\left[r,\ell\right]_{\mathrm{w}}}\left\{ x\leftrightsquigarrow\left[r,\ell\right]\right\}  & = & \mathbb{P}_{\Omega_{\delta}}^{\left[r,\ell\right]_{\mathrm{w}}}\left\{ \lambda_{\delta}\mbox{ passes at }x\right\} \\
\mathbb{P}_{\Omega_{\delta}}^{\left[r,\ell\right]_{\mathrm{w}}}\left\{ x\leftrightsquigarrow z\leftrightsquigarrow\left[r,\ell\right]\right\}  & = & \mathbb{P}_{\Omega_{\delta}}^{\left[r,\ell\right]_{\mathrm{w}}}\left\{ \lambda_{\delta}\mbox{ passes at }x\mbox{ and }z\right\} 
\end{eqnarray*}

\end{lem}

\subsection{\label{sub:proof-int-rep-spin-corr}Proof of Proposition \ref{pro:int-rep-corr}}

Let us first recall Proposition \ref{pro:int-rep-corr} (see Figures
\ref{fig:expectation-eadelta} and \ref{fig:expectation-ebdelta}).
\begin{prop*}[Proposition \ref{pro:int-rep-corr}]
We have
\[
\Phi_{\delta}\left(\Omega_{\delta},r,\ell,x,z\right)=\mathrm{E}_{\delta}^{\mathrm{A}}\left(\Omega_{\delta},r,\ell,x,z\right)+\mathrm{E}_{\delta}^{\mathrm{B}}\left(\Omega_{\delta},r,\ell,x,z\right),
\]
where
\begin{eqnarray*}
\mathrm{E}_{\delta}^{\mathrm{A}}\left(\Omega_{\delta},r,\ell,x,z\right) & := & \frac{\mathbb{E}_{\delta}^{\mathrm{A}}\left[\mathbb{E}_{\Omega_{\delta}\setminus\lambda_{\delta}}^{\mathrm{free}}\left[\sigma_{x}\sigma_{z}\right]\right]}{\mathbb{E}_{\Omega_{\delta}}^{\left[r,\ell\right]_{+}}\left[\sigma_{x}\right]},\\
\mathrm{E}_{\delta}^{\mathrm{B}}\left(\Omega_{\delta},r,\ell,x,z\right) & := & \mathbb{E}_{\delta}^{\mathrm{B}}\left[\mathbb{E}_{\Omega_{\delta}\setminus\tilde{\lambda}_{\delta}}^{\left[r,x\right]_{+}}\left[\sigma_{z}\right]\right],
\end{eqnarray*}
where 
\begin{itemize}
\item the expectation $\mathbb{E}_{\delta}^{\mathrm{A}}$ is taken over
the realizations $\lambda_{\delta}$ of a critical FK-Ising interface
in $\Omega_{\delta}$ from $\ell$ to $r$;
\item the expectation $\mathbb{E}_{\delta}^{\mathrm{B}}$ is taken over
all realizations $\tilde{\lambda}_{\delta}$ of a critical FK interface
in $\Omega_{\delta}$ from $\ell$ to $r$, conditioned to pass through
$x$, stopped at $x$;
\item the correlation $\mathbb{E}_{\Omega_{\delta}\setminus\lambda_{\delta}}^{\mathrm{free}}\left[\sigma_{x}\sigma_{z}\right]$
is the correlation of the spins at $x$ and $z$ of the critical Ising
model on $\Upsilon_{\delta}$ with free boundary conditions, where
$\Upsilon_{\delta}$ is the connected component of $\Omega_{\delta}\setminus\lambda_{\delta}$
(the graph $\Omega_{\delta}$ with the edges crossed by $\lambda_{\delta}$
removed) containing $x$ and $z$;
\item the correlation \textup{$\mathbb{E}_{\Omega_{\delta}\setminus\tilde{\lambda}_{\delta}}^{\left[r,x\right]_{+}}\left[\sigma_{z}\right]$}\textup{\emph{
is the magnetization at $z$ of the critical Ising model on the connected
component of $\Omega_{\delta}\setminus\tilde{\lambda}_{\delta}$ that
contains $z$, with $+$ boundary condition on $\left[r,x\right]$
and free boundary condition on $\left[x,r\right]$. }}
\end{itemize}
\end{prop*}
\begin{rem}
The correlation $\mathbb{E}_{\Omega_{\delta}\setminus\lambda_{\delta}}^{\mathrm{free}}\left[\sigma_{x}\sigma_{z}\right]$
is equal to $0$ if $x$ and $z$ lie in two different connected components
of $\Omega_{\delta}\setminus\lambda_{\delta}$. 
\end{rem}

\begin{rem}
The graphs $\Omega_{\delta}\setminus\lambda_{\delta}$ and $\Omega_{\delta}\setminus\tilde{\lambda}_{\delta}$
are discrete vertex domains; notice also that there is no ambiguity
in the definition of the arc $\left[r,x\right]$ in $\Omega_{\delta}\setminus\tilde{\lambda}_{\delta}$
(see Figure \ref{fig:expectation-ebdelta}).\end{rem}
\begin{proof}
By Lemmas \ref{lem:spin-corr-as-connections} and \ref{lem:interface-passage-probabilities},
we have 
\begin{eqnarray}
 &  & \mathbb{E}_{\Omega_{\delta}}^{\left[r,\ell\right]_{+}}\left[\sigma_{x}\sigma_{z}\right]\nonumber \\
 & = & \mathbb{P}_{\Omega_{\delta}}^{\left[r,\ell\right]_{\mathrm{w}}}\left[\left\{ x\leftrightsquigarrow z\right\} \cap\left\{ \lambda_{\delta}\mbox{ does not pass at }x\mbox{ or }z\right\} \right]\nonumber \\
 &  & +\mathbb{P}_{\Omega_{\delta}}^{\left[r,\ell\right]_{\mathrm{w}}}\left\{ \lambda_{\delta}\mbox{ passes at }x\mbox{ and }z\right\} .\label{eq:e-a-plus-e-b-delta}
\end{eqnarray}
Because of the domain Markov property, we have that
\begin{eqnarray}
 &  & \mathbb{P}_{\Omega_{\delta}}^{\left[r,\ell\right]_{\mathrm{w}}}\left[\left\{ x\leftrightsquigarrow z\right\} \cap\left\{ \lambda_{\delta}\mbox{ does not pass at }x\mbox{ or }z\right\} \right]\nonumber \\
 & = & \mathbb{E}_{\delta}^{\mathrm{A}}\left[\mathbb{P}_{\Omega_{\delta}\setminus\lambda_{\delta}}^{\mathrm{free}}\left\{ x\leftrightsquigarrow z\right\} \right],\label{eq:e-a-delta}
\end{eqnarray}
where $\mathbb{E}_{\delta}^{\mathrm{A}}$ is as in the statement of
the proposition and the probability $\mathbb{P}_{\Omega_{\delta}\setminus\lambda_{\delta}}^{\mathrm{free}}\left\{ x\leftrightsquigarrow z\right\} $
is zero whenever $x$ and $z$ lie in two different components of
$\Omega_{\delta}\setminus\lambda_{\delta}$ -- this happens in particular
when $\lambda_{\delta}$ passes at $x$ or $z$. 

Also using domain Markov property, and the topological fact that the
interface emanating at $\ell_{\delta}$ cannot hit $z_{\delta}$ before
$x_{\delta}$, we obtain
\begin{eqnarray}
 &  & \frac{\mathbb{P}_{\Omega_{\delta}}^{\left[r,\ell\right]_{\mathrm{w}}}\left\{ \lambda_{\delta}\mbox{ passes at }x\mbox{ and }z\right\} }{\mathbb{P}_{\Omega_{\delta}}^{\left[r,\ell\right]_{\mathrm{w}}}\left\{ \lambda_{\delta}\mbox{ passes at }x\right\} }\nonumber \\
 & = & \mathbb{P}_{\Omega_{\delta}}^{\left[r,\ell\right]_{\mathrm{w}}}\left[\left\{ \lambda_{\delta}\mbox{ passes at }z\right\} |\left\{ \lambda_{\delta}\mbox{ passes at }x\right\} \right]\\
 & = & \mathbb{E}_{\delta}^{\mathrm{B}}\left[\mathbb{P}_{\Omega_{\delta}\setminus\tilde{\lambda}_{\delta}}^{\left[r,x\right]_{+}}\left\{ z\leftrightsquigarrow\left[r,x\right]\right\} \right],\label{eq:e-b-delta}
\end{eqnarray}
where $\mathbb{E}_{\delta}^{\mathrm{B}}$ is as in the statement of
the proposition. 

The proposition directly follows from Equations \ref{eq:e-a-plus-e-b-delta},
\ref{eq:e-a-delta}, \ref{eq:e-b-delta} (and Lemma \ref{lem:interface-passage-probabilities}
for the denominator) and the representation of $\Phi_{\delta}$ given
by Proposition \ref{pro:representation-obs-as-correlation}
\[
\Phi_{\delta}\left(\Omega_{\delta},r,\ell,x,z\right)=\frac{\mathbb{E}_{\Omega_{\delta}}^{\left[r,\ell\right]_{+}}\left[\sigma_{x}\sigma_{z}\right]}{\mathbb{E}_{\Omega_{\delta}}^{\left[r,\ell\right]_{+}}\left[\sigma_{x}\right]}.
\]

\begin{figure}

\includegraphics[width=9cm]{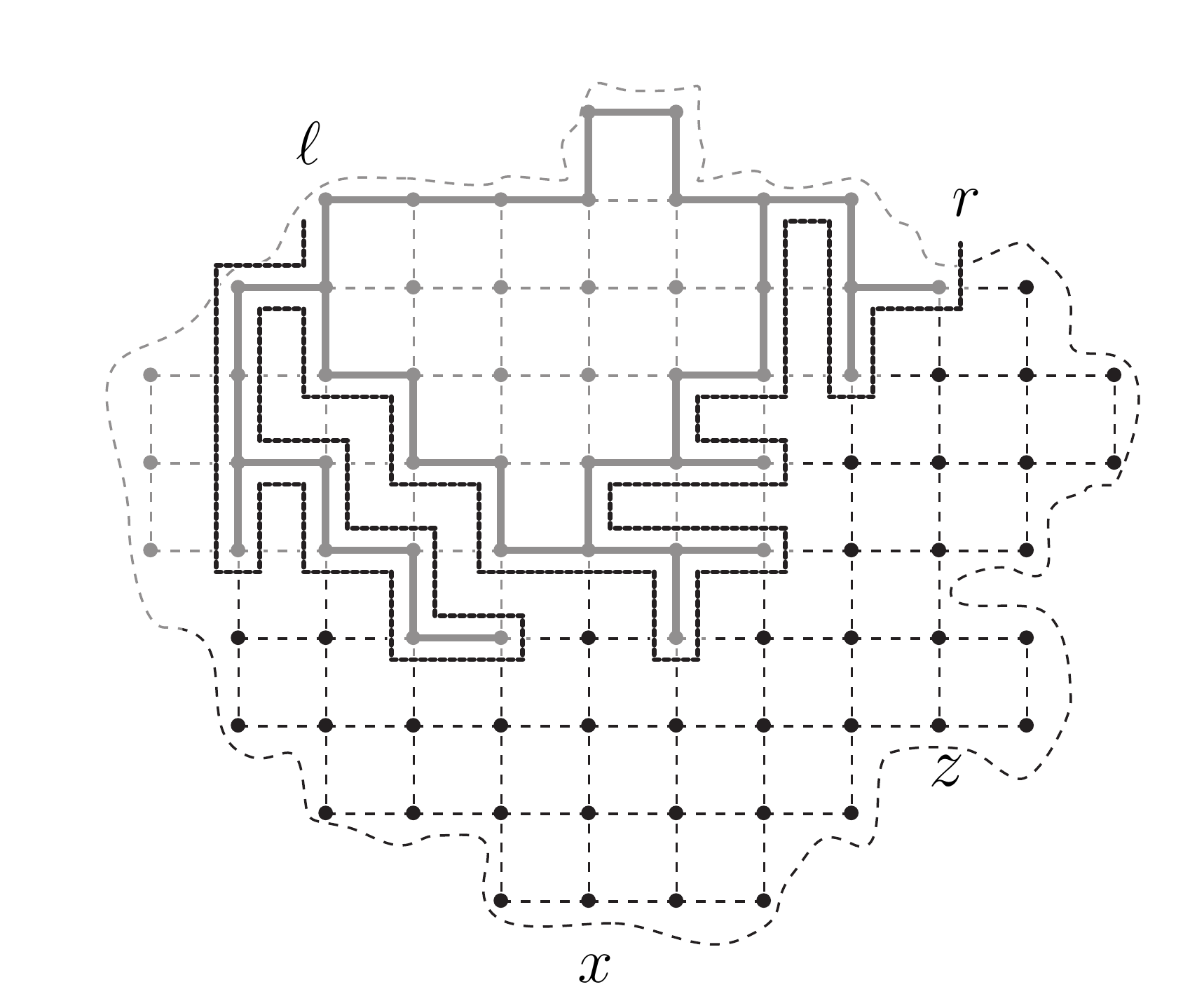}

\caption{\label{fig:expectation-eadelta}Expectation $\mathrm{E}_{\delta}^{\mathrm{A}}$:
the FK boundary conditions knowing the interface $\lambda_{\delta}$. }
\end{figure}

\end{proof}
\begin{figure}
\includegraphics[width=9cm]{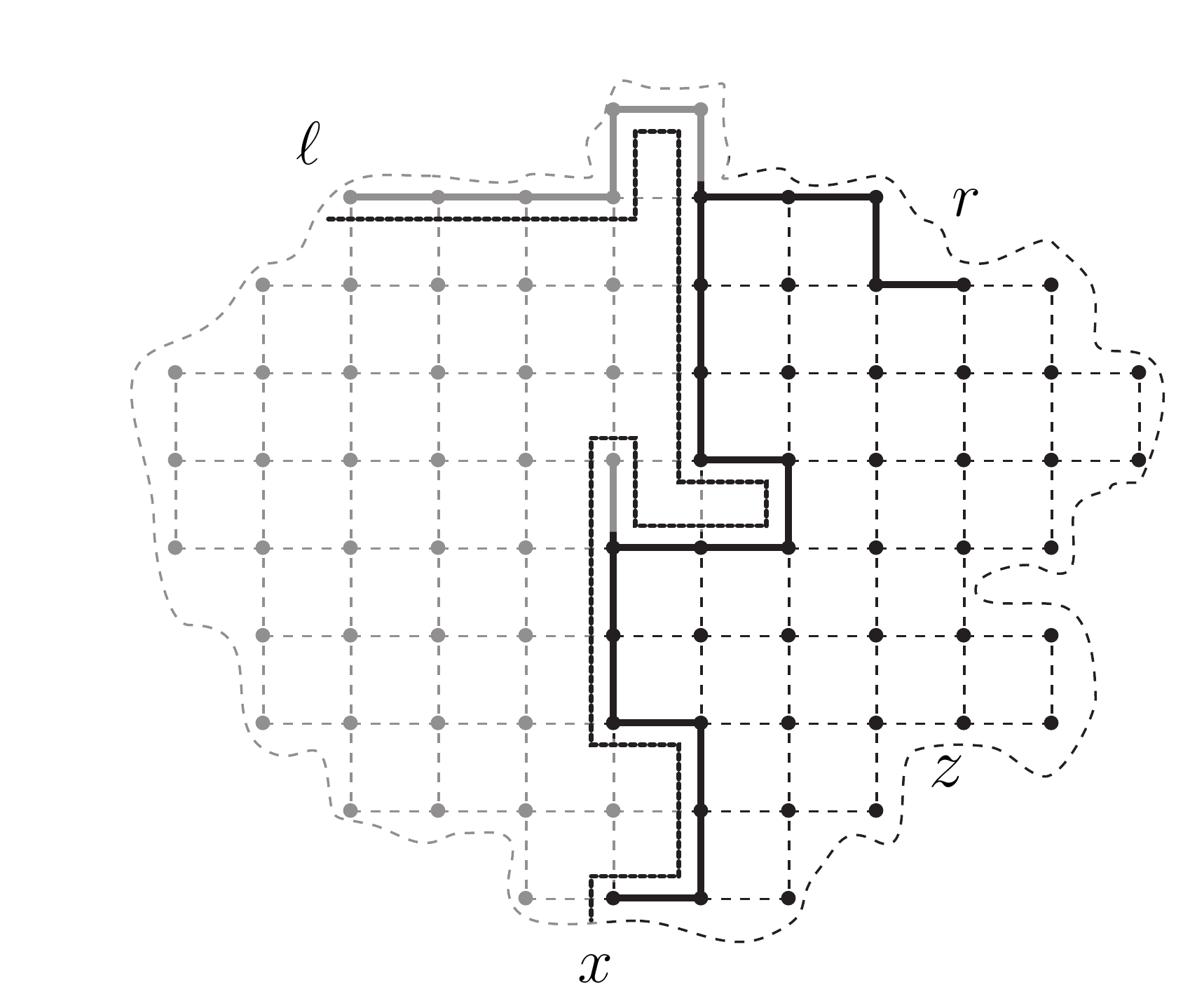}

\caption{\label{fig:expectation-ebdelta}Expectation $\mathrm{E}_{\delta}^{\mathrm{B}}$:
the FK boundary conditions knowing the interface $\tilde{\lambda}_{\delta}$. }
\end{figure}

\subsection{Law of the conditioned FK interface\label{sub:law-conditioned-fk-interface}}

Let us finish this section by the following characterization of the
conditioned FK interface, which will be useful in Section \ref{sec:sle-variants}
to pass to the scaling limit.
\begin{lem}
\label{lem:rn-deriv-disc-interfaces}Let $\left(\Omega_{\delta},r,\ell,x\right)$
be a discrete domain, $\lambda_{\delta}$ have the law of a critical
FK-Ising interface from $\ell$ to $r$  and let $\tilde{\lambda}_{\delta}$
have the law of a critical FK-Ising interface from $\ell$ to $r$,
conditioned to pass at $x$, stopped as it hits $x$. For any $\epsilon>\delta$,
and let $\tau_{\delta}^{\epsilon}\in\mathbb{N}\cup\left\{ \infty\right\} $
be the first time that $\lambda_{\delta}$ hits $\mathbf{D}_{\Omega_{\delta}}\left(x,\epsilon\right)$,
the connected component of $\left\{ z\in\Omega_{\delta}:\left|z-x\right|\leq\epsilon\right\} $
containing $x$, and let $\tilde{\tau}_{\delta}^{\epsilon}\in\mathbb{N}$
be the first time that $\tilde{\lambda}_{\delta}$ hits $\mathbf{D}_{\Omega_{\delta}}\left(x,\epsilon\right)$.
Let $\mathbb{P}_{\delta}^{\epsilon}$ and $\tilde{\mathbb{P}}_{\delta}^{\epsilon}$
denote the laws of $\lambda_{\delta}\left[0,\tau_{\delta}^{\epsilon}\right]$
and $\tilde{\lambda}_{\delta}\left[0,\tilde{\tau}_{\delta}^{\epsilon}\right]$.
Then we have
\[
\mathrm{Supp}\left(\tilde{\mathbb{P}}_{\delta}^{\epsilon}\right)=\left\{ \mu\in\mathrm{Supp}\left(\mathbb{P}_{\delta}^{\epsilon}\right):\mu\mbox{ hits }\mathbf{D}_{\Omega_{\delta}}\left(x,\epsilon\right)\mbox{ in finite time and }\mu\cap\left[x,r\right]=\emptyset\right\} 
\]
and for any $\mu_{\delta}\left[0,n\right]\in\mathrm{Supp}\left(\tilde{\mathbb{P}}_{\delta}^{\epsilon}\right)$,
we have the following expression for the Radon-Nikodym derivative:
\[
\frac{\mathrm{d}\tilde{\mathbb{P}}_{\delta}^{\epsilon}}{\mathrm{d}\mathbb{P}_{\delta}^{\epsilon}}\left(\mu_{\delta}\left[0,n\right]\right)=\frac{\mathbb{E}_{\Omega_{\delta}\setminus\mu_{\delta}}^{\left[r,\mu_{\delta}\left(n\right)\right]_{+}}\left[\sigma_{x}\right]}{\mathbb{E}_{\Omega_{\delta}}^{\left[r,\ell\right]_{+}}\left[\sigma_{x}\right]}.
\]
\end{lem}
\begin{proof}
The first part of the statement is obvious. Using the domain Markov
property and Doob's transform, for any $\mu_{\delta}\left[0,n\right]\in\mathrm{Supp}\left(\tilde{\mathbb{P}}_{\delta}^{\epsilon}\right)$
we obtain 
\[
\frac{\mathrm{d}\tilde{\mathbb{P}}_{\delta}^{\epsilon}}{\mathrm{d}\mathbb{P}_{\delta}^{\epsilon}}\left(\mu_{\delta}\left[0,n\right]\right)=\frac{\mathbb{P}_{\Omega_{\delta}\setminus\mu_{\delta}}^{\left[r,\mu_{\delta}\left(n\right)\right]_{\mathrm{w}}}\left\{ \lambda_{\delta}^{\dagger}\mbox{ passes at }x\right\} }{\mathbb{P}_{\Omega_{\delta}}^{\left[r,\ell\right]_{\mathrm{w}}}\left\{ \lambda_{\delta}\mbox{ passes at }x\right\} }.
\]
where $\lambda_{\delta}^{\dagger}$ is a critical FK interface in
$\Omega_{\delta}\setminus\mu_{\delta}\left[0,n\right]$ from $\mu_{\delta}\left(n\right)$
to $r$ (with wired boundary condition on $\left[r,\mu_{\delta}\left(n\right)\right]$
and free boundary condition on $\left[\mu_{\delta}\left(n\right),r\right]$). 

By Lemmas \ref{lem:spin-corr-as-connections} and \ref{lem:interface-passage-probabilities},
we get
\[
\frac{\mathbb{P}_{\Omega_{\delta}\setminus\mu_{\delta}}^{\left[r,\mu_{\delta}\left(n\right)\right]_{\mathrm{w}}}\left\{ \lambda_{\delta}^{\dagger}\mbox{ passes at }x\right\} }{\mathbb{P}_{\Omega_{\delta}}^{\left[r,\ell\right]_{\mathrm{w}}}\left\{ \lambda_{\delta}\mbox{ passes at }x\right\} }=\frac{\mathbb{E}_{\Omega_{\delta}\setminus\mu_{\delta}}^{\left[r,\mu_{\delta}\left(n\right)\right]_{+}}\left[\sigma_{x}\right]}{\mathbb{E}_{\Omega_{\delta}}^{\left[r,\ell\right]_{+}}\left[\sigma_{x}\right]}.
\]

Hence the result follows.
\end{proof}

\section{Scaling limit of elementary correlations\label{sec:cv-element-corr-func}}

In this section, we discuss the scaling limit of the correlation functions
appearing in Proposition \ref{pro:int-rep-corr}:
\begin{itemize}
\item The boundary spin-spin correlation with free boundary conditions:
it is the numerator of the integrand in the expectation $\mathrm{E}_{\mathrm{A}}^{\delta}$.
\item The boundary magnetization with mixed $+$ and free boundary conditions:
it is the integrand in the expectation $\mathrm{E}_{\mathrm{B}}^{\delta}$
\end{itemize}
The continuous counterparts of these quantities, the CFT correlation
functions $\left\langle \sigma_{x}\sigma_{z}\right\rangle _{\Omega}^{\mathrm{free}}$
and $\left\langle \sigma_{x}\right\rangle _{\Omega}^{\left[r,\ell\right]_{+}}$,
are given in Definition \ref{def:corr-func}.

The convergence of the discrete correlation functions to the continuous
ones is obtained by using discrete complex analysis techniques. The
first correlation function appears in \cite{hongler-i} and is closely
related to the observable used in \cite{chelkak-smirnov-ii} to prove
the convergence of the critical Ising interfaces to chordal SLE$\left(3\right)$.
The second correlation function is directly derived from the observable
used in \cite{chelkak-smirnov-iii} to prove the convergence of the
critical FK-Ising interfaces to chordal SLE$\left(16/3\right)$. 

Boundary correlation functions are very sensitive to the local geometry
of the boundary: they both depend on the geometry of the limiting
continuous domain (as they depend on the derivatives of the conformal
mappings on the boundary), and of the way the domain is discretized.
These issues are important, since we need to prove the convergence
of the observable $\Phi_{\delta}$ in domains that can a priori be
very rough, since they are slitted by an interface, which tends to
a continuous fractal. The point is that the observable $\Phi_{\delta}$
is a \emph{ratio} of two correlation functions: the dependences of
the numerator and the denominator on the fine geometry of the continuous
domain and its discretization compensate each other, and this allows
to obtain the desired result. 

Another related issue is the uniformity of the convergence which is
needed for the proof of Proposition \ref{pro:conv-discrete-expectations}:
the convergence should be uniform with respect to the shape and the
discretization of the domain: indeed, in the end, we want to be able
to say that for $\delta$ small enough, the discrete observable $\Phi_{\delta}$
is close to its continuous counterpart $\Phi$, uniformly over all
the possible realizations of the dipolar interface $\lambda_{\delta}$.

The discrete complex analysis details required to prove these results
are presented in Section \ref{sec:discrete-complex-analysis}.

\subsection{Two-point function}

The scaling limit of the boundary spin correlations on discrete vertex
domains with free boundary conditions (see Section \ref{sub:graph-domain})
is given by the following theorem. 
\begin{thm}
\label{thm:spin-spin-corr-free}Let $\left(\Omega,x,z\right)$ be
a domain such that $x$ and $z$ lie on vertical parts of $\partial\Omega$
and let $\left(\Omega_{\delta},x_{\delta},z_{\delta}\right)_{\delta>0}$
be a family of discrete vertex domains approximating $\left(\Omega,x,z\right)$.
Then we have
\[
\frac{1}{\delta}\mathbb{E}_{\Omega_{\delta}}^{\mathrm{free}}\left[\sigma_{x_{\delta}}\sigma_{z_{\delta}}\right]\underset{\delta\to0}{\longrightarrow}\left\langle \sigma_{x}\sigma_{z}\right\rangle _{\Omega}^{\mathrm{free}},
\]
where $\left\langle \sigma_{x}\sigma_{z}\right\rangle _{\Omega}^{\mathrm{free}}$
is as in Definition \ref{def:corr-func}.
\end{thm}
This result is derived in (\cite{hongler-i}, Theorem 1), when the
discretization $\Omega_{\delta}$ is the largest connected component
of $\Omega\cap\delta\mathbb{Z}^{2}$ and when $\Omega$ is assumed
$\mathcal{C}^{1}$. The question of the convergence for more general
domains is discussed in Section \ref{sub:rough-bdry}. The article
\cite{chelkak-smirnov-ii} gives (using the Kramers-Wannier duality)
the convergence of ratios of such spin correlations at different locations
on the boundary. The non-universal constant $\frac{1}{\pi}$ appearing
in Definition \ref{def:corr-func} is however not obtained there,
while it is important for our purposes. 
\begin{rem}
\label{rem:cts-spin-spin-poisson-kernel}If $x$ and $z$ lie on a
smooth part of $\Omega$, it is easy to check that $\left\langle \sigma_{x}\sigma_{z}\right\rangle _{\Omega}^{\mathrm{free}}$
is equal (up to a constant factor) to $\sqrt{E_{\Omega}\left(x,z\right)}$,
where $E$ is the excursion Poisson kernel defined by 
\[
E_{\Omega}\left(x,z\right):=\frac{\partial}{\partial\nu_{\mathrm{in}}\left(x\right)}P_{\Omega}\left(z,\cdot\right)=\frac{\partial}{\partial\nu_{\mathrm{in}}\left(z\right)}P_{\Omega}\left(x,\cdot\right),
\]
where $P_{\Omega}\left(\cdot,\cdot\right):\partial\Omega\times\Omega\to\mathbb{R}$
is the Poisson kernel and $\frac{\partial}{\partial\nu_{\mathrm{in}}\left(x\right)}$
denotes the inward normal derivative at $x$. From the monotonicity
properties of the Poisson kernel with respect to domain, we easily
get that if $\Upsilon\subset\Omega$ are two domains coinciding in
smooth neighborhoods of $x$ and $z$, we have
\[
\left\langle \sigma_{x}\sigma_{z}\right\rangle _{\Upsilon}^{\mathrm{free}}\leq\left\langle \sigma_{x}\sigma_{z}\right\rangle _{\Omega}^{\mathrm{free}}.
\]

\end{rem}

\subsection{One-point function}
\begin{thm}
\label{thm:boundary-magnetization}Let $\left(\Omega,r,\ell,x\right)$
be a domain such that such that $x$ lies on a vertical part of $\partial\Omega$
and let $\left(\Omega_{\delta},r_{\delta},\ell_{\delta},x_{\delta}\right)_{\delta>0}$
be a family of discrete vertex domains approximating $\left(\Omega,r,\ell,x\right)$.
Then we have
\[
\frac{1}{\sqrt{\delta}}\mathbb{E}_{\Omega_{\delta}}^{\left[r_{\delta},\ell_{\delta}\right]_{+}}\left[\sigma_{x_{\delta}}\right]\underset{\delta\to0}{\longrightarrow}\left\langle \sigma_{x}\right\rangle _{\Omega}^{\left[r,\ell\right]_{+}},
\]
where $\left\langle \sigma_{x}\right\rangle _{\Omega}^{\left[r,\ell\right]_{+}}$
is as in Definition \ref{def:corr-func}. 
\end{thm}
This result follows essentially from (\cite{smirnov-ii}, Remark 2.4):
using the FK interpretation of the boundary magnetization of Lemmas
\ref{lem:spin-corr-as-connections} and \ref{lem:interface-passage-probabilities},
we obtain that 
\[
\mathbb{E}_{\Omega_{\delta}}^{\left[r_{\delta},\ell_{\delta}\right]_{+}}\left[\sigma_{x_{\delta}}\right]=\mathbb{P}_{\Omega_{\delta}}^{\left[r_{\delta},\ell_{\delta}\right]_{\mathrm{w}}}\left[\gamma_{\delta}\mbox{ passes at }x_{\delta}\right],
\]
where $\gamma_{\delta}$ is the FK interface defined in the previous
section. The right-hand side is the absolute value of Smirnov's observable,
whose convergence is the main theorem (Theorem 2.2) of \cite{smirnov-ii}.
That article proves the convergence in the bulk, but as explained
in \cite{chelkak-smirnov-ii}, it is not too difficult to extend the
convergence to the straight parts of the boundary (see \cite{hongler-i}
for a version of this result specialized to the square lattice). 

Notice that it is again important for our purposes to obtain the correct
lattice-dependent constant $\sqrt{\frac{\sqrt{2}+1}{2\pi}}$ appearing
in Definition \ref{def:corr-func}; it is quite easy to derive using
the lattice construction detailed in \cite{smirnov-ii}.
\begin{rem}
\label{rem:bdary-magnet-norm-deriv}Notice that, if $x$ lies on a
smooth part of $\partial\Omega$, it is easy to check that $\left\langle \sigma_{x}\right\rangle _{\Omega}^{\left[r,\ell\right]_{+}}$
is equal to a multiple of
\[
\sqrt{\frac{\partial}{\partial\nu_{\mathrm{in}}\left(x\right)}\mathbf{H}_{\Omega}\left(\cdot,\left[r,\ell\right]\right)},
\]
where $\frac{\partial}{\partial\nu_{\mathrm{in}\left(x\right)}}$
is the inner normal derivative at $x$ and $\mathbf{H}_{\Omega}\left(z,\left[r,\ell\right]\right)$
is the harmonic measure of the arc $\left[r,\ell\right]$ in $\Omega$,
viewed from $z$ (the probability that a 2D Brownian motion starting
at $z$ exits $\Omega$ on $\left[r,\ell\right]$).
\end{rem}

\subsection{Rough boundaries and uniformity\label{sub:rough-bdry}}

A significant technical difficulty is to extend the convergence results
of Theorems \ref{thm:spin-spin-corr-free} and \ref{thm:boundary-magnetization}
to a setup that we can use to prove the convergence of the observable
$\Phi_{\delta}$. These convergence results will most of the time
not hold true for more general boundaries. 

However, the convergence of the ratio of such correlation functions
will be true if the same points on the same rough parts of the boundary
appear both in the numerator and the denominator, even if the domains
that are considered are different, or even if one divides a two-point
function by a one-point function. 

Let us state the two particular cases of this phenomenon that we will
need. Notice that the right-hand sides are well-defined for any $x$,
since we can use the same local conformal charts to look at the derivatives
of the conformal maps involved in the numerator and the denominator.
\begin{thm}
\label{thm:corr-ratios-theorem}Let $\left(\Theta,y,t,x\right)$,
$\left(\tilde{\Theta},\tilde{y},\tilde{t},x\right)$ and $\left(\Xi,x,s\right)$
be domains which coincide in a neighborhood of the boundary point
$x$. Suppose that $s$ lies on a vertical part $\mathfrak{v}$ of
$\partial\Xi$. Let $\left(\Theta_{\delta},y_{\delta},t_{\delta},x_{\delta}\right)$,
$\left(\tilde{\Theta}_{\delta},\tilde{y}_{\delta},\tilde{t}_{\delta},x_{\delta}\right)$
and $\left(\Xi_{\delta},x_{\delta},s_{\delta}\right)$ be discretizations
of these domains coinciding in a neighborhood of $x_{\delta}$, converging
to their continuous counterparts in the sense of the metric of Section
\ref{sub:uniformity-convergence}. Suppose that for each $\delta>0$,
$\partial\Xi_{\delta}$ contains a vertical part $\mathfrak{v}_{\delta}$
around $s_{\delta}$ and that as $\delta\to0$, $\mathfrak{v}_{\delta}$
converges to $\mathfrak{v}$. Then we have
\begin{eqnarray*}
\frac{1}{\sqrt{\delta}}\frac{\mathbb{E}_{\Xi_{\delta}}^{\mathrm{free}}\left[\sigma_{x_{\delta}}\sigma_{s_{\delta}}\right]}{\mathbb{E}_{\Theta_{\delta}}^{\left[y_{\delta},t_{\delta}\right]_{+}}\left[\sigma_{x_{\delta}}\right]} & \underset{\delta\to0}{\longrightarrow} & \frac{\left\langle \sigma_{x}\sigma_{s}\right\rangle _{\Xi}^{\mathrm{free}}}{\left\langle \sigma_{x}\right\rangle _{\Theta}^{\left[y,t\right]_{+}}}.\\
\frac{\mathbb{E}_{\tilde{\Theta}_{\delta}}^{\left[\tilde{y}_{\delta},\tilde{t}_{\delta}\right]_{+}}\left[\sigma_{x_{\delta}}\right]}{\mathbb{E}_{\Theta_{\delta}}^{\left[y_{\delta},t_{\delta}\right]_{+}}\left[\sigma_{x_{\delta}}\right]} & \underset{\delta\to0}{\longrightarrow} & \frac{\left\langle \sigma_{x}\right\rangle _{\tilde{\Theta}}^{\left[\tilde{y}_{\delta},\tilde{t}_{\delta}\right]_{+}}}{\left\langle \sigma_{x}\right\rangle _{\Theta}^{\left[y_{\delta},t_{\delta}\right]_{+}}}
\end{eqnarray*}
The convergence is locally uniform with respect to the domains. 
\end{thm}
To prove this result, the main idea is to cut the domains near $x$,
obtaining a domain $\left(\Upsilon,p,q,x\right)$ such that $\Upsilon\subset\Theta\cap\tilde{\Theta}\cap\Xi$,
$\left[p,q\right]$ is made of straight segments and $\left[q,p\right]\subset\partial\Theta\cap\partial\tilde{\Theta}\cap\partial\Xi$.
The discrete correlation functions of the left-hand sides come from
discrete holomorphic observables, and in particular are the boundary
values of discrete holomorphic functions, with the same type of boundary
values on (the discretizations of) $\left[q,p\right]$. 

It turns out that we can use this fact to represent the discrete holomorphic
functions involved as convolution of their boundary values on $\left[p,q\right]$
(which are well-controlled, since they are not on the boundaries of
the original domains) with the discrete holomorphic observable of
\cite{chelkak-smirnov-ii}. The convergence of the ratios of this
latter observable is addressed by Chelkak and Smirnov \cite{chelkak-smirnov-ii},
and hence we can use their result to obtain ours.

The details are presented in Section \ref{sec:discrete-complex-analysis}.

\section{SLE variants and scaling limit of FK interfaces \label{sec:sle-variants}}

In the previous section, we discussed the convergence of the elementary
correlation functions. The other main ingredient to prove Proposition
\ref{pro:conv-discrete-expectations} is the convergence of the FK
interfaces to variants of SLE$\left(16/3\right)$. Let us recall that
there are two interfaces that are considered, both in the discrete
domain $\left(\Omega_{\delta},r_{\delta},\ell_{\delta},x_{\delta},z_{\delta}\right)$.
\begin{itemize}
\item The interface $\lambda_{\delta}$, which has the law of an FK interface
from $r_{\delta}$ to $\ell_{\delta}$.
\item The interface $\tilde{\lambda}_{\delta}$, which has the law of an
FK interface from $r_{\delta}$ to $\ell_{\delta}$, conditioned to
pass at $x_{\delta}$, stopped upon hitting $x_{\delta}$.
\end{itemize}

\subsection{Chordal SLE$\left(\kappa\right)$ and SLE$\left(\kappa;\rho\right)$\label{sub:chordal-sle-kappa-rho}}

\subsubsection{Chordal SLE$\left(\kappa\right)$}

We now recall the definitions of the simplest and most standard variant
of SLE, which is chordal SLE$\left(\kappa\right)$; it is most naturally
defined on the upper half-plane $\mathbb{H}:=\left\{ z\in\mathbb{C}:\Im\mathfrak{m}\left(z\right)>0\right\} $,
in the same manner as dipolar SLE$\left(\kappa\right)$ is most naturally
defined on the strip. Chordal SLE$\left(\kappa\right)$ in $\mathbb{H}$
from $0$ (the source) to $\infty$ (the destination) is defined by
the half-plane Loewner chain (the half-plane counterpart of the strip
Loewner chain)
\begin{eqnarray*}
\partial_{t}g_{t}\left(z\right) & = & \frac{2}{g_{t}\left(z\right)-U_{t}},\\
g_{0}\left(z\right) & = & z,
\end{eqnarray*}
where the driving function $\left(U_{t}\right)_{t\geq0}$ has the
law of $\left(\sqrt{\kappa}B_{t}\right)_{t\geq0}$, where $\left(B_{t}\right)_{t\geq0}$
is a standard Brownian motion. It can be shown that for any time $t\geq0$,
$g_{t}$ is conformal map from $H_{t}\subset\mathbb{H}$ to $\mathbb{H}$,
where $H_{t}$ is the unbounded component of $\mathbb{H}\setminus\lambda\left[0,t\right]$,
where $\lambda$ is a curve from $0$ to $\infty$, called the \emph{(chordal)
SLE trace}, and that $g_{t}\left(\lambda\left(t\right)\right)=U_{t}$.
By (chordal) \emph{SLE}, we mean the trace $\lambda$ (as an oriented
unparametrized random curve). 

For $\kappa\in\left[0,4\right]$, SLE$\left(\kappa\right)$ is almost
surely a simple curve, for $\kappa\in\left(4,8\right)$, it has almost
surely double points (but it does not cross itself) and for $\kappa\geq8$,
it is almost surely space-filling \cite{lawler-ii,rohde-schramm}. 

In an arbitrary simply connected domain $\left(\Omega,a,b\right)$,
SLE$\left(\kappa\right)$ is defined as the conformal image of chordal
SLE in $\left(\mathbb{H},0,\infty\right)$ by a conformal mapping
$\varphi:\left(\mathbb{H},0,\infty\right)\to\left(\Omega,a,b\right)$.
Almost sure continuity of the SLE$\left(\kappa\right)$ trace in arbitrary
domains is shown in \cite{garban-rohde-schramm}.

\subsubsection{SLE$\left(\kappa;\rho\right)$}

A very useful variant of SLE$\left(\kappa\right)$ is SLE$\left(\kappa;\rho\right)$,
which is a process defined in a domain with three marked boundary
points. It is defined in $\left(\mathbb{H},0,x,\infty\right)$, where
$x\in\mathbb{R}\setminus\left\{ 0\right\} $, by a half-plane Loewner
chain $\left(\tilde{g}_{t}\right)_{t\geq0}$ with driving force $\left(\tilde{U}_{t}\right)_{t\geq0}$,
which is defined by the following Itô stochastic differential equation
\begin{eqnarray*}
\mathrm{d}\tilde{U}_{t} & = & \sqrt{\kappa}\mathrm{d}B_{t}+\frac{\rho\mathrm{d}t}{\tilde{U}_{t}-O_{t}}\\
\mathrm{d}O_{t} & = & \frac{2\mathrm{d}t}{O_{t}-\tilde{U}_{t}}\\
O_{0} & = & x.
\end{eqnarray*}
The process is defined up to the the first time $\tau$ when $\tilde{U}_{t}=O_{t}$.
As before it can be shown that, for each $0<t<\tau$, $\tilde{g}_{t}$
is a conformal map $H_{t}\to\mathbb{H}$, where $H_{t}$ is the unbounded
connected component of $\mathbb{H}\setminus\tilde{\lambda}\left[0,t\right]$,
where $\tilde{\lambda}$ is a curve, called the SLE$\left(\kappa;\rho\right)$
trace. The SLE$\left(\kappa;\rho\right)$ trace emanates at $0$ and
is well-defined up to the first time when it disconnects $x$ from
$\infty$. 

By Girsanov's theorem, we have that the initial segments of the SLE$\left(\kappa;\rho\right)$
trace, before the time when $x$ and $\infty$ get disconnected by
the trace, are absolutely continuous with respect to the initial segments
of chordal SLE$\left(\kappa\right)$ in $\left(\mathbb{H},0,\infty\right)$,
see Lemma \ref{lem:sle-girsanov} in Appendix B. We will simply refer
to SLE$\left(\kappa;\rho\right)$ trace as SLE$\left(\kappa;\rho\right)$. 

In an arbitrary domain $\left(\Omega,a,b,c\right)$, SLE$\left(\kappa;\rho\right)$
from $a$ (the source) to $c$ (the observation point) with force
point $b$ is defined as the image of SLE$\left(\kappa;\rho\right)$
in $\left(\mathbb{H},0,x,\infty\right)$ by the conformal mapping
$\phi:\left(\mathbb{H},0,x,\infty\right)\to\left(\Omega,a,b,c\right)$.

\subsubsection{Useful variants}

The following variants of SLE$\left(\kappa;\rho\right)$ (on $\left(\mathbb{H},0,x,\infty\right)$
for definiteness) are relevant in this paper:
\begin{itemize}
\item SLE$\left(\kappa;\kappa-6\right)$: for $\kappa\geq0$ , it has the
same law as chordal SLE$\left(\kappa\right)$ in $\left(\mathbb{H},0,x\right)$,
until the first time $x$ gets disconnected from $\infty$ by $\tilde{\gamma}$
(see \cite[Theorem 13]{schramm-wilson}). 
\item SLE$\left(\kappa;\kappa-8\right)$: has the law of a chordal SLE$\left(\kappa\right)$,
conditioned to hit the point $x$ (more precisely: to hit a ball of
radius $\epsilon$ around $x$, as $\epsilon\to0$). As we will only
be interested in the $\kappa=16/3,\rho=-8/3$ case (see Theorem \ref{thm:cond-fk-int-to-sle-kr}),
we will not make use of this general result (note that it is used
in \cite{beliaev-izyurov} for the $\kappa=6$ case).
\item SLE$\left(\kappa;\frac{\kappa-6}{2}\right)$: has the law of dipolar
SLE in $\left(\mathbb{H},0,x,\infty\right)$ \cite{kytola,schramm-wilson}.
Although we do not make use of this fact, it is worth mentioning:
in view of Lemma \ref{lem:sle-girsanov}, it tells us in particular
that chordal and dipolar SLE$\left(3\right)$ are absolutely continuous
with respect to each other. 
\end{itemize}
The following lemma provides a useful characterization of SLE$\left(\kappa;\kappa-8\right)$,
in the case when $\kappa=16/3$. It is a continuous version of Lemma
\ref{lem:rn-deriv-disc-interfaces}. Let us use, for definiteness,
the time parametrization of the SLE trace inherited from the half-plane
Loewner chain via the conformal map $\varphi:\left(\mathbb{H},0,1,\infty\right)\to\left(\Omega,r,\ell,x\right)$.
Denote by $\mathbf{D}_{\Omega}\left(x,\epsilon\right)$ the connected
component of $\left\{ z\in\overline{\Omega}:\left|z-x\right|\leq\epsilon\right\} $
containing $x$. 
\begin{lem}
\label{lem:sle-rn-deriv}Consider the domain $\left(\Omega,r,\ell,x\right)$.
Let $\lambda$ have the law of an SLE$\left(16/3\right)$ curve in
$\Omega$ from $\ell$ to $r$ and let $\tilde{\lambda}$ have the
law of an SLE$\left(16/3;-8/3\right)$ curve in $\Omega$ with starting
point $\ell$, observation point $r$ and force point $x$. 

For $\epsilon>0$, let $\tau^{\epsilon}$ be the (possibly infinite)
first time that $\lambda$ hits $\mathbf{D}_{\Omega}\left(x,\epsilon\right)$,
and let $\tilde{\tau}^{\epsilon}$ be the first time that $\tilde{\lambda}$
hits $\mathbf{D}_{\Omega}\left(x,\epsilon\right)$. Let $\mathbb{P}^{\epsilon}$
be the law of $\lambda\left[0,\tau^{\epsilon}\right]$ and $\tilde{\mathbb{P}}^{\epsilon}$
be the law of $\tilde{\lambda}\left[0,\tilde{\tau}^{\epsilon}\right]$.
Then we have

\[
\mathrm{Supp}\left(\tilde{\mathbb{P}}^{\epsilon}\right)=\left\{ \mu\in\mathrm{Supp}\left(\mathbb{P}^{\epsilon}\right):\mu\mbox{ hits }\mathbf{D}_{\Omega}\left(x,\epsilon\right)\mbox{ in finite time and }\mu\cap\left[x,r\right]=\emptyset\right\} 
\]
 and for any curve $\mu\left[0,t\right]\in\mathrm{Supp}\left(\tilde{\mathbb{P}}^{\epsilon}\right)$
we have the following expression for the Radon-Nikodym derivative:
\[
\frac{\mathrm{d}\tilde{\mathbb{P}}^{\epsilon}}{\mathrm{d}\mathbb{P}^{\epsilon}}\left(\mu\left[0,t\right]\right)=\frac{\left\langle \sigma_{x}\right\rangle _{\Omega\setminus\mu\left[0,t\right]}^{\left[r,\mu\left(t\right)\right]_{+}}}{\left\langle \sigma_{x}\right\rangle _{\Omega}^{\left[r,\ell\right]_{+}}}.
\]

\end{lem}
The proof is given in Appendix B.

\subsection{Convergence of the FK interfaces}

The convergence of $\gamma_{\delta}$ is a celebrated theorem of Smirnov
\cite{smirnov-i}.
\begin{thm}[\cite{smirnov-i}]
\label{thm:fk-interface-to-sle-cv}If $\left(\Omega,r,\ell\right)$
is a domain and $\left(\Omega_{\delta},r_{\delta},\ell_{\delta}\right)_{\delta}$
is a vertex domain discretization of it, the critical FK-Ising interface
$\gamma_{\delta}$ from $r_{\delta}$ to $\ell_{\delta}$ converges
in law to $\gamma$ as $\delta\to0$, where $\gamma$ is a chordal
SLE$\left(16/3\right)$ trace in $\Omega$ from $r$ to $\ell$. The
convergence is locally uniform with respect to the domains. 
\end{thm}
The convergence of the FK interface from $r_{\delta}$ to $\ell_{\delta}$,
conditioned to pass through the point $x_{\delta}$, can then be derived
from Theorem \ref{thm:fk-interface-to-sle-cv}.
\begin{thm}
\label{thm:cond-fk-int-to-sle-kr}Let $\left(\Omega,r,\ell,x\right)$
be a domain and let $\left(\Omega_{\delta},r_{\delta},\ell_{\delta},x_{\delta}\right)$
be a vertex domain discretization of it, and let $\tilde{\lambda}_{\delta}$
have the law of a critical FK-Ising interface in $\Omega_{\delta}$
from $\ell_{\delta}$ to $r_{\delta}$, conditioned to pass at $x_{\delta}$,
stopped as it hits $x_{\delta}$. Then $\tilde{\lambda}_{\delta}$
converges in law to $\tilde{\lambda}$ as $\delta\to0$, where $\tilde{\lambda}$
is an SLE$\left(16/3;-8/3\right)$ trace in $\left(\Omega,r,\ell,x\right)$,
with source $r$, force point $x$ and observation point $\ell$.
The convergence is locally uniform with respect to the domains $\left(\Omega,r,\ell,x\right)$.\end{thm}
\begin{proof}
Let $\lambda_{\delta}$ have the law of a critical FK-Ising interface
in $\Omega_{\delta}$ from $\ell_{\delta}$ to $r_{\delta}$ and $\lambda$
have the law of an SLE$\left(16/3\right)$ in $\Omega$ from $\ell$
to $r$. For $\epsilon>0$ and for $\delta$ small enough, let $\mathbf{D}_{\Omega_{\delta}}\left(x_{\delta},\epsilon\right)$
be the connected component of $\Omega_{\delta}\cap\left\{ z\in\mathbb{C}:\left|z-x\right|\leq\epsilon\right\} $
containing $x_{\delta}$, let $\tau_{\delta}^{\epsilon}$ be the first
time that $\lambda_{\delta}$ hits $\mathbf{D}_{\Omega_{\delta}}\left(x_{\delta},\epsilon\right)$
and let $\tilde{\tau}_{\delta}^{\epsilon}$ be the first time that
$\tilde{\lambda}_{\delta}$ hits $\mathbf{D}_{\Omega_{\delta}}\left(x_{\delta},\epsilon\right)$;
denote by $\tau^{\epsilon}$ and $\tilde{\tau}^{\epsilon}$ the corresponding
stopping times for $\lambda$ and $\tilde{\lambda}$, as in Lemma
\ref{lem:sle-rn-deriv}. Let $\mathbb{P}_{\delta}^{\epsilon}$, $\tilde{\mathbb{P}}_{\delta}^{\epsilon}$,
$\mathbb{P}^{\epsilon}$ and $\tilde{\mathbb{P}}^{\epsilon}$ denote
the respective laws of $\lambda_{\delta}\left[0,\tau_{\delta}^{\epsilon}\right]$,
$\tilde{\lambda}_{\delta}\left[0,\tilde{\tau}_{\delta}^{\epsilon}\right]$,
$\lambda\left[0,\tau^{\epsilon}\right]$ and $\tilde{\lambda}\left[0,\tilde{\tau}^{\epsilon}\right]$.
We can now put together the following four results:
\begin{itemize}
\item Let $\mu_{\delta}\left[0,n\right]\in\mathrm{Supp}\left(\tilde{\mathbb{P}}_{\delta}^{\epsilon}\right)$.
By Lemma \ref{lem:rn-deriv-disc-interfaces}, we have
\[
\frac{\mathrm{d}\tilde{\mathbb{P}}_{\delta}^{\epsilon}}{\mathrm{d}\mathbb{P}_{\delta}^{\epsilon}}\left(\mu_{\delta}\left[0,n\right]\right)=\frac{\mathbb{E}_{\Omega_{\delta}\setminus\mu_{\delta}}^{\left[r_{\delta},\mu_{\delta}\left(n\right)\right]_{+}}\left[\sigma_{x_{\delta}}\right]}{\mathbb{E}_{\Omega_{\delta}}^{\left[r_{\delta},\ell_{\delta}\right]_{+}}\left[\sigma_{x_{\delta}}\right]}.
\]

\item By Theorem \ref{thm:corr-ratios-theorem}, if $\mu_{\delta}\left[0,n_{\delta}\right]\to\mu\left[0,t\right]\in\mathrm{Supp}\left(\tilde{\mathbb{P}}^{\epsilon}\right)$,
we get 
\[
\frac{\mathbb{E}_{\Omega_{\delta}\setminus\mu_{\delta}}^{\left[r_{\delta},\mu_{\delta}\left(n\right)\right]_{+}}\left[\sigma_{x_{\delta}}\right]}{\mathbb{E}_{\Omega_{\delta}}^{\left[r_{\delta},\ell_{\delta}\right]_{+}}\left[\sigma_{x_{\delta}}\right]}\underset{\delta\to0}{\longrightarrow}\frac{\left\langle \sigma_{x}\right\rangle _{\Omega\setminus\mu\left[0,t\right]}^{\left[r,\mu\left(t\right)\right]_{+}}}{\left\langle \sigma_{x}\right\rangle _{\Omega}^{\left[r,\ell\right]_{+}}}
\]
locally uniformly with respect to $\mu$ and $\left(\Omega,r,\ell,x\right)$. 
\item From Theorem \ref{thm:fk-interface-to-sle-cv} and standard arguments,
we get
\[
\mathbb{P}_{\delta}^{\epsilon}\underset{\delta\to0}{\longrightarrow}\mathbb{P}^{\epsilon},
\]
locally uniformly with respect to $\left(\Omega,r,\ell,x\right)$.
\item By Lemma \ref{lem:sle-rn-deriv}, for any $\mu\left[0,t\right]\in\mathrm{Supp}\left(\tilde{\mathbb{P}}^{\epsilon}\right)$,
we have
\[
\frac{\mathrm{d}\tilde{\mathbb{P}}^{\epsilon}}{\mathrm{d}\mathbb{P}^{\epsilon}}\left(\mu\left[0,t\right]\right)=\frac{\left\langle \sigma_{x}\right\rangle _{\Omega\setminus\mu\left[0,t\right]}^{\left[r,\mu\left(t\right)\right]_{+}}}{\left\langle \sigma_{x}\right\rangle _{\Omega}^{\left[r,\ell\right]_{+}}}.
\]

\end{itemize}
From these four results, we easily deduce that 
\[
\tilde{\mathbb{P}}_{\delta}^{\epsilon}\underset{\delta\to0}{\longrightarrow}\tilde{\mathbb{P}}^{\epsilon},
\]
locally uniformly with respect to $\left(\Omega,r,\ell,x\right)$.
Finally, thanks to Lemma \ref{lem:control-end-conditioned-interface}
below, we can pass to the $\epsilon\to0$ limit and complete the proof
of the theorem: we get that for small $\epsilon>0$, with uniformly
high probability, $\tilde{\lambda}_{\delta}$ will always remain close
to $x_{\delta}$ after time $\tau_{\delta}^{\epsilon}$, and similarly
that $\tilde{\lambda}$ will remain close to $x$ after time $\tau^{\epsilon}$.\end{proof}
\begin{lem}
\label{lem:control-end-conditioned-interface}For any $R>0$, the
probability that $\tilde{\lambda}_{\delta}$ exits $\mathbf{D}_{\Omega_{\delta}}\left(x_{\delta},R\right)$
after the time $\tilde{\tau}_{\delta}^{\epsilon}$ tends to $0$ as
$\epsilon\to0$, uniformly with respect to $\left(\Omega_{\delta},r_{\delta},\ell_{\delta},x_{\delta}\right)$.
Similarly, for any $\rho>0$, the probability that $\tilde{\lambda}$
exits $\mathbf{D}_{\Omega}\left(x,\rho\right)$ after the time $\tilde{\tau}^{\epsilon}$
tends to $0$ as $\epsilon\to0$, uniformly with respect to $\left(\Omega_{\delta},r_{\delta},\ell_{\delta},x_{\delta}\right)$
\end{lem}
The proof is given in Appendix C.

\section{\label{sec:fk-integrals-to-sle-integrals}From FK expectations to
SLE expectations}

In this section, we put together the results of the two previous sections
to obtain Proposition \ref{pro:conv-discrete-expectations}, which
is the convergence of the FK-Ising expectations $\mathrm{E}_{\delta}^{\mathrm{A}}$
and $\mathrm{E}_{\delta}^{\mathrm{B}}$ to SLE expectations. Recall
that $\mathrm{E}_{\delta}^{\mathrm{A}}$ and $\mathrm{E}_{\delta}^{\mathrm{B}}$
are defined by
\begin{eqnarray*}
\mathrm{E}_{\delta}^{\mathrm{A}}\left(\Omega_{\delta},r_{\delta},\ell_{\delta},x_{\delta},z_{\delta}\right) & := & \mathbb{E}_{\delta}^{\mathrm{A}}\left[\frac{\mathbb{E}_{\Omega_{\delta}\setminus\lambda_{\delta}}^{\mathrm{free}}\left[\sigma_{x_{\delta}}\sigma_{z_{\delta}}\right]}{\mathbb{E}_{\Omega_{\delta}}^{\left[r_{\delta},\ell_{\delta}\right]_{+}}\left[\sigma_{x_{\delta}}\right]}\right],\\
\mathrm{E}_{\delta}^{\mathrm{B}}\left(\Omega_{\delta},r_{\delta},\ell_{\delta},x_{\delta},z_{\delta}\right) & := & \mathbb{E}_{\delta}^{\mathrm{B}}\left[\mathbb{E}_{\Omega_{\delta}\setminus\tilde{\lambda}_{\delta}}^{\left[r_{\delta},x_{\delta}\right]_{+}}\left[\sigma_{z_{\delta}}\right]\right],
\end{eqnarray*}
where $\gamma_{\delta}$ has the law of an FK interface from $r_{\delta}$
to $\ell_{\delta}$, and $\tilde{\lambda}_{\delta}$ has the law of
$\lambda_{\delta}$ conditioned to pass at $x_{\delta}$ and stopped
when it hits $x_{\delta}$. 
\begin{prop*}[Proposition \ref{pro:conv-discrete-expectations}]
As $\delta\to0$ and $\left(\Omega_{\delta},r_{\delta},\ell_{\delta},x_{\delta},z_{\delta}\right)\to\left(\Omega,r,\ell,x,z\right)$,
we have
\begin{eqnarray*}
\frac{1}{\sqrt{\delta}}\mathrm{E}_{\delta}^{\mathrm{A}}\left(\Omega_{\delta},r_{\delta},\ell_{\delta},x_{\delta},z_{\delta}\right) & \underset{\delta\to0}{\longrightarrow} & \mathrm{E}^{\mathrm{A}}\left(\Omega,r,\ell,x,z\right)\\
\frac{1}{\sqrt{\delta}}\mathrm{E}_{\delta}^{\mathrm{B}}\left(\Omega_{\delta},r_{\delta},\ell_{\delta},x_{\delta},z_{\delta}\right) & \underset{\delta\to0}{\longrightarrow} & \mathrm{E}^{\mathrm{B}}\left(\Omega,r,\ell,x,z\right),
\end{eqnarray*}
where the continuous expectations $\mathrm{E}^{\mathrm{A}}\left(\Omega,r,\ell,x,z\right)$
and $\mathrm{E}^{\mathrm{B}}\left(\Omega,r,\ell,x,z\right)$ are defined
by
\begin{eqnarray*}
\mathrm{E}^{\mathrm{A}}\left(\Omega,r,\ell,x,z\right) & := & \mathbb{E}^{\mathrm{A}}\left[\frac{\left\langle \sigma_{x}\sigma_{z}\right\rangle _{\Omega\setminus\lambda}^{\mathrm{free}}}{\left\langle \sigma_{x}\right\rangle _{\Omega}^{\left[r,\ell\right]_{+}}}\right],\\
\mathrm{E}^{\mathrm{B}}\left(\Omega,r,\ell,x,z\right) & := & \mathbb{E}^{\mathrm{B}}\left[\left\langle \sigma_{z}\right\rangle _{\Omega\setminus\tilde{\lambda}}^{\left[r,\ell\right]_{+}}\right],
\end{eqnarray*}
where 
\begin{itemize}
\item the correlation functions are as in Definition \ref{def:corr-func}.
\item the expectation $\mathbb{E}^{\mathrm{A}}$ is over the realizations
$\lambda$ of a chordal SLE$\left(16/3\right)$ trace from $\ell$
to $r$.
\item the expectation $\mathbb{E}^{\mathrm{B}}$ is over the realizations
$\tilde{\lambda}$ of an SLE$\left(16/3;-8/3\right)$ trace starting
from $\ell$, with observation point $r$ and force point $x$, stopped
upon hitting $x$. 
\item the integrand in $\mathbb{E}^{\mathrm{A}}$ is defined as $0$ if
$x$ and $z$ are in different connected components of $\Omega\setminus\lambda$.
\end{itemize}
The convergence is locally uniform with respect to the domains.
\end{prop*}
The proof of the first part of the theorem (convergence of $\mathrm{E}_{\delta}^{\mathrm{A}}$)
is given in Section \ref{sub:convergence-eadelta} and the proof of
the second part (convergence of $\mathrm{E}_{\delta}^{\mathrm{B}}$)
is given in Section \ref{sub:convergence-ebdelta}.

\subsection{Proof of convergence of $\mathrm{E}_{\delta}^{\mathrm{A}}$\label{sub:convergence-eadelta}}
\begin{proof}[Proof of Proposition \ref{pro:conv-discrete-expectations}, part A]
There are three ingredients (the types of convergence are as in Section
\ref{sub:uniformity-convergence}):
\begin{itemize}
\item Convergence of the probability measure: from Theorem \ref{thm:fk-interface-to-sle-cv},
we have that the discrete FK interface $\lambda_{\delta}$ converges
to the chordal SLE$\left(16/3\right)$ trace $\lambda$.
\item Convergence of the integrand: from Theorem \ref{thm:corr-ratios-theorem},
for any fixed curves $\lambda_{\delta}^{*}$ converging to a $\lambda^{*}$
as $\delta\to0$, we have 
\begin{equation}
\frac{1}{\sqrt{\delta}}\frac{\mathbb{E}_{\Omega_{\delta}\setminus\lambda_{\delta}^{*}}^{\mathrm{free}}\left[\sigma_{x_{\delta}}\sigma_{z_{\delta}}\right]}{\mathbb{E}_{\Omega_{\delta}}^{\left[r_{\delta},\ell_{\delta}\right]_{+}}\left[\sigma_{x_{\delta}}\right]}\to\frac{\left\langle \sigma_{x}\sigma_{z}\right\rangle _{\Omega\setminus\lambda^{*}}}{\left\langle \sigma_{x}\right\rangle _{\Omega}},\label{eq:cv-ratio-slit}
\end{equation}
where the convergence with respect to $\lambda_{\delta}^{*}$ and
$\Omega$ is locally uniform.
\item The integrand is uniformly bounded: by monotonicity of spin correlations
with free boundary conditions (which follows from FKG inequality,
see \cite{grimmett}), we have
\begin{equation}
0\leq\frac{1}{\sqrt{\delta}}\frac{\mathbb{E}_{\Omega_{\delta}\setminus\lambda_{\delta}}^{\mathrm{free}}\left[\sigma_{x_{\delta}}\sigma_{z_{\delta}}\right]}{\mathbb{E}_{\Omega_{\delta}}^{\left[r_{\delta},\ell_{\delta}\right]_{+}}\left[\sigma_{x_{\delta}}\right]}\leq\frac{1}{\sqrt{\delta}}\frac{\mathbb{E}_{\Omega_{\delta}}^{\mathrm{free}}\left[\sigma_{x_{\delta}}\sigma_{z_{\delta}}\right]}{\mathbb{E}_{\Omega_{\delta}}^{\left[r_{\delta},\ell_{\delta}\right]_{+}}\left[\sigma_{x_{\delta}}\right]}.\label{eq:monotonicity-spin-spin-corr}
\end{equation}
As the right-hand side is uniformly convergent, the left-hand side
is uniformly bounded. 
\end{itemize}
We deduce the convergence as follows: by the first point, for any
$\epsilon_{0}>0$, there exists a $\delta_{0}>0$ (locally uniform
with respect to $\Omega$) such that for any $\delta\leq\delta_{0}$,
we have 
\[
\mathbb{P}\left\{ \mathrm{d}_{\infty}\left(\lambda_{\delta},\lambda\right)>\epsilon_{0}\right\} \leq\epsilon_{0}.
\]
Combining this with the second point, we deduce that for any $\epsilon_{1}>0$,
there exists $\delta_{1}>0$ (locally uniform with respect to $\Omega$)
such that for any $\delta\leq\delta_{1}$, we have 
\begin{equation}
\mathbb{P}\left\{ \left|\frac{1}{\sqrt{\delta}}\frac{\mathbb{E}_{\Omega_{\delta}\setminus\lambda_{\delta}}^{\mathrm{free}}\left[\sigma_{x_{\delta}}\sigma_{z_{\delta}}\right]}{\mathbb{E}_{\Omega_{\delta}}^{\left[r_{\delta},\ell_{\delta}\right]_{+}}\left[\sigma_{x_{\delta}}\right]}-\frac{\left\langle \sigma_{x}\sigma_{z}\right\rangle _{\Omega\setminus\lambda}}{\left\langle \sigma_{x}\right\rangle _{\Omega}}\right|>\epsilon_{1}\right\} \leq\epsilon_{1}.\label{eq:prob-disc-int-cts-int-far-away}
\end{equation}
Together with the third point, this allows to obtain that for any
$\epsilon_{2}>0$ there exists $\delta_{2}>0$ (locally uniform with
respect to $\Omega$) such that for any $\delta\leq\delta_{2}$, we
have 
\begin{eqnarray*}
\left|\mathrm{E}_{\delta}^{\mathrm{A}}\left(\Omega_{\delta},r_{\delta},\ell_{\delta},x_{\delta},z\right)-\mathrm{E}^{\mathrm{A}}\left(\Omega,r,\ell,x,z\right)\right| & \leq & \epsilon_{2}.
\end{eqnarray*}
Indeed from the uniform bound \ref{eq:monotonicity-spin-spin-corr},
we  get that the contribution to the expectation of the event appearing
in \ref{eq:prob-disc-int-cts-int-far-away} can be made arbitrarily
small.
\end{proof}

\subsection{Proof of convergence of $\mathrm{E}_{\delta}^{\mathrm{B}}$\label{sub:convergence-ebdelta}}
\begin{proof}[Proof of Proposition \ref{pro:conv-discrete-expectations}, part B]
As in the proof of convergence of $\mathrm{E}_{\delta}^{\mathrm{A}}$,
there are three ingredients (the types of convergence are as in Section
\ref{sub:uniformity-convergence}). All statements are locally uniform
with respect to $\Omega$.
\begin{itemize}
\item Convergence of the probability measure: from Theorem \ref{thm:cond-fk-int-to-sle-kr},
we have that the conditioned discrete FK interface $\tilde{\lambda}_{\delta}$
converges to the SLE$\left(16/3;-8/3\right)$ trace $\tilde{\lambda}$.
\item Convergence of the integrand: from Theorem \ref{thm:corr-ratios-theorem},
for any fixed curves $\tilde{\lambda}_{\delta}^{*}$ converging to
$\tilde{\lambda}^{*}$ as $\delta\to0$, we have 
\[
\frac{1}{\sqrt{\delta}}\mathbb{E}_{\Omega_{\delta}\setminus\tilde{\lambda}_{\delta}^{*}}^{\left[r_{\delta},x_{\delta}\right]_{+}}\left[\sigma_{z_{\delta}}\right]\to\left\langle \sigma_{z}\right\rangle _{\Omega\setminus\tilde{\lambda}^{*}}^{\left[r,\ell\right]_{+}},
\]
where the convergence is locally uniform with respect to $\tilde{\lambda}_{\delta}^{*}$.
\item Uniform integrability of the integrand: for any $\epsilon>0$, let
$\mathfrak{N}_{\delta}^{\epsilon}$ be the event that $\tilde{\lambda}_{\delta}$
hits the $\epsilon$-neighborhood $\mathbf{D}_{\Omega_{\delta}}\left(z_{\delta},\epsilon\right)$
(as in defined in Lemma \ref{lem:rn-deriv-disc-interfaces}). Then,
we get the uniform integrability from the following observations:

\begin{itemize}
\item On the complementary of $\mathfrak{N}_{\delta}^{\epsilon}$ (i.e.
the event that $\tilde{\lambda}_{\delta}$ does not hit $\mathbf{D}_{\Omega_{\delta}}\left(z_{\delta},\epsilon\right)$),
the integrand $\frac{1}{\sqrt{\delta}}\mathbb{E}_{\Omega_{\delta}\setminus\tilde{\lambda}_{\delta}}^{\left[r_{\delta},x_{\delta}\right]_{+}}\left[\sigma_{z_{\delta}}\right]$
is uniformly bounded with respect to $\delta>0$. This follows directly
from Lemma \ref{lem:blowup-rate-magnetization} below.
\item As $\epsilon\to0$, $\frac{1}{\sqrt{\delta}}\mathbb{E}_{\Omega_{\delta}\setminus\tilde{\lambda}_{\delta}}^{\left[r_{\delta},x_{\delta}\right]_{+}}\left[\sigma_{z_{\delta}}\mathbf{1}_{\mathfrak{N}_{\delta}^{\epsilon}}\right]\to0,$
uniformly with respect to $\delta$. We write $\mathfrak{N}_{\delta}^{\epsilon}=\mathfrak{A}_{\delta}^{\epsilon}\cup\mathfrak{A}_{\delta}^{\epsilon/2}\cup\mathfrak{A}_{\delta}^{\epsilon/4}\cup\ldots$,
where for any $\eta>0$, we set $\mathfrak{A}_{\delta}^{\eta}:=\mathfrak{N}_{\delta}^{\eta}\setminus\mathfrak{N}_{\delta}^{\eta/2}$. 
\item By Lemma \ref{lem:prob-cond-curve-getting-close}, as $\eta\to0$,
the probability of $\mathfrak{A}_{\delta}^{\eta}$ behaves like $\mathcal{O}\left(\eta\right)$
uniformly with respect to $\delta$ by Lemma \ref{lem:blowup-rate-magnetization},
on $\mathfrak{A}_{\delta}^{\eta}$, the integrand is bounded by $\mathcal{O}\left(\frac{1}{\sqrt{\eta}}\right)$,
uniformly with respect to $\delta$; hence $\frac{1}{\sqrt{\delta}}\mathbb{E}_{\Omega_{\delta}\setminus\tilde{\lambda}_{\delta}}^{\left[r_{\delta},x_{\delta}\right]_{+}}\left[\sigma_{z_{\delta}}\mathbf{1}_{\mathfrak{N}_{\delta}^{\eta}}\right]=\mathcal{O}\left(\sqrt{\eta}\right)$.
Summing this over the scales $\eta=\epsilon,\epsilon/2,\epsilon/4,\ldots,\delta$
(there are $\mathcal{O}\left(\log_{2}\left(\epsilon/\delta\right)\right)$
such scales), we get the result. 
\end{itemize}
\end{itemize}
Using the exactly the same sequence of arguments as in the conclusion
of the proof of part A of the proposition (in the previous subsection),
we get the result.
\end{proof}
Let us now give the two a priori estimates used in the proof above,
which follow from results in \cite{duminil-copin-hongler-nolin} (notice
that we can directly use the results from this paper, as the boundary
of $\Omega_{\delta}$ near $z_{\delta}$ is straight).
\begin{lem}
\label{lem:prob-cond-curve-getting-close}The probability that $\tilde{\lambda}_{\delta}$
gets $\epsilon$-close to $z_{\delta}$ behaves like $\mathcal{O}\left(\epsilon\right)$,
uniformly with respect to $\delta$, locally uniformly with respect
to $\Omega$.\end{lem}
\begin{proof}
This follows readily from Proposition 12 in \cite{duminil-copin-hongler-nolin} \end{proof}
\begin{lem}
\label{lem:blowup-rate-magnetization}If $\Upsilon_{\delta}\subset\Omega_{\delta}$
is an $\epsilon$-neighborhood (in $\Omega_{\delta}$) of $z_{\delta}$,
then 
\[
\frac{1}{\sqrt{\delta}}\mathbb{E}_{\Upsilon_{\delta}}^{\left[r_{\delta},x_{\delta}\right]_{+}}\left[\sigma_{z_{\delta}}\right]=\mathcal{O}\left(\frac{1}{\sqrt{\epsilon}}\right),
\]
uniformly with respect to $\delta$, locally uniformly with respect
to $\Omega$. \end{lem}
\begin{proof}
This estimate follows directly from Proposition 18 in \cite{duminil-copin-hongler-nolin}.
\end{proof}

\section{\label{sec:sle-expectations}Computation of SLE expectations}

In this section, we compute the SLE expectations $\mathrm{E}^{\mathrm{A}}$
and $\mathrm{E}^{\mathrm{B}}$, defined as
\begin{eqnarray*}
\mathrm{E}^{\mathrm{A}}\left(\Omega,r,\ell,x,z\right) & := & \mathbb{E}^{\mathrm{A}}\left[\frac{\left\langle \sigma_{x}\sigma_{z}\right\rangle _{\Omega\setminus\lambda}^{\mathrm{free}}}{\left\langle \sigma_{x}\right\rangle _{\Omega}^{\left[r,\ell\right]_{+}}}\right],\\
\mathrm{E}^{\mathrm{B}}\left(\Omega,r,\ell,x,z\right) & := & \mathbb{E}^{\mathrm{B}}\left[\left\langle \sigma_{z}\right\rangle _{\Omega\setminus\tilde{\lambda}}^{\left[r,x\right]_{+}}\right],
\end{eqnarray*}
where $\lambda$ has the law of a chordal SLE$\left(16/3\right)$
from $r$ to $s$ and $\tilde{\lambda}$ has the law of an SLE$\left(16/3;-8/3\right)$
with starting point $r$, observation point $s$ and force point $x$,
and where the integrands are as in Definition \ref{def:corr-func}.
Let us now define two following Coulomb gas correlation functions
which will be useful to compute $\mathrm{E}^{\mathrm{A}}$ and $\mathrm{E}^{\mathrm{B}}$. 
\begin{defn}
For a simply connected domain $\left(\Omega,r,\ell,x,z\right)$ such
that $z$ is on a smooth part of $\partial\Omega$ and given any conformal
mapping $\eta_{\Omega}:\Omega\to\mathbb{H}$, we define $\mathfrak{m}^{\mathrm{A}}$
and $\mathfrak{m}^{\mathrm{B}}$ by:
\begin{eqnarray*}
\mathfrak{m}^{\mathrm{A}}\left(\Omega,r,\ell,x,z\right) & := & \left|\eta_{\Omega}'\left(x\right)\right|^{\frac{1}{2}}\left|\eta{}_{\Omega}'\left(z\right)\right|^{\frac{1}{2}}\mathfrak{m}^{\mathrm{A}}\left(\mathbb{H},\eta{}_{\Omega}\left(r\right),\eta{}_{\Omega}\left(\ell\right),\eta{}_{\Omega}\left(x\right),\eta{}_{\Omega}\left(z\right)\right)\\
\mathfrak{m}^{\mathrm{B}}\left(\Omega,r,\ell,x,z\right) & := & \left|\eta{}_{\Omega}'\left(z\right)\right|^{\frac{1}{2}}\mathfrak{m}^{\mathrm{B}}\left(\mathbb{H},\eta{}_{\Omega}\left(r\right),\eta{}_{\Omega}\left(\ell\right),\eta{}_{\Omega}\left(x\right),\eta{}_{\Omega}\left(z\right)\right)\\
\mathfrak{m^{\mathrm{A}}}\left(\mathbb{H},\eta_{\infty},\eta_{0},\eta_{1},\eta_{2}\right) & := & C_{\mathrm{\mathrm{A}}}^{\mathfrak{m}}\frac{\left(\eta_{0}-\eta_{\infty}\right)^{\frac{1}{2}}}{\left(\eta_{2}-\eta_{1}\right)^{\frac{1}{2}}\left(\eta_{2}-\eta_{\infty}\right)^{\frac{1}{2}}\left(\eta_{1}-\eta_{0}\right)^{\frac{1}{2}}}\int_{\chi}^{1}\frac{\left(1-\zeta\right)}{\zeta^{3/4}\left(\zeta-\chi\right)^{3/4}}\mathrm{d}\zeta\\
\mathfrak{m}^{\mathrm{B}}\left(\mathbb{H},\eta_{\infty},\eta_{0},\eta_{1},\eta_{2}\right) & := & C_{\mathrm{B}}^{\mathfrak{m}}\frac{\left(\eta_{1}-\eta_{\infty}\right)^{1/2}}{\left(\eta_{2}-\eta_{\infty}\right)^{1/2}\left(\eta_{2}-\eta_{1}\right)^{1/2}}\cdot\frac{1}{\sqrt{\chi}}\int_{1-\chi}^{1}\frac{\zeta^{1/2}}{\left(\zeta-1+\chi\right)^{1/4}\left(1-\zeta\right)^{1/4}}\mathrm{d}\zeta
\end{eqnarray*}
where $\chi$ is the cross-ratio defined by
\[
\chi:=\frac{\eta_{0}-\eta_{\infty}}{\eta_{2}-\eta_{0}}\frac{\eta_{2}-\eta_{1}}{\eta_{1}-\eta_{\infty}}
\]
and the constants $C_{\mathrm{A}}^{\mathfrak{m}}$ and $C_{\mathrm{B}}^{\mathfrak{m}}$
are defined by
\begin{eqnarray*}
C_{\mathrm{\mathrm{A}}}^{\mathfrak{m}} & := & \frac{\sqrt{2}+1}{\pi^{3/2}}\frac{\Gamma\left(\frac{3}{4}\right)}{\Gamma\left(\frac{1}{4}\right)}\\
C_{\mathrm{B}}^{\mathfrak{m}} & := & \frac{2\sqrt{1+\sqrt{2}}}{\pi^{3/2}}.
\end{eqnarray*}
\end{defn}
\begin{rem}
These definitions are independent of the choice of $\eta_{\Omega}$.
\end{rem}
With this notation, Proposition \ref{pro:sle-averages} simply becomes:
\begin{prop*}[Proposition \ref{pro:sle-averages}]
We have
\begin{eqnarray}
\mathrm{E}^{\mathrm{A}}\left(\Omega,r,\ell,x,z\right) & = & \frac{\mathfrak{m}^{\mathrm{A}}\left(\Omega,r,\ell,x,z\right)}{\left\langle \sigma_{x}\right\rangle _{\Omega}^{\left[r,\ell\right]_{+}}}\label{eq:ea-equals-mart-obs}\\
\mathrm{E}^{\mathrm{B}}\left(\Omega,r,\ell,x,z\right) & = & \mathfrak{m}^{\mathrm{B}}\left(\Omega,r,\ell,x,z\right)\label{eq:eb-equals-mart-obs}\\
\Phi\left(\Omega,r,\ell,x,z\right) & = & \mathrm{E}^{\mathrm{A}}\left(\Omega,r,\ell,x,z\right)+\mathrm{E}^{\mathrm{B}}\left(\Omega,r,\ell,x,z\right).\nonumber 
\end{eqnarray}

\end{prop*}
Notice that the right-hand side of \ref{eq:ea-equals-mart-obs} is
well-defined also when $x$ is on a rough boundary: the derivative
terms in the definition of $\mathfrak{m}^{\mathrm{A}}$ and $\left\langle \sigma_{x}\right\rangle $
cancel each other. On the other hand, it is required that $z$ is
on a smooth part of $\partial\Omega$. 
\begin{proof}
To compute $\mathrm{E}^{\mathrm{A}}$ and $\mathrm{E}^{\mathrm{B}}$,
notice that by conformal invariance of SLE and by conformal covariance
of the correlation functions in the integrand, it is sufficient to
compute $\mathrm{E}^{\mathrm{A}}$ and $\mathrm{E}^{\mathrm{B}}$
on the upper half-plane, and that we can choose the locations of three
of the four boundary marked points. Hence the result follows from
the computation for $\mathrm{E}^{\mathrm{A}}\left(\mathbb{H},-1,0,1,y\right)$,
for any $y\in\left(-\infty,-1\right)$ performed in Section \ref{sub:computation-e-a},
and the one of $\mathrm{E}^{\mathrm{B}}\left(\mathbb{H},\infty,0,1,w\right)$
for any $w\in\left(1,\infty\right)$, performed in Section \ref{sub:computation-e-b}.

To obtain the formula for $\Phi$, we use the following hypergeometric
representations of $\mathfrak{m}^{\mathrm{A}}$ and $\mathfrak{m}^{\mathrm{B}}$:
\begin{eqnarray*}
\mathfrak{m^{\mathrm{A}}}\left(\mathbb{H},\eta_{\infty},\eta_{0},\eta_{1},\eta_{2}\right) & = & \frac{16C_{\mathrm{A}}^{\mathfrak{m}}}{5}\frac{\left(\eta_{0}-\eta_{\infty}\right)^{\frac{1}{2}}\left(\eta_{2}-\eta_{\infty}\right)^{\frac{3}{4}}\left(\eta_{1}-\eta_{0}\right)^{\frac{3}{4}}}{\left(\eta_{1}-\eta_{\infty}\right)^{\frac{5}{4}}\left(\eta_{2}-\eta_{0}\right)^{\frac{5}{4}}\left(\eta_{2}-\eta_{1}\right)^{\frac{1}{2}}}\,\,_{2}F_{1}\left(\frac{3}{4},2;\frac{9}{4};1-\chi\right)\\
\mathfrak{m}^{\mathrm{B}}\left(\mathbb{H},\eta_{\infty},\eta_{0},\eta_{1},\eta_{2}\right) & = & \frac{2\Gamma\left(\frac{3}{4}\right)^{2}C_{\mathrm{B}}^{\mathfrak{m}}}{\sqrt{\pi}}\frac{\left(\eta_{1}-\eta_{\infty}\right)^{\frac{1}{2}}}{\left(\eta_{2}-\eta_{\infty}\right)^{\frac{1}{2}}\left(\eta_{2}-\eta_{1}\right)^{\frac{1}{2}}}\,\,_{2}F_{1}\left(-\frac{1}{2},\frac{3}{4};\frac{3}{2};\chi\right),
\end{eqnarray*}
where the branch of the hypergeometric function $_{2}F_{1}$ on $\mathbb{C}\setminus\left[1,\infty\right]$
is the usual one. Then, from the formula in \cite[Eq. 15.3.6]{abramowitz-stegun},
we get
\begin{eqnarray*}
\mathfrak{m}^{\mathrm{B}}\left(\mathbb{H},\eta_{\infty},\eta_{0},\eta_{1},\eta_{2}\right) & = & \frac{\Gamma\left(\frac{1}{4}\right)\Gamma\left(\frac{3}{4}\right)}{4}C_{\mathrm{B}}^{\mathfrak{m}}\frac{\left(\eta_{1}-\eta_{\infty}\right)^{\frac{1}{2}}}{\left(\eta_{2}-\eta_{\infty}\right)^{\frac{1}{2}}\left(\eta_{2}-\eta_{1}\right)^{\frac{1}{2}}}\,_{2}F_{1}\left(-\frac{1}{2},\frac{3}{4};-\frac{1}{4};1-\chi\right)\\
 &  & -\frac{8\Gamma\left(\frac{3}{4}\right)^{2}}{5\sqrt{\pi}}C_{\mathrm{B}}^{\mathfrak{m}}\frac{\left(1-\chi\right)^{\frac{5}{4}}\left(\eta_{1}-\eta_{\infty}\right)^{\frac{1}{2}}}{\left(\eta_{2}-\eta_{\infty}\right)^{\frac{1}{2}}\left(\eta_{2}-\eta_{1}\right)^{\frac{1}{2}}}\,_{2}F_{1}\left(\frac{3}{4},2;\frac{9}{4};1-\chi\right).
\end{eqnarray*}
Recalling that
\[
\left\langle \sigma_{\eta_{1}}\right\rangle _{\mathbb{H}}^{\left[\eta_{\infty},\eta_{0}\right]_{+}}=\sqrt{\frac{\sqrt{2}+1}{2\pi}}\frac{\left(\eta_{0}-\eta_{\infty}\right)^{\frac{1}{2}}}{\left(\eta_{1}-\eta_{0}\right)^{\frac{1}{2}}\left(\eta_{1}-\eta_{\infty}\right)^{\frac{1}{2}}},
\]
we notice a cancellation in the sum 
\[
\frac{\mathfrak{m^{\mathrm{A}}}\left(\mathbb{H},\eta_{\infty},\eta_{0},\eta_{1},\eta_{2}\right)}{\left\langle \sigma_{\eta_{1}}\right\rangle _{\mathbb{H}}^{\left[\eta_{\infty},\eta_{0}\right]_{+}}}+\mathfrak{m}^{\mathrm{B}}\left(\mathbb{H},\eta_{\infty},\eta_{0},\eta_{1},\eta_{2}\right)
\]
and then write the result of the sum, which simplifies to
\[
\sqrt{\frac{\sqrt{2}+1}{2\pi}}\frac{\left(\eta_{1}-\eta_{\infty}\right)^{\frac{1}{2}}}{\left(\eta_{2}-\eta_{\infty}\right)^{\frac{1}{2}}\left(\eta_{2}-\eta_{1}\right)^{\frac{1}{2}}}\,_{2}F_{1}\left(-\frac{1}{2},\frac{3}{4};-\frac{1}{4};1-\chi\right).
\]
Finally, using that
\[
_{2}F_{1}\left(-\frac{1}{2},\frac{3}{4};-\frac{1}{4};z\right)=\frac{1+z}{\sqrt{1-z}}
\]
and mapping conformally to the strip $\mathbb{S}$, we obtain $\mathrm{E}^{\mathrm{A}}+\mathrm{E}^{\mathrm{B}}=\Phi$. 
\end{proof}

\subsection{Computation of $\mathrm{E}^{\mathrm{A}}$\label{sub:computation-e-a}}

\begin{figure}
\includegraphics[width=11cm]{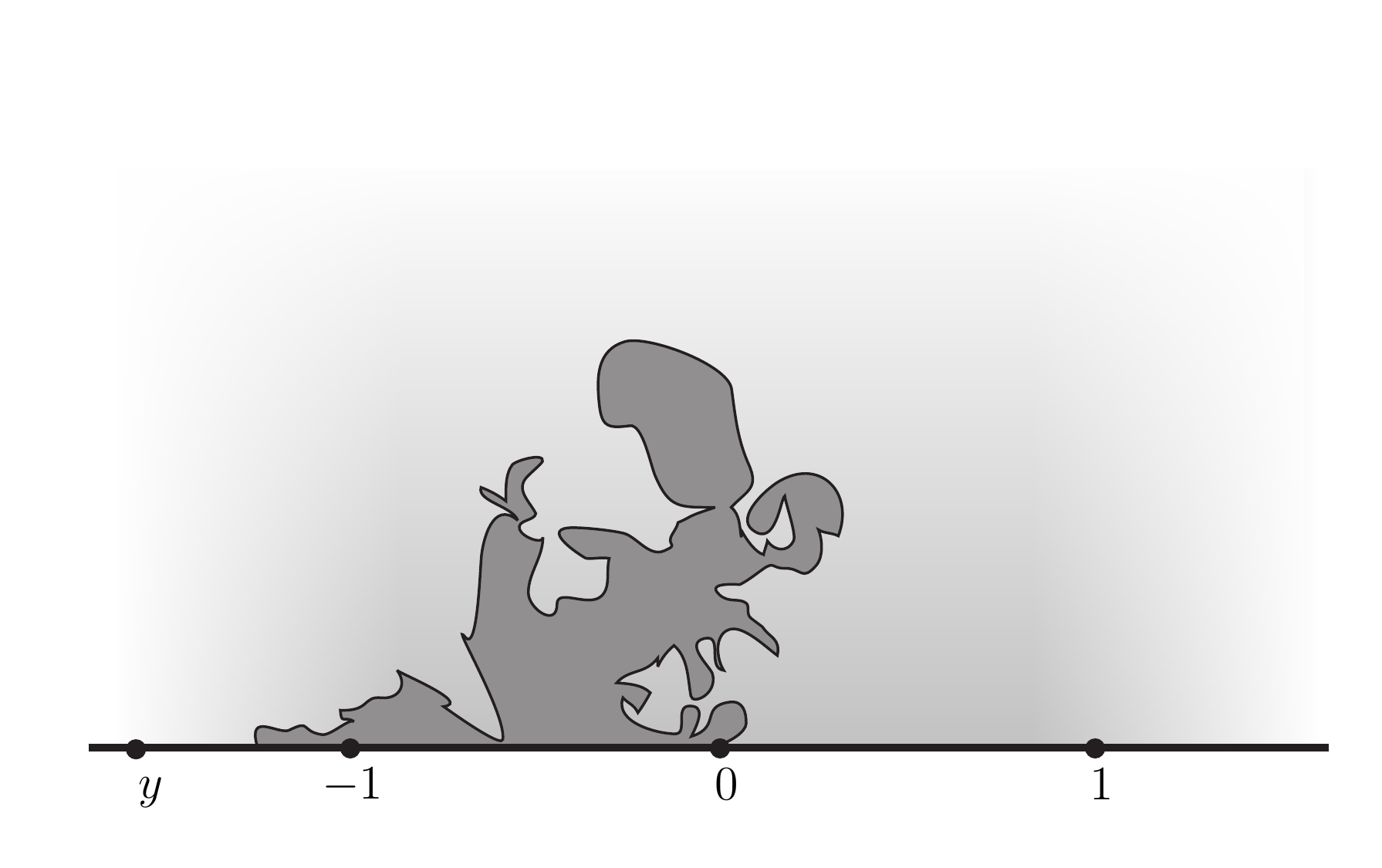}

\caption{Setup for $\mathrm{E}^{\mathrm{A}}\left(\mathbb{H},-1,0,1,y\right)$}
\end{figure}

\begin{prop}
\label{pro:computation-e-a}For any $y\in\left(-\infty,-1\right)$,
we have
\[
\mathrm{E}^{\mathrm{A}}\left(\mathbb{H},-1,0,1,y\right)=\frac{\mathfrak{m}^{\mathrm{A}}\left(\mathbb{H},-1,0,1,y\right)}{\left\langle \sigma_{1}\right\rangle _{\mathbb{H}}^{\left[-1,0\right]_{+}}}.
\]
\end{prop}
\begin{proof}
By SLE coordinate change (Section \ref{sub:chordal-sle-kappa-rho}),
we can take $\mathrm{E}^{\mathrm{A}}$ to be the expectation of
\[
\mathbb{E}\left[\frac{\left\langle \sigma_{1}\sigma_{y}\right\rangle _{\mathbb{H}\setminus\lambda}^{\mathrm{free}}}{\left\langle \sigma_{1}\right\rangle _{\mathbb{H}}^{\left[-1,0\right]_{+}}}\right],
\]
where $\lambda$ has the law of an SLE$\left(\kappa;\kappa-6\right)$
(with $\kappa=16/3$) in $\mathbb{H}$ with source $0$, observation
point $\infty$ and force point $-1$; again the integrand is set
to $0$ if $\lambda$ disconnects $y$ from $1$ before disconnecting
$-1$ from $\infty$. Parametrize $\lambda$ by the standard half-plane
capacity viewed from $\infty$. For each $t\geq0$, let $H_{t}$ be
the unbounded connected component of $\mathbb{H}\setminus\lambda\left[0,t\right]$.
Let $\tau^{\mathrm{A}}$ be the (almost surely finite) first time
when $1$ and $-1$ are disconnected by $\lambda$. Define $\left(\mathbf{M}_{t}^{\mathrm{A}}\right)_{t\geq0}$
by
\[
\mathbf{M}_{t}^{\mathrm{A}}:=\mathfrak{m}^{\mathrm{A}}\left(H_{t\wedge\tau^{\mathrm{A}}},-1,\lambda\left(t\wedge\tau^{\mathrm{A}}\right),1,y\right).
\]
By conformal covariance, if we denote by $g_{t}:H_{t}\to\mathbb{H}$
the conformal mapping with normalization $\lim_{z\to\infty}g_{t}\left(z\right)-z=0$,
we have, for $t<\tau^{\mathrm{A}}$,
\[
\mathbf{M}_{t}^{\mathrm{A}}=\sqrt{g_{t}'\left(y\right)}\sqrt{g_{t}'\left(1\right)}\mathfrak{m}^{\mathrm{A}}\left(\mathbb{H},g_{t}\left(-1\right),U_{t},g_{t}\left(1\right),g_{t}\left(y\right)\right),
\]
where $U_{t}=g_{t}\left(\lambda\left(t\right)\right)$ is the driving
force of the Loewner chain. The process $\mathbf{M}_{t}^{\mathrm{A}}$
has the following properties, shown further in this subsection:
\begin{itemize}
\item $\mathbf{M}_{t}^{\mathrm{A}}$ is a local martingale (Lemma \ref{lem:mta-local-mart});
\item $\mathbf{M}_{t}^{\mathrm{A}}$ is bounded (Lemma \ref{lem:mta-bded});
\item $\mathbf{M}_{t}^{\mathrm{A}}$ has the correct endvalue (Lemma \ref{lem:mta-endval}):
\[
\mathbf{M}_{\tau^{\mathrm{A}}}^{\mathrm{A}}=\left\langle \sigma_{1}\sigma_{y}\right\rangle _{\mathbb{H}\setminus\lambda\left[0,\tau^{\mathrm{A}}\right]}^{\mathrm{free}},
\]
where the right-hand side is zero if $y$ and $1$ are disconnected
by $\lambda\left[0,\tau^{\mathrm{A}}\right]$.
\end{itemize}
From these three lemmas, we deduce that $\mathbf{M}_{t}^{\mathrm{A}}$
is a uniformly integrable martingale and by the optional stopping
theorem, we hence obtain 
\[
\mathrm{E}^{\mathrm{A}}\left(\mathbb{H},-1,0,1,y\right)=\mathbb{E}\left[\frac{\mathbf{M}_{\tau^{\mathrm{A}}}^{\mathrm{A}}}{\left\langle \sigma_{1}\right\rangle _{\mathbb{H}}}\right]=\frac{\mathbf{M}_{0}^{\mathrm{A}}}{\left\langle \sigma_{1}\right\rangle _{\mathbb{H}}}=\frac{\mathfrak{m}^{\mathrm{A}}\left(\mathbb{H},-1,1,y\right)}{\left\langle \sigma_{1}\right\rangle _{\mathbb{H}}},
\]
which concludes the proof of the proposition.
\end{proof}
Let us show the three lemmas used in the proof of Proposition \ref{pro:computation-e-a}.
\begin{lem}
\label{lem:mta-local-mart}$\mathbf{M}_{t}^{\mathrm{A}}$ is a local
martingale.\end{lem}
\begin{proof}
With the same notation as in the proof of Proposition \ref{pro:computation-e-a},
for $t<\tau_{\mathrm{A}}$ we have 
\[
\mathbf{M}_{t}^{\mathrm{A}}=\sqrt{g_{t}'\left(y\right)}\sqrt{g_{t}'\left(1\right)}\mathfrak{m}^{\mathrm{A}}\left(\mathbb{H},g_{t}\left(-1\right),U_{t},g_{t}\left(1\right),g_{t}\left(y\right)\right),
\]
where
\begin{eqnarray*}
\mathrm{d}g_{t}\left(z\right) & = & \frac{2\mathrm{d}t}{g_{t}\left(z\right)-U_{t}},\\
\mathrm{d}U_{t} & = & \sqrt{\kappa}\mathrm{d}B_{t}+\frac{\rho\mathrm{d}t}{U_{t}-g_{t}\left(-1\right)}
\end{eqnarray*}
where $\kappa=16/3$ and $\rho=-2/3$. By Itô's calculus, we get that
the drift of $\mathbf{M}_{t}^{\mathrm{A}}$ is proportional to
\begin{eqnarray*}
 & \Bigg( & -\frac{1}{\left(\eta_{1}-\eta_{0}\right)^{2}}-\frac{1}{\left(\eta_{2}-\eta_{0}\right)^{2}}+\frac{2}{\eta_{\infty}-\eta_{0}}\frac{\partial}{\partial\eta_{\infty}}+\frac{2}{\eta_{1}-\eta_{0}}\frac{\partial}{\partial\eta_{1}}+\\
 &  & \frac{2}{\eta_{2}-\eta_{0}}\frac{\partial}{\partial\eta_{2}}-\frac{2}{3}\frac{1}{\eta_{0}-\eta_{\infty}}\frac{\partial}{\partial\eta_{0}}+\frac{8}{3}\frac{\partial^{2}}{\partial\eta_{0}^{2}}\,\,\Bigg)\,\,\mathfrak{m}^{\mathrm{A}}\left(\mathbb{H},\eta_{\infty},\eta_{0},\eta_{1},\eta_{2}\right),
\end{eqnarray*}
evaluated at $\eta_{\infty}=g_{t}\left(-1\right)$, $\eta_{0}=U_{t}$,
$g_{t}\left(y\right)$, $g_{t}\left(1\right)$. From the explicit
form of $\mathfrak{m}^{\mathrm{A}}$, it is easy to check that this
expression vanishes. \end{proof}
\begin{lem}
\label{lem:mta-bded}$\mathbf{M}_{t}^{\mathrm{A}}$ is bounded.\end{lem}
\begin{proof}
By Remark \ref{rem:cts-spin-spin-poisson-kernel}, we have the following
monotonicity property for the integrand
\[
\left\langle \sigma_{1}\sigma_{y}\right\rangle _{\mathbb{H}\setminus\lambda\left[0,t\right]}^{\mathrm{free}}\leq\left\langle \sigma_{1}\sigma_{y}\right\rangle _{\mathbb{H}}^{\mathrm{free}}
\]
for any $\lambda$. Hence, we obtain that $\mathbf{M}_{t}^{\mathrm{A}}$
is bounded by the right-hand side. \end{proof}
\begin{lem}
\label{lem:mta-endval}On stopping time $\tau^{\mathrm{A}}$, $\mathbf{M}_{t}^{\mathrm{A}}$
has the endvalue
\[
\mathbf{M}_{\tau^{\mathrm{A}}}^{\mathrm{A}}=\left\langle \sigma_{1}\sigma_{y}\right\rangle _{\mathbb{H}\setminus\lambda\left[0,\tau^{\mathrm{A}}\right]}^{\mathrm{free}}.
\]
\end{lem}
\begin{proof}
Recall that 
\[
\mathbf{M}_{t}^{\mathrm{A}}=\sqrt{g_{t}'\left(1\right)}\sqrt{g_{t}'\left(y\right)}\mathfrak{m}^{\mathrm{A}}\left(\mathbb{H};g_{t}\left(-1\right),U_{t},g_{t}\left(1\right),g_{t}\left(y\right)\right).
\]
We should show that 
\[
\lim_{t\to\tau^{\mathrm{A}}}\mathbf{M}_{t}^{\mathrm{A}}=\left\langle \sigma_{1}\sigma_{y}\right\rangle _{\mathbb{H}\setminus\lambda\left[0,\tau^{\mathrm{A}}\right]}^{\mathrm{free}}.
\]
Denote by $\chi_{t}$ the cross-ratio defined by
\[
\chi_{t}:=\frac{U_{t}-g_{t}\left(-1\right)}{U_{t}-g_{t}\left(y\right)}\frac{g_{t}\left(1\right)-g_{t}\left(y\right)}{g_{t}\left(1\right)-g_{t}\left(-1\right)}.
\]
As $t\to\tau^{\mathrm{A}}$, there are three possibilities:
\begin{itemize}
\item $\lambda$ reaches the interval $(y,-1]$: in this case, $U_{t}-g_{t}\left(-1\right)\to0$
while other distances remain positive, so in particular $\chi_{t}\to0$.
Using the same hypergeometric representation as in the proof of Proposition
\ref{pro:sle-averages} and the identity (see \cite[Eq 15.3.6]{abramowitz-stegun})
\begin{eqnarray*}
_{2}F_{1}\left(\frac{3}{4},2;\frac{9}{4};1-\chi\right) & = & -\frac{5}{4}\,_{2}F_{1}\left(\frac{3}{4},2;\frac{3}{2};\chi\right)\\
 &  & +\frac{5\sqrt{\pi}}{16}\frac{\Gamma\left(\frac{1}{4}\right)}{\Gamma\left(\frac{3}{4}\right)}\chi^{-\frac{1}{2}}\,_{2}F_{1}\left(\frac{3}{2},\frac{1}{4};\frac{1}{2};\chi\right),
\end{eqnarray*}
we see that in this case, $M_{t}^{\mathrm{A}}$ tends to
\[
C_{\mathrm{A}}^{\mathfrak{m}}\frac{\sqrt{\pi}\Gamma\left(\frac{1}{4}\right)}{\Gamma\left(\frac{3}{4}\right)}\frac{\sqrt{g_{\tau^{\mathrm{A}}}'\left(1\right)}\sqrt{g_{\tau^{\mathrm{A}}}'\left(y\right)}}{g_{\tau^{\mathrm{A}}}\left(1\right)-g_{\tau^{\mathrm{A}}}\left(y\right)},
\]
which is equal to $\left\langle \sigma_{1}\sigma_{y}\right\rangle _{\mathbb{H}\setminus\lambda\left[0,\tau^{\mathrm{A}}\right]}^{\mathrm{free}}$.
\item $\lambda$ reaches the interval $[1,\infty)$: in this case, $g_{t}\left(1\right)-U_{t}\to0$
while other distances remain positive, so $\chi_{t}\to1$. Observing
that $\left|g_{t}'\left(1\right)\right|\leq1$ and $\left|g_{t}'\left(y\right)\right|\leq1$
and using the explicit expression for $\mathfrak{m}^{\mathrm{A}}$,
we see that in this case $\mathbf{M}_{t}^{\mathrm{A}}\to0$. We have
that $1$ and $y$ get disconnected from each other by $\lambda\left[0,\tau^{\mathrm{A}}\right]$,
so 
\[
\lim_{t\to\tau^{\mathrm{A}}}\mathbf{M}_{t}^{\mathrm{A}}=0=\left\langle \sigma_{1}\sigma_{y}\right\rangle _{\mathbb{H}\setminus\lambda\left[0,\tau^{\mathrm{A}}\right]}^{\mathrm{free}}.
\]

\item $\lambda$ reaches the interval $(-\infty,y]$: in this case, both
$U_{t}-g_{t}\left(y\right)$ and $U_{t}-g_{t}\left(-1\right)$ tend
to zero and considerations of harmonic measure show that
\[
\frac{U_{t}-g_{t}\left(-1\right)}{U_{t}-g_{t}\left(y\right)}\to1,
\]
so $\chi_{t}\to1$. As in the previous case, one concludes that both
the limit of $\mathbf{M}_{t}^{\mathrm{A}}$ and the spin correlation
$\left\langle \sigma_{1}\sigma_{y}\right\rangle _{\mathbb{H}\setminus\lambda\left[0,\tau^{\mathrm{A}}\right]}^{\mathrm{free}}$
are zero. 
\end{itemize}
\end{proof}

\subsection{Computation of $\mathrm{E}^{\mathrm{B}}$\label{sub:computation-e-b}}

\begin{figure}

\includegraphics[width=11cm]{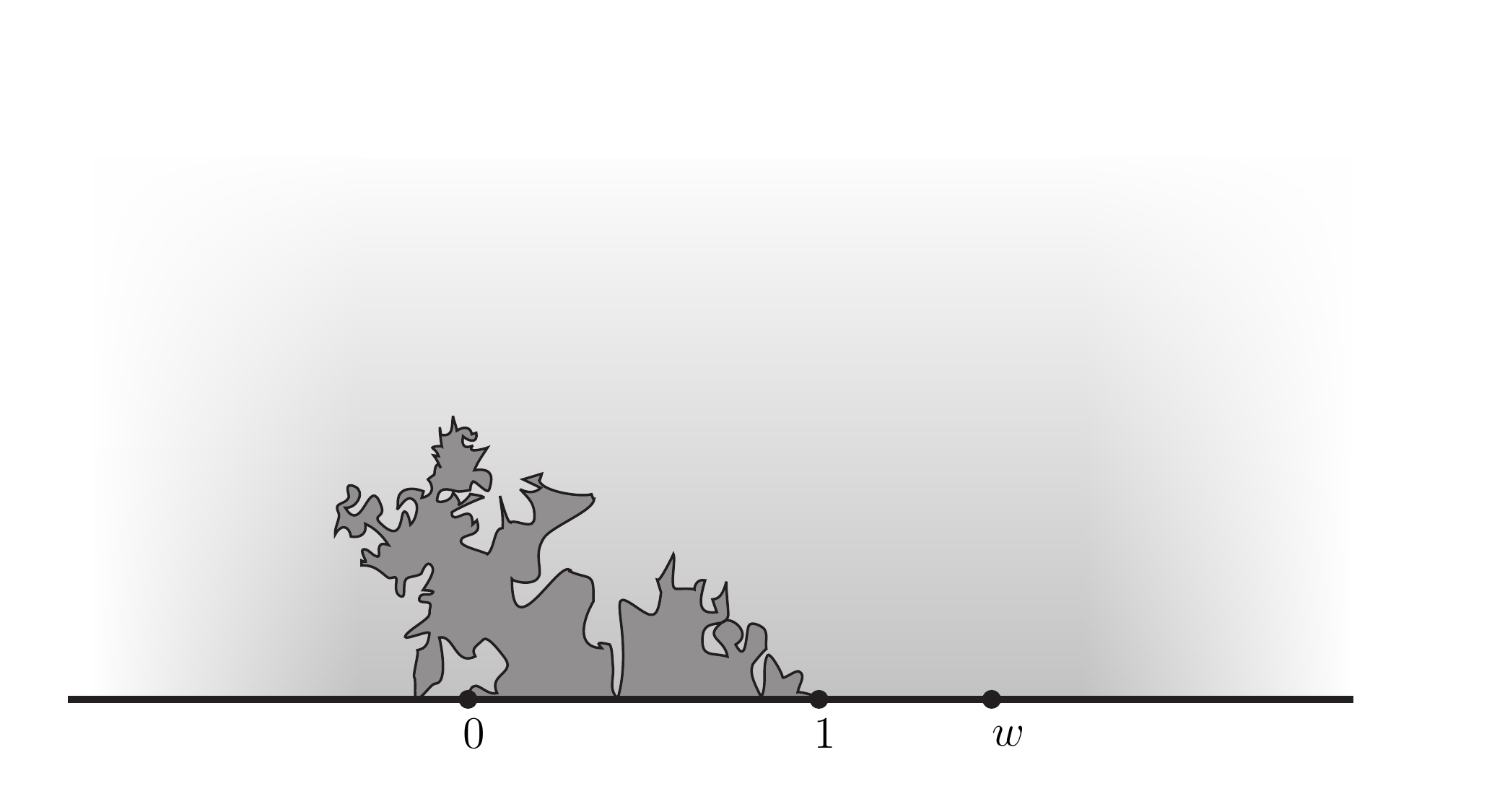}

\caption{\label{fig:setup-eb}Setup for $\mathrm{E}^{\mathrm{B}}\left(\mathbb{H},\infty,0,1,w\right)$ }
\end{figure}

\begin{prop}
\label{pro:computation-e-b}For any $w>1$, we have
\begin{eqnarray*}
\mathrm{E}^{\mathrm{B}}\left(\mathbb{H},\infty,0,1,w\right) & = & \mathfrak{m}^{\mathrm{B}}\left(\mathbb{H},\infty,0,1,w\right).
\end{eqnarray*}
\end{prop}
\begin{proof}
The proof is similar to the one of Proposition \ref{pro:computation-e-a}.
Let $\tilde{\lambda}$ have the law of an SLE$\left(16/3,-8/3\right)$
trace with source point $0$, observation point $\infty$ and force
point $1$. Consider the standard half-plane capacity parametrization
of the curve and for each $t\geq0$, let $H_{t}$ be the unbounded
connected component of $\mathbb{H}\setminus\lambda\left[0,t\right]$.
Let $\tau^{\mathrm{B}}$ be the (almost surely finite) hitting time
of $1$. Let $\mathbf{M}_{t}^{\mathrm{B}}$ be defined by
\[
\mathbf{M}_{t}^{\mathrm{B}}:=\mathfrak{m}^{\mathrm{B}}\left(H_{t\wedge\tau^{\mathrm{B}}},\infty,\tilde{\lambda}\left(t\wedge\tau^{\mathrm{B}}\right),1,w\right).
\]
By conformal covariance, if we denote by $g_{t}:H_{t}\to\mathbb{H}$
the conformal mapping with normalization $\lim_{z\to\infty}g_{t}\left(z\right)-z=0$,
we have, for $t<\tau_{\mathrm{B}}$
\[
\mathbf{M}_{t}^{\mathrm{B}}=\sqrt{g_{t}'\left(w\right)}\mathfrak{m}^{\mathrm{B}}\left(\mathbb{H},\infty,U_{t},g_{t}\left(1\right),g_{t}\left(w\right)\right),
\]
where $U_{t}=g_{t}\left(\lambda\left(t\right)\right)$ is the driving
process of the Loewner chain. We then have the following properties,
shown later in this subsection
\begin{itemize}
\item $\mathbf{M}_{t}^{\mathrm{B}}$ is a local martingale (Lemma \ref{lem:mtb-loc-mart});
\item $\mathbf{M}_{t}^{\mathrm{B}}$ is uniformly integrable (Lemma \ref{lem:mtb-unif-int});
\item $\mathbf{M}_{t}^{\mathrm{B}}$ has the correct endvalue (Lemma \ref{lem:mtb-endval}):
\[
\mathbf{M}_{\tau^{\mathrm{B}}}^{\mathrm{B}}=\left\langle \sigma_{w}\right\rangle _{\mathbb{H}\setminus\tilde{\lambda}\left[0,\tau^{\mathrm{B}}\right]}^{\left[-\infty,1\right]_{+}}.
\]

\end{itemize}
From these properties, we deduce, by the optional stopping theorem:
\[
\mathrm{E}^{\mathrm{B}}\left(\mathbb{H},\infty,0,1,w\right)=\mathbb{E}\left[\mathbf{M}_{\tau^{\mathrm{B}}}^{\mathrm{B}}\right]=\mathbf{M}_{0}^{\mathrm{B}}=\mathfrak{m}^{\mathrm{B}}\left(\mathbb{H},\infty,0,1,w\right).
\]
\end{proof}
\begin{lem}
\label{lem:mtb-loc-mart}$\mathbf{M}_{t}^{\mathrm{B}}$ is a local
martingale.\end{lem}
\begin{proof}
Writing, as in the proof of Proposition \ref{pro:computation-e-b},
for $t<\tau^{\mathrm{B}}$, 
\[
\mathbf{M}_{t}^{\mathrm{B}}=\sqrt{g_{t}'\left(w\right)}\mathfrak{m}_{t}^{\mathrm{B}}\left(\mathbb{H},\infty,U_{t},g_{t}\left(1\right),g_{t}\left(w\right)\right),
\]
and using that 
\begin{eqnarray*}
\mathrm{d}g_{t}\left(z\right) & = & \frac{2\mathrm{d}t}{g_{t}\left(z\right)-U_{t}},\\
\mathrm{d}U_{t} & = & \sqrt{\kappa}\mathrm{d}B_{t}+\frac{\rho\mathrm{d}t}{U_{t}-g_{t}\left(1\right)}
\end{eqnarray*}
where $\kappa=16/3$ and $\rho=-8/3$. Itô's calculus gives that the
drift of $\mathbf{M}_{t}^{\mathrm{B}}$ is proportional to
\begin{eqnarray*}
\left(-\frac{1}{\left(\eta_{2}-\eta_{0}\right)^{2}}+\frac{2}{\eta_{1}-\eta_{0}}\frac{\partial}{\partial\eta_{1}}+\frac{2}{\eta_{2}-\eta_{0}}\frac{\partial}{\partial\eta_{2}}-\frac{8}{3}\frac{1}{\eta_{0}-\eta_{1}}\frac{\partial}{\partial\eta_{0}}+\frac{8}{3}\frac{\partial^{2}}{\partial\eta_{0}^{2}}\right)\\
\mathfrak{m}^{\mathrm{B}}\left(\mathbb{H},\infty,\eta_{0},\eta_{1},\eta_{2}\right),
\end{eqnarray*}
evaluated at $\eta_{0}=U_{t}$, $\eta_{1}=g_{t}\left(1\right)$, $\eta_{2}=g_{t}\left(w\right)$.
From the explicit formula for $\mathfrak{m}^{\mathrm{B}}$, we get
that this expression is zero. \end{proof}
\begin{lem}
\label{lem:mtb-unif-int}$\mathbf{M}_{t}^{\mathrm{B}}$ is uniformly
integrable\end{lem}
\begin{proof}
The proof of this is completely analogous to the one of the convergence
of $\mathrm{E}_{\delta}^{\mathrm{B}}$ in Section \ref{sub:convergence-ebdelta}.
It indeed follows from Lemmas \ref{lem:prob-cond-curve-getting-close}
and \ref{lem:blowup-rate-magnetization}, passed to the $\delta\to0$
limit, using Theorems \ref{thm:spin-spin-corr-free} and \ref{thm:boundary-magnetization}:
the probability that the curve $\tilde{\lambda}$ gets $\epsilon$-close
to $w$ decays like $\mathcal{O}\left(\epsilon\right)$, while the
blow-up of the integrand as the curve gets $\epsilon$-close to $w$
is only $\mathcal{O}\left(\frac{1}{\sqrt{\epsilon}}\right)$.\end{proof}
\begin{lem}
\label{lem:mtb-endval}$\mathbf{M}_{t}^{\mathrm{B}}$ has the endvalue:
\[
\mathbf{M}_{\tau^{\mathrm{B}}}^{\mathrm{B}}=\left\langle \sigma_{w}\right\rangle _{\mathbb{H}\setminus\tilde{\lambda}\left[0,\tau^{\mathrm{B}}\right]}^{\left[-\infty,1\right]_{+}}.
\]
\end{lem}
\begin{proof}
By continuity, we should show that
\[
\mathbf{M}_{t}^{\mathrm{B}}\underset{t\to\tau^{\mathrm{B}}}{\longrightarrow}\left\langle \sigma_{w}\right\rangle _{\mathbb{H}\setminus\tilde{\lambda}\left[0,\tau^{\mathrm{B}}\right]}^{\left[-\infty,1\right]_{+}}.
\]
By conformal covariance, we have that
\[
\mathbf{M}_{t}^{\mathrm{B}}=\sqrt{g_{t}'\left(w\right)}\mathfrak{m}^{\mathrm{B}}\left(\mathbb{H},\infty,U_{t},g_{t}\left(1\right),g_{t}\left(w\right)\right)
\]
and since as $t\to\tau^{\mathrm{B}}$, we have $g_{t}\left(1\right)-U_{t}\to0$
(as the tip of the curve $\tilde{\lambda}\left(t\right)$ tends to
$1$) and $g_{t}\left(w\right)-U_{t}$ remains bounded away from $0$,
it is enough to show (again by conformal covariance) that for any
$w>1$,
\[
\lim_{z\to0_{+}}\mathfrak{m}^{\mathrm{B}}\left(\mathbb{H},\infty,0,z,w\right)=\left\langle \sigma_{w}\right\rangle _{\mathbb{H}}^{\left[-\infty,0\right]_{+}}.
\]
To obtain these asymptotics, we use the same hypergeometric representation
of $\mathfrak{m}^{\mathrm{B}}$ and then the same decomposition formula
for the hypergeometric function $_{2}F_{1}$ as in the Proof of Proposition
\ref{pro:sle-averages}: 
\begin{eqnarray*}
\mathfrak{m}^{\mathrm{B}}\left(\mathbb{H},\infty,0,z,w\right) & = & \frac{2\Gamma\left(\frac{3}{4}\right)^{2}C_{\mathrm{B}}^{\mathfrak{m}}}{\sqrt{\pi}}\frac{1}{\left(w-z\right)^{\frac{1}{2}}}\,_{2}F_{1}\left(-\frac{1}{2},\frac{3}{4};\frac{3}{2};1-\frac{z}{w}\right)\\
 & = & \frac{\Gamma\left(\frac{1}{4}\right)\Gamma\left(\frac{3}{4}\right)C_{\mathrm{B}}^{\mathfrak{m}}}{4}\frac{1}{\left(w-z\right)^{\frac{1}{2}}}\,{}_{2}F_{1}\left(-\frac{1}{2},\frac{3}{4};-\frac{1}{4};\frac{z}{w}\right)\\
 &  & -\frac{8\Gamma\left(\frac{3}{4}\right)^{2}C_{\mathrm{B}}^{\mathfrak{m}}}{5\sqrt{\pi}}\frac{1}{\left(w-z\right)^{\frac{1}{2}}}\left(\frac{z}{w}\right)^{\frac{5}{4}}\,_{2}F_{1}\left(\frac{3}{4},2;\frac{9}{4};\frac{z}{w}\right)\\
 & \underset{z\to0}{\longrightarrow} & \frac{\Gamma\left(\frac{1}{4}\right)\Gamma\left(\frac{3}{4}\right)C_{\mathrm{B}}^{\mathfrak{m}}}{4}\frac{1}{\sqrt{w}}\\
 & = & \sqrt{\frac{\sqrt{2}+1}{2\pi}}\cdot\frac{1}{\sqrt{w}}\\
 & = & \left\langle \sigma_{w}\right\rangle _{\mathbb{H}}^{\left[-\infty,0\right]_{+}}.
\end{eqnarray*}
 
\end{proof}

\section{Discrete Complex Analysis\label{sec:discrete-complex-analysis}}

In this section, we show the convergence of the ratios of elementary
Ising correlation functions to CFT correlation functions stated in
Section \ref{sec:cv-element-corr-func} (Theorem \ref{thm:corr-ratios-theorem}). 
\begin{thm*}[Theorem \ref{thm:corr-ratios-theorem}]
Let $\left(\Theta,y,t,x\right)$, $\left(\tilde{\Theta},\tilde{y},\tilde{t},x\right)$
and $\left(\Xi,x,s\right)$ be domains which coincide in a neighborhood
of the boundary point $x$. Suppose that $s$ lies on a vertical part
$\mathfrak{v}$ of $\partial\Xi$. Let $\left(\Theta_{\delta},y_{\delta},t_{\delta},x_{\delta}\right)$,
$\left(\tilde{\Theta}_{\delta},\tilde{y}_{\delta},\tilde{t}_{\delta},x_{\delta}\right)$
and $\left(\Xi_{\delta},x_{\delta},s_{\delta}\right)$ be discretizations
of these domains coinciding in a neighborhood of $x_{\delta}$, converging
to their continuous counterparts in the sense of the metric of Section
\ref{sub:uniformity-convergence}. Suppose that for each $\delta>0$,
$\partial\Xi_{\delta}$ contains a vertical part $\mathfrak{v}_{\delta}$
around $s_{\delta}$ and that as $\delta\to0$, $\mathfrak{v}_{\delta}$
converges to $\mathfrak{v}$. Then we have
\begin{eqnarray}
\frac{1}{\sqrt{\delta}}\frac{\mathbb{E}_{\Xi_{\delta}}^{\mathrm{free}}\left[\sigma_{x_{\delta}}\sigma_{s_{\delta}}\right]}{\mathbb{E}_{\Theta_{\delta}}^{\left[y_{\delta},t_{\delta}\right]_{+}}\left[\sigma_{x_{\delta}}\right]} & \underset{\delta\to0}{\longrightarrow} & \frac{\left\langle \sigma_{x}\sigma_{s}\right\rangle _{\Xi}^{\mathrm{free}}}{\left\langle \sigma_{x}\right\rangle _{\Theta}^{\left[y,t\right]_{+}}}.\label{eq:spin-over-fk-cv}\\
\frac{\mathbb{E}_{\tilde{\Theta}_{\delta}}^{\left[\tilde{y}_{\delta},\tilde{t}_{\delta}\right]_{+}}\left[\sigma_{x_{\delta}}\right]}{\mathbb{E}_{\Theta_{\delta}}^{\left[y_{\delta},t_{\delta}\right]_{+}}\left[\sigma_{x_{\delta}}\right]} & \underset{\delta\to0}{\longrightarrow} & \frac{\left\langle \sigma_{x}\right\rangle _{\tilde{\Theta}}^{\left[\tilde{y}_{\delta},\tilde{t}_{\delta}\right]_{+}}}{\left\langle \sigma_{x}\right\rangle _{\Theta}^{\left[y_{\delta},t_{\delta}\right]_{+}}}\label{eq:fk-over-fk-cv}
\end{eqnarray}
The convergence is locally uniform with respect to the domains. 
\end{thm*}
The proof of Theorem \ref{thm:corr-ratios-theorem} is given in Section
\ref{sub:correlation-theorem-proof}. The key tool, introduced in
\cite{smirnov-ii} and further developed in \cite{chelkak-smirnov-i,chelkak-smirnov-ii,hongler-smirnov-ii}
is (a type of) discrete complex analysis. More precisely, the structure
of this section is as follows:
\begin{itemize}
\item In Section \ref{sub:disc-comp-analysis-notation}, we precisely define
the graphs and notations that are suited for the discrete complex
analysis tools that we use.
\item In Section \ref{sub:discrete-holomorphic-observables}, we define
and give basic properties of the discrete holomorphic observables
that are instrumental to compute the discrete correlation functions
of Theorem \ref{thm:corr-ratios-theorem}.
\item In Section \ref{sub:bv-as-corr-func}, we obtain the discrete correlation
functions as the boundary values of the discrete holomorphic observables.
\item In Section \ref{sub:discrete-riemann-bvp}, we formulate a discrete
Riemann boundary value problem which provides a convenient local representation
of the discrete holomorphic observables in terms of a convolution
kernel.
\item In Section \ref{sub:continous-holomorphic-obs}, we introduce the
continuous counterpart of the discrete holomorphic observables.
\item In Section \ref{cft-corr-functions}, we obtain the CFT correlation
functions appearing in Theorem \ref{thm:corr-ratios-theorem} as boundary
values of the continuous holomorphic observables.
\item In Section \ref{sub:cts-riemann-bvp}, we formulate the continuous
Riemann boundary value problem which gives a local representation
of the continuous observables.
\item In Section \ref{sub:conv-obs}, we derive the convergence (in the
bulk and on straight parts of the boundary) of the discrete holomorphic
observables to the continuous ones.
\item In Section \ref{sub:correlation-theorem-proof}, we show Theorem \ref{thm:corr-ratios-theorem}.
To do this, we extend to the boundary the convergence results of Section
\ref{sub:conv-obs} for appropriate ratios.
\item In Section \ref{sub:apriori-est-hol-obs}, we show the a priori estimates
used in Section \ref{sub:correlation-theorem-proof} to prove Theorem
\ref{thm:corr-ratios-theorem}.
\end{itemize}

\subsection{Graphs, notation and definitions\label{sub:disc-comp-analysis-notation}}

Let $\Omega_{\delta}$ be a discrete vertex domain: a connected subgraph
of the square grid $\mathbb{C}_{\delta}:=\delta\mathbb{Z}^{2}$. 
\begin{itemize}
\item We denote by $\mathcal{V}_{\Omega_{\delta}}$ the set of vertices
and $\mathcal{E}_{\Omega_{\delta}}$ the set of edges of $\Omega_{\delta}$.
\item We denote by $\mathrm{Int}\left(\Omega_{\delta}\right)$ the complex
domain bounded by the dual circuit made of edges of $\mathcal{E}_{\mathbb{C}_{\delta}^{*}\setminus\Omega_{\delta}^{*}}$
(see Figure \ref{fig:graph-notation-discrete-analysis}), by $\partial\hat{\Omega}_{\delta}$
the set of its prime ends and by $\hat{\Omega}_{\delta}$ its Carathédory
compactification (for a definition of these notions, see \cite[Chapters 1,2]{pommerenke},
for instance).
\item We denote by $\partial\mathcal{E}_{\Omega_{\delta}}$ the set of edges
of $\mathcal{E}_{\mathbb{C}_{\delta}}\setminus\mathcal{E}_{\Omega_{\delta}}$
that are incident to a vertex of $\mathcal{V}_{\Omega_{\delta}}$,
counted with multiplicity: if an edge $e\in\mathcal{E}_{\mathbb{C}_{\delta}}\setminus\mathcal{E}_{\Omega_{\delta}}$
is incident to two vertices of $\mathcal{V}_{\Omega_{\delta}}$, it
appears as two distinct elements of $\partial\mathcal{E}_{\Omega_{\delta}}$.
\item We denote by $\partial\mathcal{V}_{\Omega_{\delta}}$ the set of vertices
of $\mathcal{V}_{\mathbb{C}_{\delta}}\setminus\mathcal{V}_{\Omega_{\delta}}$
incident to $\partial\mathcal{E}_{\Omega_{\delta}}$, counted with
multiplicity: if two edges of $\partial\mathcal{E}_{\delta}$ are
incident to a vertex $v\in\mathcal{V}_{\mathbb{C}_{\delta}}\setminus\mathcal{V}_{\Omega_{\delta}}$,
that vertex counts as two elements of $\partial\mathcal{V}_{\Omega_{\delta}}$. 
\item We define $\overline{\mathcal{V}}_{\Omega_{\delta}}:=\mathcal{V}_{\Omega_{\delta}}\cup\partial\mathcal{V}_{\Omega_{\delta}}$.
\item We denote by $\Omega_{\delta}^{*}$ the dual graph of $\Omega_{\delta}$,
whose vertex set $\mathcal{V}_{\Omega_{\delta}^{*}}$ consists of
the midpoints of the bounded faces of $\Omega_{\delta}$ and whose
edge set $\mathcal{E}_{\Omega_{\delta}^{*}}$ consists of all pairs
of dual vertices of $\mathcal{V}_{\Omega_{\delta}^{*}}$ corresponding
to adjacent faces of $\Omega_{\delta}$.
\item We denote by $\partial\mathcal{V}_{\Omega_{\delta}^{*}}$ the set
of dual vertices of $\mathcal{V}_{\mathbb{C}_{\delta}^{*}}\setminus\mathcal{V}_{\Omega_{\delta}^{*}}$
that are adjacent to a face of $\mathcal{V}_{\Omega_{\delta}^{*}}$,
counted with multiplcity: a face of $\mathcal{V}_{\mathbb{C}_{\delta}^{*}}\setminus\mathcal{V}_{\Omega_{\delta}^{*}}$
appear as as many elements of $\partial\mathcal{V}_{\Omega_{\delta}^{*}}$
as there are faces of $\mathcal{V}_{\Omega_{\delta}}$ it is adjacent
to.
\item We define $\overline{\mathcal{V}}_{\Omega_{\delta}^{*}}:=\mathcal{V}_{\Omega_{\delta}^{*}}\cup\partial\mathcal{V}_{\Omega_{\delta}^{*}}.$
\item We denote $\Omega_{\delta}^{m}$ the medial graph, whose vertex set
$\mathcal{V}_{\Omega_{\delta}^{m}}$ consists of the midpoints of
edges of $\mathcal{E}_{\Omega_{\delta}}\cup\partial\mathcal{E}_{\Omega_{\delta}}$
and whose edge set $\mathcal{E}_{\Omega_{\delta}^{m}}$ consists of
all pairs of medial vertices $\mathcal{V}_{\Omega_{\delta}^{m}}$
corresponding to incident edges of $\mathcal{E}_{\Omega_{\delta}}\cup\partial\mathcal{E}_{\Omega_{\delta}}$.
\item We denote by $\partial_{0}\mathcal{V}_{\Omega_{\delta}^{m}}$ the
set of midpoints of edges of $\partial\mathcal{E}_{\Omega_{\delta}}$.
The vertices of $\partial_{0}\mathcal{V}_{\Omega_{\delta}^{m}}$ get
naturally identified with prime ends of $\partial\hat{\Omega}_{\delta}$.
\item We denote by $\partial_{0}\mathcal{E}_{\Omega_{\delta}^{m}}$ the
set of medial edges incident to a medial vertex of $\partial_{0}\mathcal{V}_{\Omega_{\delta}^{m}}$.
\item We denote by $\Omega_{\delta}^{m*}$ the dual of the medial graph,
whose vertex set $\mathcal{V}_{\Omega_{\delta}^{m*}}$ is identified
with $\mathcal{V}_{\Omega_{\delta}}\cup\mathcal{V}_{\Omega_{\delta}^{*}}$. 
\item For two boundary medial vertices $v_{1},v_{2}\in\partial_{0}\mathcal{V}_{\Omega_{\delta}^{m}}$,
we denote by $\partial_{0}\mathcal{V}_{\left[v_{1},v_{2}\right]_{\delta}^{m}}\subset\partial_{0}\mathcal{V}_{\Omega_{\delta}^{m}}$
the set boundary medial vertices identified with prime ends of the
counterclockwise arc $\left[v_{1},v_{2}\right]\subset\partial\hat{\Omega}_{\delta}$.
We denote by $\partial_{0}\mathcal{E}_{\left[v_{1},v_{2}\right]_{\delta}^{m}}\subset\partial_{0}\mathcal{E}_{\Omega_{\delta}^{m}}$
the set of boundary medial edges incident to a vertex of $\partial_{0}\mathcal{V}_{\left[v_{1},v_{2}\right]_{\delta}^{m}}$.
\item With each medial edge $e\in\mathcal{E}_{\Omega_{\delta}^{m}}$, we
associate a line $\ell\left(e\right)\subset\mathbb{C}$ in the complex
plane, defined by $\ell\left(e\right):=\left(\mathfrak{m}-\mathfrak{c}\right)^{-\frac{1}{2}}\mathbb{R}$,
where $\mathfrak{m}$ is the midpoint of $e$ and $\mathfrak{c}\in\mathcal{V}_{\Omega_{\delta}}$
is the vertex of $\Omega_{\delta}$ that is the closest to $e$. 
\item We say that a vertex $v_{1}$ and a medial vertex $v_{2}$ are adjacent
if $v_{1}$ is incident to the edge whose midpoint is $v_{2}$.
\item For a line $\ell=e^{i\theta}\mathbb{R}$ in the complex plane, we
denonte $\mathsf{P}_{\ell}$ the orthgonal projection onto that line,
defined by
\[
\mathsf{P}_{\ell}\left[z\right]:=\frac{1}{2}\left(z+e^{2i\theta}\overline{z}\right)\,\,\,\,\forall z\in\mathbb{C}
\]

\item Given two complex numbers $z_{1},z_{2}\in\mathbb{C}$, we write $z_{1}\parallel z_{2}$
if $z_{1}$ is a real multiple of $z_{2}$.
\item For each boundary medial vertex $z\in\partial_{0}\mathcal{V}_{\Omega_{\delta}^{m}}$,
we denote by $\nu_{\mathrm{out}}\left(z\right)$ the unit outward-pointing
normal of $\mathcal{D}_{\Omega_{\delta}}$ at $z$, i.e. the complex
number $\frac{2}{\delta}\left(z-v\right)$, where $v\in\mathcal{V}_{\Omega_{\delta}}$
is the vertex incident to the edge $e\in\partial\mathcal{E}_{\Omega_{\delta}^{m}}$
whose midpoint is $z$. We define $\nu_{\mathrm{in}}\left(z\right)$
as $-\nu_{\mathrm{out}}\left(z\right)$. 
\item We say that a function $f:\mathcal{V}_{\Omega_{\delta}^{m}}\to\mathbb{C}$
is s-holomorphic if for each $e=\left\langle v_{1},v_{2}\right\rangle \in\mathcal{E}_{\Omega_{\delta}^{m}}$,
we have
\[
\mathsf{P}_{\ell\left(e\right)}\left[f\left(v_{1}\right)\right]=\mathsf{P}_{\ell\left(e\right)}\left[f\left(v_{2}\right)\right].
\]

\end{itemize}
\begin{figure}
\includegraphics[width=10cm]{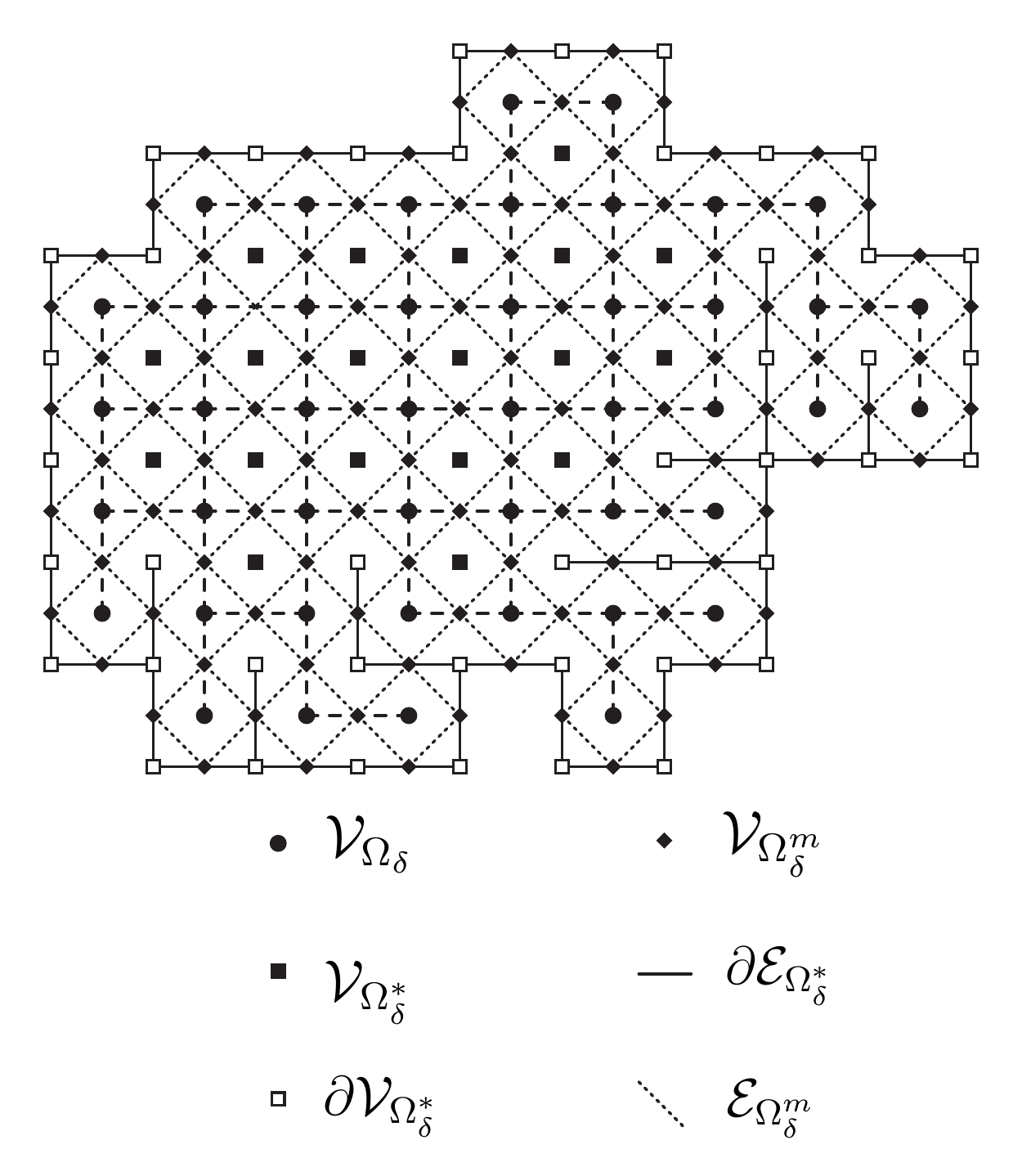}

\caption{\label{fig:graph-notation-discrete-analysis}Vertex domain, its medial
and its dual.}
\end{figure}

\subsection{Discrete holomorphic observables\label{sub:discrete-holomorphic-observables}}

We now define the two discrete holomorphic observables that are instrumental
in our analysis\@. These functions are defined on the medial graph
of a discrete vertex domain, and their boundary values give the correlation
functions appearing in Theorem \ref{thm:corr-ratios-theorem}. As
one of these observables is more naturally defined in terms of the
FK-Ising model and the other in terms of the high-temperature expansion
of the spin correlations of the Ising model, we will refer to these
as the FK(-Ising) and spin observables.

\subsubsection{FK observable\label{sub:fk-observable-def}}

The FK observable was originally introduced in \cite{smirnov-i} and
studied in \cite{smirnov-ii,chelkak-smirnov-ii} to show the convergence
of the critical FK-Ising interfaces to SLE$\left(16/3\right)$. The
key result in this proof is the scaling limit of the observable (see
Theorem \ref{thm:fk-obs-cv} below). The observable has also proven
to be useful to obtain estimates for crossing probabilities \cite{duminil-copin-hongler-nolin}.
Its boundary values of are of particular interest, as they give the
boundary magnetization with mixed $+/\mathrm{free}$ boundary conditions
(see Section \ref{sub:bv-as-corr-func}).

Let $\left(\Omega_{\delta},r,\ell\right)$ be a discrete domain and
consider the FK-Ising model on $\Omega_{\delta}$, with wired boundary
condition on $\left[r,\ell\right]$ (see Section \ref{sec:fk-representation}).
Let $r_{m},\ell_{m}\in\partial_{0}\mathcal{V}_{\Omega_{\delta}^{m}}$
be the medial vertices separating $\left[r,\ell\right]$ from $\partial\hat{\Omega}_{\delta}\setminus\left[r,\ell\right]$
(see Figure \ref{fig:fk-obs}).
\begin{defn}
\label{def:fk-obs-def}We define the FK-Ising observable $g_{\delta}^{\mathrm{FK}}\left(\Omega_{\delta},r,\ell,\cdot\right)$
on $\mathcal{E}_{\Omega_{\delta}^{m}}\setminus\partial_{0}\mathcal{E}_{\left[r,\ell\right]_{\delta}^{m}}$
by 
\[
g_{\delta}^{\mathrm{FK}}\left(\Omega_{\delta},r,\ell,e\right):=\frac{e^{\frac{\pi i}{4}}}{\sqrt{2}}\cdot\mathbb{E}_{\Omega_{\delta}}^{\left[r,\ell\right]_{\mathrm{w}}}\left[\mathbf{1}_{e\in\gamma}e^{-\frac{i}{2}\mathbf{W}\left(\lambda_{\delta}:r_{m}\leadsto e\right)}\right],
\]
where $\gamma$ is the FK interface linking $r_{m}$ to $\ell_{m}$,
rounded as in Figure \ref{fig:fk-obs}, and $\mathbf{W}\left(\lambda_{\delta}:r_{m}\leadsto e\right)$
is the winding  (i.e. the total turning) of the interface $\lambda_{\delta}$
(running backwards) from $r_{m}$ to the midpoint of $e$ (hence we
have $\mathbf{W}\left(\lambda_{\delta}:r_{m}\leadsto e\right)\in\left\{ \frac{\pi}{4}+k\frac{\pi}{2}:k\in\mathbb{Z}\right\} $).\end{defn}
\begin{rem}
The factor $\frac{e^{\pi i/4}}{\sqrt{2}}$ is introduced in order
to follow existing conventions. \end{rem}
\begin{defn}
\label{def:fk-obs-rephased-def}We define
\[
f_{\delta}^{\mathrm{FK}}\left(\Omega_{\delta},r,\ell,e\right):=\frac{1}{\sqrt{\nu_{\mathrm{in}}\left(r_{m}\right)}}g_{\delta}^{\mathrm{FK}}\left(\Omega_{\delta},r,\ell,e\right),
\]
where we take the following branch of the square root $\sqrt{e^{i\theta}}:=e^{\frac{i\theta}{2}}$
for $\theta\in(-\pi,\pi]$.\end{defn}
\begin{rem}
The branch choice of $\sqrt{\nu_{\mathrm{in}}\left(r_{m}\right)}$
is somewhat arbitrary and is made for definiteness.

\begin{figure}
\includegraphics[width=10cm]{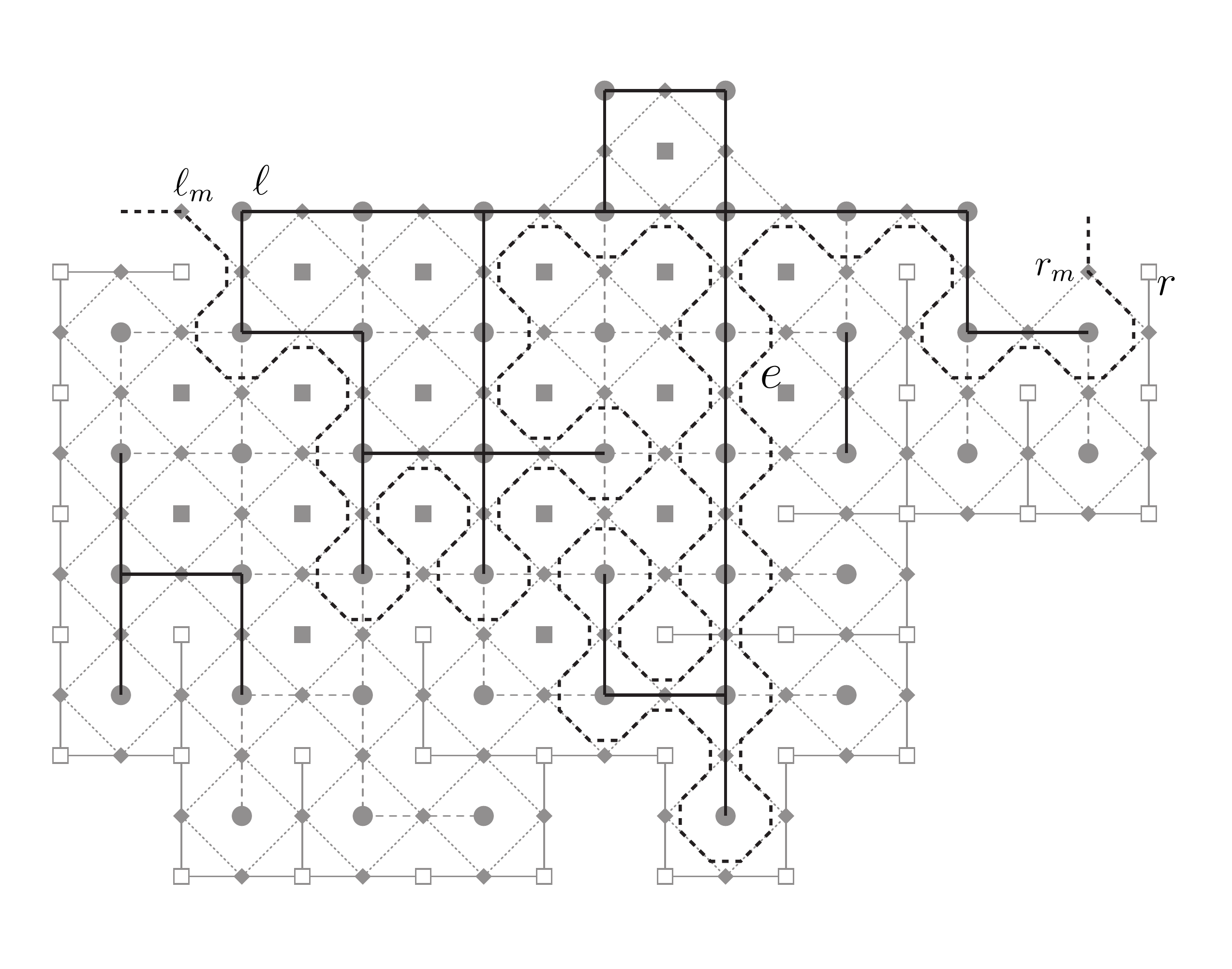}

\caption{\label{fig:fk-obs}Setup for the FK observable. On the figure, we
have $\mathbf{W}\left(\gamma:r_{m}\leadsto e\right)=-\frac{\pi}{4}$.}

\end{figure}
\end{rem}
\begin{lem}
\label{lem:bulk-argument-fk-obs-edges}We have
\[
f_{\delta}^{\mathrm{FK}}\left(\Omega_{\delta},r,\ell,e\right)\in\ell\left(e\right)\,\,\,\,\forall e\in\mathcal{E}_{\Omega_{\delta}^{m}}\setminus\partial_{0}\mathcal{E}_{\left[r,\ell\right]_{\delta}^{m}}.
\]
\end{lem}
\begin{proof}
This follows from topological considerations (see \cite[Lemma 4.1]{smirnov-ii}).\end{proof}
\begin{lem}
\label{lem:fk-relation}Let $e_{\mathrm{NE}},e_{\mathrm{NW}},e_{\mathrm{SW}},e_{\mathrm{SE}}\in\mathcal{E}_{\Omega_{\delta}^{m}}$
be the four medial edges incident to a medial vertex $v\in\mathcal{V}_{\Omega_{\delta}^{m}}\setminus\partial_{0}\mathcal{V}_{\Omega_{\delta}^{m}}$
as in Figure \ref{fig:fk-s-hol-rel}. Then we have 
\begin{eqnarray*}
f_{\delta}^{\mathrm{FK}}\left(\Omega_{\delta},r,\ell,e_{\mathrm{NE}}\right)+f_{\delta}^{\mathrm{FK}}\left(\Omega_{\delta},r,\ell,e_{\mathrm{SW}}\right) & = & f_{\delta}^{\mathrm{FK}}\left(\Omega_{\delta},r,\ell,e_{\mathrm{SE}}\right)+f_{\delta}^{\mathrm{FK}}\left(\Omega_{\delta},r,\ell,e_{\mathrm{NW}}\right).
\end{eqnarray*}
\end{lem}
\begin{proof}
The proof of this is based on combinatorial considerations (see \cite[Equation 12]{smirnov-ii}).

\begin{figure}
\includegraphics[width=6cm]{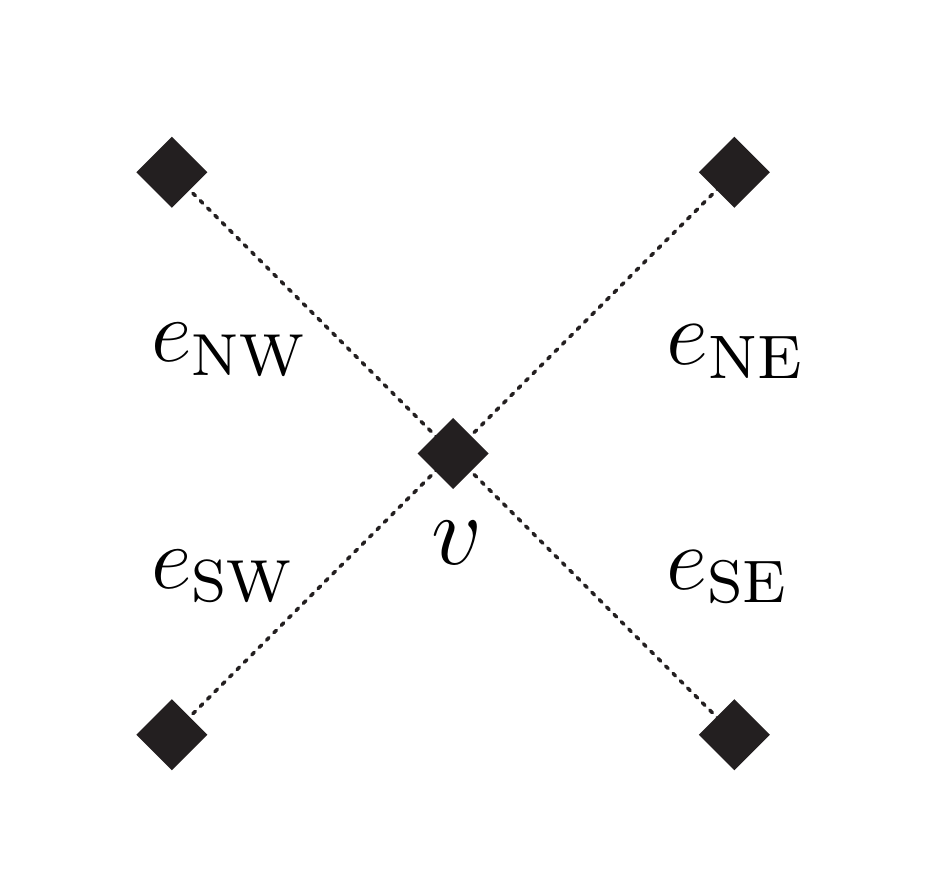}

\caption{\label{fig:fk-s-hol-rel}Medial vertex and edges in Lemma \ref{lem:fk-relation}.}

\end{figure}

\end{proof}
This lemma allows us to define $f_{\delta}^{\mathrm{FK}}$ on $\mathcal{V}_{\Omega_{\delta}^{m}}\setminus\partial_{0}\mathcal{V}_{\left[r,\ell\right]_{\delta}^{m}}$
in the following way:
\begin{defn}
\label{def:fk-obs-on-med-def}For each medial vertex $z\in\mathcal{V}_{\Omega_{\delta}^{m}}\setminus\partial_{0}\mathcal{V}_{\left[r,\ell\right]_{\delta}^{m}}$,
we define $f_{\delta}^{\mathrm{FK}}\left(\Omega_{\delta},r,\ell,z\right)$
as the unique complex number such that for each medial edge $e\in\mathcal{E}_{\Omega_{\delta}^{m}}\setminus\left[r,\ell\right]_{m}$
that is incident to $z$, we have 
\begin{equation}
f_{\delta}^{\mathrm{FK}}\left(\Omega_{\delta},r,\ell,e\right)=\mathsf{P}_{\ell\left(e\right)}\left[f_{\delta}^{\mathrm{FK}}\left(\Omega_{\delta},r,\ell,z\right)\right].\label{eq:s-hol-for-fk}
\end{equation}
\end{defn}
\begin{rem}
For any $z\in\mathcal{V}_{\Omega_{\delta}^{m}}$ having four neighbors
in $\mathcal{V}_{\Omega_{\delta}^{m}}\setminus\partial_{0}\mathcal{V}_{\left[r,\ell\right]_{\delta}^{m}}$,
if we denote by $e_{\mathrm{NE}},e_{\mathrm{NW}},e_{\mathrm{SW}},e_{\mathrm{SE}}\in\mathcal{E}_{\Omega_{\delta}^{m}}\setminus\partial_{0}\mathcal{E}_{\left[r,\ell\right]_{\delta}^{m}}$
the medial edges incident to $z$, we have the following orthogonal
decompositions:
\begin{eqnarray*}
f_{\delta}^{\mathrm{FK}}\left(\Omega_{\delta},r,\ell,z\right) & = & f_{\delta}^{\mathrm{FK}}\left(\Omega_{\delta},r,\ell,e_{\mathrm{NE}}\right)+f_{\delta}^{\mathrm{FK}}\left(\Omega_{\delta},r,\ell,e_{\mathrm{SW}}\right)\\
 & = & f_{\delta}^{\mathrm{FK}}\left(\Omega_{\delta},r,\ell,e_{\mathrm{SE}}\right)+f_{\delta}^{\mathrm{FK}}\left(\Omega_{\delta},r,\ell,e_{\mathrm{NW}}\right).
\end{eqnarray*}

\end{rem}
We will also use a rephased version of the FK observable:
\begin{defn}
We define $g_{\delta}^{\mathrm{FK}}\left(\Omega_{\delta},r,\ell,\cdot\right)$
on the medial vertices by
\begin{eqnarray*}
g_{\delta}^{\mathrm{FK}}\left(\Omega_{\delta},r,\ell,\cdot\right) & := & \sqrt{\nu_{\mathrm{in}}\left(r_{m}\right)}f_{\delta}^{\mathrm{FK}}\left(\Omega_{\delta},r,\ell,\cdot\right),
\end{eqnarray*}
where $\sqrt{\nu_{\mathrm{in}}\left(r_{m}\right)}$ is as in Definition
\ref{def:fk-obs-rephased-def}. 
\end{defn}
The most fundamental analytical property of the FK observable is the
following:
\begin{lem}
\label{lem:s-hol-fk-obs}The function $f_{\delta}^{\mathrm{FK}}\left(\Omega_{\delta},r,\ell,\cdot\right):\mathcal{V}_{\Omega_{\delta}^{m}}\setminus\partial_{0}\mathcal{V}_{\left[r,\ell\right]_{\delta}^{m}}\to\mathbb{C}$
is s-holomorphic.\end{lem}
\begin{proof}
This follows directly from the construction of $f_{\delta}^{\mathrm{FK}}$
(Equation \ref{eq:s-hol-for-fk}).\end{proof}
\begin{lem}
\label{lem:fk-obs-bdry-condition}We have
\[
f_{\delta}^{\mathrm{FK}}\left(\Omega_{\delta},r,\ell,z\right)\in\nu_{\mathrm{out}}^{-\frac{1}{2}}\left(z\right)\mathbb{R}\,\,\,\,\forall z\in\partial_{0}\mathcal{V}_{\left[\ell,r\right]_{\delta}^{m}}
\]
\end{lem}
\begin{rem}
Near $\partial_{0}\mathcal{V}_{\left[r,\ell\right]_{\delta}^{m}}$,
a boundary condition analogous to the one of Lemma \ref{lem:fk-obs-bdry-condition}
holds (see \cite[Lemma 4.12]{smirnov-ii} or \cite[Remark 2.3]{chelkak-smirnov-ii}),
but we will not need to study it for our purposes.
\end{rem}

\subsubsection{Spin observable\label{sub:spin-observable-def}}

We now define the spin observable, first introduced in \cite{smirnov-i}
and studied in \cite{chelkak-smirnov-ii}, which is instrumental in
the original proof of Chelkak and Smirnov to obtain the convergence
of the spin interfaces of the Ising model (with $+$ and $-$ boundary
conditions) to chordal SLE$\left(3\right)$. A variant of this observable
can be used to derive the correlation functions of the energy field
of the Ising model \cite{hongler-smirnov-ii,hongler-i}. Like the
FK observable, its boundary values are particularly interesting as
they give boundary spin-spin correlations with free boundary conditions
(see Section \ref{sub:bv-as-corr-func} or \cite{hongler-i}).

Let $\Omega_{\delta}$ be a discrete domain. We denote by $\mathcal{Z}\left(\Omega_{\delta}\right)$
the low-temperature expansion of the partition function of the critical
Ising model on the faces of $\Omega_{\delta}$, defined by 
\[
\mathcal{Z}\left(\Omega_{\delta}\right):=\sum_{\omega\in\mathcal{C}\left(\Omega_{\delta}\right)}\alpha^{\left|\omega\right|},
\]
where $\mathcal{C}\left(\Omega_{\delta}\right)$ is the set of contours
$\omega\subset\mathcal{E}_{\Omega_{\delta}}$ such that every vertex
of $\mathcal{V}_{\Omega_{\delta}}$ is incident to an even number
of edges of $\omega$ and where $\alpha:=\sqrt{2}-1$ and $\left|\omega\right|$
is the total number of edges of $\omega$.

Let $x\in\partial_{0}\mathcal{V}_{\Omega_{\delta}^{m}}$ be a boundary
medial vertex and let $z\in\mathcal{V}_{\Omega_{\delta}^{m}}$ be
a medial vertex. We define the collection $\mathcal{C}\left(\Omega_{\delta},x,z\right)$
as the set of $\gamma$'s consisting of edges of $\mathcal{E}_{\Omega_{\delta}}\setminus\left\{ z\right\} $
and of two \emph{half-edges} (half of an edge, between its midpoint
and one of its ends) such that
\begin{itemize}
\item one of the half-edges is the unique half-edge incident to $x$.
\item the other half-edge is incident to $z$;
\item every vertex $v\in\mathcal{V}_{\Omega_{\delta}}$ belongs to an even
number of edges or half-edges of $\gamma$.
\end{itemize}
For a contour $\gamma\in\mathcal{C}\left(\Omega_{\delta},x,z\right)$,
we define its winding $\mathbf{W}\left(\gamma\right)$ as the total
rotation (the cumulative angle of turn) of the walk on the edges and
half-edges of $\gamma$ from $x$ to $z$, which turns left whenever
there is an ambiguity (i.e. we arrive at a vertex such that it belongs
to four edges or half-edges of $\gamma$). See Figure \ref{fig:spin-obs}. 
\begin{rem}
As shown in \cite[Lemma 4]{hongler-smirnov-ii}, the complex number
$e^{-\frac{i}{2}\mathbf{W}\left(\gamma\right)}$ is essentially independent
of the choice of the walk on $\gamma$.\end{rem}
\begin{defn}
\label{def:spin-obs-def}We define the \emph{spin observable} $g_{\delta}^{\mathrm{SPIN}}$
by 
\[
g_{\delta}^{\mathrm{SPIN}}\left(\Omega_{\delta},x,z\right):=\frac{1}{\mathcal{Z}\left(\Omega_{\delta}\right)}\sum_{\gamma\in\mathcal{C}\left(\Omega_{\delta},x,z\right)}\alpha^{\left|\gamma\right|}e^{-\frac{i}{2}\mathbf{W}\left(\gamma\right)},
\]
for any $z\in\mathcal{V}_{\Omega_{\delta}^{m}}\setminus\left\{ x\right\} $,
where $\left|\gamma\right|$ is the number of edges in $\gamma$,
with the two half-edges of $\gamma$ contributing $\frac{1}{2}$ each.
We set 
\[
g_{\delta}^{\mathrm{SPIN}}\left(\Omega_{\delta},x,x\right):=1.
\]

We define $f_{\delta}^{\mathrm{SPIN}}$ by
\[
f_{\delta}^{\mathrm{SPIN}}\left(\Omega_{\delta},x,z\right):=\frac{1}{\sqrt{\nu_{\mathrm{in}}\left(x\right)}}g_{\delta}^{\mathrm{SPIN}}\left(\Omega_{\delta},x,z\right),
\]
where $\sqrt{\nu_{\mathrm{in}}\left(x\right)}$ is the principal determination
of the square root as in Definition \ref{def:fk-obs-def}.

\begin{figure}

\includegraphics[width=9cm]{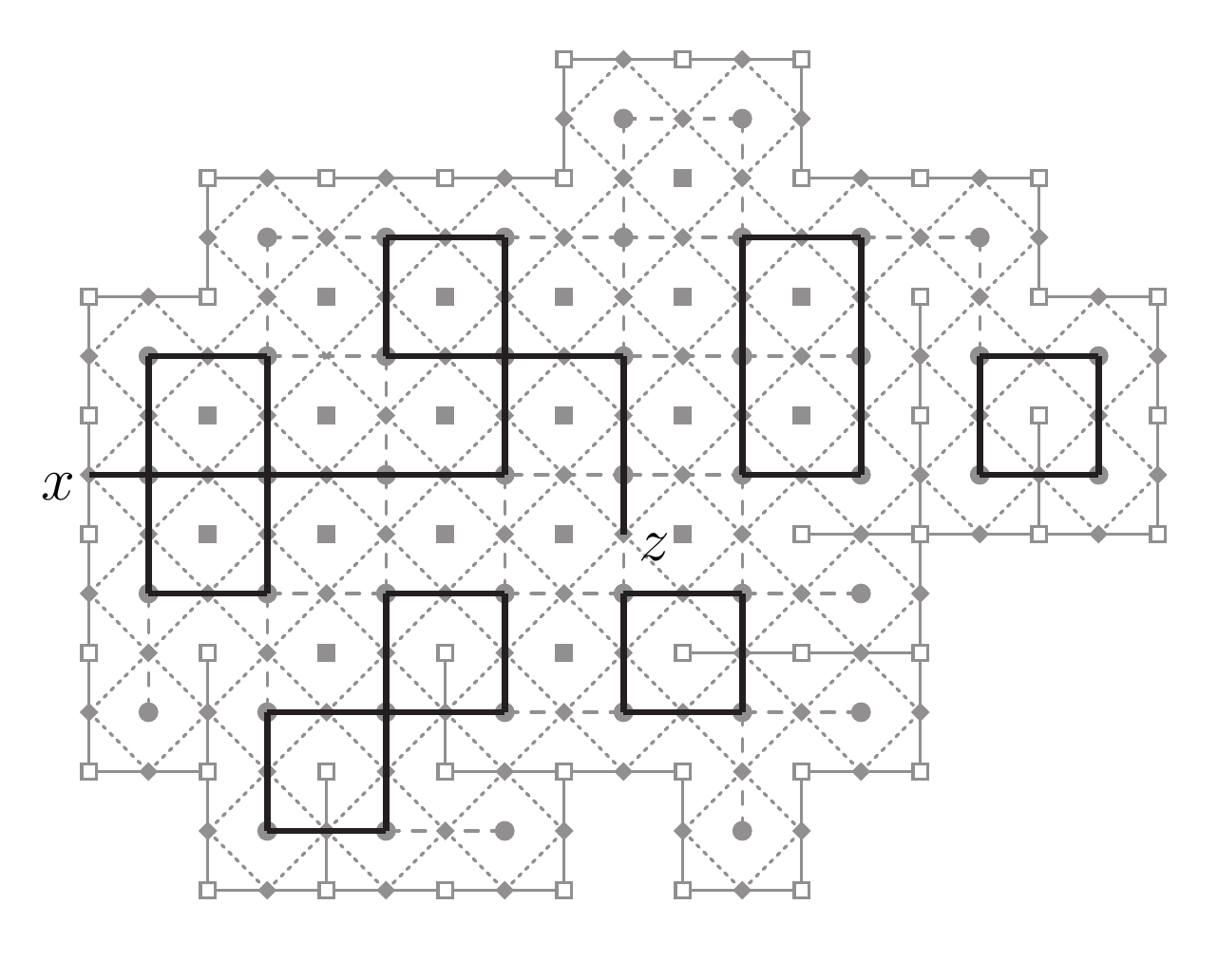}

\caption{\label{fig:spin-obs}A configuration $\gamma\in\mathcal{C}_{\Omega_{\delta}}\left(x,z\right)$
with three ambiguities and with $\mathbf{W}\left(\gamma\right)=-\frac{\pi}{2}$.}

\end{figure}
\end{defn}
\begin{rem}
If $z\in\partial_{0}\mathcal{V}_{\Omega_{\delta}^{m}}$, by symmetry,
we get
\[
g_{\delta}^{\mathrm{SPIN}}\left(\Omega_{\delta},x,z\right)=\overline{g_{\delta}^{\mathrm{SPIN}}\left(\Omega_{\delta},z,x\right)}.
\]
\end{rem}
\begin{lem}
\label{lem:spin-obs-s-hol}For any $x\in\partial_{0}\mathcal{V}_{\Omega_{\delta}^{m}}$,
the function $f_{\delta}^{\mathrm{SPIN}}\left(\Omega_{\delta},x,\cdot\right):\mathcal{V}_{\Omega_{\delta}^{m}}\to\mathbb{C}$
is s-holomorphic. \end{lem}
\begin{proof}
This follows from combinatorial considerations. See \cite[Proposition 2.5]{chelkak-smirnov-ii}
or \cite[Proposition 74]{hongler-i}.\end{proof}
\begin{lem}
\label{lem:spin-obs-boundary-values}For any $x\in\partial_{0}\mathcal{V}_{\Omega_{\delta}^{m}}$
and all $z\in\partial_{0}\mathcal{V}_{\Omega_{\delta}^{m}}\setminus\left\{ x\right\} $,
we have
\[
f_{\delta}^{\mathrm{SPIN}}\left(\Omega_{\delta},x,z\right)\in\nu_{\mathrm{out}}^{-\frac{1}{2}}\left(z\right)\mathbb{R}.
\]
\end{lem}
\begin{proof}
This follows from topological considerations. See \cite[Equation 2.11]{chelkak-smirnov-ii}
or \cite[Proposition 78]{hongler-i}.
\end{proof}

\subsection{Discrete correlation functions\label{sub:bv-as-corr-func}}

What makes the FK and spin observables particularly relevant to our
analysis is that their boundary values give the correlation functions
appearing in Theorem \ref{thm:corr-ratios-theorem}.
\begin{lem}
\label{lem:fk-obs-bdry-vals-corr-fct}Let $\left(\Omega_{\delta},r,\ell\right)$
be a discrete domain. For $z\in\partial_{0}\mathcal{V}_{\Omega_{\delta}^{m}}$
and any boundary medial edge $e\in\partial_{0}\mathcal{E}_{\Omega_{\delta}^{m}}$
incident to $z$, we have
\begin{eqnarray*}
\cos\left(\frac{\pi}{8}\right)\cdot\left|f_{\delta}^{\mathrm{FK}}\left(\Omega_{\delta},r,\ell,z\right)\right| & = & \left|f_{\delta}^{\mathrm{FK}}\left(\Omega_{\delta},r,\ell,e\right)\right|\\
 & = & \mathbb{P}_{\Omega_{\delta}}^{\left[r,\ell\right]_{\mathrm{w}}}\left\{ \lambda_{\delta}\mbox{ separates }z\mbox{ from }\left[\ell,r\right]\right\} \\
 & = & \mathbb{P}_{\Omega_{\delta}}^{\left[r,\ell\right]_{\mathrm{w}}}\left\{ z\leftrightsquigarrow\left[\ell.r\right]\right\} \\
 & = & \mathbb{E}_{\Omega_{\delta}}^{\left[r,\ell\right]_{+}}\left[\sigma_{z}\right]
\end{eqnarray*}
\end{lem}
\begin{proof}
The first identity follows from the definition of $f_{\delta}^{\mathrm{FK}}$
(Equation \ref{eq:s-hol-for-fk}) and the boundary condition (see
Lemma \ref{lem:fk-obs-bdry-condition}). For the second identity,
see \cite[Remark 2.4]{smirnov-i} or \cite[proof of Proposition 5.6]{duminil-copin-hongler-nolin}.
The remaining identies are derived in Lemma \ref{lem:interface-passage-probabilities}
in Section \ref{sub:fk-model}.\end{proof}
\begin{lem}
\label{lem:bdry-spin-obs-spin-corr}Let $\left(\Omega_{\delta},a,z\right)$
be a discrete domain. Let $a_{m}\in\partial_{0}\mathcal{V}_{\Omega_{\delta}^{m}}$
and $z_{m}\in\partial_{0}\mathcal{V}_{\Omega_{\delta}^{m}}$ be boundary
medial vertices that are adjacent to $a$ and $z$. Then we have
\[
\frac{1}{\alpha}\left|f_{\delta}^{\mathrm{SPIN}}\left(\Omega_{\delta},a_{m},z_{m}\right)\right|=\mathbb{E}_{\Omega_{\delta}}^{\mathrm{free}}\left[\sigma_{a}\sigma_{z}\right].
\]
 where $\alpha=\sqrt{2}-1$.\end{lem}
\begin{proof}
This, lemma, which can be found in \cite[Proposition 71]{hongler-i},
follows from the fact that the winding $\mathbf{W}\left(\gamma\right)$
is the same for all $\gamma\in\mathcal{C}_{\Omega}\left(a_{m},z_{m}\right)$,
and from the high-temperature expansion of the spin correlations (the
techniques of Proposition \ref{pro:representation-obs-as-correlation}
can be adapted to get the result). The $\alpha$ denominator comes
from the fact that we have to remove the two half-edges incident to
$a_{m}$ and $z_{m}$ to get the same contours as in the high-temperature
expansions. 
\end{proof}

\subsection{Convolution representation of discrete Riemann BVP\label{sub:discrete-riemann-bvp}}

In this subsection we discuss the Riemann-type boundary values taken
by the observables. These boundary values, together with the s-holomorphicity,
allow for a local representation of the observables in terms of a
convolution kernel, which happens to be the spin observable.

\subsubsection{Convolution kernel\label{sub:disc-conv-kernel}}
\begin{defn}
\label{def:discrete-conv-kernel-def}Let $\Omega_{\delta}$ be a discrete
domain and $u_{\delta}:\partial_{0}\mathcal{V}_{\Omega_{\delta}^{m}}\to\mathbb{C}$.
We denote by $\mathbb{K}_{\delta}\left[\Omega_{\delta},u_{\delta}\right]:\mathcal{V}_{\Omega_{\delta}^{m}}\to\mathbb{C}$
the function defined by
\[
\mathbb{K}_{\delta}\left[\Omega_{\delta},u_{\delta}\right]\left(z\right)\,:=\sum_{x\in\partial_{0}\mathcal{V}_{\Omega_{\delta}^{m}}}\mathsf{P}_{\nu_{\mathrm{in}}^{-\frac{1}{2}}\left(x\right)}\left[u_{\delta}\left(x\right)\right]\cdot g_{\delta}^{\mathrm{SPIN}}\left(\Omega_{\delta},x,z\right).
\]

\end{defn}

\subsubsection{Uniqueness\label{sub:uniqueness-sol-bvp}}

The kernel $\mathbb{K}_{\delta}$ provides us with a local representation
of s-holomorphic functions.
\begin{lem}
\label{lem:discrete-conv-rep}Let $\Omega_{\delta}$ be a discrete
domain and let $u_{\delta}:\partial_{0}\mathcal{V}_{\Omega_{\delta}^{m}}\to\mathbb{C}$
be any given function. Then $\mathbb{K}_{\delta}\left[\Omega_{\delta},u_{\delta}\right]:\mathcal{V}_{\Omega_{\delta}^{m}}\to\mathbb{C}$
is the unique s-holomorphic function such that 
\begin{equation}
\left(\mathbb{K}_{\delta}\left[\Omega_{\delta},u_{\delta}\right]-u_{\delta}\right)\left(z\right)\in\nu_{\mathrm{out}}^{-\frac{1}{2}}\left(z\right)\mathbb{R}\,\,\,\,\forall z\in\partial_{0}\mathcal{V}_{\Omega_{\delta}^{m}}\label{eq:bdry-cond-convol}
\end{equation}
\end{lem}
\begin{proof}
By Lemma \ref{lem:spin-obs-s-hol}, we easily obtain that $\mathbb{K}_{\delta}\left[\Omega_{\delta},u_{\delta}\right]$
is s-holomorphic, as $\mathsf{P}_{\nu_{\mathrm{in}}^{-\frac{1}{2}}\left(x\right)}\left[u_{\delta}\left(x\right)\right]\cdot g_{\delta}^{\mathrm{SPIN}}\left(\Omega_{\delta},x,\cdot\right)$
is a real multiple of $f_{\delta}^{\mathrm{SPIN}}\left(\Omega_{\delta},x,\cdot\right)$
for any $x\in\partial_{0}\mathcal{V}_{\Omega_{\delta}^{m}}$. From
Lemma \ref{lem:spin-obs-boundary-values}, we have that $\mathbb{K}_{\delta}\left[\Omega_{\delta},u_{\delta}\right]$
satisfies the boundary condition \ref{eq:bdry-cond-convol}: for $z\in\partial_{0}\mathcal{V}_{\Omega_{\delta}^{m}}$,
notice that we have 
\[
\mathsf{P}_{\nu_{\mathrm{in}}^{-\frac{1}{2}}\left(z\right)}\left[\mathbb{K}_{\delta}\left[\Omega_{\delta},u_{\delta}\right]\left(z\right)\right]=\mathsf{P}_{\nu_{\mathrm{in}}^{-\frac{1}{2}}\left(z\right)}\left[u_{\delta}\left(z\right)\right]
\]
since $\nu_{\mathrm{in}}^{-\frac{1}{2}}$ and $\nu_{\mathrm{out}}^{-\frac{1}{2}}$
are orthogonal. For the uniqueness, see \cite[Proposition 48]{hongler-i}.
\end{proof}
Let us reformulate the above lemma in a form that will be directly
useful later.
\begin{lem}
\label{lem:useful-form-convolution-lemma}Let $\Omega_{\delta}$ be
a discrete domain, suppose that $\partial_{0}\mathcal{V}_{\Omega_{\delta}^{m}}=\partial_{0}^{\mathrm{s}}\mathcal{V}_{\Omega_{\delta}^{m}}\cup\partial_{0}^{\mathrm{r}}\mathcal{V}_{\Omega_{\delta}^{m}}$
and let $v_{\delta}:\mathcal{V}_{\Omega_{\delta}^{m}}\to\mathbb{C}$
be an s-holomorphic function such that 
\[
v_{\delta}\left(z\right)\in\nu_{\mathrm{out}}^{-\frac{1}{2}}\left(z\right)\mathbb{R}\,\,\,\,\forall z\in\partial_{0}^{\mathrm{r}}\mathcal{V}_{\Omega_{\delta}^{m}}.
\]
Let $u_{\delta}\left(z\right):\partial_{0}\mathcal{V}_{\Omega_{\delta}^{m}}\to\mathbb{C}$
be defined by
\[
u_{\delta}\left(z\right)=\begin{cases}
\mathsf{P}_{\nu_{\mathrm{in}}^{-\frac{1}{2}}\left(z\right)}\left[v_{\delta}\left(z\right)\right] & \forall z\in\partial_{0}^{\mathrm{s}}\mathcal{V}_{\Omega_{\delta}^{m}},\\
0 & \forall z\in\partial_{0}^{\mathrm{r}}\mathcal{V}_{\Omega_{\delta}^{m}}.
\end{cases}
\]
Then we have 
\[
v_{\delta}\left(z\right)=\mathbb{K}_{\delta}\left[\Omega_{\delta},u_{\delta}\right]\left(z\right)\,\,\,\,\forall z\in\mathcal{V}_{\Omega_{\delta}^{m}}.
\]

\end{lem}

\subsection{\label{sub:continous-holomorphic-obs}Continuous holomorphic observables}

\subsubsection{\label{sub:cts-fk-observable-def}Continuous FK observable}

Let us now define the continuous FK observable, following \cite{smirnov-ii}.
By Theorem \ref{thm:fk-obs-cv} in Section \ref{sub:scaling-limit-fk-obs},
it is the scaling limit of the discrete FK observable.
\begin{defn}
\label{def:cts-fk-obs-def}Let $\left(\Omega,r,\ell\right)$ be a
simply connected domain and let $\varphi_{\Omega}$ be a conformal
map $\left(\Omega,r,\ell\right)\to\left(\mathbb{S},-\infty,\infty\right)$.
We define thesquare of the continuous FK observable $f^{\mathrm{FK}}$
by
\[
\left(f^{\mathrm{FK}}\right)^{2}\left(\Omega,r,\ell,z\right):=\frac{1}{i}\cdot\varphi_{\Omega}'\left(z\right).
\]
\end{defn}
\begin{rem}
As the conformal map $\varphi_{\Omega}$ is unique up to an additive
constant, $\varphi_{\Omega}'$ is independent of the choice of $\varphi_{\Omega}$.
\end{rem}

\begin{rem}
We prefer to define $\left(f^{\mathrm{FK}}\right)^{2}$ rather than
$f^{\mathrm{FK}}$ in order to avoid to choose a square root branch.
\end{rem}

\subsubsection{\label{sub:continuous-spin-observable}Continuous Spin observable}

To properly define the continuous observable $g^{\mathrm{SPIN}}\left(\Omega_{\delta},x,z\right)$
as the continuous counterpart of the discrete observable introduced
in Section \ref{sub:spin-observable-def}, we need to make an assumption
on the regularity of the boundary of $\Omega$ near the point $x$.
Notice that the normalization in \cite{chelkak-smirnov-ii} is different
(see Remark \ref{rem:chsm-spin-obs-normalization} below).
\begin{defn}
\label{def:cts-spin-obs-def}Let $\left(\Omega,x\right)$ be a simply
connected domain, with $x$ being on a smooth part of $\partial\Omega$.
Let $\eta_{\Omega}$ be a conformal mapping from $\left(\Omega,x\right)$
to $\left(\mathbb{H},0\right)$. We define the continuous spin observable
$g^{\mathrm{SPIN}}$ by
\[
g^{\mathrm{SPIN}}\left(\Omega,x,z\right):=\frac{i}{\pi}\overline{\sqrt{\eta_{\Omega}'\left(x\right)}}\sqrt{\eta_{\Omega}'\left(z\right)}\frac{1}{\eta_{\Omega}\left(z\right)},
\]
this definition being independent from $\eta_{\Omega}$ and from the
branch choice of $\sqrt{\eta_{\Omega}'}$.

We define $f^{\mathrm{SPIN}}$ by
\[
f^{\mathrm{SPIN}}\left(\Omega,x,z\right):=\frac{1}{\sqrt{\nu_{\mathrm{in}}\left(x\right)}}g^{\mathrm{SPIN}}\left(\Omega,x,z\right),
\]
where we take the principal branch of the square root (i.e. setting
$\sqrt{re^{i\theta}}:=\sqrt{r}e^{\frac{i\theta}{2}}$, where $\theta$
is chosen in $(-\pi,\pi]$).\end{defn}
\begin{rem}
\label{rem:chsm-spin-obs-normalization}In \cite{chelkak-smirnov-ii},
a normalization based on the following observation is used: if $\left(\Omega,y\right)$
is a simply connected domain, with $y$ being on a smooth part of
$\partial\Omega$ and $x\in\partial\Omega$, we have that
\[
z\mapsto\frac{g^{\mathrm{SPIN}}\left(\Omega,x,z\right)}{g^{\mathrm{SPIN}}\left(\Omega,x,y\right)}
\]
is well-defined on $\Omega$, even if $g^{\mathrm{SPIN}}\left(x,\cdot\right)$
might not be well-defined (the possibly ill-defined derivative $\eta_{\Omega}'\left(x\right)$
in Definition \ref{def:cts-fk-obs-def} appears in both the numerator
and the numerator and hence cancels).
\end{rem}

\subsection{CFT correlation functions\label{cft-corr-functions}}

Let us now give a representation of the CFT correlation functions
in terms of the continuous observables:
\begin{lem}
\label{lem:bdry-val-cts-obs-as-cft-corr-fct}Let $\left(\Theta,y,t\right)$,
$\left(\tilde{\Theta},\tilde{y},\tilde{t}\right)$ and $\left(\Xi,x,s\right)$
be domains as in Theorem \ref{thm:corr-ratios-theorem}. Then we have
\begin{eqnarray*}
\frac{\left\langle \sigma_{x}\sigma_{s}\right\rangle _{\Xi}^{\mathrm{free}}}{\left\langle \sigma_{x}\right\rangle _{\Theta}^{\left[y,t\right]_{+}}} & = & \frac{\left(\sqrt{2}+1\right)}{\cos\left(\frac{\pi}{8}\right)}\left|\frac{f^{\mathrm{SPIN}}\left(\Xi,s,x\right)}{f^{\mathrm{FK}}\left(\Theta,y,t,x\right)}\right|,\\
\frac{\left\langle \sigma_{x}\right\rangle _{\tilde{\Theta}}^{\left[\tilde{y}_{\delta},\tilde{t}_{\delta}\right]_{+}}}{\left\langle \sigma_{x}\right\rangle _{\Theta}^{\left[y_{\delta},t_{\delta}\right]_{+}}} & = & \left|\frac{f^{\mathrm{FK}}\left(\tilde{\Theta},\tilde{y},\tilde{t},x\right)}{f^{\mathrm{FK}}\left(\Theta,y,t,x\right)}\right|.
\end{eqnarray*}
\end{lem}
\begin{proof}
This follows from the definitions of the continuous observables $f^{\mathrm{FK}}$
and $f^{\mathrm{SPIN}}$ (Definitions \ref{def:cts-fk-obs-def} and
\ref{def:cts-spin-obs-def}) and of the continuous continuous correlation
functions (Definition \ref{def:corr-func}). Notice that all these
ratios are well-defined, even when $x$ is on a rough part of the
boundary.
\end{proof}

\subsection{Convolution representation of continuous Riemann BVP\label{sub:cts-riemann-bvp}}

In this subsection, we discuss the continuous version of the Riemann
boundary value problems introduced in Section \ref{sub:discrete-riemann-bvp}.
For those, we will only introduce a restricted framework, which is
enough for our purposes.

\subsubsection{Continuous convolution kernel\label{sub:cts-conv-kern}}

Let us first introduce the continuous version of the operator $\mathbb{K}_{\delta}$
introduced in Section \ref{sub:disc-conv-kernel}.
\begin{defn}
\label{def:cts-conv-kern-def}Let $\Omega$ be a simply connected
domain, with $\partial\Omega=\partial^{\mathrm{s}}\Omega\cup\partial^{\mathrm{r}}\Omega$,
$\partial^{\mathrm{s}}\Omega$ being compact and piecewise smooth.
Let $u:\partial^{\mathrm{s}}\Omega\to\mathbb{C}$ be an arbitrary
continuous function. We denote by $\mathbb{K}\left[\Omega,u\right]:\Omega\to\mathbb{C}$
the function defined by
\[
\mathbb{K}\left[\Omega,u\right]\left(z\right):=\int_{\partial^{\mathrm{s}}\Omega}\mathsf{P}_{\nu_{\mathrm{in}}^{-\frac{1}{2}}\left(x\right)}\left[u\left(x\right)\right]\cdot g^{\mathrm{SPIN}}\left(\Omega,x,z\right)\mathrm{d}\left|x\right|.
\]

\end{defn}

\subsubsection{\label{sub:local-conformal-charts}Local conformal charts}

Let us first define the continuous version of the Riemann-type boundary
condition $f\left(z\right)\in\nu_{\mathrm{out}}^{-\frac{1}{2}}\left(z\right)\mathbb{R}$
discussed in Section \ref{sub:discrete-riemann-bvp}. As the normal
vector is not necessarily well-defined anymore, we use local conformal
charts.
\begin{defn}
\label{def:local-chart-continuous-rbvp}Let $\Omega$ be a domain
and let $f:\Omega\to\mathbb{C}$ be a holomorphic function. Let $\mathfrak{b}\subset\partial\Omega$
be an arc of the boundary. We say that
\[
f\left(z\right)\in\nu_{\mathrm{out}}^{-\frac{1}{2}}\left(z\right)\mathbb{R}\,\,\,\,\mbox{on }\mathfrak{b}
\]
 if there is a simply connected domain $\Upsilon$ coinciding with
$\Omega$ in a neighborhood of $\mathfrak{b}$ and a point $a$ on
a smooth part of $\partial\Upsilon\setminus\partial\Omega$ such that
for any $p\in\mathfrak{b}$, we have that 
\begin{equation}
\lim_{z\to p}\frac{f\left(z\right)}{f^{\mathrm{SPIN}}\left(\Upsilon,a,z\right)}\in\mathbb{R}.\label{eq:imaginary-part-ratio-goto-zero}
\end{equation}
\end{defn}
\begin{rem}
\label{rem:smooth-rh-bdry-cond}When $\partial\Omega$ is smooth,
this boundary condition equivalent to having $f$ extending continuously
to $\mathfrak{b}$ and satisfying $f\left(z\right)\in\nu_{\mathrm{out}}^{-\frac{1}{2}}\left(z\right)\mathbb{R}$
for each $z\in\mathfrak{b}$.
\end{rem}
Thanks to the following lemma, the above definition does not depend
on the choice of $\Upsilon$ or of $a$.
\begin{lem}
\label{lem:equiv-local-charts}If there exists such a domain $\Upsilon$
and a point $a\in\partial\Upsilon$ on a smooth part of $\partial\Upsilon\setminus\partial\Omega$,
then the for all $\tilde{\Upsilon}$ coinciding with $\Omega$ in
neighborhood of $\mathfrak{b}$ and any $\tilde{a}\in\partial\tilde{\Upsilon}$
on a smooth part of $\partial\tilde{\Upsilon}\setminus\partial\Omega$,
the condition of Equation \ref{eq:imaginary-part-ratio-goto-zero}
is satisfied.\end{lem}
\begin{proof}
It is enough to prove that for any $p\in\mathfrak{b}$, we have 
\begin{equation}
\lim_{z\to p}\frac{f^{\mathrm{SPIN}}\left(\Upsilon,a,z\right)}{f^{\mathrm{SPIN}}\left(\tilde{\Upsilon},\tilde{a},z\right)}\in\mathbb{R}\label{eq:im-part-ratio-fs-obs}
\end{equation}
It is enough to prove \ref{eq:im-part-ratio-fs-obs} in the following
two cases:
\begin{itemize}
\item When $a=\tilde{a}$ and $\Upsilon,\tilde{\Upsilon}$ moreover coincide
in a neighborhood of $a$, to prove \ref{eq:im-part-ratio-fs-obs},
we can assume that $\tilde{\Upsilon}\subset\Upsilon$ (otherwise replace
$\tilde{\Upsilon}$ by $\Upsilon\cap\tilde{\Upsilon}$). Let $\psi_{\Upsilon}:\Upsilon\to\mathbf{D}\left(0,1\right)$
be a conformal map. Setting $\tilde{\mathbf{D}}:=\psi_{\Upsilon}\left(\tilde{\Upsilon}\right)\subset\mathbf{D}\left(0,1\right)$,
by conformal covariance of $g^{\mathrm{SPIN}}$ (which follows from
its definition), we have (noticing that the derivative terms cancel):
\[
\frac{f^{\mathrm{SPIN}}\left(\Upsilon,a,z\right)}{f^{\mathrm{SPIN}}\left(\tilde{\Upsilon},\tilde{a},z\right)}=\frac{g^{\mathrm{SPIN}}\left(\Upsilon,a,z\right)}{g^{\mathrm{SPIN}}\left(\tilde{\Upsilon},\tilde{a},z\right)}=\frac{g^{\mathrm{SPIN}}\left(\mathbf{D}\left(0,1\right),\psi_{\Upsilon}\left(a\right),\psi_{\Upsilon}\left(z\right)\right)}{g^{\mathrm{SPIN}}\left(\tilde{\mathbf{D}},\psi_{\Upsilon}\left(a\right),\psi_{\Upsilon}\left(z\right)\right)}.
\]
As $z\to p$, $\psi_{\Upsilon}\left(z\right)\to\partial\mathbf{D}\left(0,1\right)\cap\partial\tilde{\mathbf{D}}$,
and as $\partial\mathbf{D}\left(0,1\right)$ is smooth, it is easy
to check that the right-hand side tends to a purely real number.
\item When $\Upsilon=\tilde{\Upsilon}$, taking a conformal map $\psi_{\Upsilon}:\Upsilon\to\mathbf{D}\left(0,1\right)$,
by conformal covariance of $g^{\mathrm{SPIN}}$, we get
\begin{eqnarray*}
\frac{f^{\mathrm{SPIN}}\left(\Upsilon,a,z\right)}{f^{\mathrm{SPIN}}\left(\Upsilon,\tilde{a},z\right)} & = & \frac{\sqrt{\nu_{\mathrm{in}}\left(\tilde{a}\right)}}{\sqrt{\nu_{\mathrm{in}}\left(a\right)}}\frac{g^{\mathrm{SPIN}}\left(\Upsilon,a,z\right)}{g^{\mathrm{SPIN}}\left(\Upsilon,\tilde{a},z\right)}\\
 & = & \frac{\sqrt{\nu_{\mathrm{in}}\left(\psi_{\Upsilon}\left(\tilde{a}\right)\right)}}{\sqrt{\nu_{\mathrm{in}}\left(\psi_{\Upsilon}\left(a\right)\right)}}\frac{\sqrt{\left|\psi_{\Upsilon}'\left(a\right)\right|}}{\sqrt{\left|\psi_{\Upsilon}'\left(\tilde{a}\right)\right|}}\frac{g^{\mathrm{SPIN}}\left(\mathbf{D}\left(0,1\right),\psi_{\Upsilon}\left(a\right),\psi_{\Upsilon}\left(z\right)\right)}{g^{\mathrm{SPIN}}\left(\mathbf{D}\left(0,1\right),\psi_{\Upsilon}\left(\tilde{a}\right),\psi_{\Upsilon}\left(z\right)\right)}\\
 & = & \pm\frac{\sqrt{\left|\psi_{\Upsilon}'\left(a\right)\right|}}{\sqrt{\left|\psi_{\Upsilon}'\left(\tilde{a}\right)\right|}}\frac{f^{\mathrm{SPIN}}\left(\mathbf{D}\left(0,1\right),\psi_{\Upsilon}\left(a\right),\psi_{\Upsilon}\left(z\right)\right)}{f^{\mathrm{SPIN}}\left(\mathbf{D}\left(0,1\right),\psi_{\Upsilon}\left(\tilde{a}\right),\psi_{\Upsilon}\left(z\right)\right)},
\end{eqnarray*}
where the $\pm$ sign comes the branch of the square root. Since $\partial\mathbf{D}\left(0,1\right)$
is smooth, it is easy to check that the the right-hand side tends
to a purely real number as $z\to p$.
\end{itemize}
\end{proof}

\begin{lem}
\label{lem:bdry-cond-cts-obs-local-charts}Let $\left(\Omega^{\left(1\right)},r,\ell\right)$
be a simply connected domain. The function $f^{\mathrm{FK}}\left(\Omega^{\left(1\right)},r,\ell,\cdot\right)$
satisfies the boundary condition
\[
f^{\mathrm{FK}}\left(\Omega^{\left(1\right)},r,\ell,z\right)\in\nu_{\mathrm{out}}^{-\frac{1}{2}}\left(z\right)\mathbb{R}\,\,\,\,\mbox{on the compact subsets of }\left[\ell,r\right]\setminus\left\{ \ell,r\right\} .
\]

Let $\left(\Omega^{\left(2\right)},x\right)$ be a simply connected
domain such that $x$ is on a smooth part of $\partial\Omega^{\left(2\right)}$.
The function $f^{\mathrm{SPIN}}\left(\Omega^{\left(2\right)},x,\cdot\right)$
satifies the boundary condition
\[
f^{\mathrm{SPIN}}\left(\Omega^{\left(2\right)},x,z\right)\in\nu_{\mathrm{out}}^{-\frac{1}{2}}\left(z\right)\mathbb{R}\,\,\,\,\mbox{on the compact subsets of }\partial\Omega^{\left(2\right)}\setminus\left\{ x\right\} .
\]
\end{lem}
\begin{proof}
The proof is essentially the same as the one of Lemma \ref{lem:equiv-local-charts}:
we obtain a representation of the observables in terms of the same
conformal map, and the possibly ill-defined derivative terms appearing
in the numerator and the denominator of the fractions cancel.
\end{proof}

\subsubsection{Convolution representation and uniqueness\label{sub:conv-rep-uniqueness}}

We now give the lemma which provides us with a local representation
of functions satisfying Riemann-type boundary conditions in terms
of the convolution kernel $\mathbb{K}$. 
\begin{lem}
\label{lem:cts-conv-rep-uniqueness}Let $\Omega$ be a simply connected
domain with $\partial\Omega=\partial^{\mathrm{s}}\Omega\cup\partial^{\mathrm{r}}\Omega$,
$\partial^{\mathrm{s}}\Omega$ being compact and piecewise smooth.
Let $u:\partial^{\mathrm{s}}\Omega\to\mathbb{C}$ be an arbitrary
continuous function. Then the function $\mathbb{K}\left[\Omega,u\right]:\Omega\to\mathbb{C}$
is the unique holomorphic function satisfying
\begin{eqnarray}
\mathbb{K}\left[\Omega,u\right]\left(z\right)-u\left(z\right) & \in & \nu_{\mathrm{out}}^{-\frac{1}{2}}\left(z\right)\,\,\forall z\in\partial^{\mathrm{s}}\Omega\label{eq:bdry-cond-cts-conv-kern}\\
\mathbb{K}\left[\Omega,u\right]\left(z\right) & \in & \nu_{\mathrm{out}}^{-\frac{1}{2}}\left(z\right)\mbox{ on }\partial^{\mathrm{r}}\Omega\nonumber 
\end{eqnarray}
\end{lem}
\begin{proof}
We have that $\mathbb{K}\left[\Omega,u\right]$ is a real convolution
of $f^{\mathrm{SPIN}}$. It follows from Lemma \ref{lem:bdry-cond-cts-obs-local-charts}
that $\mathbb{K}\left[\Omega,z\right]$ satisfies the boundary condition
\ref{eq:bdry-cond-cts-conv-kern}. Now, for the uniqueness, suppose
there are two holomorphic functions solving this problem and denote
by $f$ their difference. Fix $a\in\partial^{\mathrm{s}}\Omega$.
We have that
\[
z\mapsto\frac{f\left(z\right)}{f^{\mathrm{SPIN}}\left(\Omega,a,z\right)}
\]
extends continuously to $z=a$ (where it is equal to $0$, as $\left|f^{\mathrm{SPIN}}\left(\Omega,a,z\right)\right|\to\infty$
when $z\to a$) and that its imaginary part tends to $0$ as $z\to\partial\Omega$.
Hence this function is identically equal to $0$.
\end{proof}

\subsubsection{Well-definedness of ratios on the boundary\label{sub:well-defined-ratios-bdry}}

Thanks to Lemmas \ref{lem:bdry-cond-cts-obs-local-charts} and \ref{lem:cts-conv-rep-uniqueness}
above, we get a convenient convolution representation of the observables
introduced in Section \ref{sub:continous-holomorphic-obs} in a neighborhood
of the boundary. We can now access the boundary values of these observables
taking ratios with a given reference observable (it is convenient
to choose the spin observable for our purposes). 
\begin{lem}
\label{lem:well-defined-ratios-bdry-carastability}Let $\Omega$ be
a simply connected domain with $\partial\Omega=\partial^{\mathrm{s}}\Omega\cup\partial^{\mathrm{r}}\Omega$,
$\partial^{\mathrm{s}}\Omega$ being compact and piecewise smooth.
Let $u:\partial^{\mathrm{s}}\Omega\to\mathbb{C}$ be a continuous
function and let $x\in\partial^{\mathrm{s}}\Omega$ be on a smooth
part of $\partial^{\mathrm{s}}\Omega$. Then the ratio
\begin{equation}
z\mapsto\frac{\mathbb{K}\left[\Omega,u\right]\left(z\right)}{f^{\mathrm{SPIN}}\left(\Omega,x,z\right)}\label{eq:ratio-conv-kernel-spin-obs}
\end{equation}
extends continuously to $\partial^{\mathrm{r}}\Omega$ and is purely
real there. This ratio varies continuously with respect to $u$.

The ratio \ref{eq:ratio-conv-kernel-spin-obs} is is also Carathéodory-stable
with respect to perturbation of $\partial^{\mathrm{r}}\Omega$: fix
a smooth curve $\gamma$, a continuous function $u:\gamma\to\mathbb{C}$,
a compact set $\mathcal{K}$ such that $\gamma\subset\partial\mathcal{K}$,
$x\in\gamma$ and $w\in\mathrm{Int}\left(\mathcal{K}\right)$; for
any $\epsilon>0$, there exists $\mu>0$ such that if $\Omega^{\left(1\right)}$
and $\Omega^{\left(2\right)}$ are two domains such that $\mathcal{K}\subset\Omega_{1}\cap\Omega_{2}$
and $\gamma\subset\partial\Omega_{1}\cap\partial\Omega_{2}$ that
are $\mu$-close in Carathéodory metric with respect to $w$, then
\[
\left|\frac{\mathbb{K}\left[\Omega^{\left(1\right)},u\right]\left(z\right)}{f^{\mathrm{SPIN}}\left(\Omega^{\left(1\right)},x,z\right)}-\frac{\mathbb{K}\left[\Omega^{\left(2\right)},u\right]\left(z\right)}{f^{\mathrm{SPIN}}\left(\Omega^{\left(2\right)},x,z\right)}\right|\leq\epsilon\,\,\,\,\forall z\in\mathcal{K}.
\]
\end{lem}
\begin{proof}
This follows from Lemma \ref{lem:bdry-cond-cts-obs-local-charts}.
The Carathéodory-stability follows from the definitions of the observables
in terms of conformal mappings.
\end{proof}

\subsection{\label{sub:conv-obs}Convergence of observables}

In this subsection, we discuss and adapt some results of \cite{smirnov-ii,chelkak-smirnov-ii,hongler-smirnov-ii,hongler-i}
to get the convergence of the discrete observables to their continuous
counterparts.

\subsubsection{\label{sub:scaling-limit-fk-obs}Scaling limit of FK observable}

Let us now state the important result concerning scaling limit of
the FK observable \cite{smirnov-ii}. In \cite{smirnov-i}, it is
the key result allowing for the proof of Theorem \ref{thm:fk-interface-to-sle-cv}.
\begin{thm}
\label{thm:fk-obs-cv}Let $d,D>0$. Let $\left(\Omega_{\delta},r,\ell\right)$
be a discrete domain with $\mathrm{diam}\left(\Omega_{\delta}\right)\leq D$.
Then for any $\epsilon>0$, there exists $\delta_{0}>0$ function
of $d,D$ only such that for any $\delta\leq\delta_{0}$ and any $z\in\mathcal{V}_{\Omega_{\delta}^{m}}$
with $\mathrm{dist}\left(z,\partial\Omega_{\delta}\right)\geq d$,
we have
\[
\left|\frac{1}{\delta}\left(f_{\delta}^{\mathrm{FK}}\right)^{2}\left(\Omega_{\delta},r,\ell,z\right)-\left(f^{\mathrm{FK}}\right)^{2}\left(\Omega,r,\ell,z\right)\right|\leq\epsilon
\]
\end{thm}
\begin{proof}
This is the main result of \cite[Theorem 2.2]{smirnov-ii}. It is
generalized in \cite[Theorem 4.3]{chelkak-smirnov-ii} in a form closer
to the form that we use here. Remark that the normalization is slightly
different there: the mesh size $\delta$ in our paper corresponds
to $\sqrt{2}\delta$ with the notation of \cite{smirnov-ii}.
\end{proof}

\subsubsection{Scaling limit of spin observable}

To obtain the scaling limit of the spin observable, we merge two existing
results: in \cite{hongler-i}, the convergence of the observable is
derived, with the additional assumption that the boundary is piecewise
smooth, while in \cite{chelkak-smirnov-ii}, the convergence of is
derived for general domains, but with a different normalization (see
Remark \ref{rem:chsm-spin-obs-normalization} above). 

Let us first specialize a result of \cite{hongler-i} to the case
of a\emph{ straight domain}, i.e. a polygonal domain with horizontal
and vertical sides only. Such a domain will serve us as reference
domain. This result will be used in the proof of Theorem \ref{thm:spin-obs-conv},
both to get precompactness of the spin observable on general domains
and to identify the limit. 
\begin{lem}
\label{lem:cv-spin-obs-in-rect}Let $\mathbf{Q}$ be a straight domain.
Let $a\in\partial\mathbf{Q}$ be at the midpoint of a side. For each
$\delta>0$, denote by $\mathbf{Q}_{\delta}$ the discrete domain
defined by $\mathbf{Q}_{\delta}:=\mathbf{Q}\cap\delta\mathbb{Z}^{2}$
and by $a_{\delta}\in\partial_{0}\mathcal{V}_{\mathbf{\mathbf{Q}}_{\delta}}$
boundary medial vertex that is the closest to $a$ and let $\left\{ c_{\delta}^{1},\ldots,c_{\delta}^{m}\right\} $
be the corners of $\mathbf{\mathbf{Q}}_{\delta}$. Let $d>0$. Then
for each $\epsilon>0$, there exists $\delta_{0}>0$ such that for
any $\delta\leq\delta_{0}$, we have
\begin{eqnarray*}
\left|\frac{1}{\delta}f_{\delta}^{\mathrm{SPIN}}\left(\mathbf{Q}_{\delta},a_{\delta},z_{\delta}\right)-f^{\mathrm{SPIN}}\left(\mathbf{Q}_{\delta},a_{\delta},z_{\delta}\right)\right| & \leq & \epsilon\\
\forall z_{\delta}\in\mathcal{V}_{\mathbf{Q}_{\delta}}:\mathrm{dist}\left(z_{\delta},\left\{ a_{\delta},c_{\delta}^{1},\ldots,c_{\delta}^{m}\right\} \right)\geq d.
\end{eqnarray*}
\end{lem}
\begin{proof}
Suppose for definiteness that $a$ is the midpoint of the left side
of $\mathbf{Q}$. We can apply Theorem 90 in \cite{hongler-i}, which
holds for piecewise smooth domains. The function $f_{\delta}^{\mathrm{SPIN}}\left(\mathbf{Q}_{\delta},a_{\delta},z_{\delta}\right)$
corresponds, in the notation of \cite{hongler-i} to the function
$h_{\mathbf{S}_{\delta}}\left(a_{\delta}^{\left(1\right)^{2}},z_{\delta}^{\left(\sqrt{o}\right)^{2}}\right)+h_{\mathbf{S}_{\delta}}\left(a_{\delta}^{\left(1\right)^{2}},z_{\delta}^{-\left(\sqrt{o}\right)^{2}}\right)$,
where $\sqrt{o}=1$ if $z_{\delta}$ the midpoint of a horizontal
edge and $\sqrt{o}=e^{i\pi/4}$ if it is the midpoint of a vertical
one. From \cite[Theorem 90]{hongler-i}, we obtain the convergence
of $f_{\delta}^{\mathrm{SPIN}}\left(\mathbf{Q}_{\delta},a_{\delta},\cdot\right)$
to $f^{\mathrm{SPIN}}\left(\mathbf{Q},a,\cdot\right)$. To obtain
the lemma, notice that $f^{\mathrm{SPIN}}\left(\mathbf{Q},a,\cdot\right)$
and $f^{\mathrm{SPIN}}\left(\mathbf{Q}_{\delta},a_{\delta},\cdot\right)$
are uniformly close (see Lemma \ref{lem:well-defined-ratios-bdry-carastability}
above).
\end{proof}
The next result that we need is the convergence of ratios of the spin
observable in arbitrary domains, obtained in \cite{chelkak-smirnov-ii}.
\begin{thm}
\label{thm:ratio-spin-obs-cv}Let $\varrho,D>0$. For each $\delta>0$,
let $\left(\Omega_{\delta},a_{\delta},z_{\delta}\right)$ be a discrete
domain with $\partial\Omega_{\delta}=\mathfrak{s}_{\delta}\cup\mathfrak{r}_{\delta}$
such that $\mathfrak{s}_{\delta}$ is made of a finite number of horizontal
and vertical segments $\left\{ \left[p_{\delta}^{j},q_{\delta}^{j}\right]:j=1,\ldots,n\right\} $,
$\mathrm{diam}\left(\left[p_{\delta}^{j},q_{\delta}^{j}\right]\right)\geq\varrho$
for each $j\in\left\{ 1,\ldots,n\right\} $, $a_{\delta}\in\mathfrak{r}_{\delta}$,
$\mathrm{dist}\left(a_{\delta},\mathfrak{s}_{\delta}\right)\geq\varrho$
and $\mathrm{diam}\left(\Omega_{\delta}\right)\leq D$. Then for any
$d>0$ and any $\epsilon>0$, there exists $\delta_{0}>0$ (function
of $d,D,\epsilon$ only) such that for any $\delta\leq\delta_{0}$,
we have

\[
\left|\frac{f_{\delta}^{\mathrm{SPIN}}\left(\Omega_{\delta},a_{\delta},z_{\delta}\right)}{f_{\delta}^{\mathrm{SPIN}}\left(\Omega_{\delta},a_{\delta},y_{\delta}\right)}-\frac{f^{\mathrm{SPIN}}\left(\Omega_{\delta},a_{\delta},z_{\delta}\right)}{f^{\mathrm{SPIN}}\left(\Omega_{\delta},a_{\delta},y_{\delta}\right)}\right|\leq\epsilon
\]
for any $z_{\delta}\in\mathcal{V}_{\Omega_{\delta}}$ such that 
\[
\mathrm{dist}\left(z_{\delta},\mathfrak{r}_{\delta}\cup\bigcup_{j=1}^{n}\left\{ p_{\delta}^{j},q_{\delta}^{j}\right\} \right)\geq d.
\]
 \end{thm}
\begin{proof}
This follows from Theorem 5.9 and Corollary 5.10 in \cite{chelkak-smirnov-ii}.
\end{proof}
Thanks to Lemma \ref{lem:cv-spin-obs-in-rect} and Theorem \ref{thm:ratio-spin-obs-cv},
we can derive the following convergence theorem for the spin observable
when $a$ is on a straight part of the boundary:
\begin{thm}
\label{thm:spin-obs-conv}Let $d,D>0$. For each $\delta>0$, let
$\left(\Omega_{\delta},p_{\delta},a_{\delta},q_{\delta}\right)$ be
a discrete domain such that $\left[p_{\delta},q_{\delta}\right]$
is either horizontal or vertical with $\mathrm{dist}\left(p_{\delta},q_{\delta}\right)\geq d$,
$\mathrm{dist}\left(a_{\delta},\left[q_{\delta},p_{\delta}\right]\right)\geq d$
and with $\mathrm{diam}\left(\Omega_{\delta}\right)\leq D$. Then
for any $\epsilon>0$, there exists $\delta_{0}>0$ (function of $d,D,\epsilon$
only) such that for $\delta\leq\delta_{0}$

\begin{eqnarray*}
\left|\frac{1}{\delta}f_{\delta}^{\mathrm{SPIN}}\left(\Omega_{\delta},a_{\delta},z_{\delta}\right)-f^{\mathrm{SPIN}}\left(\Omega_{\delta},a_{\delta},z_{\delta}\right)\right| & \leq & \epsilon\\
\left|\frac{1}{\delta}g_{\delta}^{\mathrm{SPIN}}\left(\Omega_{\delta},a_{\delta},z_{\delta}\right)-g^{\mathrm{SPIN}}\left(\Omega_{\delta},a_{\delta},z_{\delta}\right)\right| & \leq & \epsilon
\end{eqnarray*}
for all $z_{\delta}\in\mathcal{V}_{\Omega_{\delta}^{m}}$ with $\mathrm{dist}\left(z_{\delta},\partial\Omega_{\delta}\right)\geq d$. \end{thm}
\begin{proof}
The statements for $f^{\mathrm{SPIN}}$ and $g^{\mathrm{SPIN}}$ are
obviously equivalent. Suppose for definiteness that $\left[p_{\delta},q_{\delta}\right]$
is vertical and that the domain $\Omega_{\delta}$ lies on the right
of $\left[p_{\delta},q_{\delta}\right]$. Set $f_{\delta}:=f_{\delta}^{\mathrm{SPIN}}=g_{\delta}^{\mathrm{SPIN}}$
and $f:=f^{\mathrm{SPIN}}=g^{\mathrm{SPIN}}$. Take a straight domain
$\mathbf{Q}$ as in Lemma \ref{lem:cv-spin-obs-in-rect} containing
$\Omega$ (see Figure \ref{fig:domain-within-straight-domain}), with
$a$ being the midpoint of one of its sides, and denote by $\mathbf{Q}_{\delta}$
its discretizations as in Lemma \ref{lem:cv-spin-obs-in-rect}, aligned
in such a way that the $\left[p_{\delta},q_{\delta}\right]\subset\partial\mathbf{Q}_{\delta}$
and that the points $a_{\delta}$ of $\partial_{0}\mathcal{V}_{\mathbf{Q}_{\delta}^{m}}$
and $\partial_{0}\mathcal{V}_{\Omega_{\delta}^{m}}$ coincide. 

\begin{figure}

\includegraphics[width=7cm]{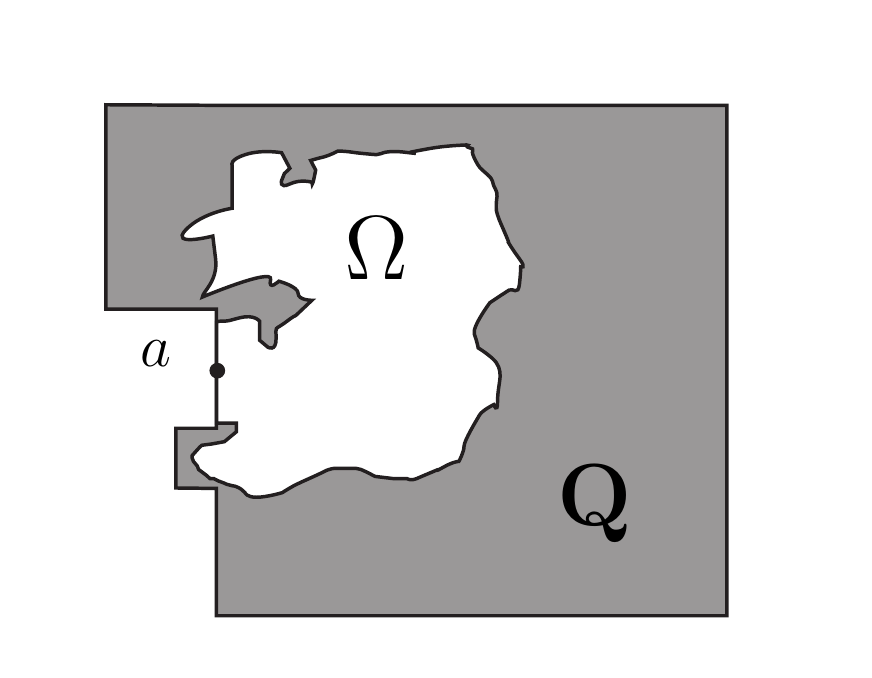}

\caption{\label{fig:domain-within-straight-domain}The domain $\Omega$ and
the straight domain $\mathbf{Q}$.}

\end{figure}

To prove the result, we proceed by contradiction, as in the proof
of \cite{chelkak-smirnov-ii}, Theorem 5.9. Suppose that we can find
an $\epsilon>0$ and a sequence $\left(\Omega_{\delta_{n}},a_{\delta_{n}}\right)$
of discrete domains of mesh size $\delta_{n}\to0$ satisfying the
assumptions of the theorem and a sequence of points $z_{\delta_{n}}\to z$
such that the conclusion fails. For each $n\geq0$, we can moreover
choose a point $y_{\delta_{n}}$ such that $y_{\delta_{n}}\in\left[a_{\delta_{n}},q_{\delta_{n}}\right]$
such that $\mathrm{dist}\left(y_{\delta_{n}},a_{\delta_{n}}\right)\geq\frac{1}{3}d$
and $\mathrm{dist}\left(y_{\delta_{n}},\left[q_{\delta_{n}},p_{\delta_{n}}\right]\right)\geq\frac{1}{3}d$.
\begin{enumerate}
\item Notice first that the sequence of discrete domains $\left(\Omega_{\delta_{n}},p_{\delta_{n}},a_{\delta_{n}},y_{\delta_{n}},q_{\delta_{n}}\right)_{n\geq0}$
is precompact in Carathéodory topology with respect to $z$ and hence
that there is a continuous domain $\left(\Omega,p,a,y,q\right)$ such
that (a susbsequence of) this sequence converges to $\left(\Omega,p,a,y,q\right)$.
\item Precompactness: we show that the family of functions is uniformly
bounded

\begin{enumerate}
\item We have $\left|f_{\delta_{k}}\left(\Omega_{\delta_{k}},a_{\delta_{k}},y_{\delta_{k}}\right)\right|\leq\left|f_{\delta_{k}}\left(\mathbf{Q}_{\delta_{k}},a_{\delta_{k}},y_{\delta_{k}}\right)\right|$,
by Lemma \ref{lem:bdry-spin-obs-spin-corr}, as we have
\begin{eqnarray*}
\left|f_{\delta_{k}}\left(\Omega_{\delta_{k}},a_{\delta_{k}},y_{\delta_{k}}\right)\right| & = & \left(\sqrt{2}-1\right)\mathbb{E}_{\Omega_{\delta_{k}}}^{\mathrm{free}}\left[\sigma\left(a_{\delta_{k}}\right)\sigma\left(y_{\delta_{k}}\right)\right]\\
 & \leq & \left(\sqrt{2}-1\right)\mathbb{E}_{\mathbf{Q}_{\delta_{k}}}^{\mathrm{free}}\left[\sigma\left(a_{\delta_{k}}\right)\sigma\left(y_{\delta_{k}}\right)\right]\\
 & = & \left|f_{\delta_{k}}\left(\mathbf{Q}_{\delta_{k}},a_{\delta_{k}},y_{\delta_{k}}\right)\right|,
\end{eqnarray*}
where we used on the second line that $\Omega_{\delta_{k}}\subset\mathbf{Q}_{\delta_{k}}$
and that two-spin correlations with free boundary conditions are monotone
increasing with respect to the domains (this follows from FKG inequality,
see \cite{grimmett}). 
\item Hence $\frac{1}{\delta_{k}}f_{\delta_{k}}\left(\Omega_{\delta_{k}},a_{\delta_{k}},y_{\delta_{k}}\right)$
is uniformly bounded as $\frac{1}{\delta_{k}}f_{\delta_{k}}\left(\mathbf{Q}_{\delta_{k}},a_{\delta_{k}},y_{\delta_{k}}\right)$
is uniformly bounded by Lemma \ref{lem:cv-spin-obs-in-rect}, being
uniformly convergent.
\item By Theorem \ref{thm:ratio-spin-obs-cv} and Lemma \ref{lem:well-defined-ratios-bdry-carastability},
we have that 
\begin{equation}
\frac{f_{\delta_{k}}\left(\Omega_{\delta_{k}},a_{\delta_{k}},\cdot\right)}{f_{\delta_{k}}\left(\Omega_{\delta_{k}},a_{\delta_{k}},y_{\delta_{k}}\right)}\label{eq:spin-obs-ratio-subseq}
\end{equation}
is uniformly convergent on every compact set of $\overline{\Omega}\setminus\left(\left[q,p\right]\cup\left\{ a\right\} \right)$,
in the sense that for each compact set $K\subset\overline{\Omega}\setminus\left(\left[q,p\right]\cup\left\{ a\right\} \right)$,
the restriction of \ref{eq:spin-obs-ratio-subseq} to $K\cap\mathcal{V}_{\Omega_{\delta_{k}}}$
is uniformly convergent. Hence, by the previous point, we deduce that
the family $\frac{1}{\delta_{k}}f_{\delta_{k}}\left(\Omega_{\delta_{k}},a_{\delta_{k}},\cdot\right)$
is precompact for the topology of uniform convergence on the compact
subsets of $\overline{\Omega}\setminus\left(\left[q,p\right]\cup\left\{ a\right\} \right)$. 
\item By extracting once more a subsequence, we can suppose that 
\[
\frac{1}{\delta_{k}}f_{\delta_{k}}\left(\Omega_{\delta_{k}},a_{\delta_{k}},\cdot\right)
\]
 is uniformly convergent on the compact subsets of $\overline{\Omega}\setminus\left(\left[q,p\right]\cup\left\{ a\right\} \right)$
. Denote by $\tilde{f}$ this limit.
\end{enumerate}
\item Identification of the limit. Let us now show that $\tilde{f}\left(\cdot\right)=f\left(\Omega,a,\cdot\right)$.
We show the following two properties: both functions have the same
boundary conditions and the same pole at $a$ and this characterizes
them uniquely (to check this, take the difference of two functions
satisfying these properties and get that it is equal to $0$ using
Lemma \ref{lem:cts-conv-rep-uniqueness}).

\begin{enumerate}
\item The function $\tilde{f}$ satisfies the boundary condition $\tilde{f}\left(z\right)\in\nu_{\mathrm{out}}^{-\frac{1}{2}}\left(z\right)\mathbb{R}$
on $\partial\Omega\setminus\left\{ a\right\} $: this follows directly
from \cite[Theorem 5.9]{chelkak-smirnov-ii} as 
\[
\frac{f_{\delta_{k}}\left(\Omega_{\delta_{k}},a_{\delta_{k}},\cdot\right)}{f_{\delta_{k}}\left(\Omega_{\delta_{k}},a_{\delta_{k}},y_{\delta_{k}}\right)}\to\frac{f\left(\Omega,a,\cdot\right)}{f\left(\Omega,a,y\right)},
\]
and $f\left(\Omega,a,y\right)\in\mathbb{R}$. 
\item The function $v:=\tilde{f}\left(\cdot\right)-f\left(\mathbf{Q},a,\cdot\right)$
is uniformly bounded in a neighborhood of $a$: take indeed a small
rectangle $R_{\delta_{k}}\subset\Omega_{\delta_{k}}$ such that $\partial R_{\delta_{k}}=\partial_{\delta_{k}}^{1}\cup\partial_{\delta_{k}}^{2}$
with $\partial_{\delta_{k}}^{1}\subset\Omega_{\delta_{k}}$ and $a_{\delta_{k}}\in\partial_{\delta_{k}}^{2}\subset\left[p_{\delta_{k}},q_{\delta_{k}}\right]$
(see Figure \ref{fig:small-rectangle-inside}). On $\partial_{\delta_{k}}^{1}$,
we have that $v_{\delta_{k}}:=\frac{1}{\delta_{k}}\left(f_{\delta_{k}}\left(\Omega_{\delta_{k}},a_{\delta_{k}},\cdot\right)-f_{\delta_{k}}\left(\mathbf{Q}_{\delta_{k}},a,\cdot\right)\right)$
is uniformly bounded. On $\partial_{\delta_{k}}^{2}$, we have the
boundary condition
\[
v_{\delta_{k}}\left(z\right)\in\nu_{\mathrm{out}}^{-\frac{1}{2}}\left(z\right)\mathbb{R}\,\,\,\,\forall z\in\partial_{\delta_{k}}^{2},
\]
as $v_{\delta_{k}}\left(a_{\delta_{k}}\right)=0$. We then have that
the restriction of $v_{\delta_{k}}$ to $R_{\delta_{k}}$ is equal
to 
\[
\mathbb{K}_{\delta_{k}}\left[R_{\delta_{k}},g_{\delta_{k}}|_{\partial_{\delta_{k}}^{1}}\right]
\]
Hence, it follows easily from the representation of $\mathbb{K}_{\delta_{k}}$
(the integrands appearing in it being uniformly bounded) that $v_{\delta_{k}}$
is uniformly bounded near $a$. And hence $\lim_{k\to\infty}v_{\delta_{k}}=\tilde{f}\left(\cdot\right)-f\left(\mathbf{Q},a,\cdot\right)$
is also uniformly bounded. 
\end{enumerate}
\end{enumerate}
Finally, using Lemma \ref{lem:well-defined-ratios-bdry-carastability},
we obtain $\frac{1}{\delta_{k}}f_{\delta_{k}}\left(\Omega_{\delta_{k}},a_{\delta_{k},}z_{\delta_{k}}\right)\to f\left(\Omega,a,z\right)$
as $k\to\infty$, which contradicts the definition of $z_{\delta_{k}}$. 

\begin{figure}
\includegraphics[width=9cm]{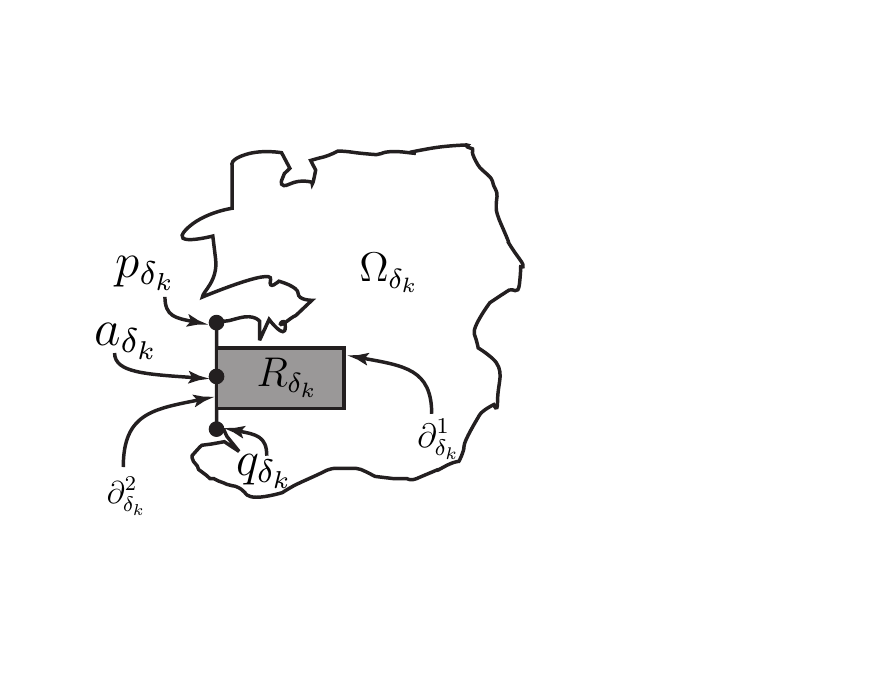}

\caption{\label{fig:small-rectangle-inside}The rectangle $R_{\delta_{k}}\subset\Omega_{\delta_{k}}$.}

\end{figure}

\end{proof}

\subsection{Proof of Theorem \ref{thm:corr-ratios-theorem}\label{sub:correlation-theorem-proof}}

In this subsection, we prove the main convergence theorem of Section
\ref{sec:discrete-complex-analysis} (Theorem \ref{thm:corr-ratios-theorem}).
The central idea is to localize the convergence results of Theorems
\ref{thm:fk-obs-cv} and \ref{thm:ratio-spin-obs-cv} on the boundary,
by representing them in terms of the convolution kernel introduced
in Section \ref{sub:disc-conv-kernel}. 

Let us first introduce some notation. Recall that $\Xi$, $\Theta$
and $\tilde{\Theta}$ are domains coinciding in a neighborhood of
a boundary point $x$. 
\begin{defn}
\label{def:square-boxes}Let $\Xi_{\delta}$ be a discrete vertex
domain. Let $x_{\delta}\in\partial\Xi_{\delta}$ and let $s_{\delta}\in\overline{\Xi_{\delta}}$.
We denote by $\mathbf{Q}_{\delta}\left(x_{\delta},\varrho\right)$
the discrete domain consisting of the square of sidelength $\varrho$,
centered at $x_{\delta}$, with horizontal and vertical sides. Let
$\Lambda_{\delta}$ be the connected component of $\Xi_{\delta}\cap\mathbf{Q}\left(x_{\delta},\varrho\right)$
containing $x_{\delta}$, and suppose $\varrho>0$ is small enough
so that $s_{\delta}\notin\Lambda_{\delta}$. Denote by $l_{\delta}$
the arc of $\partial\Lambda_{\delta}$ that separates $x_{\delta}$
from $s_{\delta}$ in $\Xi_{\delta}$. We denote by $\mathbf{Q}_{\Xi_{\delta}}\left(x_{\delta},\varrho,s_{\delta}\right)$
the connected component of $\Xi_{\delta}\setminus l_{\delta}$ containing
$x_{\delta}$ (see Figure \ref{fig:localization}). We denote by $\mathrm{\Gamma}\left(l_{\delta}\right)$
the set of corners of $l_{\delta}$, i.e. the points of $l_{\delta}$
where a horizontal and a vertical segment of $l_{\delta}$ intersect.

\begin{figure}
\includegraphics[width=11cm]{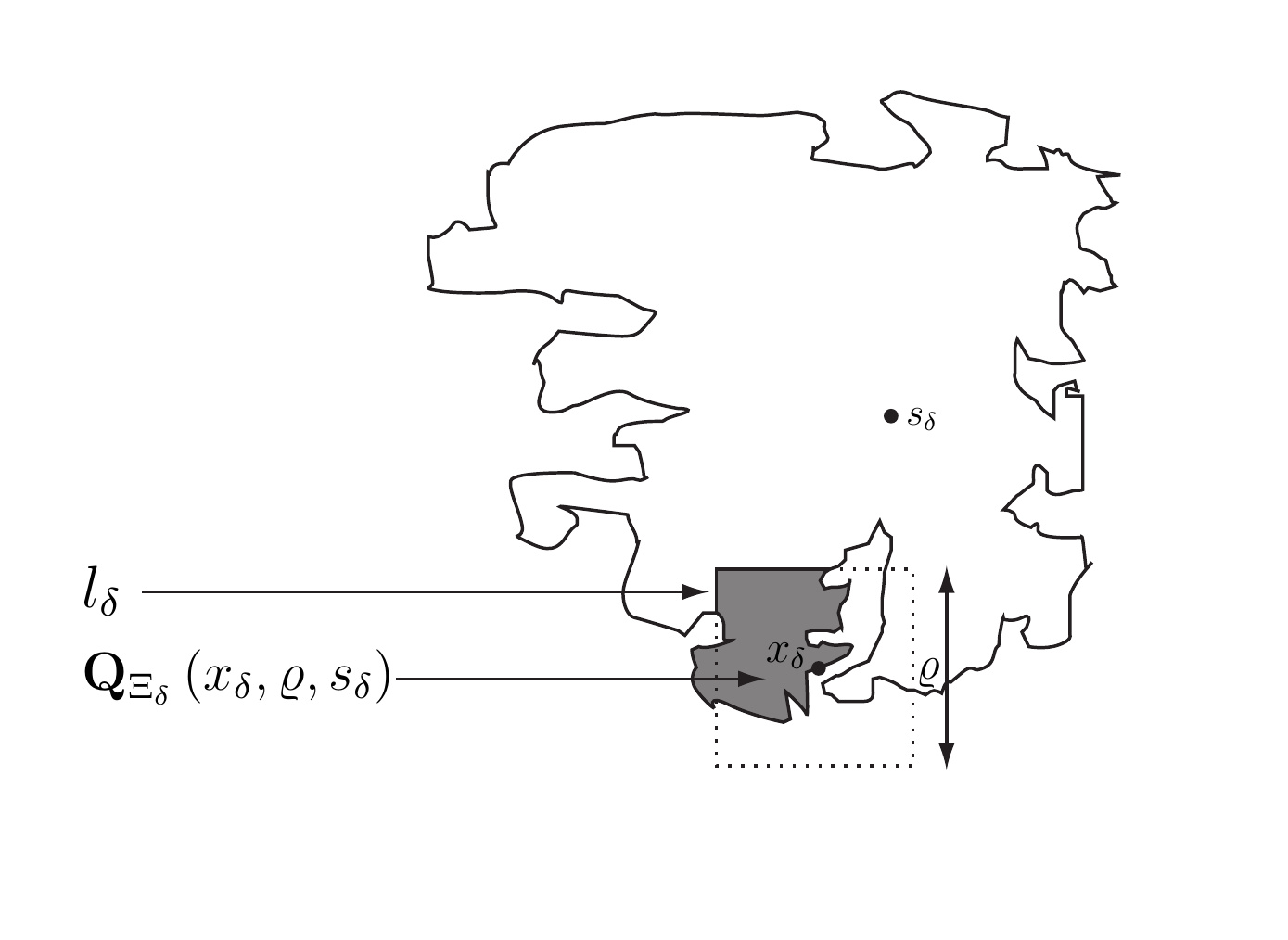}

\caption{\label{fig:localization}Localization near the point $x_{\delta}$}

\end{figure}
\end{defn}
\begin{proof}[Proof of Theorem \ref{thm:corr-ratios-theorem}]
First make the following observations:
\begin{itemize}
\item By Lemmas \ref{lem:fk-obs-bdry-vals-corr-fct} and \ref{lem:bdry-spin-obs-spin-corr},
we have
\begin{eqnarray*}
\mathbb{E}_{\Theta_{\delta}}^{\left[y_{\delta},t_{\delta}\right]_{+}}\left[\sigma_{x_{\delta}}\right] & = & \cos\left(\frac{\pi}{8}\right)\left|f_{\delta}^{\mathrm{FK}}\left(\Theta_{\delta},y_{\delta},t_{\delta},x_{\delta}\right)\right|\\
\mathbb{E}_{\Xi_{\delta}}^{\mathrm{free}}\left[\sigma_{x_{\delta}}\sigma_{s_{\delta}}\right] & = & \left(\sqrt{2}+1\right)\left|f_{\delta}^{\mathrm{SPIN}}\left(\Xi_{\delta},s_{\delta},x_{\delta}\right)\right|.
\end{eqnarray*}

\item By Lemma \ref{lem:bdry-val-cts-obs-as-cft-corr-fct}, we have
\begin{eqnarray*}
\frac{\left\langle \sigma_{x}\sigma_{s}\right\rangle _{\Xi}^{\mathrm{free}}}{\left\langle \sigma_{x}\right\rangle _{\Theta}^{\left[y,t\right]_{+}}} & = & \frac{\sqrt{2}+1}{\cos\left(\frac{\pi}{8}\right)}\left|\frac{f^{\mathrm{SPIN}}\left(\Xi,s,x\right)}{f^{\mathrm{FK}}\left(\Theta,y,t,x\right)}\right|,\\
\frac{\left\langle \sigma_{x}\right\rangle _{\tilde{\Theta}}^{\left[\tilde{y}_{\delta},\tilde{t}_{\delta}\right]_{+}}}{\left\langle \sigma_{x}\right\rangle _{\Theta}^{\left[y_{\delta},t_{\delta}\right]_{+}}} & = & \left|\frac{f^{\mathrm{FK}}\left(\tilde{\Theta},\tilde{y},\tilde{t},x\right)}{f^{\mathrm{FK}}\left(\Theta,y,t,x\right)}\right|.
\end{eqnarray*}

\item By Theorem \ref{thm:fk-obs-cv} we have that 
\[
\frac{1}{\delta}\left(f_{\delta}^{\mathrm{FK}}\right)^{2}\left(\Theta_{\delta},y_{\delta},t_{\delta},\cdot\right)\underset{\delta\to0}{\longrightarrow}\left(f^{\mathrm{FK}}\right)^{2}\left(\Theta,y,t,\cdot\right),
\]
on the compact subsets of $\Theta$. By changing if necessary the
sign of $f_{\delta}^{\mathrm{FK}}$ (and choosing an arbitrary branch
of the square root to define $f^{\mathrm{FK}}$) we can suppose that
\[
\frac{1}{\sqrt{\delta}}f_{\delta}^{\mathrm{FK}}\left(\Theta_{\delta},y,t,\cdot\right)\underset{\delta\to0}{\longrightarrow}f^{\mathrm{FK}}\left(\Theta,y,t,\cdot\right),
\]

\item By Theorem \ref{thm:spin-obs-conv}, we have 
\[
\frac{1}{\delta}f_{\delta}^{\mathrm{SPIN}}\left(\Xi_{\delta},s_{\delta},\cdot\right)\underset{\delta\to0}{\longrightarrow}f^{\mathrm{SPIN}}\left(\Xi,s,\cdot\right)
\]
on the compact subsets of $\Xi$.
\end{itemize}
We now want to localize our observables in a neighborhood of $x_{\delta}$.
\begin{itemize}
\item Take $\varrho>0$ small enough so that $\mathbf{Q}_{\Theta_{\delta}}\left(x_{\delta},\varrho,y_{\delta}\right)=\mathbf{Q}_{\Xi_{\delta}}\left(x_{\delta},\varrho,s_{\delta}\right)$
(see Definition \ref{def:square-boxes} and Figure \ref{fig:localization}). 
\item Set $\Omega_{\delta}:=\mathbf{Q}_{\Theta_{\delta}}\left(x_{\delta},\varrho,y_{\delta}\right)$. 
\item Let $\partial^{\mathrm{r}}\Omega_{\delta}\subset\partial\Omega_{\delta}$
(the {}``rough part of the boundary'') be such that 
\[
\partial^{\mathrm{r}}\Omega_{\delta}=\partial\Omega_{\delta}\cap\partial\Theta_{\delta}=\partial\Omega_{\delta}\cap\partial\Xi_{\delta},
\]
In other words, the arc $\partial^{\mathrm{r}}\Omega_{\delta}$ is
the arc of the boundary of $\Omega_{\delta}$ that is common to $\Theta_{\delta}$
and $\Xi_{\delta}$. 
\item Let $\partial^{\mathrm{s}}\Omega_{\delta}$ be $\partial\Omega_{\delta}\setminus\partial^{\mathrm{r}}\Omega_{\delta}$:
it is the arc which is made of the sides of the square $\mathbf{Q}_{\delta}\left(x_{\delta},\varrho\right)$
(it is equal to $l_{\delta}$ in Definition \ref{def:square-boxes}). 
\item Let $D>0$ be such that $\Omega_{\delta}\subset\mathbf{D}\left(0,D\right)$
for all $\delta>0$.
\item Let $d>0$ be such that $\mathrm{dist}\left(x_{\delta},\partial^{\mathrm{s}}\Omega_{\delta}\right)\geq d$
for all $\delta>0$ and such that we can choose a point $w_{\delta}$
away from the corners and the {}``rough part of the boundary'',
with $\mathrm{dist}\left(w_{\delta},\partial^{\mathrm{r}}\Omega_{\delta}\cup\Gamma\left(l_{\delta}\right)\right)\geq d$
for all $\delta>0$. 
\end{itemize}
Define $u_{\delta}^{\mathrm{FK}}:\partial_{0}\mathcal{V}_{\Omega_{\delta}^{m}}\to\mathbb{C}$
by 
\begin{eqnarray*}
u_{\delta}^{\mathrm{FK}}\left(z\right): & = & \begin{cases}
\mathsf{P}_{\nu_{\mathrm{in}}^{-\frac{1}{2}}\left(z\right)}\left[\frac{1}{\sqrt{\delta}}f_{\delta}^{\mathrm{FK}}\left(\Theta_{\delta},y,t,z\right)\right] & \,\,\,\, z\in\partial_{0}^{\mathrm{s}}\mathcal{V}_{\Omega_{\delta}^{m}},\\
0 & \,\,\,\, z\in\partial_{0}^{\mathrm{r}}\mathcal{V}_{\Omega_{\delta}^{m}}
\end{cases}
\end{eqnarray*}
 and $u^{\mathrm{FK}}\left(\cdot\right):\partial\Omega\to\mathbb{C}$
by
\[
u^{\mathrm{FK}}\left(z\right):=\begin{cases}
\mathsf{P}_{\nu_{\mathrm{in}}^{-\frac{1}{2}}\left(z\right)}\left[f^{\mathrm{FK}}\left(\Theta_{\delta},y,t,z\right)\right] & \,\,\,\, z\in\partial_{0}^{\mathrm{s}}\mathcal{V}_{\Omega_{\delta}^{m}},\\
0 & \,\,\,\, z\in\partial_{0}^{\mathrm{r}}\mathcal{V}_{\Omega_{\delta}^{m}}.
\end{cases}
\]
 Assume without loss of generality that $\nu_{\mathrm{in}}\left(w_{\delta}\right)=1$,
so that $f_{\delta}^{\mathrm{SPIN}}\left(\Omega_{\delta},w_{\delta},\cdot\right)=g_{\delta}^{\mathrm{SPIN}}\left(\Omega_{\delta},w_{\delta},\cdot\right)$
and $f^{\mathrm{SPIN}}\left(\Omega_{\delta},w_{\delta},\cdot\right)=g^{\mathrm{SPIN}}\left(\Omega_{\delta},w_{\delta},\cdot\right)$.

Let $\epsilon>0$. We want to show that there exists a $\delta_{0}$
(depending only on $\varrho,d,D$) such that for any $\delta\leq\delta_{0}$,
\begin{equation}
\left|\frac{\frac{1}{\sqrt{\delta}}f_{\delta}^{\mathrm{FK}}\left(\Theta_{\delta},y,t,x_{\delta}\right)}{\frac{1}{\delta}f_{\delta}^{\mathrm{SPIN}}\left(\Omega_{\delta},w_{\delta},x_{\delta}\right)}-\frac{f^{\mathrm{FK}}\left(\Theta_{\delta},y,t,x_{\delta}\right)}{f^{\mathrm{SPIN}}\left(\Omega_{\delta},w_{\delta},x_{\delta}\right)}\right|\leq\epsilon.\label{eq:cv-ratio-fk-local-spin-obs}
\end{equation}
Let us first use the convolution representations introduced in Sections
\ref{sub:disc-conv-kernel} and \ref{sub:cts-conv-kern}.
\begin{itemize}
\item By Lemma \ref{lem:useful-form-convolution-lemma}, we have that
\begin{eqnarray*}
\frac{\frac{1}{\sqrt{\delta}}f_{\delta}^{\mathrm{FK}}\left(\Theta_{\delta},y,t,x_{\delta}\right)}{\frac{1}{\delta}f_{\delta}^{\mathrm{SPIN}}\left(\Omega_{\delta},w_{\delta},x_{\delta}\right)} & = & \frac{\mathbb{K}_{\delta}\left[\Omega_{\delta},u_{\delta}^{\mathrm{FK}}|_{\partial_{0}^{\mathrm{s}}\mathcal{V}_{\Omega_{\delta}^{m}}}\right]\left(x_{\delta}\right)}{\frac{1}{\delta}g_{\delta}^{\mathrm{SPIN}}\left(\Omega_{\delta},w_{\delta},x_{\delta}\right)}\\
 & = & \sum_{z_{\delta}\in\partial_{0}^{\mathrm{s}}\mathcal{V}_{\Omega_{\delta}^{m}}}u_{\delta}^{\mathrm{FK}}\left(z_{\delta}\right)\frac{g_{\delta}^{\mathrm{SPIN}}\left(\Omega_{\delta},z_{\delta},x_{\delta}\right)}{g_{\delta}^{\mathrm{SPIN}}\left(\Omega_{\delta},w_{\delta},x_{\delta}\right)}\delta\\
 & = & \sum_{z_{\delta}\in\partial_{0}^{\mathrm{s}}\mathcal{V}_{\Omega_{\delta}^{m}}}u_{\delta}^{\mathrm{FK}}\left(z_{\delta}\right)\frac{\overline{g_{\delta}^{\mathrm{SPIN}}\left(\Omega_{\delta},x_{\delta},z_{\delta}\right)}}{\overline{g_{\delta}^{\mathrm{SPIN}}\left(\Omega_{\delta},x_{\delta},w_{\delta}\right)}}\delta.
\end{eqnarray*}

\item Similarly, by Proposition \ref{lem:cts-conv-rep-uniqueness}, we have
\begin{eqnarray*}
\frac{f^{\mathrm{FK}}\left(\Theta_{\delta},y,t,x_{\delta}\right)}{f^{\mathrm{SPIN}}\left(\Omega_{\delta},w_{\delta},x_{\delta}\right)} & = & \frac{\mathbb{K}\left[\Omega,u^{\mathrm{FK}}|_{\partial^{\mathrm{s}}\Omega_{\delta}}\right]\left(x_{\delta}\right)}{g^{\mathrm{SPIN}}\left(\Omega_{\delta},w_{\delta},x_{\delta}\right)}\\
 & = & \int_{\partial_{0}^{\mathrm{s}}\Omega_{\delta}}u^{\mathrm{FK}}\left(z\right)\frac{g^{\mathrm{SPIN}}\left(\Omega_{\delta},z,x_{\delta}\right)}{g^{\mathrm{SPIN}}\left(\Omega_{\delta},w_{\delta},x_{\delta}\right)}\mathrm{d}\left|z\right|\\
 & = & \int_{\partial_{0}^{\mathrm{s}}\Omega_{\delta}}u^{\mathrm{FK}}\left(z\right)\frac{\overline{g^{\mathrm{SPIN}}\left(\Omega_{\delta},x_{\delta},z\right)}}{\overline{g^{\mathrm{SPIN}}\left(\Omega_{\delta},x_{\delta},w_{\delta}\right)}}\mathrm{d}\left|z\right|,
\end{eqnarray*}
where $u^{\mathrm{FK}}:\partial_{0}^{\mathrm{s}}\Omega_{\delta}\to\mathbb{C}$
is defined by 
\[
u^{\mathrm{FK}}\left(z\right):=\mathsf{P}_{\nu_{\mathrm{in}}\left(z\right)}\left[f^{\mathrm{FK}}\left(\Theta_{\delta},y,t,z\right)\right].
\]

\end{itemize}
In order to prove the theorem, we prove the convergence of the discrete
convolution representation above to the continuous one. The integrand
in the convolution converges away from the {}``rough part'' and
the corners (Theorems \ref{thm:fk-obs-cv} and \ref{thm:ratio-spin-obs-cv}),
so we just need to control the values of this integrand near the corners
and the {}``rough part''. For this, we use a priori estimates which
will be proven in the next subsection. 

Now, set $\psi_{\delta}\left(\cdot\right):=g_{\delta}^{\mathrm{SPIN}}\left(\Omega_{\delta},x_{\delta},\cdot\right)$
and $\psi\left(\cdot\right):=g^{\mathrm{SPIN}}\left(\Omega_{\delta},x_{\delta},\cdot\right)$.
All the estimates below will be depend on $\varrho,d,D$ only. 
\begin{itemize}
\item By Proposition \ref{lem:fk-and-spin-a-priori-est} in the next subsection,
there exists $C^{\mathrm{FK}}$ and $\vartheta>0$ such that
\begin{eqnarray*}
\left|u_{\delta}^{\mathrm{FK}}\left(z_{\delta}\right)\right| & \leq & \frac{C^{\mathrm{FK}}}{\mathrm{dist}\left(z_{\delta},\left[t_{\delta},y_{\delta}\right]\right)^{\frac{1}{2}-\vartheta}}.
\end{eqnarray*}

\item By Lemma \ref{lem:spin-ratio-apriori-ctrl} in the next subsection,
there exists $C>0$ such that 
\[
\left|\frac{\psi_{\delta}\left(z_{\delta}\right)}{\psi_{\delta}\left(w_{\delta}\right)}\right|\leq\frac{C}{\sqrt{\mathrm{dist}\left(z_{\delta},\partial^{\mathrm{r}}\Omega_{\delta}\cup\Gamma\left(l_{\delta}\right)\right)}}.
\]

\end{itemize}
Hence, for any $\epsilon>0$, we can find $\theta>0$ and split $\partial^{\mathrm{s}}\Omega_{\delta}$
into $\partial^{\mathrm{b}}\Omega_{\delta}\cup\partial^{\mathrm{i}}\Omega_{\delta}$
in such a way that for any $\delta>0$:
\begin{enumerate}
\item $\mathrm{dist}\left(\partial^{\mathrm{i}}\Omega_{\delta},\partial^{\mathrm{r}}\Omega_{\delta}\cup\Gamma\left(\Omega_{\delta}\right)\right)\geq\theta$
\item $\sum_{z_{\delta}\in\partial_{0}^{\mathfrak{b}}\mathcal{V}_{\Omega_{\delta}^{m}}}\left|u_{\delta}^{\mathrm{FK}}\left(z_{\delta}\right)\right|\left|\frac{\psi_{\delta}\left(z_{\delta}\right)}{\psi\left(w_{\delta}\right)}\right|\cdot\delta\leq\frac{\epsilon}{6}$
\item $\int_{\partial_{0}^{\mathrm{b}}\Omega_{\delta}}\left|u^{\mathrm{FK}}\left(z\right)\right|\left|\frac{\psi\left(z\right)}{\psi\left(w_{\delta}\right)}\right|\mathrm{d}\left|z\right|\leq\frac{\epsilon}{6}$
.
\end{enumerate}
We obtain
\begin{eqnarray*}
 &  & \left|\frac{\frac{1}{\sqrt{\delta}}f_{\delta}^{\mathrm{FK}}\left(\Theta_{\delta},y,t,x_{\delta}\right)}{\frac{1}{\delta}f_{\delta}^{\mathrm{SPIN}}\left(\Omega_{\delta},w_{\delta},x_{\delta}\right)}-\frac{f^{\mathrm{FK}}\left(\Theta_{\delta},y,t,x_{\delta}\right)}{f^{\mathrm{SPIN}}\left(\Omega_{\delta},w_{\delta},x_{\delta}\right)}\right|\\
 & = & \left|\sum_{z_{\delta}\in\partial_{0}^{\mathrm{s}}\mathcal{V}_{\Omega_{\delta}^{m}}}u_{\delta}^{\mathrm{FK}}\left(z_{\delta}\right)\frac{\psi_{\delta}\left(z_{\delta}\right)}{\psi_{\delta}\left(w_{\delta}\right)}\cdot\delta-\int_{\partial^{\mathrm{s}}\Omega_{\delta}}u^{\mathrm{FK}}\left(z\right)\frac{\psi\left(z\right)}{\psi\left(w_{\delta}\right)}\mathrm{d}\left|z\right|\right|\\
 & \leq & \frac{\epsilon}{3}+\mathbf{A}_{\delta}
\end{eqnarray*}
where 
\[
\mathbf{A}_{\delta}:=\left|\sum_{z_{\delta}\in\partial_{0}^{\mathfrak{i}}\mathcal{V}_{\Omega_{\delta}^{m}}}u_{\delta}^{\mathrm{FK}}\left(z_{\delta}\right)\frac{\psi_{\delta}\left(z_{\delta}\right)}{\psi_{\delta}\left(w_{\delta}\right)}\cdot\delta-\int_{\partial^{\mathrm{i}}\Omega_{\delta}}u^{\mathrm{FK}}\left(z\right)\frac{\psi\left(z\right)}{\psi\left(w_{\delta}\right)}\mathrm{d}\left|z\right|\right|.
\]
We can write 
\[
\mathbf{A}_{\delta}\leq\mathbf{B}_{\delta}+\mathbf{C}_{\delta},
\]
where 
\begin{eqnarray*}
\mathbf{B}_{\delta} & := & \left|\sum_{z_{\delta}\in\partial_{0}^{\mathfrak{i}}\mathcal{V}_{\Omega_{\delta}^{m}}}u_{\delta}^{\mathrm{FK}}\left(z_{\delta}\right)\frac{\psi_{\delta}\left(z_{\delta}\right)}{\psi_{\delta}\left(w_{\delta}\right)}\cdot\delta-\sum_{z_{\delta}\in\partial_{0}^{\mathfrak{i}}\mathcal{V}_{\Omega_{\delta}^{m}}}u^{\mathrm{FK}}\left(z_{\delta}\right)\frac{\psi\left(z_{\delta}\right)}{\psi\left(w_{\delta}\right)}\cdot\delta\right|,\\
\mathbf{C}_{\delta} & := & \left|\sum_{z_{\delta}\in\partial_{0}^{\mathfrak{i}}\mathcal{V}_{\Omega_{\delta}^{m}}}u^{\mathrm{FK}}\left(z_{\delta}\right)\frac{\psi\left(z_{\delta}\right)}{\psi\left(w_{\delta}\right)}\cdot\delta-\int_{\partial^{\mathrm{i}}\Omega_{\delta}}u^{\mathrm{FK}}\left(z\right)\frac{\psi\left(z\right)}{\psi\left(w_{\delta}\right)}\mathrm{d}\left|z\right|\right|.
\end{eqnarray*}

Let us estimate $\mathbf{B}_{\delta}$ first. From the following two
observations, we get that there exists $\delta_{\mathbf{B}}>0$ such
that for any $\delta\leq\delta_{\mathbf{B}}$, $\mathbf{B}_{\delta}\leq\epsilon/3$.
\begin{itemize}
\item By Theorem \ref{thm:fk-obs-cv}, for any $\epsilon_{1}>0$, we have
that there exists $\delta_{1}>0$ such that for $\delta\leq\delta_{1}$,
$\left|u_{\delta}^{\mathrm{FK}}\left(z\right)-u^{\mathrm{FK}}\left(z\right)\right|\leq\delta_{1}$
on $\partial_{0}^{\mathrm{i}}\mathcal{V}_{\Omega_{\delta}^{m}}$.
\item By Theorem \ref{thm:ratio-spin-obs-cv} and Theorem \ref{thm:ratio-spin-obs-cv},
for any $\epsilon_{2}>0$, there exists $\delta_{2}>0$ such that
for $\delta\leq\delta_{2}$, 
\[
\left|\frac{\psi_{\delta}\left(z_{\delta}\right)}{\psi_{\delta}\left(w_{\delta}\right)}-\frac{\psi\left(z_{\delta}\right)}{\psi\left(w_{\delta}\right)}\right|\leq\epsilon_{2}\,\,\,\,\forall z_{\delta}\in\partial_{0}^{\mathrm{i}}\mathcal{V}_{\Omega_{\delta}^{m}}.
\]

\end{itemize}
From standard calculus, there exists $\delta_{\mathbf{C}}>0$ such
that for any $\delta\leq\delta_{\mathbf{C}}$, $\mathbf{C}_{\delta}\leq\epsilon/3$.
Hence taking $\delta_{0}=\min\left(\delta_{\mathbf{B}},\delta_{\mathbf{C}}\right)$,
we obtain the desired inequality (Eq. \ref{eq:cv-ratio-fk-local-spin-obs}).
We therefore obtain
\begin{equation}
\frac{\frac{1}{\sqrt{\delta}}f_{\delta}^{\mathrm{FK}}\left(\Theta_{\delta},y_{\delta},t_{\delta},x\right)}{\frac{1}{\delta}f_{\delta}^{\mathrm{SPIN}}\left(\Omega_{\delta},w_{\delta},x_{\delta}\right)}\to\frac{f^{\mathrm{FK}}\left(\Theta,y,t,x\right)}{f^{\mathrm{SPIN}}\left(\Omega,w,x\right)},\label{eq:first-ratio}
\end{equation}
noticing that the right-hand side is Carathéodory-stable (see Lemma
\ref{lem:well-defined-ratios-bdry-carastability}). 

Repeating the arguments above, replacing $\frac{1}{\sqrt{\delta}}f_{\delta}^{\mathrm{FK}}\left(\Theta_{\delta},y_{\delta},t_{\delta},x_{\delta}\right)$
by $\frac{1}{\delta}f_{\delta}^{\mathrm{SPIN}}\left(\Xi_{\delta},s_{\delta},x_{\delta}\right)$,
we obtain

\begin{equation}
\frac{\frac{1}{\delta}f_{\delta}^{\mathrm{SPIN}}\left(\Xi_{\delta},s_{\delta},x_{\delta}\right)}{\frac{1}{\delta}f_{\delta}^{\mathrm{SPIN}}\left(\Omega_{\delta},w_{\delta},x_{\delta}\right)}\to\frac{f^{\mathrm{SPIN}}\left(\Xi,s,x\right)}{f^{\mathrm{SPIN}}\left(\Omega,w,x\right)}.\label{eq:second-ratio}
\end{equation}

Taking the ratio of the convergence results \ref{eq:first-ratio}
and \ref{eq:second-ratio} and noticing that the limits are both nonzero
(Lemma \ref{lem:bdry-cond-cts-obs-local-charts}), we obtain the first
convergence result of the theorem. The second one is obtained by using
exactly the same method.
\end{proof}

\subsection{A priori estimates for holomorphic observables\label{sub:apriori-est-hol-obs}}

In this subsection, we derive the technical lemmas that have been
used in Section \ref{sub:correlation-theorem-proof}. Most importantly,
these estimates are uniform with respect to the mesh size $\delta>0$.
\begin{lem}
\label{lem:fk-and-spin-a-priori-est}There exists a universal $\vartheta>0$
such that for any $d,D,\varrho>0$, there exists $C^{\mathrm{FK}}\left(d,D,r\right)$
such that for any discrete domain $\left(\Omega_{\delta},r_{\delta},\ell_{\delta}\right)$
, we have
\[
\left|\frac{1}{\sqrt{\delta}}f_{\delta}^{\mathrm{FK}}\left(\Omega_{\delta},r_{\delta},\ell_{\delta},z_{\delta}\right)\right|\leq\frac{C^{\mathrm{FK}}}{\left(\mathrm{dist}\left(z_{\delta},\partial\Omega_{\delta}\right)\right)^{\frac{1}{2}-\vartheta}}
\]
for any $z_{\delta}\in\mathcal{V}_{\Omega_{\delta}^{m}}$ with $\mathrm{dist}\left(z_{\delta},\left\{ r_{\delta},\ell_{\delta}\right\} \right)\geq d$.
Similarly there exists $C^{\mathrm{SPIN}}\left(d,D,r\right)$ such
that for any discrete domain $\left(\Omega_{\delta},p_{\delta}^{1},a_{\delta},p_{\delta}^{2}\right)$
with $\mathrm{dist}\left(p_{\delta}^{1},p_{\delta}^{2}\right)\geq d$,
$\mathrm{dist}\left(a_{\delta},\left[p_{\delta}^{2},p_{\delta}^{1}\right]\right)\geq d$
and with $\mathrm{diam}\left(\Omega_{\delta}\right)\leq D$, and for
any $z_{\delta}\in\mathcal{V}_{\Omega_{\delta}^{m}}$ with $\mathrm{dist}\left(z_{\delta},a_{\delta}\right)\geq\varrho$,
we have
\[
\left|\frac{1}{\delta}f_{\delta}^{\mathrm{SPIN}}\left(\Omega_{\delta},a_{\delta},z_{\delta}\right)\right|\leq\frac{C^{\mathrm{SPIN}}}{\left(\mathrm{dist}\left(z_{\delta},\partial\Omega_{\delta}\right)\right)^{\frac{1}{2}-\vartheta}}.
\]

\end{lem}
The proof is given in Section \ref{sub:proofs-shol-est-ctrl-lemmas}.
\begin{lem}
\label{lem:spin-ratio-apriori-ctrl}Let $r,R>0$. Let $\left(\Omega_{\delta},p_{\delta}^{1},p_{\delta}^{2},x_{\delta}\right)$
be a discrete domain with $\mathrm{diam}\left(\Omega_{\delta}\right)\leq R$,
with $\left[p_{\delta}^{1},p_{\delta}^{2}\right]$ being a straight
part of $\partial\Omega_{\delta}$, with $\mathrm{dist}\left(p_{\delta}^{1},p_{\delta}^{2}\right)\geq r$,
$w_{\delta}\in\left[p_{\delta}^{1},p_{\delta}^{2}\right]$, $\mathrm{dist}\left(w_{\delta},\left[p_{\delta}^{2},p_{\delta}^{1}\right]\right)\geq r$,
$\mathrm{dist}\left(x_{\delta},\left[p_{\delta}^{1},p_{\delta}^{2}\right]\right)\geq r$.
Then there exists $C$, dependent on $r,R$ only. 
\[
\left|\frac{g_{\delta}^{\mathrm{SPIN}}\left(\Omega_{\delta},x_{\delta},z_{\delta}\right)}{g_{\delta}^{\mathrm{SPIN}}\left(\Omega_{\delta},x_{\delta},w_{\delta}\right)}\right|\leq\frac{C}{\sqrt{\mathrm{dist}\left(z_{\delta},\left\{ p_{\delta}^{1},p_{\delta}^{2}\right\} \right)}}
\]
 for all $z_{\delta}$ on $\left[p_{\delta}^{1},p_{\delta}^{2}\right]$.
\end{lem}
The proof is given in Section \ref{sub:proofs-shol-est-ctrl-lemmas}.

\subsubsection{Discrete complex analysis techniques}
\begin{defn}
For a function $f:\overline{\mathcal{V}}_{\Omega_{\delta}^{*}}\to\mathbb{C}$,
we define the Laplacian by $\Delta_{\circ}f:\mathcal{V}_{\Omega_{\delta}^{*}}\to\mathbb{C}$
by
\[
\Delta_{\circ}f\left(v\right):=\sum_{w\in\overline{\mathcal{V}}_{\Omega_{\delta}^{*}}:v\sim w}f\left(w\right)-f\left(v\right).
\]

For a function $f:\overline{\mathcal{V}}_{\Omega_{\delta}}\to\mathbb{C}$
and an arc $\mathfrak{a}\subset\partial\mathcal{V}_{\Omega_{\delta}}$
following \cite{chelkak-smirnov-ii} (see also \cite[Section 2.6.1]{hongler-i}
for a setup closer to the one employed here), we define the (modified)
Laplacian $\tilde{\Delta}_{\bullet}f:\mathcal{V}_{\Omega_{\delta}}\to\mathbb{C}$
by
\[
\tilde{\Delta}_{\bullet}f_{\delta}\left(v\right):=2\alpha\left(\sum_{w\in\mathfrak{\partial\mathcal{V}_{\Omega_{\delta}}}:w\sim v}f\left(w\right)-f\left(v\right)\right)+\sum_{w\in\mathcal{V}_{\Omega_{\delta}}:w\sim v}f\left(w\right)-f\left(v\right),
\]
where $\alpha=\sqrt{2}-1$ as usual.
\end{defn}

\subsubsection{Discrete integral of the square}

Except in specific cases, the product (or even the square) of s-holomorphic
functions is no longer discrete holomorphic. However, a remarkable
feature of s-holomorphic functions is that the (real part of the)
antiderivative of the square of an s-holomorphic function can be still
defined. It turns out to be particularly useful to integrate the square
of the observables $g_{\delta}^{\mathrm{FK}}$ and $g_{\delta}^{\mathrm{SPIN}}$,
as the boundary of values of the resulting functions are much simpler.

The following lemma was introduced in \cite[Lemma 3.6]{smirnov-ii}.
An important simplification of the boundary value was introduced in
\cite[Section 3.6]{chelkak-smirnov-ii}.
\begin{prop}
\label{prop:antideriv-square-fk}Let $\left(\Omega_{\delta},r,\ell\right)$
be a discrete domain and set $f_{\delta}\left(\cdot\right):=\frac{1}{\sqrt{\delta}}f_{\delta}^{\mathrm{FK}}\left(\Omega_{\delta},r,\ell,\cdot\right)$.
There exists a unique function $\mathbb{I}f_{\delta}:\mathcal{V}_{\Omega_{\delta}}\cup\mathcal{V}_{\Omega_{\delta}^{*}}\cup\partial\mathcal{V}_{\left[\ell,r\right]^{*}}\to\mathbb{R}$
such that 
\begin{itemize}
\item $\mathbb{I}f_{\delta}$ is equal to $0$ on $\mathcal{V}_{\left[\ell,r\right]^{*}}$.
\item $\mathbb{I}f_{\delta}$ is equal to $1$ on $\mathcal{V}_{\left[r,\ell\right]}$.
\item If $e=\left\langle x,y\right\rangle \in\mathcal{E}_{\Omega_{\delta}^{m}}$
is a edge and $b\in\mathcal{V}_{\Omega_{\delta}}$ and $w\in\mathcal{V}_{\Omega_{\delta}^{*}}\cup\partial\mathcal{V}_{\left[\ell,r\right]^{*}}$
are such that $\left\langle b,w\right\rangle =e^{*}$, then 
\begin{equation}
\mathbb{I}f_{\delta}\left(b\right)-\mathbb{I}f_{\delta}\left(w\right)=\sqrt{2}\delta\left|\mathsf{P}_{\ell\left(e\right)}\left[f_{\delta}\left(x\right)\right]\right|^{2}=\sqrt{2}\delta\left|\mathsf{P}_{\ell\left(e\right)}\left[f_{\delta}\left(y\right)\right]\right|^{2}.\label{eq:anti-deriv-square-1}
\end{equation}

\item For any $v,w\in\mathcal{V}_{\Omega_{\delta}}$ with $v\sim w$ or
any $v,w\in\mathcal{V}_{\Omega_{\delta}^{*}}\cup\partial\mathcal{V}_{\left[\ell,r\right]^{*}}$
with $v\sim w$, we have $\mathbb{I}f_{\delta}\left(v\right)-\mathbb{I}f_{\delta}\left(w\right)=-\Re\mathfrak{e}\left(f_{\delta}^{2}\left(m\right)\left(v-w\right)\right)$,
where $m\in\mathcal{V}_{\Omega_{\delta}^{m}}$ is the midpoint of
$\left\langle v,w\right\rangle $.
\item The function $\mathbb{I}f_{\delta}$ can be extended to $\partial\mathcal{V}_{\left[\ell,r\right]}$
by setting its value to $0$ there, in such a way that

\begin{itemize}
\item $\Delta_{\circ}\mathbb{I}f_{\delta}\left(z\right)\leq0$, for all
$z\in\mathcal{V}_{\Omega_{\delta}^{*}}\setminus\mathcal{V}_{\left[r,\ell\right]^{*}}$. 
\item $\tilde{\Delta}_{\bullet}\mathbb{I}f_{\delta}\left(z\right)\geq0$,
for all $z\in\mathcal{V}_{\Omega_{\delta}}\setminus\mathcal{V}_{\left[r,\ell\right]}$.
\end{itemize}
\end{itemize}
\end{prop}
\begin{proof}
This follows from \cite[Remark 3.15]{chelkak-smirnov-ii}. The boundary
modification trick that we use is of the same form as in \cite[Section 2.6.1]{hongler-i}.\end{proof}
\begin{rem}
The function $\mathbb{I}f$ can also be extended to $\partial\mathcal{V}_{\left[r,\ell\right]^{*}}$
by the value $1$ but we will not need this here.
\end{rem}
Similarly, we can construct the antiderivative of the square of the
spin observable.
\begin{prop}
\label{prop:antideriv-square-spin-obs}Let $\left(\Omega_{\delta},x\right)$
be a discrete domain and set $f_{\delta}\left(\cdot\right):=\frac{1}{\delta}f_{\delta}^{\mathrm{SPIN}}\left(\Omega_{\delta},x,\cdot\right)$.
Denote by $x_{\mathrm{in}}\in\mathcal{V}_{\Omega_{\delta}}$ and $x_{\mathrm{out}}\in\partial\mathcal{V}_{\Omega_{\delta}}$
the endpoints of the edge of $\partial\mathcal{E}_{\Omega_{\delta}}$
whose midpoint is $x$. There exists a unique function $\mathbb{I}f_{\delta}:\mathcal{V}_{\Omega_{\delta}}\cup\mathcal{V}_{\Omega_{\delta}^{*}}\cup\partial\mathcal{V}_{\Omega_{\delta}^{*}}\to\mathbb{R}$
such that
\begin{itemize}
\item $\mathbb{I}f_{\delta}$ is equal to $0$ on $\partial\mathcal{V}_{\Omega_{\delta}^{*}}$
\item If $e=\left\langle x,y\right\rangle \in\mathcal{E}_{\Omega_{\delta}^{m}}$
is a edge and $b\in\mathcal{V}_{\Omega_{\delta}}$ and $w\in\mathcal{V}_{\Omega_{\delta}^{*}}\cup\partial\mathcal{V}_{\Omega_{\delta}^{*}}$
are such that $\left\langle b,w\right\rangle =e^{*}$, then 
\begin{equation}
\mathbb{I}f_{\delta}\left(b\right)-\mathbb{I}f_{\delta}\left(w\right)=\sqrt{2}\delta\left|\mathsf{P}_{\ell\left(e\right)}\left[f_{\delta}\left(x\right)\right]\right|^{2}=\sqrt{2}\delta\left|\mathsf{P}_{\ell\left(e\right)}\left[f_{\delta}\left(y\right)\right]\right|^{2}.\label{eq:anti-deriv-square-2}
\end{equation}

\item For any $v,w\in\mathcal{V}_{\Omega_{\delta}}$ with $v\sim w$ or
any $v,w\in\mathcal{V}_{\Omega_{\delta}^{*}}\cup\partial\mathcal{V}_{\Omega_{\delta}^{*}}$
with $v\sim w$, we have $\mathbb{I}f_{\delta}\left(v\right)-\mathbb{I}f_{\delta}\left(w\right)=-\Re\mathfrak{e}\left(f_{\delta}^{2}\left(m\right)\left(v-w\right)\right)$,
where $m\in\mathcal{V}_{\Omega_{\delta}^{m}}$ is the midpoint of
$\left\langle v,w\right\rangle $.
\item The function $\mathbb{I}f_{\delta}$ can be extended to $\partial\mathcal{V}_{\Omega_{\delta}}\setminus\left\{ x_{\mathrm{out}}\right\} $,
by the value $0$, in such a way that

\begin{itemize}
\item $\Delta_{\circ}\mathbb{I}f_{\delta}\left(z\right)\leq0$ for all $z\in\mathcal{V}_{\Omega_{\delta}^{*}}$,
\item $\tilde{\Delta}_{\bullet}\mathbb{I}f_{\delta}\left(z\right)\geq0$
for all $z\in\mathcal{V}_{\Omega_{\delta}}\setminus\left\{ x_{\mathrm{in}}\right\} $.
\end{itemize}
\end{itemize}
\end{prop}
\begin{proof}
See \cite[Remark 3.15]{chelkak-smirnov-ii} or \cite[Section 2.6.1]{hongler-i}.\end{proof}
\begin{defn}
\label{def:antideriv-square}Given a discrete domain $\Omega_{\delta}$
and an s-holomorphic function $h_{\delta}:\mathcal{V}_{\Omega_{\delta}^{m}}\to\mathbb{C}$,
we define the discrete antiderivative of $g^{2}$ as the function
on $\mathcal{V}_{\Omega_{\delta}}\cup\mathcal{V}_{\Omega_{\delta}^{*}}\cup\partial\mathcal{V}_{\Omega_{\delta}^{*}}$
obtained by integrating Equation \ref{eq:anti-deriv-square-2} and
denote it by $\mathbb{I}h_{\delta}$\end{defn}
\begin{rem}
It is always possible to integrate the Equation \ref{eq:anti-deriv-square-2}
if $g$ is s-holomorphic, and this defines $\mathbb{I}g$ uniquely,
up to an additive constant.
\end{rem}

\subsubsection{Control of s-holomorphic functions}

Let us give a very useful lemma, introduced in \cite{chelkak-smirnov-ii},
that allows to control the s-holomorphic functions given the integral
of their square.
\begin{prop}
\label{pro:int-ctrl-impl-ctrl}Let $f_{\delta}:\mathcal{V}_{\Omega_{\delta}^{m}}\to\mathbb{C}$
be an s-holomorphic function and let $\mathbb{I}f_{\delta}$ be a
discrete antiderivative of $f_{\delta}^{2}$. Then there exists a
constant $C>0$ such that for $\delta>0$ and any $x\in\mathcal{V}_{\Omega_{\delta}^{m*}}$,
we have
\begin{eqnarray*}
\left|f_{\delta}\left(v\right)\right|^{2} & \leq & C\cdot\frac{\max_{w\in\mathcal{V}_{\Omega_{\delta}^{m*}}}\left|\mathbb{I}\left[f_{\delta}\right]\left(w\right)\right|}{\mathrm{dist}\left(v,\partial_{0}\mathcal{V}_{\Omega_{\delta}^{m}}\right)}\;\forall v\in\mathcal{V}_{\Omega_{\delta}^{m}}.\\
\frac{1}{\delta}\left\Vert \nabla_{\delta}f_{\delta}\left(x\right)\right\Vert ^{2} & \leq & C\cdot\frac{\max_{w\in\mathcal{V}_{\Omega_{\delta}^{m*}}}\left|\mathbb{I}\left[f_{\delta}\right]\left(w\right)\right|}{\mathrm{dist}\left(x,\partial_{0}\mathcal{V}_{\Omega_{\delta}^{m}}\right)^{3}}\;\forall x\in\mathcal{V}_{\Omega_{\delta}}\cup\mathcal{V}_{\Omega_{\delta}^{*}},
\end{eqnarray*}
where 
\[
\nabla_{\delta}f_{\delta}\left(x\right):=\left(f_{\delta}\left(x+\frac{\delta}{2}\right)-f_{\delta}\left(x-\frac{\delta}{2}\right),f_{\delta}\left(x+i\frac{\delta}{2}\right)-f_{\delta}\left(x-\frac{i\delta}{2}\right)\right).
\]
\end{prop}
\begin{proof}
See \cite[Theorem 3.12]{chelkak-smirnov-ii} or \cite[Proposition 27]{hongler-smirnov-i}.
\end{proof}

\subsubsection{Proofs of the lemmas\label{sub:proofs-shol-est-ctrl-lemmas}}

We can now give the proof of the lemmas given at the beginning of
this subsection.
\begin{proof}[Proof of Lemma \ref{lem:fk-and-spin-a-priori-est}]
Set $F_{\delta}^{\mathrm{FK}}:=\mathbb{I}\left[\frac{f_{\delta}^{\mathrm{FK}}\left(\Omega_{\delta},r_{\delta},\ell_{\delta},\cdot\right)}{\sqrt{\delta}}\right]$
and $F_{\delta}^{\mathrm{SPIN}}:=\mathbb{I}\left[\frac{f_{\delta}^{\mathrm{SPIN}}\left(\Omega_{\delta},a_{\delta},\cdot\right)}{\delta}\right]$
(as in Propositions \ref{prop:antideriv-square-fk} and \ref{prop:antideriv-square-spin-obs}). 

By Proposition \ref{prop:antideriv-square-fk} and maximum principle,
we have that $F_{\delta}^{\mathrm{FK}}$ is uniformly bounded by $1$.
By the discrete Beurling estimate \cite[Proposition 2.10]{chelkak-smirnov-i},
we have 
\[
\left|F_{\delta}^{\mathrm{FK}}\left(z_{\delta}\right)\right|\leq C\cdot\mathrm{dist}\left(z_{\delta},\left[\ell_{\delta},r_{\delta}\right]\right)^{2\vartheta}.
\]
for any $z_{\delta}\in\mathcal{V}_{\Omega_{\delta}^{m}}$ with $\mathrm{dist}\left(z_{\delta},\left\{ r_{\delta},\ell_{\delta}\right\} \right)\geq r$.
Proposition \ref{pro:int-ctrl-impl-ctrl} allows to deduce the first
estimate.

Let us first show that $F_{\delta}^{\mathrm{SPIN}}$ is uniformly
bounded for $\left(\Omega_{\delta},p_{\delta}^{1},a_{\delta},p_{\delta}^{2}\right)$
and $z_{\delta}$ satisfying the conditions above. By Theorem \ref{thm:spin-obs-conv}
we have $\frac{1}{\delta}f_{\delta}^{\mathrm{SPIN}}$ is uniformly
bounded on a contour separating $a_{\delta}$ from $z_{\delta}$,
and by Proposition \ref{prop:antideriv-square-spin-obs} $F_{\delta}^{\mathrm{SPIN}}$
is constant on $\partial\mathcal{V}_{\Omega_{\delta}^{*}}\cup\left(\partial\mathcal{V}_{\Omega_{\delta}}\setminus\left\{ a\right\} \right)$.
By subharmonicity and superharmonicity of the respective restrictions
of $F_{\delta}^{\mathrm{SPIN}}$ to $\mathcal{V}_{\Omega_{\delta}^{*}}$
and $\mathcal{V}_{\Omega_{\delta}}\setminus\left\{ a\right\} $, we
obtain the uniform boundedness of $F_{\delta}^{\mathrm{SPIN}}$. Hence,
again by Beurling estimate, $F_{\delta}^{\mathrm{SPIN}}\left(z_{\delta}\right)=O\left(\mathrm{dist}\left(z_{\delta},\partial\Omega_{\delta}\right)^{2\vartheta}\right)$
as $z\to\partial\Omega$. By using Proposition \ref{pro:int-ctrl-impl-ctrl},
we obtain the second estimate.
\end{proof}

\begin{proof}[Proof of Lemma \ref{lem:spin-ratio-apriori-ctrl}]
Note that
\[
\left|\frac{g_{\delta}^{\mathrm{SPIN}}\left(\Omega_{\delta},x_{\delta},\cdot\right)}{g_{\delta}^{\mathrm{SPIN}}\left(\Omega_{\delta},x_{\delta},w_{\delta}\right)}\right|=\left|\frac{f_{\delta}^{\mathrm{SPIN}}\left(\Omega_{\delta},x_{\delta},\cdot\right)}{\left|f_{\delta}^{\mathrm{SPIN}}\left(\Omega_{\delta},x_{\delta},w_{\delta}\right)\right|}\right|,
\]
so we will instead estimate the right-hand side. Suppose without loss
of generality that $\left[p_{\delta}^{1},p_{\delta}^{2}\right]$ is
a vertical part and that $\Omega_{\delta}$ lies on the right of $\left[p_{\delta}^{1},p_{\delta}^{2}\right]$.
Consider the antiderivative $F_{\delta}$ of the square of $\frac{f_{\delta}^{\mathrm{SPIN}}\left(\Omega_{\delta},x_{\delta},\cdot\right)}{\left|f_{\delta}^{\mathrm{SPIN}}\left(\Omega_{\delta},x_{\delta},w_{\delta}\right)\right|}$
with $0$ boundary value, as defined by Proposition \ref{prop:antideriv-square-spin-obs}.
By \cite[proof of Theorem 5.9]{chelkak-smirnov-ii}, we have that
at distance $r/2$ from $x_{\delta}$ $F_{\delta}$ is bounded by
a constant $M=M\left(r,R\right)$. Let $S_{\delta}$ be a square of
sidelength $\mathrm{dist}\left(z_{\delta},\left\{ p_{\delta}^{1},p_{\delta}^{2}\right\} \right)$
such that $z_{\delta}$ lies at the middle of the left side of $S_{\delta}$.
Let $H_{\delta}$ be a harmonic function with respect to a modified
Laplacian (as discussed in \cite{chelkak-smirnov-ii,hongler-i,duminil-copin-hongler-nolin})
on $\overline{\mathcal{V}}_{S_{\delta}}$, with boundary value $0$
on the left side of $S_{\delta}$ and boundary value $M$ on the remaining
three sides of $S_{\delta}$. We have that $\left|F_{\delta}^{\bullet}\right|\leq H_{\delta}$
($F_{\delta}^{\bullet}$ is the restriction of $F_{\delta}$ to $\overline{\mathcal{V}}_{\Omega_{\delta}}$),
by subharmonicity and monotonicity of harmonic measure (as $F_{\delta}^{\bullet}$
takes boundary value $0$ on $\partial\mathcal{V}_{\Omega_{\delta}}$).
Hence, by standard harmonic measure estimates (see \cite[Lemma 3.5]{duminil-copin-hongler-nolin}
for instance), we deduce that there exists an absolute constant $C_{1}>0$
such that 
\[
\left|F_{\delta}^{\bullet}\left(z+\frac{\delta}{2}\right)\right|\leq\frac{C_{1}M}{\mathrm{dist}\left(z_{\delta},\left\{ p_{\delta}^{1},p_{\delta}^{2}\right\} \right)}\delta.
\]
By again using Proposition \ref{pro:int-ctrl-impl-ctrl} and noticing
that
\[
\frac{1}{\sqrt{2}\delta}\left|F_{\delta}^{\bullet}\left(z+\frac{\delta}{2}\right)-0\right|=\left|\frac{f_{\delta}^{\mathrm{SPIN}}\left(\Omega_{\delta},x_{\delta},\cdot\right)}{\left|f_{\delta}^{\mathrm{SPIN}}\left(\Omega_{\delta},x_{\delta},w_{\delta}\right)\right|}\right|^{2},
\]
we obtain that there exists an absolute constant $C_{2}>0$ such that
\[
\left|\frac{f_{\delta}^{\mathrm{SPIN}}\left(\Omega_{\delta},x_{\delta},\cdot\right)}{\left|f_{\delta}^{\mathrm{SPIN}}\left(\Omega_{\delta},x_{\delta},w_{\delta}\right)\right|}\right|\leq C_{2}\sqrt{\frac{M}{\mathrm{dist}\left(z_{\delta},\left\{ p_{\delta}^{1},p_{\delta}^{2}\right\} \right)}},
\]
 which is the desired result.
\end{proof}

\section*{Appendix A: Assumption on the vertical part of the boundary}

In this subsection, we prove Lemma \ref{lem:vert-boundary-assumption},
which is used in Section \ref{sub:identification-limit} to get our
main result.
\begin{lem*}[Lemma \ref{lem:vert-boundary-assumption}]
To prove Theorem \ref{thm:main-thm}, we can assume that the domain
$D$ is such that the arc $\left[b,r\right]$ contains a vertical
part $\mathfrak{v}$ and that the discrete domains $D_{\delta}$ are
such that the arc $\left[b_{\delta},r_{\delta}\right]$ contains a
vertical part $\mathfrak{v}_{\delta}$ converging to $\mathfrak{v}$
as $\delta\to0$. \end{lem*}
\begin{proof}
This follows from the following monotonicity property of the Ising
model: if we
\begin{itemize}
\item take $\left(D_{\delta}^{\left(1\right)},r_{\delta},\ell_{\delta},b_{\delta}\right)$
and $\left(D_{\delta}^{\left(2\right)},r_{\delta},\ell_{\delta},b_{\delta}\right)$
be two discrete domains such that $D_{\delta}^{\left(1\right)}\subset D_{\delta}^{\left(2\right)}$
such that the domains $D_{\delta}^{\left(1\right)}$ and $D_{\delta}^{\left(2\right)}$
share the same boundary arcs $\left[r_{\delta},\ell_{\delta}\right]$
and $\left[r_{\delta},b_{\delta}\right]$ (but not necessarily $\left[b_{\delta},r_{\delta}\right]$),
\item consider the Ising models on $\left(D_{\delta}^{\left(1\right)},r_{\delta},\ell_{\delta},b_{\delta}\right)$
and $\left(D_{\delta}^{\left(2\right)},r_{\delta},\ell_{\delta},x_{\delta}\right)$
with dipolar boundary conditions (free on $\left[r_{\delta},\ell_{\delta}\right]$,
$-$ on $\left[\ell_{\delta},b_{\delta}\right]$ and $+$ on $\left[b_{\delta},r_{\delta}\right]$),
at the same inverse temperature,
\end{itemize}
then we find a coupling $\left(\left(\sigma_{y}^{\left(1\right)}\right)_{y\in D_{\delta}^{\left(1\right)}},\left(\sigma_{y}^{\left(2\right)}\right)_{y\in D_{\delta}^{\left(2\right)}}\right)$
of both Ising models in such a way that 
\begin{equation}
\sigma_{y}^{\left(1\right)}\geq\sigma_{y}^{\left(2\right)}\,\,\,\,\forall y\in D_{\delta}^{\left(1\right)}\label{eq:monotone-coupling}
\end{equation}
 and hence that $\gamma_{\delta}^{\left(1\right)}$, the dipolar interface
on $\left(D_{\delta}^{\left(1\right)},r_{\delta},\ell_{\delta},b_{\delta}\right)$,
is always to the left of $\gamma_{\delta}^{\left(2\right)}$, the
dipolar interface on $\left(D_{\delta}^{\left(2\right)},r_{\delta},\ell_{\delta},b_{\delta}\right)$. 

We can construct a coupling of two Markov chains $\left(\sigma_{y}^{\left(1\right)}\left(t\right)\right)_{y\in D_{\delta}^{\left(1\right)},t\geq0}$,
$\left(\sigma_{y}^{\left(2\right)}\left(t\right)\right)_{y\in D_{\delta}^{\left(2\right)},t\geq0}$
such that $\sigma_{y}^{\left(1\right)}\left(t\right)\geq\sigma_{y}^{\left(2\right)}\left(t\right)$
for any $y\in D_{\delta}^{\left(1\right)}$ and any time $t\geq0$,
and such that the laws of $\left(\sigma_{y}^{\left(1\right)}\left(t\right)\right)_{y\in D_{\delta}^{\left(1\right)}}$
and $\left(\sigma^{\left(2\right)}\left(t\right)\right)_{y\in D_{\delta}^{\left(2\right)}}$
converge, as $t\to\infty$, to the Ising model probability measures
$\left(\sigma_{y}^{\left(1\right)}\right)_{y\in D_{\delta}^{\left(1\right)}}$on
$D_{\delta}^{\left(1\right)}$ and $\left(\sigma_{y}^{\left(2\right)}\right)_{y\in D_{\delta}^{\left(2\right)}}$
$D_{\delta}^{\left(2\right)}$. To construct this Markov chain coupling,
we can use Glauber dynamics, starting from a configuration with all
spins set to $+1$, for instance. 

From the above coupling, we deduce that for any increasing function
$f:\left\{ \pm1\right\} ^{D_{\delta}^{\left(1\right)}}\to\mathbb{R}$,
\begin{eqnarray*}
\mathbb{E}\left[f\left(\left(\sigma_{y}^{\left(1\right)}\right)_{y\in D_{\delta}^{\left(1\right)}}\right)\right] & = & \lim_{t\to\infty}\mathbb{E}\left[f\left(\left(\sigma_{y}^{\left(1\right)}\left(t\right)\right)_{y\in D_{\delta}^{\left(1\right)}}\right)\right]\\
 & \geq & \lim_{t\to\infty}\mathbb{E}\left[f\left(\left(\sigma_{y}^{\left(2\right)}\left(t\right)\right)_{y\in D_{\delta}^{\left(1\right)}}\right)\right]\\
 & = & \mathbb{E}\left[f\left(\left(\sigma_{y}^{\left(2\right)}\right)_{y\in D_{\delta}^{\left(1\right)}}\right)\right].
\end{eqnarray*}
In other words, $\left(\sigma_{y}^{\left(1\right)}\right)_{y\in D_{\delta}^{\left(1\right)}}$
dominates $\left(\sigma_{y}^{\left(2\right)}\right)_{y\in D_{\delta}^{\left(1\right)}}$
stochastically. By Strassen's theorem \cite{strassen}, this is equivalent
to the existence of a coupling satisfying the inequality \ref{eq:monotone-coupling}
above.

With this monotonicity property, we can now approximate from inside
and outside the domain $D$ (and its discretizations) by domains having
a vertical part on $\left[b,r\right]$, we obtain convergence to dipolar
SLE$\left(3\right)$ on those domains. The desired result follows
readily (the interface can be squeezed between two interfaces whose
limit is dipolar SLE$\left(3\right)$, and these two interfaces are
arbitrarily close to each other. 
\end{proof}

\section*{Appendix B: SLE$\left(\kappa;\rho\right)$ lemmas}

In this subsection, we prove the characterization of SLE$\left(16/3;-8/3\right)$
provided by Lemma \ref{lem:sle-rn-deriv}. Let us first give the following
general characterization of the SLE$\left(\kappa;\rho\right)$ processes.
\begin{lem}
\label{lem:sle-girsanov}Let $\gamma$ be a chordal SLE$\left(\kappa\right)$
in $\mathbb{H}$ from $0$ to $\infty$, and $\tilde{\gamma}$ be
an SLE$\left(\kappa;\rho\right)$ in $\mathbb{H}$ starting from $0$
with force point $x>0$ and observation point $\infty$. For $\epsilon>0$
and $\varrho>0$, let $\vartheta^{\epsilon}$ (respectively $\tilde{\vartheta}^{\epsilon}$)
be the first time that $\gamma$ (respectively $\tilde{\gamma}$)
hits 
\[
\left\{ z:\left|z\right|\geq\frac{1}{\epsilon}\right\} \cup\left\{ z:\mathrm{dist}\left(z,[x,\infty)\right)\leq\epsilon\right\} .
\]
Let $\mathcal{P}^{\epsilon}$ be the law of $\gamma\left[0,\vartheta^{\epsilon}\right]$
and $\tilde{\mathcal{P}}^{\epsilon}$ be the law of $\tilde{\gamma}\left[0,\tilde{\vartheta}^{\epsilon}\right]$. 

Then $\mathcal{P}^{\epsilon}$ and $\tilde{\mathcal{P}}^{\epsilon}$
are absolutely continuous with respect to each other, and for any
$\mu\left[0,t\right]$ in their support, we have the follow expression
for the Radon-Nikodym derivative:
\begin{equation}
\frac{\mathrm{d}\tilde{\mathcal{P}}^{\epsilon}}{\mathrm{d}\mathcal{P}^{\epsilon}}\left(\mu\left[0,t\right]\right)=G_{t}'\left(x\right)^{h}\left(\frac{G_{t}\left(x\right)}{x}\right)^{\rho/\kappa},\label{eq:rn-deriv-slekapparho}
\end{equation}
where $h=\frac{\rho\left(\rho+4-\kappa\right)}{4\kappa}$ and $G_{t}$
is the conformal mapping from (the unbounded component of) $\mathbb{H}\setminus\mu\left[0,t\right]$
to $\mathbb{H}$, normalized such that $G_{t}\left(z\right)\sim z$
as $z\to\infty$ and $G_{t}\left(\mu\left(t\right)\right)=0$.\end{lem}
\begin{rem}
\label{rem:identification-rn-deriv}When $\kappa=16/3$ and $\rho=-8/3$,
by a straightforward computation, we obtain that the right-hand side
of Equation \ref{eq:rn-deriv-slekapparho} can be expressed as 
\[
G_{t}'\left(x\right)^{h}\left(\frac{G_{t}\left(x\right)}{x}\right)^{\rho/\kappa}=\frac{\left\langle \sigma_{x}\right\rangle _{H_{t}}^{\left[-\infty,\mu\left(t\right)\right]_{+}}}{\left\langle \sigma_{x}\right\rangle _{\mathbb{H}}^{\left[-\infty,0\right]_{+}}},
\]
where $H_{t}$ is the unbounded connected component of $\mathbb{H}\setminus\mu\left[0,t\right]$. \end{rem}
\begin{proof}[Proof of Lemma \ref{lem:sle-girsanov}]
This result is a consequence of Girsanov's theorem. Proofs of similar
statements can be found in \cite[Section 3]{werner-ii} or in \cite[Section 1.2.4]{kytola-ii}.
\end{proof}
We can now prove Lemma \ref{lem:sle-rn-deriv}. For definiteness,
we take the time parametrization inherited from the time parametrization
in the half-plane via the conformal mapping $\varphi:\left(\mathbb{H},0,x,\infty\right)\to\left(\Omega,r,\ell,x\right)$.
Also recall that we denote by $\mathbf{D}_{\Omega}\left(x,\epsilon\right)$
the connected component of $\left\{ z\in\overline{\Omega}:\left|z-x\right|\leq\epsilon\right\} $
that contains $x$. Set also $\mathbf{B}_{\Omega}\left(x,r,\zeta\right):=\bigcup_{z\in\left[x,r\right]}\mathbf{D}_{\Omega}\left(z,\zeta\right)$.

\begin{figure}
\includegraphics[width=9cm]{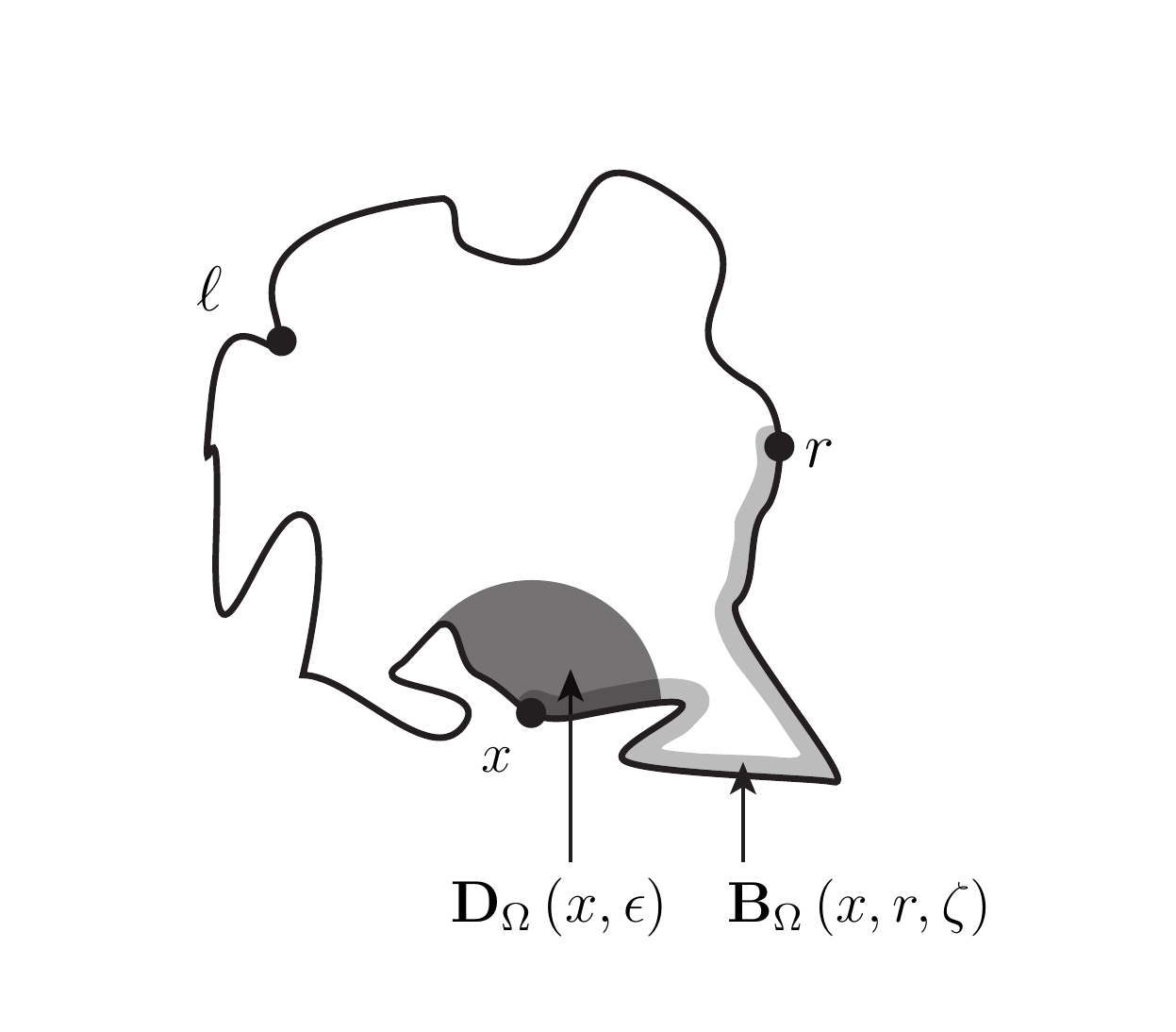}

\caption{The domain $\left(\Omega,r,\ell,x\right)$ and the neighborhoods $\mathbf{D}_{\Omega}\left(x,\epsilon\right)$
and $\mathbf{B}_{\Omega}\left(x,r,\zeta\right)$.}
\end{figure}

\begin{lem*}[Lemma \ref{lem:sle-rn-deriv}]
Consider the domain $\left(\Omega,r,\ell,x\right)$. Let $\lambda$
have the law of an SLE$\left(16/3\right)$ curve in $\Omega$ from
$\ell$ to $r$ and let $\tilde{\lambda}$ have the law of an SLE$\left(16/3;-8/3\right)$
curve in $\Omega$ with starting point $\ell$, observation point
$r$ and force point $x$. 

For $\epsilon>0$, let $\tau^{\epsilon}$ be the (possibly infinite)
first time that $\lambda$ hits $\mathbf{D}_{\Omega}\left(x,\epsilon\right)$
and let $\tilde{\tau}^{\epsilon}$ be the first time that $\tilde{\lambda}$
hits $\mathbf{D}_{\Omega}\left(x,\epsilon\right)$. Let $\mathbb{P}^{\epsilon}$
be the law of $\lambda\left[0,\tau^{\epsilon}\right]$ and $\tilde{\mathbb{P}}^{\epsilon}$
be the law of $\tilde{\lambda}\left[0,\tilde{\tau}^{\epsilon}\right]$.
Then we have
\[
\mathrm{Supp}\left(\tilde{\mathbb{P}}^{\epsilon}\right)=\left\{ \mu\in\mathrm{Supp}\left(\mathbb{P}^{\epsilon}\right):\mu\mbox{ hits }\mathbf{D}_{\Omega}\left(x,\epsilon\right)\mbox{ in finite time and }\mu\cap\left[x,r\right]=\emptyset\right\} 
\]
and for any curve $\mu\left[0,t\right]\in\mathrm{Supp}\left(\tilde{\mathbb{P}}^{\epsilon}\right)$
\[
\frac{\mathrm{d}\tilde{\mathbb{P}}^{\epsilon}}{\mathrm{d}\mathbb{P}^{\epsilon}}\left(\mu\left[0,t\right]\right)=\frac{\left\langle \sigma_{x}\right\rangle _{\Omega\setminus\mu\left[0,t\right]}^{\left[r,\mu\left(t\right)\right]_{+}}}{\left\langle \sigma_{x}\right\rangle _{\Omega}^{\left[r,\ell\right]_{+}}}.
\]
\end{lem*}
\begin{proof}
For $\epsilon>0$ and $\zeta>0$, let us denote by $\mathbb{P}^{\epsilon,\zeta}$
(respectively $\tilde{\mathbb{P}}^{\epsilon,\zeta}$) the law of $\lambda$
(respectively of $\tilde{\lambda}$) stopped as it hits 
\[
\mathbf{D}_{\Omega}\left(x,\epsilon\right)\cup\mathbf{B}_{\Omega}\left(\zeta\right),
\]

It is the enough to show that
\begin{itemize}
\item $\mathbb{P}^{\epsilon,\zeta}$ and $\tilde{\mathbb{P}}^{\epsilon,\zeta}$
are absolutely continuous with respect to each other and for any $\mu\left[0,t\right]$
in their support, 
\begin{equation}
\frac{\mathrm{d}\tilde{\mathbb{P}}^{\epsilon,\zeta}}{\mathrm{d}\mathbb{P}^{\epsilon,\zeta}}\left(\mu\left[0,t\right]\right)=\frac{\left\langle \sigma_{x}\right\rangle _{\Omega\setminus\mu\left[0,t\right]}^{\left[r,\mu\left(t\right)\right]_{+}}}{\left\langle \sigma_{x}\right\rangle _{\Omega}^{\left[r,\ell\right]_{+}}}.\label{eq:rn-deriv-formula}
\end{equation}
By conformal invariance of SLE and of the right-hand side of this
formula, it is enough to show this on the half-plane and the result
follows from Lemma \ref{lem:sle-girsanov} and Remark \ref{rem:identification-rn-deriv}
above.
\item As $\zeta\to0$, $\tilde{\mathbb{P}}^{\epsilon,\zeta}\to\tilde{\mathbb{P}}^{\epsilon}$:
in other words, $\tilde{\lambda}\left[0,\tilde{\tau}^{\epsilon}\right]$
almost surely does not hit $\left[x,r\right]$. This gives
\[
\mathrm{Supp}\left(\tilde{\mathbb{P}}^{\epsilon}\right)\subset\left\{ \mu\in\mathrm{Supp}\left(\mathbb{P}^{\epsilon}\right):\mu\mbox{ hits }\mathbf{D}_{\Omega}\left(x,\epsilon\right)\mbox{ and }\mu\cap\left[x,r\right]=\emptyset\right\} .
\]
The inclusion in the other direction immediately follows from Lemma
\ref{lem:sle-girsanov}.
\end{itemize}
To conclude the proof it hence remains to show that for any fixed
$\epsilon>0$, we have 
\begin{equation}
\tilde{\mathbb{P}}^{\epsilon,\zeta}\left\{ \tilde{\lambda}\left[0,\tilde{\tau}^{\epsilon,\zeta}\right]\cap\mathbf{B}_{\Omega}\left(x,r,\zeta\right)\neq\emptyset\right\} \underset{\zeta\to0}{\longrightarrow}0\label{eq:cond-sle-doesnt-hit-bdry}
\end{equation}
By conformal invariance, it is sufficient to prove this when $\Omega$
is the rectangle 
\[
\left\{ z\in\mathbb{C}:\Re\mathfrak{e}\left(z\right)\in\left(-1,1\right),\Im\mathfrak{m}\left(z\right)\in\left(0,1\right)\right\} ,
\]
with $\ell=-1$, $x=0$ and $r=1$. For $\mu\left[0,t\right]\in\mathrm{Supp}\left(\tilde{\mathbb{P}}^{\epsilon,\zeta}\right)$,
by Equation \ref{eq:rn-deriv-formula} and Remark \ref{rem:bdary-magnet-norm-deriv},
we have
\[
\frac{\mathrm{d}\tilde{\mathbb{P}}^{\epsilon,\zeta}}{\mathrm{d}\mathbb{P}^{\epsilon,\zeta}}\left(\mu\left[0,t\right]\right)=\frac{\left\langle \sigma_{x}\right\rangle _{\Omega\setminus\mu\left[0,t\right]}^{\left[r,\mu\left(t\right)\right]_{+}}}{\left\langle \sigma_{x}\right\rangle _{\Omega}^{\left[r,\ell\right]_{+}}}=\mathrm{Cst}\cdot\sqrt{\frac{\partial}{\partial\nu_{\mathrm{int}}\left(x\right)}\mathbf{H}_{\Omega\setminus\mu\left[0,t\right]}\left(\cdot,\left[r,\mu\left(t\right)\right]\right)},
\]
where $\mathbf{H}_{\Omega\setminus\mu\left[0,t\right]}\left(z,\left[r,\mu\left(t\right)\right]\right)$
is the harmonic measure of $\left[r,\mu\left(t\right)\right]$ in
$\Omega\setminus\mu\left[0,t\right]$ viewed from $z$ (see Figure
\ref{fig:harmonic-measure}) and $\frac{\partial}{\partial\nu_{\mathrm{int}}\left(x\right)}$
is the inward normal derivative at $x$.

By standard harmonic measure estimates, we have that if $\mu\left(t\right)\in\mathbf{B}_{\Omega}\left(x,r,\zeta\right)$,
then the right-hand side is bounded by $\mathrm{Cst}\cdot\zeta^{\alpha}$
for some $\alpha>0$, uniformly with respect to $\mu\left[0,t\right]$.
From there, we immediately deduce \ref{eq:cond-sle-doesnt-hit-bdry}.
\end{proof}
\begin{figure}
\includegraphics[width=9cm]{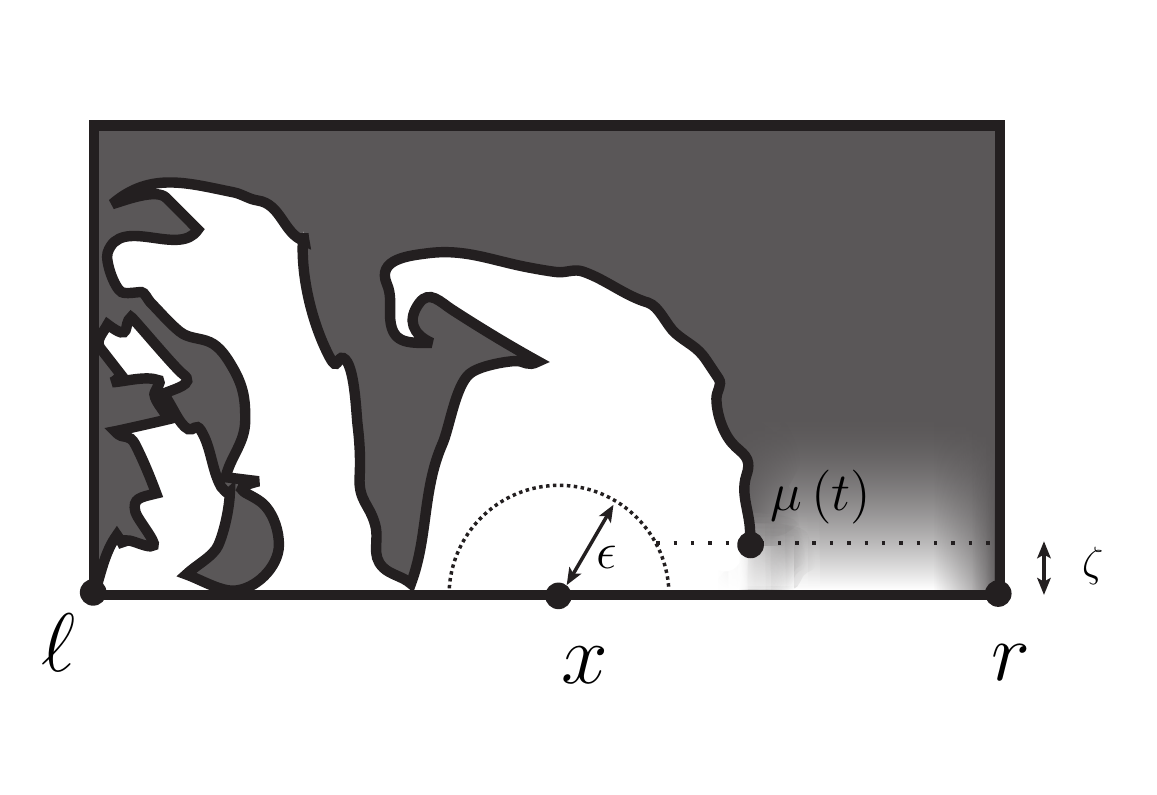}

\caption{\label{fig:harmonic-measure}Harmonic measure $\mathbf{H}_{\Omega\setminus\mu\left[0,t\right]}\left(\cdot,\left[r,\mu\left(t\right)\right]\right)$
(dark gray is $1$, white is $0$).}
\end{figure}

\section*{Appendix C: Control of the end of the conditioned interface}

The goal of this subsection is to prove Lemma \ref{lem:control-end-conditioned-interface},
which is used to conclude the proof of Theorem \ref{thm:cond-fk-int-to-sle-kr}.
Recall that $\tilde{\lambda}_{\delta}$ has the law of a critical
FK-Ising interface in $\left(\Omega_{\delta},r_{\delta},\ell_{\delta},x_{\delta}\right)$,
from $\ell_{\delta}$ to $r_{\delta}$, conditioned to pass at $x_{\delta}$,
stopped as it hits $x_{\delta}$ and that $\tilde{\lambda}$ has the
law of an SLE$\left(16/3;-8/3\right)$ in $\Omega$ with starting
point $\ell$, observation point $r$ and force point $x$. 

For $\epsilon>0$, we defined $\mathbf{D}_{\Omega_{\delta}}\left(x_{\delta},\epsilon\right)$
as the connected component of $\left\{ z\in\Omega_{\delta}:\left|z-x\right|\leq\epsilon\right\} $
that contains $x_{\delta}$ and $\mathbf{D}_{\Omega}\left(x,\epsilon\right)$
as the connected component of $\left\{ z\in\overline{\Omega}:\left|z-x\right|\leq\epsilon\right\} $
containing $x$. For $\rho>\epsilon>0$ let us also define $\mathbf{A}_{\Omega_{\delta}}\left(x_{\delta},\epsilon,\rho\right)$
as $\mathbf{D}_{\Omega_{\delta}}\left(x_{\delta},\rho\right)\setminus\mathbf{D}_{\Omega_{\delta}}\left(x_{\delta},\epsilon\right).$
\begin{lem*}[Lemma \ref{lem:control-end-conditioned-interface}]
For any $\rho>0$, the probability that $\tilde{\lambda}_{\delta}$
exits $\mathbf{D}_{\Omega_{\delta}}\left(x_{\delta},\rho\right)$
after the time $\tilde{\tau}_{\delta}^{\epsilon}$ tends to $0$ as
$\epsilon\to0$, uniformly with respect to $\left(\Omega_{\delta},r_{\delta},\ell_{\delta},x_{\delta}\right)$.
Similarly, for any $\rho>0$, the probability that $\tilde{\lambda}$
exits $\mathbf{D}_{\Omega}\left(x,\rho\right)$ after time $\tilde{\tau}^{\epsilon}$
tends to $0$ as $\epsilon\to0$, uniformly with respect to $\left(\Omega,r,\ell,x\right)$. 
\end{lem*}
To prove these lemmas, we will use the following property of FK interfaces
conditioned to pass at a boundary point (see Figure \ref{fig:fk-mutual-conditioning}).
\begin{lem}
\label{lem:conditioning-fk-interfaces}Let $\left(\Omega_{\delta},r_{\delta},\ell_{\delta},x_{\delta}\right)$
be a discrete domain and let $\lambda_{\delta}^{*}$ have the law
of an FK interface from $\ell_{\delta}$ to $r_{\delta}$, conditioned
to pass at $x_{\delta}$. Let $\lambda_{\delta}^{\ell}$ denote the
part of $\lambda_{\delta}^{*}$ from $\ell_{\delta}$ to $x_{\delta}$
and let $\lambda_{\delta}^{r}$ denote the part of $\lambda_{\delta}^{*}$
from $x_{\delta}$ to $r_{\delta}$. Then, conditionally on $\lambda_{\delta}^{r}$,
$\lambda_{\delta}^{\ell}$ has the law of an (unconditioned) FK interface
in $\left(\Omega_{\delta}^{\ell},\ell_{\delta},x_{\delta}\right)$
from $\ell_{\delta}$ to $x_{\delta}$, where $\Omega_{\delta}^{\ell}$
is the connected component of $\Omega_{\delta}\setminus\lambda_{\delta}^{r}$
containing $\ell_{\delta}$. The law of $\lambda_{\delta}^{r}$ conditionally
on $\lambda_{\delta}^{\ell}$ is described symmetrically (echanging
$r$ and $\ell$).\end{lem}
\begin{proof}
This is a direct consequence of the domain Markov property of the
FK model.
\end{proof}
\begin{figure}
\includegraphics[width=13cm]{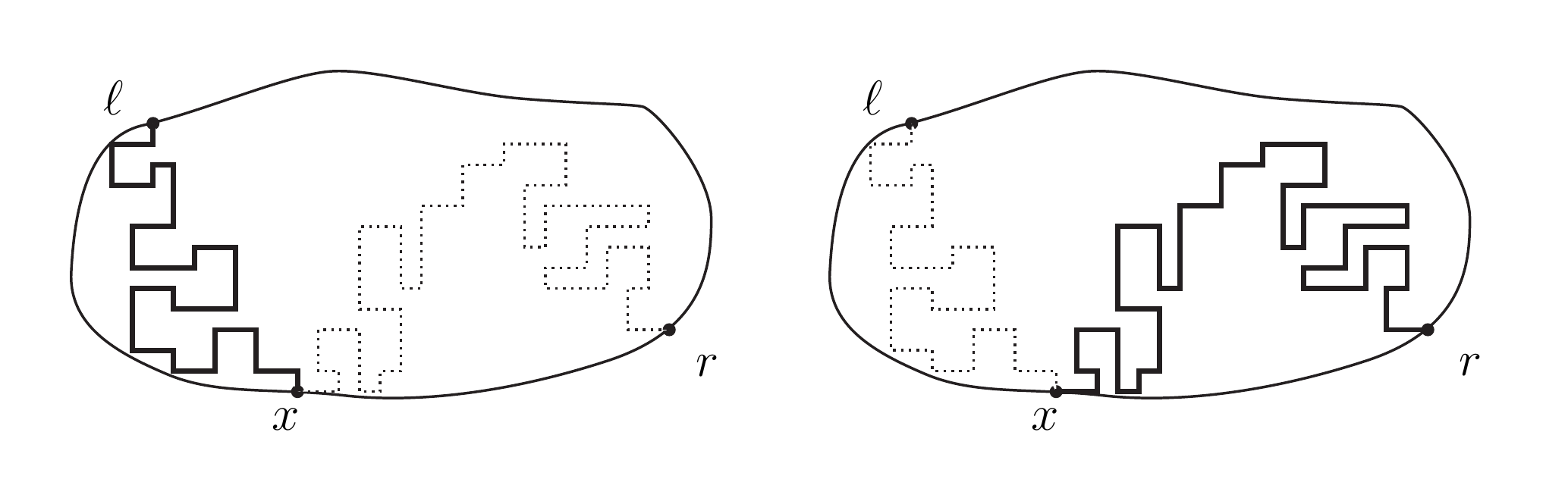}

\caption{\label{fig:fk-mutual-conditioning}Law of $\lambda_{\delta}^{\ell}$
knowing $\lambda_{\delta}^{r}$ and conversely.}
\end{figure}

We will also use the following uniform lemma concerning the behavior
of (unconditioned) FK interface near their endpoints:
\begin{lem}
\label{lem:fk-int-unif-bound-target}Let $\left(\Omega_{\delta},r_{\delta},\ell_{\delta}\right)$
be a discrete domain and let $\lambda_{\delta}$ have the law of a
critical FK-Ising interface from $\ell_{\delta}$ to $r_{\delta}$.
Then for any $\rho>0$, the probability that $\lambda_{\delta}$ exits
$\mathbf{D}_{\Omega_{\delta}}\left(r_{\delta},\rho\right)$ after
entering $\mathbf{D}_{\Omega_{\delta}}\left(r_{\delta},\epsilon\right)$
is bounded by $C\left(\epsilon/\rho\right)^{\vartheta}$, for some
universal constants $C,\vartheta>0$. \end{lem}
\begin{proof}
The result follows from standard techniques, using RSW crossing type
estimates, as explained in \cite[Section 3.3]{kemppainen-smirnov}.
The RSW estimates are given by \cite{chelkak-smirnov-ii} or \cite{duminil-copin-hongler-nolin}.
The idea is to show that once $\lambda$ has crossed an annulus, with
uniformly positive probability, one can guarantee that $\lambda$
will not cross this annulus anymore, and to do this for a family of
concentric annuli (see Figure \ref{fig:ctrl-fk-int-endpoint}).
\end{proof}
\begin{figure}
\includegraphics[width=8cm]{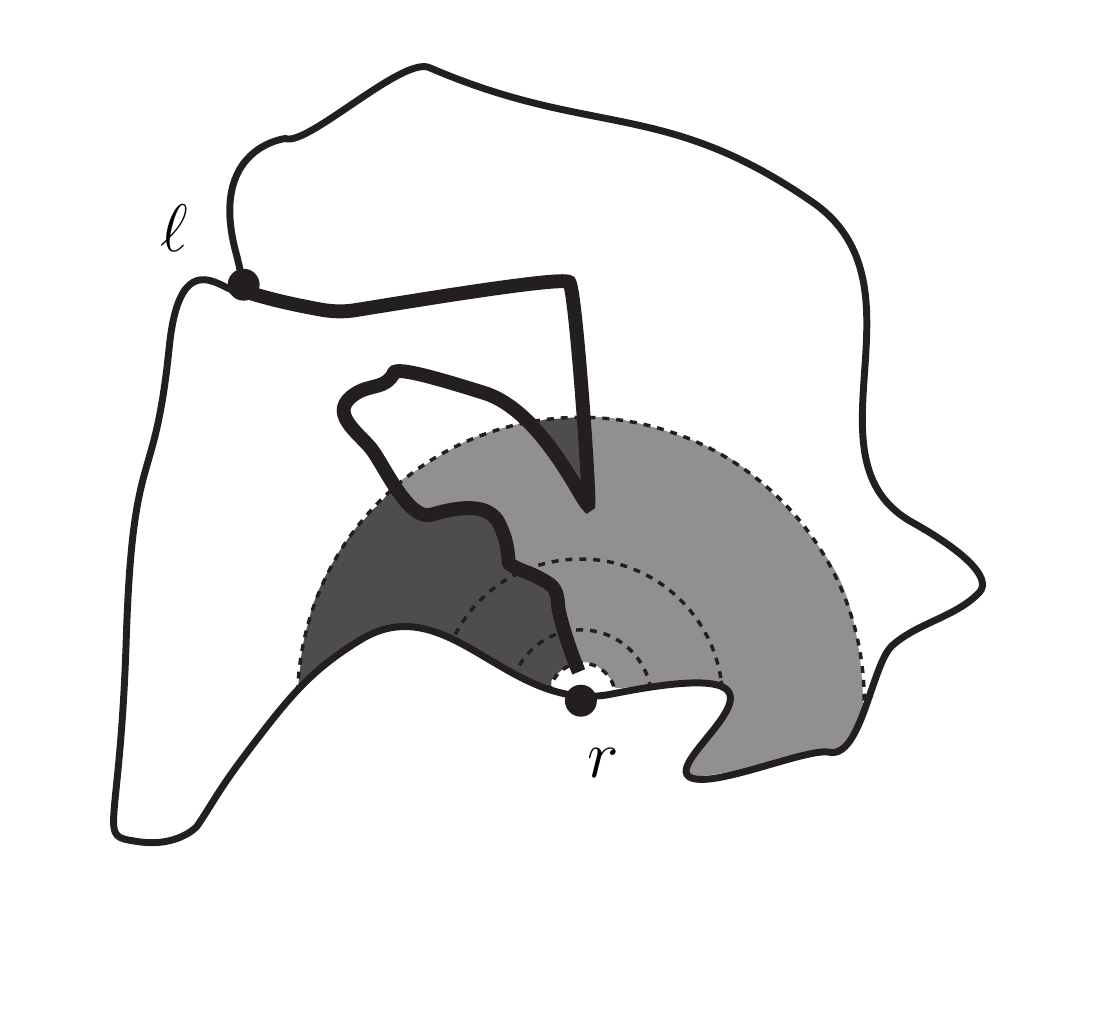}

\caption{\label{fig:ctrl-fk-int-endpoint}Nested family of annuli near the
target point.}
\end{figure}

The third ingredient we will need is the following monotonicity lemma:
\begin{lem}
\label{lem:fk-monotonicity}Let $\left(\Omega_{\delta}^{\left(1\right)},r_{\delta},\ell_{\delta}\right)$
and $\left(\Omega_{\delta}^{\left(2\right)},r_{\delta},\ell_{\delta}\right)$
be two discrete domains such that $\Omega_{\delta}^{\left(2\right)}\subset\Omega_{\delta}^{\left(1\right)}$
and such that the arcs $\left[\ell_{\delta},r_{\delta}\right]$ in
$\Omega_{\delta}^{\left(1\right)}$ and in $\Omega_{\delta}^{\left(2\right)}$
are the same. Consider the FK-Ising interfaces $\lambda_{\delta}^{\left(1\right)}$
in $\left(\Omega_{\delta}^{\left(1\right)},r_{\delta},\ell_{\delta}\right)$
and $\lambda_{\delta}^{\left(2\right)}$ in $\left(\Omega_{\delta}^{\left(2\right)},r_{\delta},\ell_{\delta}\right)$,
both oriented from $\ell_{\delta}$ to $r_{\delta}$. Then there is
a coupling of $\lambda_{\delta}^{\left(1\right)}$ and $\lambda_{\delta}^{\left(2\right)}$
such that $\lambda_{\delta}^{\left(2\right)}$ is always to the right
of $\lambda_{\delta}^{\left(1\right)}$ (i.e. $\lambda_{\delta}^{\left(2\right)}$
separates $\lambda_{\delta}^{\left(1\right)}$ from $\left[\ell_{\delta},r_{\delta}\right]$
in $\Omega_{\delta}^{\left(1\right)}$). \end{lem}
\begin{proof}
This follows from the strong positive association property of the
FK model (see \cite{grimmett}).
\end{proof}
Using the three lemmas we can now prove the desired result:
\begin{proof}[Proof of Lemma \ref{lem:control-end-conditioned-interface}]
Let us first prove the discrete statement. Denote by $\lambda_{\delta}^{*}$
the full interface from $\ell_{\delta}$ to $r_{\delta}$, conditioned
to pass at $x_{\delta}$; denote by $\lambda_{\delta}^{\ell}$ the
part of $\lambda_{\delta}^{*}$ form $\ell_{\delta}$ to $x_{\delta}$
(hence $\lambda_{\delta}^{\ell}$ is the same as $\tilde{\lambda}_{\delta}$
in the statement) and by $\lambda_{\delta}^{r}$ the part of $\lambda_{\delta}^{*}$
from $x_{\delta}$ to $r_{\delta}$. 

If $\lambda_{\delta}^{\ell}$ exits $\mathbf{D}_{\Omega_{\delta}}\left(x_{\delta},\rho\right)$
after the time $\tilde{\tau}_{\delta}^{\epsilon}$, it either exits
on the left or on the right of $\lambda_{\delta}^{\ell}\left[0,\tilde{\tau}_{\delta}^{\epsilon}\right]$,
i.e. in the connected component of $\mathbf{A}_{\Omega_{\delta}}\left(x_{\delta},\epsilon,\rho\right)\setminus\lambda_{\delta}^{\ell}\left[0,\tilde{\tau}_{\delta}^{\epsilon}\right]$
that is bounded by the left side (respectively the right side) of
$\lambda_{\delta}^{\ell}$ (see Figures \ref{fig:fk-left-exit} and
\ref{fig:fk-right-exit}). Let us decompose along these two cases:
\begin{itemize}
\item The probability of the event $\mathfrak{L}$ that $\lambda_{\delta}^{\ell}$
exits $\mathbf{D}_{\Omega_{\delta}}\left(x_{\delta},\rho\right)$
on the left of $\lambda_{\delta}^{\ell}\left[0,\tilde{\tau}_{\delta}^{\epsilon}\right]$
can be bounded as follows (see Figure \ref{fig:fk-left-exit}). 

\begin{itemize}
\item If we condition on (the whole path) $\lambda_{\delta}^{r}$ and on
$\lambda_{\delta}\left[0,\tilde{\tau}_{\delta}^{\epsilon}\right]$,
then by Lemma \ref{lem:conditioning-fk-interfaces}, the remaining
part of $\lambda_{\delta}^{\ell}$ has the same law as an unconditioned
FK interface $\lambda_{\delta}^{\left(2\right)}$ in $\left(\Omega_{\delta}^{\dagger},\lambda_{\delta}\left(\tilde{\tau}_{\delta}^{\epsilon}\right),x_{\delta}\right)$
from $\lambda_{\delta}\left(\tilde{\tau}_{\delta}^{\epsilon}\right)$
to $x_{\delta}$, where $\Omega_{\delta}^{\dagger}$ is the connected
component of $\Omega_{\delta}\setminus\left(\lambda_{\delta}\left[0,\tilde{\tau}_{\delta}^{\epsilon}\right]\cup\lambda_{\delta}^{r}\right)$
containing $\lambda_{\delta}\left(\tilde{\tau}_{\delta}^{\epsilon}\right)$. 
\item By Lemma \ref{lem:fk-monotonicity}, the interface $\lambda_{\delta}^{\left(2\right)}$
can be coupled to always be on the right of an unconditioned FK interface
$\lambda_{\delta}^{\left(1\right)}$ in $\left(\Omega_{\delta}\setminus\lambda_{\delta}\left[0,\tilde{\tau}_{\delta}^{\epsilon}\right],\lambda_{\delta}\left(\tilde{\tau}_{\delta}^{\epsilon}\right),x_{\delta}\right)$,
oriented from $\lambda_{\delta}\left(\tilde{\tau}_{\delta}^{\epsilon}\right)$
to $x_{\delta}$. Hence, the probability that $\lambda_{\delta}^{\left(2\right)}$
exits $\mathbf{D}_{\Omega_{\delta}}\left(x_{\delta},\rho\right)$
on the left of $\lambda_{\delta}^{\ell}\left[0,\tilde{\tau}_{\delta}^{\epsilon}\right]$
is smaller than the probability that $\lambda_{\delta}^{\left(1\right)}$
exits $\mathbf{D}_{\Omega_{\delta}}\left(x_{\delta},\rho\right)$
(through the coupling, the first event implies the second one). 
\item By Lemma \ref{lem:fk-int-unif-bound-target}, the probability that
$\lambda_{\delta}^{\left(1\right)}$ exits $\mathbf{D}_{\Omega_{\delta}}\left(x_{\delta},\rho\right)$
is bounded by $C\left(\epsilon/\rho\right)^{\vartheta}$, where $C$
and $\vartheta$ are universal positive constants. 
\end{itemize}
\item The probability of the event $\mathfrak{R}$ that $\lambda_{\delta}^{\ell}$
exits $\mathbf{D}_{\Omega_{\delta}}\left(x_{\delta},\rho\right)$
on the right of $\lambda_{\delta}^{\ell}\left[0,\tilde{\tau}_{\delta}^{\epsilon}\right]$
can be bounded in the following way (see Figure \ref{fig:fk-right-exit}). 

\begin{itemize}
\item For $\phi>\epsilon$, let $E_{\delta}^{\phi}$ be the event that $\lambda_{\delta}^{r}$\emph{
}exits $\mathbf{D}_{\Omega_{\delta}}\left(x_{\delta},\phi\right)$
to the right of $\lambda_{\delta}^{r}$ after entering $\mathbf{D}_{\Omega_{\delta}}\left(x_{\delta},\epsilon\right)$
(we parametrize $\lambda_{\delta}^{r}$ from $r$ to $x$). By the
bound of the previous paragraph, applied to $\lambda_{\delta}^{r}$,
the probability of $E_{\delta}^{\phi}$ is bounded by $C\left(\frac{\phi}{\rho}\right)^{\vartheta}$. 
\item By Lemma \ref{lem:conditioning-fk-interfaces}, for $0<\epsilon<\phi<\rho$,
conditionally on $\lambda_{\delta}^{r}$ and on the event that $E^{\phi}$
\emph{does not occur}, the probability that $\lambda_{\delta}^{\ell}$
exits $\mathbf{D}_{\Omega_{\delta}}\left(x_{\delta},\rho\right)$
after time $\tilde{\tau}_{\delta}^{\epsilon}$ is bounded by the probability
that an unconditioned FK interface (in $\Omega_{\delta}^{\ell}$,
from $\ell_{\delta}$ to $x_{\delta}$) exits $\mathbf{D}_{\Omega_{\delta}^{\ell}}\left(x_{\delta},\rho\right)$
after entering $\mathbf{D}_{\Omega_{\delta}^{r}}\left(x_{\delta},2\phi\right)$.
By Lemma \ref{lem:fk-int-unif-bound-target}, this probability is
bounded by $C\left(\frac{2\phi}{\rho}\right)^{\vartheta}$. 
\item Writing $\mathsf{E}_{k}$ for $E_{\delta}^{2^{k}\epsilon}$, and summing
over successive scales, we get:
\begin{eqnarray*}
\mathbb{P}\left(\mathfrak{R}\right) & = & \sum_{k=0}^{\lfloor\log_{2}\left(\rho/\epsilon\right)\rfloor}\mathbb{P}\left\{ \mathfrak{R}|\mathsf{E}_{k}\setminus\mathsf{E}_{k+1}\right\} \mathbb{P}\left\{ \mathsf{E}_{k}\setminus\mathsf{E}_{k+1}\right\} \\
 & \leq & \sum_{k=0}^{\lfloor\log_{2}\left(\rho/\epsilon\right)\rfloor}\mathbb{P}\left\{ \mathfrak{R}|\mathsf{E}_{k}\setminus\mathsf{E}_{k+1}\right\} \mathbb{P}\left\{ \mathsf{E}_{k}\right\} \\
 & \leq & \sum_{k=0}^{\lfloor\log_{2}\left(\rho/\epsilon\right)\rfloor}C^{2}\left(\frac{2^{k-1}\epsilon}{\rho}\right)^{\vartheta}\left(\frac{\epsilon}{2^{k}\epsilon}\right)^{\vartheta}\\
 & \leq & \tilde{C}\left(\frac{\epsilon}{\rho}\right)^{\vartheta}\left(1+\log\left(\frac{\rho}{\epsilon}\right)\right)
\end{eqnarray*}

\end{itemize}
\end{itemize}
Given the two uniform bounds for the probabilities of $\mathfrak{L}$
and $\mathfrak{R}$ above, we obtain the desired result for the conditioned
FK interface $\tilde{\lambda}_{\delta}$. 

For the SLE$\left(16/3;-8/3\right)$ curve $\tilde{\lambda}$, we
get the same uniform bound as for the FK interface. For $\epsilon,\rho>0$,
let us denote by $\mathbf{T}_{\delta}\left(\epsilon,\rho\right)$
(respectively $\mathbf{T}\left(\epsilon,\rho\right)$) the possibly
infinite first time when $\tilde{\lambda}_{\delta}$ (respectively
$\tilde{\lambda}$) exits $\mathbf{D}_{\Omega_{\delta}}\left(x_{\delta},\rho\right)$
(respectively $\mathbf{D}_{\Omega}\left(x,\rho\right)$) after time
$\tilde{\tau}_{\delta}^{\epsilon}$ (respectively $\tilde{\tau}^{\epsilon}$).
We have
\begin{eqnarray*}
\mathbb{P}\left\{ \tilde{\tau}^{\epsilon}<\mathbf{T}\left(\epsilon,\rho\right)<\infty\right\}  & = & \mathbb{P}\left(\bigcup_{\alpha\in\left(0,\epsilon\right)}\left\{ \tilde{\tau}^{\epsilon}<\mathbf{T}\left(\epsilon,\rho\right)<\tilde{\tau}^{\alpha}\right\} \right)\\
 & = & \sup_{\alpha\in\left(0,\epsilon\right)}\mathbb{P}\left\{ \tilde{\tau}^{\epsilon}<\mathbf{T}\left(\epsilon,\rho\right)<\tilde{\tau}^{\alpha}\right\} .
\end{eqnarray*}
But for any $\alpha\in\left(0,\epsilon\right)$, as $\tilde{\lambda}\left[0,\tilde{\tau}^{\alpha}\right]$
is the scaling limit of $\tilde{\lambda}_{\delta}\left[0,\tilde{\tau}_{\delta}^{\alpha}\right]$,
we have
\begin{eqnarray*}
\mathbb{P}\left\{ \tilde{\tau}^{\epsilon}<\mathbf{T}\left(\epsilon,\rho\right)<\tilde{\tau}^{\alpha}\right\}  & \leq & \limsup_{\delta\to0}\mathbb{P}\left\{ \tilde{\tau}_{\delta}^{\epsilon}<\mathbf{T}_{\delta}\left(\epsilon,\rho\right)<\tilde{\tau}_{\delta}^{\alpha}\right\} \\
 & \leq & \limsup_{\delta\to0}\mathbb{P}\left\{ \tilde{\tau}_{\delta}^{\epsilon}<\mathbf{T}_{\delta}\left(\epsilon,\rho\right)\right\} .
\end{eqnarray*}
But $\mathbb{P}\left\{ \tilde{\tau}_{\delta}^{\epsilon}<\mathbf{T}_{\delta}\left(\epsilon,\rho\right)\right\} \to0$
uniformly as $\epsilon\to0$ by the first part of the lemma, so we
obtain the desired result.

\begin{figure}
\includegraphics[width=8cm]{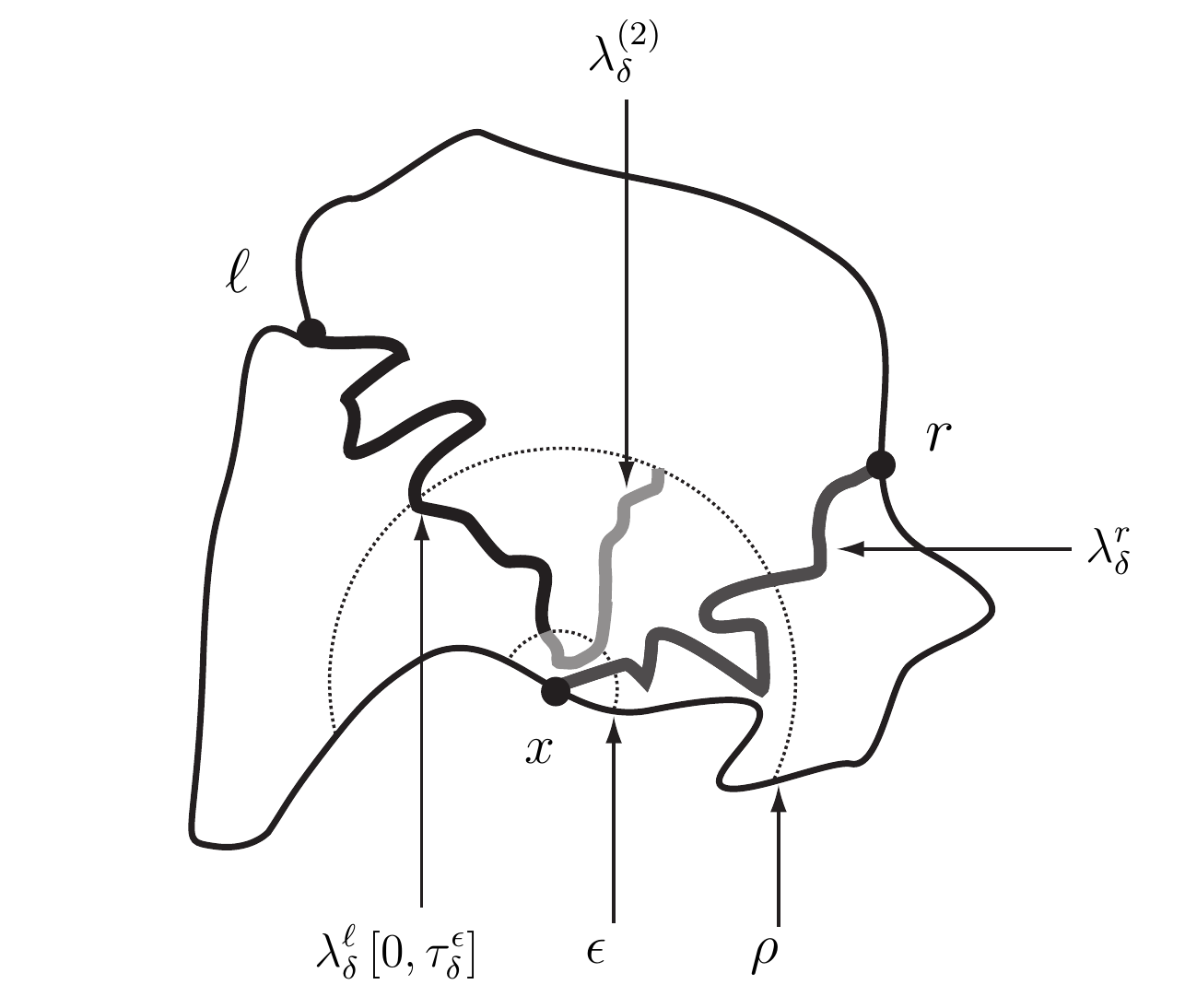}

\caption{\label{fig:fk-left-exit}Event $\mathfrak{L}$: exit on the left of
$\lambda_{\delta}^{\ell}\left[0,\tilde{\tau}_{\delta}^{\epsilon}\right]$}
\end{figure}

\begin{figure}

\includegraphics[width=8cm]{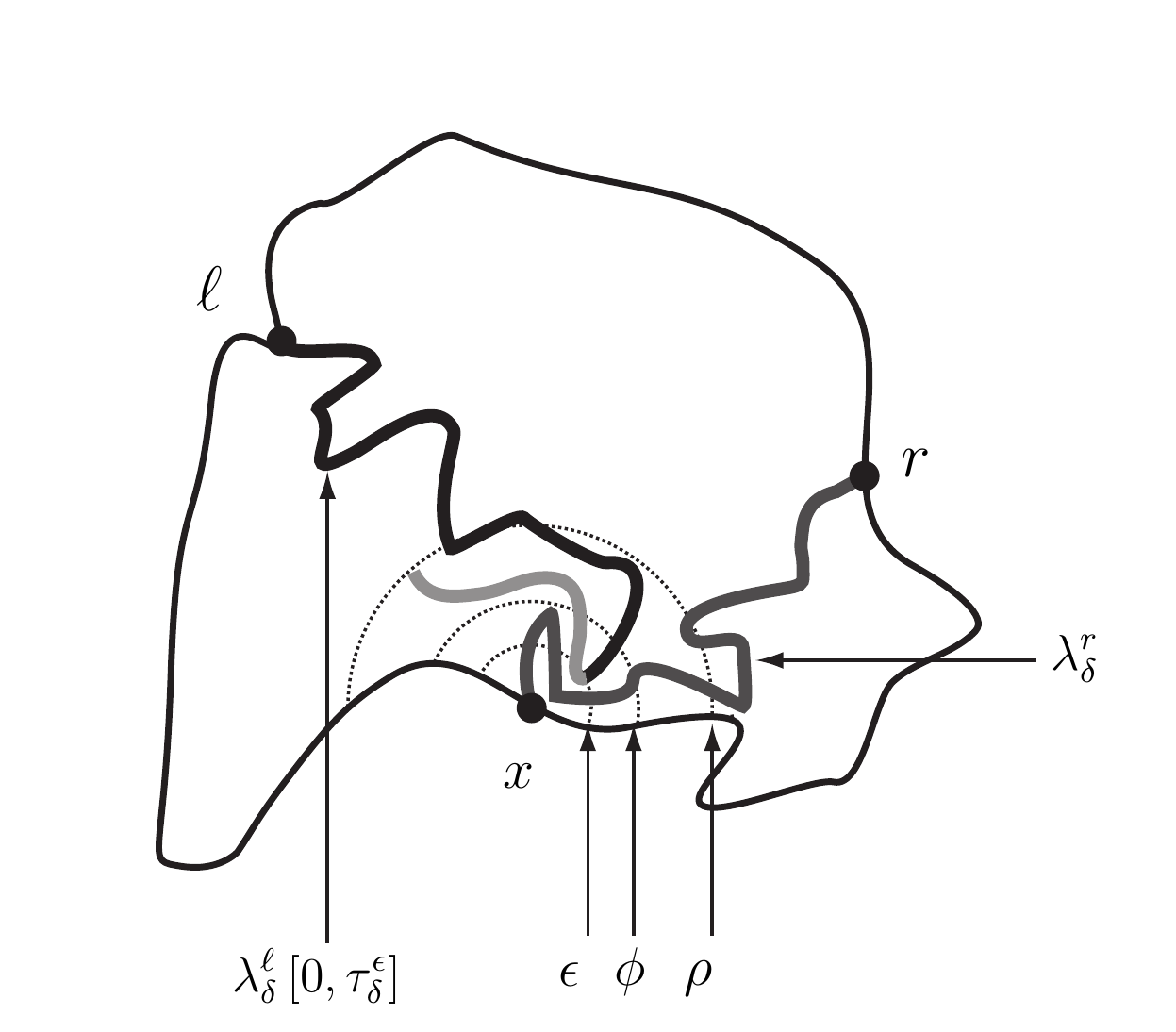}

\caption{\label{fig:fk-right-exit}Event $\mathfrak{R}$: exit on the right
of $\lambda_{\delta}^{\ell}\left[0,\tilde{\tau}_{\delta}^{\epsilon}\right]$}

\end{figure}
\end{proof}

\end{document}